\documentclass[11pt,twoside]{report} 
\usepackage[margin=2.4cm]{geometry}
\usepackage{amssymb,multicol}
\usepackage{amsmath,amsthm}
 \usepackage{latexsym}
\usepackage{amsmath}
\usepackage{graphicx}
\usepackage[all]{xy}  
\newtheorem{theorem}{Theorem}[chapter]
\newtheorem{proposition}[theorem]{Proposition}
\newtheorem{lemma}[theorem]{Lemma}
\newtheorem{corollary}[theorem]{Corollary}
 \theoremstyle{definition}
\newtheorem{definition}[theorem]{Definition}
\newtheorem{example}[theorem]{Example}
 \newcommand{\norm}{\left\Vert\,\cdot\,\right\Vert}
\newcommand{\Norm}{\left\vert\left\vert\left\vert\,\cdot\,\right\vert\right\vert\right\vert}
\newcommand {\C}{\mathbb C}
\newcommand {\T}{\mathbb T}

\newcommand {\R}{\mathbb R}
\newcommand {\N}{\mathbb N}
\newcommand {\Z}{\mathbb Z}

\newcommand {\K}{\mathbb K}
\newcommand {\M}{\mathbb M}
\newcommand {\I}{\mathbb I}
\newcommand {\D}{\mathbb D}
\newcommand {\B}{\mathcal B}
\newcommand {\s}{\smallskip}
\newcommand{\lv}{\left\vert}
\newcommand{\rv}{\right\vert}
\newcommand{\lV}{\left\Vert}
\newcommand{\rV}{\right\Vert}
\newcommand{\LV}{\left\vert\left\vert\left\vert}
\newcommand{\RV}{\right\vert\right\vert\right\vert}
\newcommand{\dd}{\,{\rm d}}
\newcommand{\mn}{multi-normed}
\newcommand{\ex}{\mathop{\rm ex}}
\newcommand{\WK}{\widetilde{K}}
\newcommand{\TLS}{topological linear space}

\newcommand{\widecheck}[1]{\stackrel{\textrm{\footnotesize \reflectbox{\rotatebox[origin=c]{180}{$\widehat{\phantom{#1}}$}}}}{#1}}
\newcommand{\projectivetensor}{\,\widehat{\otimes}\,}
\newcommand{\injectivetensor}{\,\widecheck{\otimes}\,}

\def\lin{\mathop{\rm lin\,}}

\def\ker{\mathop{\rm ker\,}}
\def\sup{\mathop{\rm sup\,}}
\def\Sup{\mathop{\rm Sup\,}}
\def\supp{\mathop{\rm supp\,}}

\def\Lim{\mathop{\rm Lim\,}}
\def\olim{\mathop{\rm o\!-\!lim\,}}

\begin{document}

\title{Multi-normed spaces}
\author{H.\ G.\ Dales, M.\ E.\ Polyakov}
\maketitle

\setcounter{page}{1}
\noindent
H. G. Dales\\
Department of  Mathematics and Statistics\\
Fylde College\\
University of Lancaster\\
Lancaster LA1 4YF\\
United Kingdom\\
E-mail: g.dales@lancaster.ac.uk

\medskip

\noindent M.\ E.\ Polyakov\\
(deceased)

\tableofcontents


\newpage\bigskip

\begin{centering}

{\Large\textbf{Abstract}}

\end{centering}

\bigskip

\begin{quote}
We modify the very well known theory of normed spaces $(E, \norm)$ within functional analysis by  considering a sequence 
$(\norm_n : n\in\N)$ of norms, where $\norm_n$ is defined on the product space $E^n$ for each $n\in\N$. 

Our theory is analogous to, but distinct from, an existing theory of `operator spaces'; it is designed to relate to general spaces $L^p$ for 
$p\in [1,\infty]$, and in particular to  $L^1$-spaces, rather than to $L^2$-spaces.

 After recalling in Chapter 1 some results in functional analysis, especially in Banach space, Hilbert space, Banach algebra, and Banach lattice theory, that we shall use, 
 we shall present in Chapter 2 our axiomatic definition of a `multi-normed space' $((E^n, \norm_n) : n\in \N)$, where $(E, \norm)$ is a normed space. 
Several different, equivalent, characterizations of multi-normed spaces are given, some involving the theory of tensor products; 
key examples of multi-norms are the minimum,   maximum, and $(p,q)$- multi-norms based on a given space.  Multi-norms measure `geometrical features' of normed spaces,
 in particular by considering their `rate of growth'.  There is a strong connection  between multi-normed spaces and the theory of absolutely summing operators.

A substantial number of examples of multi-norms will be  presented.

Following the pattern of standard presentations of the foundations of functional analysis, we consider  generalizations to `multi-topological linear spaces' 
through `multi-null sequences', and to `multi-bounded' linear operators, which are exactly the `multi-continuous' operators. 
We define a new  Banach space ${\mathcal M}(E,F)$ of multi-bounded operators,  and show that it generalizes well-known spaces, especially 
in the theory of Banach lattices.

We conclude with a theory  of `orthogonal decompositions'  of a normed space with respect to a multi-norm, and apply this to construct  
a `multi-dual' space.

Applications of this theory will be presented elsewhere.

\noindent\emph{2000 Mathematics Subject Classification:}
Primary 43A10, 43A20; secondary  46J10.

\noindent\emph{Key words and phrases:}
Banach space, tensor products, Banach algebra, Banach lattice, $AL_p$-space, $AM$-space, positive operator, regular operator, Dedekind complete, Riesz space, Nak\-ano property, multi-norm, multi-Banach space, dual multi-norm, maximum multi-norm, minimum multi-norm,  matrices, tensor norms, condition (P), summing norms, weak $p$-summing norm,  $(p,q)-$multi-norm, standard $q$-multi-norm, summing constant, multi-topological linear space, multi-null sequence, multi-bounded set, multi-bounded operator, multi-continuous operator,  extensions of multi-norms, hermitian decomposition, small decomposition, orthogonal decomposition,  multi-dual space, multi-reflexive.
\end{quote}

\newpage
\chapter{Introduction}

\noindent In this introductory chapter, we shall recall some background that we shall require, and establish our notation; many of the results are well known. 
We shall conclude the chapter with  a summary, with some history of our project, and with some acknowledgements.\medskip

\section{Basic notation}

\noindent  We begin by recalling some standard notation that will be fixed  throughout this memoir.\s 

\subsection{Sets and sequences}  
We write $\N$,  $\Z$, and  $\Z^+ $ for the three sets $\{1,2,\dots\}$ of natural numbers,   $\{0, \pm 1,\pm 2,\dots\}$ of integers,
 and $\{0,  1, 2,\dots\}$ of  positive integers, res\-pectively.   
For each $n \in \N$, we denote by $\N_n$  and $\Z^+_n$ the sets $\{1,\dots, n\}$ and $\{0,1,\dots, n\}$, respect\-ively.
 Also,  we denote by  ${\mathfrak S}_n$   the group of perm\-utations on 
$n$ symbols; we  write ${\mathfrak S}_{\N}$ for the group of all permutations of $\N$.  

 The real field is $\R$, and $\R^+ =[0,\infty)$; the {\it unit interval\/} $[0,1]$  in $\R$ is
 denoted by $\I$.
The complex field is $\C$;  the {\it open unit disc\/}  in $\C$ is always denoted by 
$\D = \{z \in \C : \lv z \rv< 1\}$,
 and its  closure is $\overline{\D} = \{z \in \C : \lv z \rv\leq 1\}$,  the {\it closed unit disc\/}. 
We write $[x]$  for the integer part of $x\in \R^+$.

For $i\in\N_n$, the  $i ^{\rm th}$-{\it coordinate functional\/} on $\C^n$ or $\R^n$ is denoted by $Z_i$, so that 
  $$
  Z_i : (z_1,\dots, z_n)\mapsto z_i\,,\quad\C^n\to \C\,.
  $$
  
  The {\it cardinality\/} of set $S$ is denoted by 
$\lv S \rv$, and the {\it symmetric difference\/} of two sets $S$
 and $T$ is $S\Delta T$.\s

 The space of all complex-valued sequences on $\N$ is $\C^{\N}$, and we often write $(\alpha_i)$ for
 $\alpha= (\alpha_i  : i\in\N) \in \C^{\N}$.  Let 
$\alpha, \beta \in  \C^{\N}$. Then: \s

 $\alpha = O(\beta)$ if there is a constant $K$ with $\lv \alpha_i\rv \leq K\lv\beta_i\rv\,\;(i\in \N)$; \s
 
  $\alpha = o(\beta)$ if, for each $\varepsilon >0$, there exists $i_0\in \N$ with 
$\lv \alpha_i\rv \leq \varepsilon \lv\beta_i\rv\,\;(i\geq i_0)$; \s
 
  $\alpha \sim  \beta$ if $\alpha = O(\beta)$  and $\beta = O(\alpha)$,  in which case  $\alpha$ and $  \beta$
  are  said to be {\it similar} sequences.

\medskip

\subsection{Inequalities} We shall use various inequalities; for an attractive discussion of many inequalities in related areas, see \cite{Ga}.

Take $p$ with $1<p < \infty$.  Then the {\it conjugate index\/} to $p$ is $q$, where
$$\frac{1}{p} +\frac{1}{q} = 1\,;
$$
we also regard  $1$ and $\infty$ as being conjugates of each other; later we shall sometimes  denote the conjugate of $p$ by $p{\,}'$. 
We shall interpret $(\alpha_1^q +\cdots+ \alpha_n^q)^{1/q}$, where $\alpha_1,\dots,\alpha_n \in \R^+$, as
 $\max\{\alpha_1,\dots,\alpha_n\}$ when $q=\infty$.

 First, an easy form of {\it H{\"o}lder's inequality\/} gives the following.
Let  $p,q \in [1,\infty]$ be conjugate indices.  Then, for each   $n\in \N$ and  each $x_1,\dots,x_n, y_1, \dots, y_n\in\C$, we have
\begin{equation}\label{(3.2da)}
\sum_{j=1}^n\lv x_jy_j\rv \leq \left(\sum_{j=1}^n \lv x_j\rv^p\right)^{1/p}\,\left(\sum_{j=1}^n \lv y_j\rv^q\right)^{1/q}\,.
\end{equation}
Now take $a_1,\dots,a_n \in \R^+$ and $r,s$ with $1\leq r \leq s $. Then  (in  the  case where $r<s$) we apply (\ref{(3.2da)}) with $x_j= a_j^r$ and
 $y_j =1$  for $j\in\N_n$ and with $p= s/r$ and $q = s/(s-r)$ to see that 
\begin{equation}\label{(3.2d)}
\frac{1}{n^{1/r}}(a_1^r + \cdots +a_n^r)^{1/r} \leq \frac{1}{n^{1/s}}(a_1^s + \cdots +a_n^s)^{1/s}\,.
\end{equation}\s

 For $k\in \N$ with $k\geq 2$, set $\zeta =\exp(2\pi{\rm i}/k)$, so that $1 + \zeta^t + \cdots + \zeta^{t(k-1)}=0$ for $\pm t\in \N_{k-1}$,
 and then take $\zeta_1,\dots,\zeta_k\in\C$ and set
$$
z_i  = \sum_{j=1}^k \zeta_j\zeta^{ij}\quad (i\in\N_k)\,.
$$

 \begin{lemma}\label{4.1af}
Let  $k\in \N$, and let $q\in [1,2]$. \s

{\rm (i)} Take  $\zeta_1,\dots,\zeta_k\in\C$ with $\sum_{i=1}^k\lv \zeta_i\rv^2 = 1$. Then 
$\sum_{i=1}^k\lv z_i\rv^2 =k$ and $$\left(\sum_{i=1}^k\lv z_i\rv^q \right)^{1/q} \leq k^{1/q}\,.$$\s

{\rm (ii)} Take  $\zeta_1,\dots,\zeta_k\in\T$.  Then $\sum_{i=1}^k\lv z_i\rv^2 =k^2$  and 
$$\left(\sum_{i=1}^k\lv z_i\rv^q\right)^{1/q}  \leq k^{1/2+1/q}\,.$$ 
 \end{lemma}

 \begin{proof}   For $r,s \in \N_k$ with $r\neq s$, the coefficient of $\zeta_r\overline{\zeta}_s$ in the expansion of $\sum_{i=1}^k  z_i\overline{z}_i$ is
  $\sum_{i=1}^k\zeta^{it}$, where $t=r-s$, so that $ \lv t\rv \in \N_{k-1}$. Hence this coefficient is $0$.
For $r\in \N_k$, the coefficient of  $\zeta_r\overline{\zeta}_r$ in the expansion is $k$, so that 
$
\sum_{i=1}^k\lv z_i\rv^2 =k\sum_{i=1}^k\lv \zeta_i\rv^2\,, 
$
  and this is $k$ in case (i) and  $k^2$ in case (ii), giving the equalities in the two results.   The subsequent inequalities  follow from (\ref{(3.2d)}).
 \end{proof}
\medskip

\subsection{Linear spaces} Let $E$ be a linear space over  the real or complex field.  In fact, we shall usually implicitly assume that 
$E$ is taken over the complex field $\C$; 
small modifications  usually give the same result for spaces over the real field $\R$, but at a few points it will be  important to specify the
 underlying field.    Note that  a linear space $E$ over $\C$ can be regarded as a linear space over $\R$ by restricting the scalars to $\R$; 
we obtain the {\it underlying real-linear space\/}.

A real-linear space $V$ has a standard {\it complexification\/} of the form  $E= V \oplus {\rm i}V$,
where $(\alpha + {\rm i}\beta)(x+ {\rm i}y) = \alpha x -\beta y + {\rm i}(\beta x + \alpha y)$ for $\alpha, \beta \in \R$ and
 $x,y \in V$,  so that $E$ is a complex linear space;  we set $E_{\R} = V$.

 The {\it dimension\/} of $E$ over the underlying field and the linear subspace spanned by a subset $S$ of $E$ are denoted by
$$
\dim E\quad \;{\rm and}\; \quad \lin S\,,
$$ respectively.

Let $F$ and $G$ be linear subspaces of a linear space $E$. Then we set 
$$
F+G =\{x+y: x\in F,\,y\in G\}\,,
$$
so that $F+G$ is a linear subspace of $E$; further, we write  $E = F\oplus G$ if  $F\cap G = \{0\}$ and $F+G = E$.  More generally,
let $E_1,\dots,E_n$ be linear subspaces of $E$ such that $E_1+\cdots + E_n=E$ and $E_i\cap E_j= \{0\}$ whenever $i,j\in\N_n$ with $i\neq j$. Then we write 
$$
E = E_1 \oplus\cdots \oplus E_n\,;
$$
this is a {\it direct sum decomposition\/} of $E$.  In this case, each $x\in E$ has a unique expression as $x=x_1+\cdots + x_n$,
 where $x_i\in E_i\,\,(i\in\N_n)$.  Two direct sum decompositions $E_1 \oplus\cdots \oplus E_m$  and $F_1 \oplus\cdots \oplus F_n$ of $E$ are {\it equal\/} 
if $n=m$ and $F_i=E_i\,\;(i\in\N_m)$.

 Let $E$ be a linear space. For $x,y \in E$, define
 $$
 [x,y] = \{ tx + (1-t)y : t \in \I\}\,.
$$
A non-empty subset $K$ of a linear space $E$ is {\it convex\/} if $[x,y]\subset K$ whenever
$x,y\in K$. The {\it convex hull\/}  of a non-empty subset $S$ of $E$ is the intersection of the convex subsets of $E$
 that contain $S$; it is denoted by ${\rm co\/}(S)$, so that 
$$
{\rm co\/}(S) = \left\{ t_1x_1 + \cdots + t_nx_n : t_1, \dots,t_n \in \I,\, \sum_{i=1}^nt_i =1,\, x_1, \dots,x_n \in S\right\}\,.
$$
The set of {\it extreme points\/} of a convex subset $K$ of $E$ is denoted by $\ex K$,
so that, for $x\in K$, we have $x \in  \ex K$ if and only if $K\setminus\{x\}$ is convex.

Now suppose that $E$ is a complex linear space. For   $\alpha \in \C$ and a subset $S$ of $E$,  we write 
 $\alpha S =\{\alpha x : x \in S\}$;  $S$ is {\it absorbing\/} if
$$
\bigcup\{\alpha S : \alpha >0\}= E\,,
$$
 {\it balanced\/} if $\alpha S \subset S\,\;(\alpha \in \overline{\D})$,  and {\it absolutely convex\/}
 if $S$ is convex and balanced. Equivalently, $S$ is absolutely convex if $\alpha x + \beta y \in S$ 
 whenever $x,y \in S$ and $\alpha, \beta \in \C$ with $\lv \alpha \rv +\lv \beta \rv \leq 1$. 
 The {\it absolutely convex hull\/} of a non-empty subset $S$ of $E$ is the 
intersection of the absolutely convex subsets of $E$
 that contain $S$; it is denoted by ${\rm aco\/}(S)$, so that 
$$
{\rm aco\/}(S) = \left\{ \alpha_1x_1 + \cdots + \alpha_nx_n :  
\sum_{i=1}^n\lv \alpha_i\rv \leq 1,\, x_1, \dots,x_n \in S\right\}\,, 
$$
where $\alpha_1, \dots,\alpha_n \in \C$. In the case where $S$ is balanced,   ${\rm aco\/}(S) = {\rm co\/}(S)$.

Let $K$ be an absolutely  convex,  absorbing subset of the space $E$. Then the {\it Minkowski  functional\/} $p_K$ of $K$, defined by
$$
p_K(x) = \inf\{\alpha >0 : x \in \alpha K\}\quad (x\in E)\,,
$$
is a seminorm on $E$; $p_K$ is a norm if and only if
$$
\bigcap\{(1/n)K : n\in \N\} = \{0\}\,.
 $$
  Of course, we have
$$
\{x\in E : p_K(x) < 1\} \subset K \subset \{x\in E : p_K(x) \leq 1\}\,.
$$

Let $S$ be a non-empty set. The linear spaces  of all functions from $S$ to $\C$ and $\R$ are denoted by
 $\C^S$ and $\R^S$, respectively; $\C^S$ and $\R^S$ are    complex and real  algebras, respectively, for the pointwise operations. 
There is an obvious ordering on the space $\R^S$: for each $f,g \in \R^S$, we set $f\leq g$ if $f(s)\leq g(s)\;\,(s\in S)$, so that $(\R^S, \leq)$
 is a partially ordered linear space. Indeed, $fg\geq 0$ whenever $f,g\geq 0$ in $\R^S$, and so $(R^S, \leq)$
 is a partially ordered algebra. For a subset $F$ of $\R^S$, we set 
$$
F^+=\{f\in  F : f\geq 0\}\,.
$$
The functions $\lv f \rv$ and  $\exp f$, etc., for functions
$f,g \in \C^S$, and   $f\vee g$ and  $f\wedge g$ for functions $f,g \in \R^S$,  are defined pointwise.  For example,
$$
(f \vee g)(s) =\max\{f(s), g(s)\}\,,\quad  (f \wedge g)(s) =\min\{f(s), g(s)\}\quad (s\in S)\,.
$$
We then define the functions \label{positive}$f^+= f\vee 0$, $f^-= (-f)\vee 0$, and 
$$
\lv f\rv = f^+ + f^- = f \vee (-f)\,,
$$
so that $f =f^+-f^-$ and $f^+f^-=0$.\s

 Let $E$ be a linear space, and take $n \in \N$. Then we denote by $E^n$ the linear space
$$
\stackrel{n}{\overbrace{E\times \cdots \times E}}\,,
$$
 where there are $n$ copies of the space $E$. Thus $E^n$ consists of $n$-tuples $(x_1,\dots,x_n)$,
where $x_1,\dots,x_n \in E$. As a matter of notational convenience, we regard the generic element 
$(x_1, \dots,x_{k-1}, y_1, \dots, y_m)$ for $k,m\in\N$ as $(y_1, \dots, y_m)$ in the special case where $k=1$, and we write $x$, rather 
than $(x)$, in the case where $n=1$. The linear operations on $E^n$ are defined coordinatewise.
 The zero element of either $E$ or $E^n$ is denoted by $0$.  When we write 
$$
(0,\dots,0,x_i,0,\dots,0)
$$
 for an element in $E^n$, we understand that $x_i$ appears in  the $i^{\rm th.}$ coordinate, unless we say other\-wise. An element $x$
 of $E^n$ is often written as either $(x_1,\dots,x_n)$ or $(x_i)$. For each $x\in E$, the 
{\it constant sequence with value $x$\/} is the sequence $(x)=  (x,\dots,x)\in E^n$.

\begin{definition} \label{1.1aa}
Let $E$ be a linear space.

Take $n\in\N$ and $k\in \N_n$,  and let $ (x_1,\dots, x_n ) \in E^n$. Then an element 
$(y_1,\dots, y_k)\in E^k$ is a {\it coagulation\/} of $(x_1,\dots, x_n)$ \label{coagulation} if there is
 a partition $\{S_j : j\in \N_k\}$ of $\N_n$ such that $y_j = \sum\{x_i : i \in S_j\}$ for each $j\in \N_k$.

Let   $n,k \in \N$, and   take  $x = (x_1,\dots,x_k) \in E^k$.   Then 
$$
x^{[n]} = (x_1, \dots, x_k, x_1, \dots, x_k, \, \dots\,,x_1, \dots, x_k) \in E^{nk}\,,
$$
 where there are $n$ copies of each block $(x_1,\dots,x_k)$;  $x^{[n]}$ is the $n^{\rm th.}$-{\it ampli\-fication\/} of $x$.
 \end{definition}\s

Let $E$ be a linear space, and consider the space $E^{\,\N}$, which is also a linear space.\label{prod}
 A generic element of $E^{\,\N}$ is often written as
$$
x = (x_i) = (x_i : i\in \N)\,;
$$
 the {\it zero} element of $E^{\,\N}$ is $0= (0,0,0,\dots)$, and, for $x\in E$, the `constant sequence with value $x$' is again $(x)$.  
Define $\iota : x \mapsto (x), \;\, E \to E^{\,\N}$, so that $\iota(E)$ is a linear subspace of $E^{\,\N}$.\medskip

\subsection{Linear operators and matrices} \label{Linear operators} Let $E$ and $F$ be linear spaces. Then the linear  space of all linear
 operators from  $E$ to $F$ is denoted by  ${\mathcal L}(E,F)\/$; we set ${\mathcal L}(E)= {\mathcal L}(E,E)\/$.  The identity operator on
 $E$ is denoted by $I_E$.  Thus ${\mathcal L}(E)$ is a unital algebra with respect to the composition of operators.
 
 Now let $V$ and $W$ be real-linear spaces, and let $T$ be a real-linear map from $V$ to $W$. Set $E=V\/\oplus\/{\rm i}V$ and $F = W\/\oplus\/{\rm i}W$.
  The {\it complexification\/}   $T_{\C}$ of $T$ is defined by 
$$
T_{\C}(x+{\rm i}y) = Tx +{\rm i}Ty\quad (x,y \in V)\,,
 $$
 so that $T_{\C}$ is a complex-linear map from $E$ to $F$. 
 
 Let $E$ be a linear space, and take  $m,n\in\N$. Then we denote by ${\mathbb M}_{\,m,n}(E)$
the linear space of all $m\times n $ matrices with coefficients in $E$; also, we write  ${\mathbb M}_{\,n}(E)$ for  ${\mathbb M}_{\,n,n}(E)$.
We  write ${\mathbb M}_{\,m,n}$  and  ${\mathbb M}_{\,n}$  for ${\mathbb M}_{\,m,n}(\C)$  and ${\mathbb M}_{\,n}(\C)$, respectively.
Let  $v \in  {\mathbb M}_{m}(E)$  and  $w \in {\mathbb M}_{n}(E)$.   Then $v\oplus w$ is the matrix in ${\mathbb M}_{\,m+n}(E)$ of the form
$$
\left[
\begin{array}{cc}
v & 0\\
0& w
\end{array}
\right]\,.
$$
  Let $x =(x_{ij}) \in  {\mathbb M}_{\,m,n}(E)$. Then the {\it transpose\/}   of $x$  is the matrix
$$
x^t = (x_{ji}) \in {\mathbb M}_{\,n,m}(E)\,.
$$ 

Let $E$  be a linear space, and take  $m,n \in \N$. Then each element $a \in  {\mathbb M}_{\,m,n}$  defines an element of
 ${\mathcal L}(E^n,E^m)\/$ by matrix multiplication.

Let $E_1,\dots,E_n$ and $F$  be linear spaces.  Then the linear space of $n$-linear maps from $E_1\times \cdots \times E_n$ to $F$ is denoted by
${\mathcal L}^n(E_1,\dots,E_n; F)$. 

Let $E$ be a linear space, take $n\in\N$, and let  $S$ be a subset of $\N_n$.  For  $x = (x_i) \in E^n$, we set 
\begin{eqnarray*}
P_S(x) &=& (y_i),\quad {\rm where} \quad y_i = x_i \quad   (i\in S)\quad {\rm and } \quad y_i = 0 \quad (i\not \in S)\,,\\
Q_S(x) &=& (y_i),\quad {\rm where} \quad y_i = x_i \quad  (i\not\in S)\quad {\rm and } \quad y_i = 0 \quad (i \in S)\,.
\end{eqnarray*}
Thus $P_S$ is the {\it projection onto\/} $S$  and  $Q_S$ is the
 {\it projection onto the complement of\/} $S$. Clearly $P_S$ and $Q_S$ are idempotents in the algebra ${\mathcal L}(E^n)$, and 
$P_S +Q_S =I_{E^n}$.  Also, for $i\in \N_n$, we set  
$$
\left. \begin{array}{rcl}P_i(x)&= &(0,\dots,0,x_i,0,\dots,0)\,,\\
Q_i(x)&= &(x_1,\dots, x_{i-1},0, x_{i+1}, \dots, x_n)\end{array}\right\}\quad (x = (x_1,\dots, x_n)\in E^n)\,,
$$
so that $P_i = P_{\{i\}}$ and $Q_i = Q_{\{i\}}$.

We conclude this section by defining more formally some operators that will be important for us.\s

\begin{definition} \label{1.1a}
Let $E$ be a linear space, and take   $n \in \N$. For $\sigma \in {\mathfrak S}_n$, define  
$$
A_\sigma (x) = (x_{\sigma(1)}, \dots , x_{\sigma(n)}) \quad (x = (x_1,\dots, x_n)\in E^n)\,.
$$
 For $\alpha = (\alpha_i) \in \C^n$, define  
$$
M_\alpha (x) = (\alpha_i x_i)\quad (x = (x_1,\dots, x_n)\in E^n)\,.
$$

Let $E$ and $F$ be   linear spaces, and let $T \in {\mathcal L}(E,F)$. For $n\in \N$,   define
\begin{equation}\label{(1.6a)}
  T^{(n)}: (x_1,\dots,x_n)\mapsto(Tx_1,\dots,Tx_n)\,,\quad E^n\to F^n\, ;
\end{equation}
$T^{(n)}$  is the  $n^{\rm th.}$-{\rm amplification\/}  of $T$.
\end{definition}\s
 
Thus we see that $A_\sigma \in {\mathcal L} (E^n)$ for each $\sigma \in {\mathfrak S}_n$,  that  $M_\alpha  \in {\mathcal L} (E^n)$ for each $\alpha \in \C^n$, and  that 
  $T^{(n)}\in {\mathcal L}(E^n,F^n)$. 
\medskip

\section{Banach spaces and Banach algebras}

\noindent  We recall some basic facts about Banach spaces and algebras that we shall use.\medskip

\subsection{Banach spaces and operators}\label{Banach spaces and operators}
  For attractive introductions to Banach space theory,  
see \cite{AK,GRA,  Meg, W}, for example;  standard and beautiful classical texts on functional analysis are \cite{DS} and \cite{Ru2}.  
Most of the results on these topics that we shall use are summarized in \cite[Appendix A.3\/]{D}.

Suppose that $(E, \norm )$ is a normed space (over  a scalar field $\K$, always taken to be $\R$ or $\C$). 
We denote by $E_{[r]}$  the closed ball in $E$ with centre $0$ and radius $r\geq 0$.  We recall that  each $E_{[r]}$ 
 is an absolutely convex, absorbing, and closed neighbourhood of $0$. We also denote by $S_E$ the {\it unit sphere}  of $E$, so that 
$$
S_E = \{ x \in E : \lV x \rV =1\}\,.
$$

We shall later consider {\it direct sum decompositions\/}  of a Banach space $E$, say
$$E = E_1 \oplus\cdots \oplus E_n
\,.
$$
 In this situation, we shall always suppose that each of the linear subspaces $E_1, \dots, E_n$ is closed in $E$.  

A sequence $(x_n : n\in\N)$ in a normed space $E$ is a {\it null sequence\/}  if  $$\lim_{n\to \infty}x_n =0\,; 
$$  the
  subspace of $E^{\,\N}$ consisting of all null sequences in $E$ is denoted by  $c_{\,0}(E)$.

The dual space of a normed space $(E, \norm)$ is denoted by $E'$;  the action of $\lambda \in E'$ on $x \in E$ gives the
   number $\langle x,\,\lambda\rangle$. We  shall sometimes denote the {\it dual norm\/}  on $E'$ by $\norm'$.  The second dual space of $E$ is denoted by 
$E''$, and the action of $\Phi\in E''$ on $\lambda \in E'$ gives $\langle \Phi,\,\lambda\rangle$ in our notation; we shall sometimes denote the dual 
norm on $E''$ by $\norm''$. The {\it canonical embedding\/} $\iota :E\to E''$ is defined by  the equation
$$
\langle \iota (x),\,\lambda\rangle = \langle x,\,\lambda\rangle\quad (x\in E,\,\lambda \in E')\,,
$$
so that $\iota$ is an isometry; the space $E$ is {\it reflexive\/} if $\iota$ is a surjection.
 In fact,  we shall usually identify $x$ with $\iota (x)$ and sometimes write $\norm$ for the second dual norm on $E''$.
 
    The weak topology on $E$ is denoted by $\sigma(E,E')$, the weak-$*$ topology on $E'$ is  $\sigma(E',E)$, and the weak-$*$ topology on $E''$ is   
$\sigma(E'',E')$, so that $(E',\sigma(E',E))$ is a locally convex space whose dual space is $E$.
 Of course, by  {\it Goldstein's  theorem\/},
$E_{[1]}$ is  $\sigma(E'',E')$-dense in $E''_{[1]}$, and, by the {\it Banach--Alaoglu theorem\/}, 
$E'_{[1]}$ is $\sigma(E',E)$-compact.

For a subset $X\subset E$, we define its {\it annihilator\/}  $X^{\circ}$ to be
$$
X^{\circ}  = \{\lambda\in E: \langle x,\,\lambda\rangle = 0\;\;(x\in X)\}\,.
$$
Evidently $X^{\circ}$ is a $\sigma(E',E)$-closed linear subspace of $E'$.

A form of the {\it Hahn--Banach separation theorem\/}  \cite[Theorem 3.7]{Ru2} is the following.  
Let  $(E,\tau)$ be a locally convex space. Suppose that $S$ 
is a closed,  absolutely convex subset of $E$ and that $x_0\in E\setminus S$. Then there exists $\lambda \in (E, \tau)'$ such that
 $\langle x_0,\, \lambda\rangle >1$ and  $\lv \langle x,\, \lambda\rangle \rv \leq 1\,\;(x \in S)$.
 
  Let $E$ and $F$ be normed spaces. We denote by ${\mathcal B}(E,F)$ the normed space (with respect to the operator norm) of bounded linear operators from
 $E$ to $F$; ${\mathcal B}(E,F)$   is a Banach space whenever $F$ is a Banach space. Let $T \in   {\mathcal B}(E,F)$. Then we denote
 the operator norm by $\lV T \rV$ or, occasionally, by
$$
\lV T: E\to F \rV\,.
$$
 We set ${\mathcal B}(E) = {\mathcal B} (E,E)$, so that  ${\mathcal B}(E) $ is a unital normed algebra.  A map  $T \in {\mathcal B}(E,F)$ is 
 an {\it isometry\/}  if $\lV T x\rV  = \lV   x\rV\,\;(x\in E)$; $T$ is a {\it contraction\/}   if 
$\lV T x\rV  \leq \lV   x\rV\,\;(x\in E)$; $T$ is an {\it isometric isomorphism\/}  if 
$T$ is a bijection and $T$ and $T^{-1}$ are isometries.

Let $E$ and $F$ be two Banach  spaces. The space $E$ is {\it linearly homeo\-morphic\/}, or {\it isomorphic\/},  {\it to\/}
 $F$ if there exists a bijection $T \in {\mathcal B}(E,F)$ (so that we have
 $T^{-1} \in {\mathcal B}(F,E)$); such a map $T$ is a {\it linear homeomorphism\/} or an
{\it isomorphism}. In this case, we write 
$$
E\sim F\,;
$$
 the {\it Banach--Mazur distance\/} from $E$ to $F$ is
$$
d(E,F) = \inf\{ \lV T \rV \lV T^{-1}\rV   :  T \in {\mathcal B}(E,F)\;\;\mbox{\rm is an isomorphism}\};
$$
see \cite[Definition 7.4.5]{AK}, for example.  The space $E$ is {\it isometrically isomorphic\/}
 to $F$ if there is an isometric isomorphism $T \in \B(E,F)$, so that $d(E,F)=1$; in this case, we shall write 
$$
E\cong F\,.
$$

   For $\lambda_0 \in E'$ and $y_0\in F$, set
$$
y_0\otimes \lambda_0 :x \mapsto \langle x, \lambda_0\rangle y_0\,,\quad E\to F\,.
$$
Then  $y_0\otimes \lambda_0$ is a {\it rank-one operator\/} in ${\mathcal B}(E,F)$ with  $ \lV y_0\otimes \lambda_0\rV =\lV y_0\rV\lV \lambda_0\rV$, and each
 {\it finite-rank operator\/}  in ${\mathcal B}(E,F)$ is a finite sum of such operators. The linear subspace of ${\mathcal B}(E,F)$ consisting of the
 finite-rank operators is denoted by ${\mathcal F}(E,F)$.  An operator $T\in  {\mathcal B}(E,F)$ is {\it nuclear\/}  if 
it can be expressed in the form $T= \sum_{i=1}^{\infty}y_i\otimes \lambda_i$, where $(\lambda_i)$ is a sequence in $E'$, $(y_i)$ is a sequence in $F$, and
$$
 \sum_{i=1}^{\infty}\lV y_i\rV\lV \lambda_i\rV < \infty\,;
$$
 the {\it nuclear norm\/}  $\nu(T)$ of the operator $T$ is defined to be the infimum of the specified sums
 $ \sum_{i=1}^{\infty}\lV y_i\rV\lV \lambda_i\rV$. In particular,
$$
\nu (y_0\otimes \lambda_0) = \lV y_0\otimes \lambda_0\rV =\lV y_0\rV\lV \lambda_0\rV\quad (\lambda_0\in E',\,y_0\in F)\,.
$$
 The space of   nuclear operators is denoted by ${\mathcal N}(E,F)\/$; 
$({\mathcal N}(E,F), \nu)$ is a Banach space when $E$ and $F$ are   Banach spaces, and ${\mathcal F}(E,F)$  is dense in $({\mathcal N}(E,F), \nu)$.  
 
The closure of the space ${\mathcal F}(E,F)$ in $({\mathcal B}(E,F),\norm)$ forms the closed subspace of {\it approximable operators}.  The spaces of approximable 
and compact operators from $E$ to $F$ are denoted by 
 $$
  {\mathcal A}(E,F)\quad {\rm and}\quad {\mathcal K}(E,F)\,,
 $$
 respectively. In the case where $F=E$, we write  ${\mathcal F}(E)$, ${\mathcal N}(E)$, ${\mathcal A}(E)$, and $ {\mathcal K}(E)$   for 
 ${\mathcal F}(E,E)$, ${\mathcal N}(E,E)$, ${\mathcal A}(E, E)$, and $ {\mathcal K}(E, E)$
respectively; each of these is an ideal in the normed algebra ${\mathcal B}(E)$. 

 For  $T \in \B(E,F)$,  the {\it dual operator\/}  $T'$ of $T$ is defined by the equation
$$
\langle x,\,T' \lambda \rangle =\langle Tx,\,\lambda\rangle \quad(x\in E,\,\lambda \in F')\,;
$$
we have $T' \in \B(F',E')$  and $\lV T'\rV  =\lV T \rV$.   The dual of an isometry is also an isometry.

A closed subspace $F$ of a Banach space $E$ is {\it  complemented\/} if there is a   projection $P\in {\B}(E,F)$ with $P(E) =F$, and 
 {\it $\lambda$-complemented\/} (for $\lambda\geq 1$)  if there is a projection  $P$ of $E$ onto $F$ with $\lV P \rV \leq \lambda$.

We shall sometimes use the following  {\it Principle of Local Reflexivity}, proved  in 
\cite[Theorem 11.2.4]{AK} and \cite[Theorem 5.54]{Ry}, for example.  \smallskip

   \begin{theorem}\label{1.6}
Let $E$ be a Banach space,  let $X$ and $Y$ be   finite-dimen\-sional subspaces of $E''$  and $E'$, respectively,  and take $\varepsilon >0$.
 Then there is an injective, bounded  linear map $S :X \to E$ with the following properties:\s
 
  {\rm (i)}  $Sx =x\,\;(x\in X\cap E)\,$;\s
 
 {\rm (ii)}  $ \langle S(\Lambda),\,\lambda\rangle =   \langle  \Lambda,\,\lambda\rangle\quad (\lambda \in Y,\,\Lambda \in X)$\,;\s
 
  {\rm (iii)} $(1-\varepsilon)\lV \Lambda\rV \leq \lV S(\Lambda)\rV \leq (1+\varepsilon)\lV  \Lambda\rV\;\,(\Lambda \in X)\,$.\s
 \qed  \end{theorem}\s

Let $E_1,\dots,E_n$ and $F$  be normed spaces.  Then the space of bounded $n$-linear maps from $E_1\times \cdots \times E_n$ to $F$ is denoted by
 ${\mathcal B}^n(E_1,\dots,E_n; F)$.   This is a normed space for the norm $\norm$ defined by
 $$
 \lV T \rV =\sup\{\lV T(x_1,\dots, x_n)\rV : x_j \in (E_j)_{[1]},\,\;j\in\N_n\}
 $$
 for $T \in {\mathcal B}^n(E_1,\dots,E_n; F)$, and it is a Banach space whenever $F$ is complete.\medskip

\subsection{Tensor products}  Let $E$ and $F$ be linear  spaces.  Each element of the (algebraic) tensor product $E\otimes F$  has the form 
$\sum_{i=1}^m x_i\otimes y_i$ for some  $m\in \N$, $x_1,\dots, x_m\in E$, 
and $y_1,\dots, y_m\in F$; such a representation is not unique. 

Let $G$ be a third linear space. For each bilinear map \mbox{$T : E\times F \to G$,} 
there is a unique linear map $\widetilde{T} : E\otimes F \to G$ such that   $$\widetilde{T}(x\otimes y) =T(x,y)\quad(x\in E,\,y\in F)\,.
$$
  Let $S \in {\mathcal L}(E)$ and  $T\in {\mathcal L}(F)$. Then there exists a map $S\otimes T \in {\mathcal L}(E\otimes F)$  such that
$$
(S\otimes T)(x\otimes y)= Sx\otimes Ty\quad (x\in E,\,y\in F)\,.
$$

   Now  suppose that  $E$ and $F$ are normed spaces, and that $\norm$ is a norm on the 
linear space  $E\otimes F$. Then $\norm $ is a {\it sub-cross-norm\/} if 
$$
\lV x\otimes y\rV \leq \lV x \rV \lV y\rV\quad(x\in E,\,y\in F)
$$
 and a  {\it  cross-norm\/} if 
$$
\lV x\otimes y\rV = \lV x \rV \lV y\rV\quad(x\in E,\,y\in F)\,.
$$
Further, a sub-cross-norm $\norm$ on $E\otimes F$ is a {\it reasonable cross-norm} if the linear functional 
$\lambda \otimes \mu$ is bounded and $\lV \lambda \otimes \mu\rV \leq \lV \lambda \rV\lV \mu\rV$  for each $\lambda \in E'$ and $\mu\in F'$.
 For these definitions  and the properties stated below,   see \cite[\S VIII,1]{DU} and  \cite[\S6.1]{Ry}.\s

\begin{proposition}\label{1.6a}
Let $E$ and $F$ be  normed spaces. Then each  reasonable cross-norm on $E\otimes F$ is a cross-norm,
 and 
$$\vspace{-\baselineskip}\lV \lambda \otimes \mu \rV= \lV \lambda \rV\lV \mu\rV\quad(\lambda \in E',\,\mu\in F')\,.{}
$$
\hspace*{\stretch{1}}\qed
\end{proposition}

 The {\it projective norm\/}  $\norm_\pi$ on $E \otimes F$ is defined by
$$
\lV z \rV_\pi = \inf\left\{ \sum_{i=1}^m \lV x_i\rV \lV y_i\rV :
 z = \sum_{i=1}^m x_i\otimes y_i \in E\otimes F\right\}\,,
$$
where the infimum is taken over all representations $z = \sum_{i=1}^m x_i\otimes y_i$ of $z\in E\otimes F$; $(E\otimes F, \norm_\pi)$  is then a normed space, 
and its completion
$$
(E\projectivetensor F, \norm_\pi)
$$
 is the {\it projective tensor product} of $E$ and $F$. We note that
\begin{equation}\label{(1.2)}
\lV x\otimes y \rV_\pi = \lV x \rV \lV y\rV\quad( x\in E,\,y\in F)\,,
\end{equation}
so that $\norm_\pi$ is a cross-norm on $E\otimes F$. In fact, $\norm_\pi$ is a reasonable cross-norm, and $\lV z\rV \leq \lV z\rV_\pi\;(z\in E\otimes F)$
 for each reasonable cross-norm $\norm$ on $E\otimes F$.  The key property of this tensor product is the following.\s

\begin{proposition}\label{1.3a}
Let $E$, $F$, and $G$ be three Banach spaces. Then, for each bilinear operator $T \in {\B}(E,F; G)$, there exists a unique linear  operator
$\widetilde{T} \in {\B}(E\projectivetensor F, G)$ such that
$$
\widetilde{T}(x\otimes y) = T(x,y)\quad (x\in E,\,y\in F)\,,
$$
and the  map $T\mapsto \widetilde{T},\,\;{\B}(E,F; G) \to {\B}(E\projectivetensor F, G)$, is an isometric isomorphism.\qed
\end{proposition}\s 

 Let $E$ and $F$ be two Banach  spaces.  For $\mu \in (E\projectivetensor F)'$, define $T_\mu $ by
$$
\langle y,\,T_\mu x\rangle = \langle x\otimes y,\,\mu\rangle\quad (x\in E,\,y\in F)\,.
$$
Then $T_\mu x \in F'\,\;(x\in E)$, $T_\mu \in \B(E,F')$, and the map  $$\mu\mapsto T_\mu\,,\quad(E\projectivetensor F)'\to  \B(E,F')\,,
$$
is an isometric isomorphism, and so
\begin{equation}\label{(1.5a)}
(E\projectivetensor F)'\cong  \B(E,F')\,.
  \end{equation}
  
  Let $E$ and $F$ be  normed spaces over a field $\K$. For $x\in E$ and $y\in F$, set 
$$
T_{x,y} (\lambda, \mu) =\langle x,\,\lambda\rangle\langle y,\,\mu\rangle\quad (\lambda\in E',\,\mu \in F')\,,
$$
so that  $T_{x,y} \in {\B}(E',F'; \K)$;  the map $$(x,y)\mapsto T_{x,y}\,,\quad E\times F \to {\B}(E',F'; \K)\,,
$$
 is bilinear. There is an injective  linear map
$\iota : E\otimes F \to {\B}(E',F'; \K)$ such that $$\iota(x\otimes y) = T_{x,y}\quad ( x\in E,\,y\in F)\,,
$$
 and so we may  regard $E\otimes F$ as a linear subspace of $ {\B}(E',F'; \K)$.  The {\it injective norm\/}  $\norm_\varepsilon$  on $E \otimes F$ 
is the norm  inherited from  $ {\B}(E',F'; \K)$, and so
$$
\lV z \rV_\varepsilon = \sup\left\{\lv \sum_{i=1}^m \langle x_i,\,\lambda\rangle\langle y_i\,,\mu\rangle\rv :\lambda \in E'_{[1]},\, \mu \in F'_{[1]}
\right\}\,,
$$
for any  representation  $z  = \sum_{i=1}^m x_i\otimes y_i$ of $z\in E\otimes F$.  The closure of $E \otimes F$  in $ {\B}(E',F'; \K)$, denoted by
$$
(E{\injectivetensor} F, \norm_\varepsilon)\,,
$$
is the {\it injective tensor product}  of $E$ and $F$.  We note that
\begin{equation}\label{(1.2a)}
\lV x\otimes y \rV_\varepsilon = \lV x \rV \lV y\rV\quad( x\in E,\,y\in F)\,,
\end{equation}
so that $\norm_\varepsilon$ is also  cross-norm on $E\otimes F$. In fact, $\norm_\varepsilon$ is a reasonable cross-norm, and  
$\lV z\rV_\varepsilon \leq \lV z\rV\,\;(z\in E\otimes F)$  for each reasonable cross-norm $\norm$ on $E\otimes F$.

It is shown in \cite[Proposition 6.1]{Ry} that a  norm $\norm $ on $E\otimes F$ is a reasonable cross-norm if and only if
\begin{equation}\label{(1.2b)}
\lV z\rV_\varepsilon \leq \lV z\rV\leq \lV z\rV_\pi\quad (z\in E\otimes F)\,.
 \end{equation}\s

\subsection{Direct sum decompositions}
Let $(E, \norm)$ be a normed space, and suppose that  $E = E_1\oplus \cdots \oplus E_k$  is a direct sum decomposition
  of $E$, where $E_1, \dots, E_k$ are closed subspaces of $E$; we allow the possibility that $E_j=\{0\}$
 for some $j \in\N_k$. We say that the decomposition has  {\it length\/} $k$ in this case. 
Thus each element $x\in E$ has a unique expression as $x = x_1 + \cdots+ x_k$, where $x_j\in E_j\,\;(j\in\N_k)$.  The  decomposition  is {\it  trivial\/}
  if $E = E_j$ for some $j \in \N_k$.  We write $P_j : E \to E_j \,\;(j\in \N_k)$ for the natural projections.

Now suppose that  $E = E_1\oplus \cdots \oplus E_k$  is a Banach space. Then, for each $j\in\N_k$,  the map  $ P_j$ is continuous, and is  
regarded as a member of the Banach space   ${\mathcal B}(E, E_j)$.
It  is not necessarily true that $\lV P_j\rV \leq 1$.  
\smallskip

\begin{definition}\label{10.5}
Let   $(E, \norm)$ be a  normed space, and consider a family  $\mathcal K$  of direct sum decompositions of $E$.
The family $\mathcal K$ is {\it closed} provided that  the following conditions are satisfied for each $k\in\N$:\smallskip

{\rm (C1)} $ E_{\sigma (1)}\oplus \dots\oplus E_{\sigma (k)}  \in  \mathcal K$
 whenever  $ E_{1}\oplus \dots\oplus E_{k}  \in  \mathcal K$,   $\sigma \in {\mathfrak S}_{k}\/$, and $k\in\N$;\smallskip

{\rm (C2)} $ F\oplus E_{3}\oplus \dots\oplus E_{k}  \in  \mathcal K$ whenever $ E_{1}\oplus \dots\oplus E_{k}  \in  \mathcal K$, $F=E_1\oplus E_2$, 
 and $k \geq 3$;\smallskip

{\rm (C3)}    $\mathcal K$ contains   all trivial direct sum decompositions.
\end{definition}\smallskip

It follows from (C3) that,  for each $k\in \N$, there exists an element of $\mathcal K $ with  length $k$.

For example, the families of all direct sum decompositions and of all trivial direct sum decompositions of $E$ are closed families.

We see immediately that the intersection of a collection of   closed families of direct sum decompositions  of a normed space is also 
a closed family of direct sum decompositions.  Thus the following notion is well-defined.\smallskip

 \begin{definition}\label{10.5b}
 Let   $(E, \norm)$ be a  normed space, and consider a family $\mathcal K $ of direct sum decompositions of $E$. Then the smallest closed  family $\mathcal L$ 
of direct sum decompositions of $E$  such that $\mathcal L$  contains $\mathcal K $ is the {\it closed family generated by}  $\mathcal K$.
 \end{definition}\s
 
Let $E = E_1\oplus \cdots \oplus E_k$  be a direct sum decomposition. For $j\in\N_k$, the dual map
$$
P_j' : \lambda \mapsto \lambda \,\circ\,P_j\,, \quad E_j'\to E'\,,
$$
 is a continuous linear  embedding,  and the image $P_j'(E_j')$ is a closed subspace of $E'$; we shall usually regard  $E_j'$ as a subspace of 
$E'$ by identifying $\lambda \in E_j'$ with $\lambda \,\circ\, P_j \in E'$, and then $E' = E_1'\oplus \cdots \oplus E_k'$. 

\begin{definition}\label{10.4c}
Let   $(E, \norm) $ be a normed space, and let $\mathcal K$ be a  closed family of direct sum  decompositions of $E$. The {\it  dual\/}  to the family $\mathcal K $ is  
$$
 {\mathcal K}' =\{E'_1\oplus\cdots \oplus E'_k : E_1\oplus\cdots \oplus E_k \in \mathcal K\}\,.
$$
\end{definition}\s

Thus    ${\mathcal K}'$ is a closed family of direct sum decompositions of $E'$.
\medskip

\subsection{Duals of products of Banach spaces} Let $(E,\norm)$ be  a  normed space, and take  $k\in \N$. Let $\Norm$ be any norm on the linear space $E^k$ such that
\begin{equation}\label{(1.5)}
\LV x \RV \geq \max\{\lV x_i\rV: i\in\N_k\}\quad (x =(x_i) \in E^k)
\end{equation}
and
\begin{equation}\label{(1.6)}
\LV (0,\dots,0, x_i,0,\dots,0)\RV = \lV x_i\rV\quad(x_i \in E,\,i\in\N_k)\,.
\end{equation}
  For $\lambda_1, \dots, \lambda_k\in E'$,  define $\lambda $ on $E^k$ by
\begin{equation}\label{(1.6e)}
\langle x,\,\lambda\rangle =\sum_{i=1}^k\langle x_i,\,\lambda_i\rangle\quad (x=(x_1,\dots,x_k) \in E^k)\,.
\end{equation}
Then $\lambda $ is a linear functional on $E^k$, and
$$
\lv \langle x,\,\lambda\rangle\rv  \leq   \left(\sum_{i=1}^k \lV \lambda_i\rV\right)
 \max\{\lV x_i\rV: i\in\N_k\} \leq \left(\sum_{i=1}^k \lV \lambda_i\rV\right) \LV x \RV
 $$
for each $x =(x_1,\dots,x_k) \in E^k$. Thus $\lambda \in (E^k, \Norm)'$ with
\begin{equation}\label{(1.6d)}
\max\{\lV \lambda_i\rV :i\in\N_k\} \leq  \LV \lambda\RV' \leq  \sum_{i=1}^k \lV \lambda_i\rV\,, 
\end{equation}  
where $\Norm'$ is the dual norm to $\Norm$. Further, each element in $(E^k, \Norm)'$  arises in this way. 
 Thus we may regard $(E')^k$ as a Banach space for the norm $\Norm'$, identifying $\lambda \in (E^k, \Norm)'$ with
$(\lambda_1, \dots, \lambda_k) \in (E')^k$.

In this case, it is easily seen that $\Norm'$ is a norm on  $(E')^k$ that also satisfies equations (\ref{(1.5)}) and  (\ref{(1.6)}), 
 and so we may also regard $(E'')^k$ as a Banach space for the norm $\Norm''$.  The weak-$*$ topology on $(E^k, \Norm)'$ as the dual of $(E^k, \Norm)$ is 
 equal to the product topology on $(E', \sigma(E',E))^k$.
 
 Let  $E$ be a normed space, and suppose that, for each $k\in\N$, $\norm_k$ is a norm on $E^k$  satisfying (\ref{(1.5)}) and (\ref{(1.6)}), so that 
$\norm'_k$ is a norm on $(E')^k$. Then  $(\norm_k': k\in\N)$ is the {\it dual sequence\/}  to   $(\norm_k: k\in\N)$.
\medskip

\subsection{Families of Banach spaces} Let   $ \{(E_\alpha, \norm_\alpha) : \alpha \in A\}$ be a family of normed spaces, defined for each $\alpha$ in a non-empty index set
 $A$ (perhaps finite).   Then we shall consider the following spaces.

First set  
$$
  \ell^{\,\infty}(E_\alpha) = \{(x_\alpha  : \alpha \in A):  \lV (x_\alpha) \rV =\sup_\alpha\lV x_\alpha\rV_\alpha < \infty\}\,.
$$
Similarly, for $p$ with $1\leq p< \infty$, we define
$$
  \ell^{\,p}(E_\alpha) = \left\{ (x_\alpha  : \alpha \in A):  \lV (x_\alpha) \rV =
\left(\sum_\alpha\lV x_\alpha\rV_\alpha^{\,p} \right)^{1/p}< \infty\right\}\,.
$$
Clearly, $  \ell^{\,\infty}(E_\alpha)$ and $\ell^{\,p}(E_\alpha)$ are normed spaces; they are Banach spaces if each of the spaces $E_\alpha$ 
is a Banach space.  We write 
$$
F\oplus_\infty G\quad {\rm  and}\quad   F\oplus_p G
$$ 
 for the sum of two normed spaces $F$ and $G$ with the appropriate norms,  etc., and we write $\ell^{\,p}_n(E)$ for $E^n$ with the norm given by
$$
\lV (x_1,\dots,x_n)\rV = \left(\sum_{i=1}^n \lV x_i\rV^p\right)^{1/p}\quad (x_1,\dots,x_n \in E)\,.
$$
\medskip

\subsection{Hilbert spaces and $C^*$-algebras}\label{Hilbert spaces} We recall some basic facts about Hilbert spaces; 
for further background, see  \cite{GRA,KR}, for example.

Let $H$ be a Hilbert space,  with inner product denoted by $[\,\cdot\,,\,\cdot\,]$.  For example, let $H= \ell^{\,2}$, where
the inner product is specified by
$$
[(z_j),\,(w_j)] = \sum_{j=1}^{\infty}z_j\overline{w}_j\quad ((z_j),\,(w_j)\in \ell^{\,2})\,.
$$
We recall that $\lV x \rV ^2 = [x,\,x]\;\, (x \in H)$
and that
\begin{equation}\label{(4.3e)}
\lV x +y\rV ^2 = \lV x \rV ^2 +2\Re\, [x,y]+ \lV y \rV ^2 \quad (x,y\in H)\,.
\end{equation}
The {\it Cauchy--Schwarz inequality\/} asserts that
$$
\lv\,[x,y]\,\rv\leq \lV x \rV\/\lV y\rV \quad(x,y \in H)\,.
$$
 
Two vectors $x,y \in H$ are {\it orthogonal\/},  written $x\perp y$,   if  $[x,\,y]=0\,$; a subset $S$ of $S_H$ is {\it orthonormal\/}  if $x\perp y$ 
whenever $x,y\in S$ with $x\neq y$, and an $n$-tuple $(e_1, \dots,e_n)$ of elements in $S_H$ is  {\it orthonormal\/} if $e_i\perp e_j$ 
whenever $i,j\in\N_n$ with $i\neq j$. 

Let $S$ be an orthonormal set in $H$. Then 
$$
\sum_{e \in S} \lv \,[x,e]\,\rv^2 \leq \lV x \rV^2 \quad (x\in H)\,,
$$
 with equality if and only if $x\in \overline{\lin}S$. A maximal orthonormal set is an {\it orthonormal basis\/} 
 for $H$; an orthonormal set is an orthonormal  basis if and only if its closed linear span is $H$.  The 
{\it Hilbert dimension\/}  of $H$ is the cardinality of such a basis;  it is independent of the choice of the basis.  Every Hilbert space is isomorphic  
to one of the form $\ell^{\,2}(I)$, where $I$ is an index set  with $\lv I \rv$ equal to the Hilbert dimension of $H$.

Two linear subspaces $F$ and $G$ of $H$ are {\it orthogonal\/} if
$$
[x,\,y]=0\quad (x\in F, \,y\in G)\,,
$$
 and we write $F\perp G$ in this case.  Suppose that $H=F\oplus G$, where   $F\perp G$. Then $H=F\oplus G$ is an 
{\it orthogonal decomposition\/},   and we  write $$H=F\oplus_{\perp} G\,.$$
 
Let $H$ be a Hilbert space. There is a standard involution  $*$ on  ${\B}(H)$, defined by the condition that 
$$[T^*x,y]= [x,Ty]\quad(x,y\in H,\, T \in {\B}(H))\,,
$$
and then $$\lV T^*T\rV =\lV T \rV ^2\quad(T\in {\B}(H))\,,
$$
 showing that ${\B}(H)$ is a $C^*$-algebra; see \cite[Chapter 4]{KR}.  Subalgebras of  ${\B}(H)$ that are $*$-closed and norm-closed
 are also $C^*$-algebras, and the {\it Gel'fand--Naimark representation theorem\/}   
asserts that every abstractly defined $C^*$-algebra has this form.
 
 Let $A$ be a unital $C^*$-algebra, with identity $e_A$. An element $u\in A$ is {\it unitary\/} if $u^*u=uu^*=e_A$;
 the set of unitary elements is the {\it unitary group\/}, ${\mathcal U}(A)$, of $A$. We shall use the {\it Russo--Dye theorem\/} 
\cite[Theorem 3.2.18]{D}, which asserts that 
\begin{equation}\label{Russo-Dye}A_{[1]} =  \overline{\rm co}({\mathcal U}(A))\,.
\end{equation}
 Suppose that $(e_1,\dots,e_n)$ is an orthonormal $n$-tuple in $H^n$ and that $U$ is a unitary operator on $H$. Then $(Ue_1,\dots,Ue_n)$ 
is also an orthonormal $n$-tuple in $H^n$.

Let $H$ be a Hilbert space. A {\it projection}  in ${\B}(H)$  is an element $P$ in ${\B}(H)$ such that  $P=P^*=P^2$.
In the case where $H= F\oplus_{\perp} G$,  set  $y= P_Fx$ and $z =P_Gx$, so that  $x=y+z$.  Then $P_F$ and $P_G$ are projections in 
${\B}(H)$  such that $P_F+ P_G =I_H$ and $P_FP_G=P_GP_F=0$, so  that $P_F$ and $P_G$ are {\it orthogonal projections}. 
 Conversely, each pair of orthogonal projections gives an orthogonal decomposition of $H$. 
Take $x\in H =F\oplus_{\perp} G$, and set $e= P_Fx/\lV P_Fx\rV$, with $e =0$ when $P_Fx=0$. Then
\begin{equation}\label{(4.3c)}
e\in F\quad{\rm and}\quad   \lV P_Fx \rV  =[e,x]\,.
\end{equation}

We set  
$$
H = H_1\oplus_{\perp} \cdots \oplus_{\perp} H_n
$$
when $H_1, \dots,H_n $  are closed subspaces of $H$ with $H = H_1 \oplus \cdots \oplus H_n$  and $H_i\perp H_j$ whenever $i,j\in \N_n$ with $i\neq j$; 
this is an {\it orthogonal decomposition}. It corresponds to an orthogonal family $\{P_1, \dots,P_n\}$ of projections, where $P_iP_j=0$ whenever
 $i,j\in  \N_n$ with $i\neq j$.\medskip

\subsection{Standard Banach spaces}\label{Standard Banach spaces} Throughout we have certain fixed notations for some standard elements and Banach spaces.

 Consider the space $\C^{\N}$, which consists of all complex-valued sequences, regarded as functions from $\N$ to $\C$. 
 For $n\in \N$, set  
$$
\delta _n = (\delta_{m,n} : m\in \N)\in \C^{\N}\,,
$$
where $\delta_{m,n} =1\,\; (m=n)$ and $\delta_{m,n} =0\,\; (m\neq n)$. Define 
$$
c_{\,00}= \lin \{\delta_n : n\in \N\}\subset \C^{\N}\,,
$$
and, for $p$ with $1\leq p< \infty$, set 
$$
{\ell}^{\,p} = \left\{(\alpha_i) \in \C^{\N} : \sum_{i=1}^\infty \lv \alpha_i \rv^{\,p} < \infty\right\}\,,
$$
 so that    ${\ell}^{\,p} $ is a Banach space for the norm given by
$$
\lV (\alpha_i) \rV_{\ell^{\,p}} =
\lV (\alpha_i) \rV = \left(\sum_{i=1}^{\infty} \lv \alpha_i \rv^{\,p} \right)^{1/p}\quad ((\alpha_i) \in {\ell}^{\,p})\,.
$$
 (We shall  usually suppress the dependence of the norm $\norm$ on the index $p$ in the notation, but we shall occasionally write
 $\norm_{\ell^{\,p}}$ when there is a possibility of confusion.) 

  Further, we set  $$
{\ell}^{\,\infty} = \left\{(\alpha_i) \in \C^{\N} : \lv (\alpha_i) \rv_{\N}= \sup_{i\in \N} \lv \alpha_i \rv < \infty\right\}\,,
$$
so that    $({\ell}^{\,\infty}, \lv\,\cdot\,\rv_{\N}) $ is a Banach space;  the spaces
$$
c_{\,0} = \{(\alpha_i)\in \C^{\N} : \lim_{i\to\infty}\alpha_i =0\}\quad{\rm and}\quad
 c= \{(\alpha_i)\in \C^{\N} : \lim_{i\to\infty}\alpha_i \;\;{\rm exists}\}
$$
of {\it null sequences\/}  and {\it convergent sequences\/},  respectively, are each closed subspaces of $({\ell}^{\,\infty}, \lv\,\cdot\,\rv_{\N}) $. 
 Of course, $c=c_{\,0}\oplus \C 1$, where $1$ is the sequence identically equal to $1$, $c_{\/00}$ is a dense linear subspace of each 
${\ell}^{\,p}$ for $p\geq 1$ and of $c_{\,0}$, and $\{\delta _n :n \in \N\}$
 is a Schauder basis for each of these spaces; we call it the {\it standard basis}.   Note that 
$\lV \delta _n \rV = 1\,\;(n\in \N)$, where $\norm$ is calculated in any of the spaces  ${\ell}^{\,p}$ (for $p\geq 1$) or $c_{\,0} $.

Similarly, we regard $\{\delta_1, \dots, \delta_n\}$ as the standard basis of $\C^n$ for $n\in\N$.

The real-valued versions of these spaces are ${\ell}^{\,p}_\R$, ${\ell}^{\,\infty}_\R$,  $c_{\,0, \R}$ and $c_{\,\R}$,  regarded as subspaces of $\R^{\N}$.
 
 We note that the spaces $ {\ell}^{\,p}$ for  $1< p< \infty$ are reflexive,  that  the spaces $ {\ell}^{\,p}$ for 
$1\leq p< \infty$ and $c_{\,0}$ are separable, that $ \ell^{\,\infty}$ is not separable, and that
 $$
 {\ell}^{\,1} \subset {\ell}^{\,p} \subset {\ell}^{\,q} \subset c_{\,0}\subset \ell^{\,\infty}\quad\;{\rm whenever }\;\quad 1\leq p\leq q< \infty\,.
 $$

Of course, $c_{\,0}' \cong {\ell}^{\,1}$,  $({\ell}^{\,1})' \cong {\ell}^{\,\infty}$, and $({\ell}^{\,p})' \cong {\ell}^{\,q}$ 
for $1< p< \infty$, with the standard duality, where $q$ is the conjugate index  to $p$.

Let $n\in\N$. The $n$-dimensional versions of the above  spaces are denoted by ${\ell}^{\,p}_n$ (for $p\geq 1$) and by ${\ell}^{\,\infty}_n$.  Now
 $({\ell}^{\,\infty}_n)' = \ell_n^{\,1}$. \smallskip
 
 Let $m,n \in \N$. Then we can  identify ${\mathbb M}_{\,m,n}$ with the Banach space ${\mathcal B}(\ell_n^{\,\infty},\ell_m^{\,\infty})$, so that 
$({\mathbb M}_{m,n}, \norm)$ is a Banach space.  Indeed, the formula for the norm in ${\mathbb M}_{\,m,n}$  of an element $a= (a_{ij})$  is then
\begin{equation}\label{(2.3)}
\lV a : \ell_n^{\,\infty}\to \ell_m^{\,\infty}\rV = \max\left\{\sum_{j=1}^n\lv a_{ij}\rv: i\in\N_m\right\}\,.
\end{equation}
In  the case where $m=n$, we obtain a unital Banach algebra $({\mathbb M}_{\,n}, \norm )$.   More generally, let
$p, q \in [1,\infty]$. Then we can also identify ${\mathbb M}_{\,m,n}$ with ${\mathcal B}(\ell_n^{\,p},\ell_m^{\,q})$, and in this case 
we may denote the norm of $a \in {\mathbb M}_{\,m,n}$ by 
$$
\lV a : \ell_n^{\,p}\to \ell_m^{\,q}\rV\,.
$$
For example,
\begin{equation}\label{(2.3b)}
\lV a : \ell_n^{\,1}\to \ell_m^{\,1}\rV = \max\left\{\sum_{i=1}^m\lv a_{ij}\rv: j\in\N_n\right\}\,.
\end{equation}
Let $p_1,p_2 \in [1,\infty]$, and take $q_1,q_2$ to be the two conjugate indices to $p_1$ and $p_2$, respectively. 
For each $a \in  {\mathbb M}_{\,m,n}$, we have $a^t= a'$ and
\begin{equation}\label{(2.3d)}
\lV a : \ell_n^{\,p_1}\to \ell_m^{\,p_2}\rV =\lV a^t : \ell_m^{\,q_2}\to \ell_n^{\,q_1}\rV\,.
\end{equation}

Let $\mu$ be a positive measure on a measure space $\Omega$. (We use the terminology concerning measures of \cite{Co,Ru1,Ru2}.) 
An {\it ordered partition\/}  of $\Omega$ is an $n$-tuple $(S_1, \dots, S_n)$ of measurable subsets of $\Omega$ such that
 $S_1 \cup\cdots\cup S_n =\Omega$ and $S_i\cap S_j=\emptyset$ whenever $i,j\in\N_n$ with $i\neq j$. (We allow  some of the sets $S_j$ to be empty.)

We shall  consider  measurable  functions $f: \Omega  \to \C$. For $p\geq 1$, we  set
$$
L^{p}(\Omega,\mu) =\left\{f : \int_\Omega \lv f \rv^{\,p} {\dd}\mu < \infty\right\}\,,
$$
so that $L^{p}(\Omega,\mu)$ is a Banach space for the norm $\norm$, where
$$
\lV f \rV_{L^{p}} =
\lV f \rV = \left(\int_\Omega \lv f \rv^{\,p} {\dd}\mu\right)^{1/p}= \left(\int_\Omega \lv f(x) \rv^{\,p} {\dd}\mu(x)\right)^{1/p}\quad (f\in L^{p}(\Omega,\mu))\,.
$$
Here we equate functions that are equal almost everywhere with respect to $\mu$ in the usual way.
We shall often write $L^{p}(\Omega)$ for $L^{p}(\Omega,\mu)$ and  $\int_\Omega f$ or $\int f$ for $\int_\Omega  f{\dd}\mu$.  

The space $L^\infty(\Omega, \mu)$  consists of the essentially bounded functions on $\Omega$, with the essential supremum norm.

The real-linear subspaces of $L^{p}(\Omega,\mu)$ and $L^\infty(\Omega, \mu)$ consisting of the real-valued
 functions are denoted by $L^{p}_\R(\Omega,\mu)$  and $L^\infty_\R(\Omega,\mu)$,  respectively.

We shall use  {\it H{\"o}lder's inequality\/} in the following form. Take $p\geq  1$, with conjugate index $q$. 
Then, for $f\in L^{p}(\Omega, \mu)$ and $g \in L^q(\Omega, \mu)$, we have $fg \in L^1(\Omega, \mu)$ and
\begin{equation}\label{(1.6b)}
\int_\Omega \lv fg\rv  \leq \left(\int_\Omega \lv f \rv^{\,p}\right)^{1/p}\left(\int_\Omega \lv g \rv^q\right)^{1/q}\,.
\end{equation}
We shall identify the dual space $L^{p}(\Omega,\mu)'$  with $L^q(\Omega, \mu)$  in the cases where  $p>1$,   where $p=1$ and $\mu$ is $\sigma$-finite, 
and where $\mu$ is counting measure on a non-empty set  $S$, so that $\ell^{\,1}(S)'= \ell^{\,\infty}(S)\/$; the duality is specifed by 
$$
\langle f,\, g\rangle = \int_\Omega fg  {\dd}\mu\quad (f\in L^{p}(\Omega, \mu),\, g \in L^q(\Omega, \mu))\,.
$$
 See \cite[Theorem 4.5.1]{Co} or \cite[Theorem 6.16]{Ru1}.  Again  the spaces $ L^{p}(\Omega, \mu)$  for $1< p< \infty$ are reflexive.

When we consider the spaces $L^{p}(\I)$,  we always suppose that the measure on $\I$ is the Lebesgue measure.\s

 Throughout, a locally compact topological space is supposed to be Hausdorff.

Let $K$ be a non-empty, locally compact space. Then $C_0(K)$ is  the space of all complex-valued, continuous functions on $K$
 that vanish at infinity, and  $C_{0,\R}(K)$ is  the real-linear  subspace of real-valued functions
 in $C_0(K)$.  We write $C(K)$  for $C_0(K)$ in the case where $K$ is compact.  Thus $C_0(K)$  is a 
Banach space with respect to the uniform norm $\lv \,\cdot\,\rv_K$ on $\Omega$, defined by
\begin{equation}\label{(1.6c)}
 \lv f \rv_K =\sup\{\lv f(x)\rv : x\in K\}\quad (f\in C_0(K))\,.
\end{equation}
Let $f,g \in C_{0,\R}(K)$. Then  $\lv f \rv$, $f^+$, $f^-$, $\lv f \rv \vee \lv g \rv $,   $\lv f \rv \wedge \lv g \rv $  belong to $C_{0,\R}(K)$.

Let $K$ be a non-empty, locally  compact space. We denote by   $M(K)$  the space of all 
complex-valued, regular Borel measures  on $K$, taken with the total variation norm $\norm$, so that
$$
\lV \mu\rV = \lv \mu\rv (K) \quad (\mu \in M(K))\,;
 $$  the subspace of real-valued measures  is $M_\R(K)$.  We shall write $\delta_x$ for the
 measure which is the point mass at $x$ for $x\in K$.  A subset $S$ of $K$ is said to be {\it measurable\/} if it 
is measurable with respect to the $\sigma$-algebra of Borel subsets of $K$.   For a (Borel) measurable subset $X$ of $K$, 
we define the restriction measure $\mu \mid X$ for $\mu \in M(K)$ by $(\mu\mid X)(B) =\mu(X\cap B)$ for each Borel subset $B$ of $K$. 
We identify the dual space $C_0(K)'$  with $M(K)$; the duality is
$$
\langle f,\, \mu\rangle = \int_K f   {\dd}\mu\quad (f\in C_0(K),\, \mu \in M(K))\,.
$$ 
 See \cite[\S4.1]{Co} and \cite[Chapter 6]{Ru1}.

  A measure $\mu \in M(K)$ is {\it discrete\/}  if there is a countable subset $S$ of $K$ such that $\lv \mu\rv(K\setminus S)=0$;
 the closed subspace of  $M(K)$ consisting of the discrete measures is denoted by $M_d(K)$, and identified with $\ell^{\,1}(K)$. 
  A measure $\mu \in M(K)$ is {\it continuous\/} if $\mu(\{x\})= 0\,\;(x\in K)$; the closed subspace of 
 $M(K)$ consisting of the continuous measures is denoted by $M_c(K)$.  We have  $M(K)=M_d(K)\oplus M_c(K)$, and
 $\lV \mu + \nu\rV =\lv \mu \rV+ \lV \nu\rV$ for each $\mu \in M_d(K)$ and $\nu \in M_c(K)$, so that $$M(K)=\ell^{\,1}(K)\oplus_1 M_c(K)\,.
 $$
 
We shall use  {\it Hahn's decomposition theorem\/} \label{Hahndecomp} in the following form.
 Let $\mu\in M_\R(K)$.  Then there exist measurable subsets $P$ and $N$ of $K$ such that $\mu(S)\geq 0$ for each measurable subset $S$
 of $P$ and  $\mu(S)\leq 0$ for each measurable subset $S$ of $N$.\s

\subsection{Banach algebras}\label{Banach algebras}   We  shall sometimes  refer to Banach algebras. As a standard  reference 
for  this topic, we shall cite  \cite{D}, and we shall use the terminology of that book.  For an introduction to the theory of Banach algebras  
that is sufficient for our purposes, see \cite[Part II]{GRA}.

Thus a {\it Banach algebra\/} is a linear, associative  algebra $A$ over $\C$  such that $A$ is also a Banach space and 
$$\lV ab \rV \leq \lV a \rV \lV b\rV\quad (a,b\in A)\,.
$$

Let $G$ be a locally compact group.  Then the {\it group algebra} $L^1(G)$ and the {\it measure 
algebra\/}  $M(G)$ are Banach algebras with respect to convolution  multiplication. For details of these algebras,
 see \cite{D, DP, DLS2, HR}.

The   {\it spectrum\/}  of an element $a$ in a Banach algebra $A$ is denoted by $\sigma_A(a)$ or $\sigma(a)$  \cite[Definition 1.5.27]{D}; $\sigma_A(a)$
 is always a non-empty, compact subset of $\C$. The corresponding {\it spectral radius\/}  is denoted by  $\nu(a)$; by definition, 
$\nu(a)=\sup\{\lv z \rv: z \in \sigma_A(a)\}$, and the {\it spectral radius formula} 
\cite[Theorem 2.3.8(iii)]{D} states that 
$$
\nu(a)  = \lim_{n\to \infty} \lV a ^n\rV ^{1/n}\,.
$$

For example, the spaces ${\ell}^{\,p} $ (for $p\geq 1$), ${\ell}^{\,\infty} $, and $c_{\,0}$  are Banach algebras with 
res\-pect to the product defined by   coordinatewise multiplication; indeed, they are {\it Banach sequence algebras\/}
in the sense of \cite[\S4.1]{D}. 

Let $E$ be a Banach space. A {\it Banach operator algebra\/} is a subalgebra $\mathfrak A$ of ${\B}(E)$ containing 
${\mathcal F}(E)$  such that $\mathfrak A$  is a Banach algebra with respect to a norm, say $\Norm$; necessarily
 $\LV T \RV \geq \lV T \rV\,\;(T\in   \mathfrak A$). For example, $({\mathcal K}(E), \norm)$, $({\B}(E), \norm)$, and $({\mathcal N}(E), \nu)$ 
are Banach operator algebras.  The spectrum $\sigma(T)$ of $T \in{\mathcal K}(E)$ is always either finite or a sequence converging  to $0$,
 together with $\{0\}$.  See \cite[$\S$2.5]{D}, for example.

  For each compact space $K$ 
 (always assumed to be Hausdorff), the algebra $C(K)$ is a commutative, unital $C^*$-algebra, and  each  commutative $C^*$-algebra $A$ has the form
 $C_0( \Phi_A)$, where $\Phi_A$ is the (locally compact) character space of $A$. 
 
  We quote the following form of the {\it Banach--Stone theorem\/}, as stated in \cite[V.8.8]{DS}, 
for example. For an account of related results, see \cite[Chapter 2]{DLS2} and \cite{FJ}.\s

 \begin{theorem}\label{1.15a}
 Let $K$ and $L$ be two compact spaces, and suppose that  $T: C(K)\to C(L)$ is an isometric isomorphism.  Then there is a homeomorphism 
$\eta: L\to K$ and a function  $h\in C(L)$ with $h(L)\subset \T$ such that 
$$\vspace{-\baselineskip}
 (Tg)(x) = h(x)(g\,\circ\,\eta)(x)\quad (g\in C(K),\,x\in L)\,.
 $$\hspace*{\stretch{1}}\qed
\end{theorem}\s

A related   result   is described in  \cite[\S 3.2]{FJ}, which is largely an exposition of results of Lamperti \cite{La}; see
  \cite[Theorem 3.2.5]{FJ}. We first recall some background.

Let $(\Omega_1, \mu_1 )$  and $(\Omega_2, \mu_2 )$  be   measure spaces.  A   map $\sigma$  from the measurable subsets of $\Omega_1$ 
to the measurable subsets of $\Omega_2$ is defined to be a {\it regular set isomorphism\/}  
if $\sigma (\Omega_1 \setminus X) = \sigma(\Omega_1) \setminus \sigma (X)$ for each measurable subset of $\Omega_1$, if $\sigma (\bigcup X_n) = \bigcup \sigma(X_n)$
 for pairwise-disjoint families $\{X_n : n\in \N \}$ of  measurable subsets of    $\Omega_1$,  and, for a measurable subset $X$ of $\Omega_1$, 
 $\sigma (X)$ is a $\mu_2$-null set if and  only if $X$ is a $\mu_1$-null set.  In the case where $\Omega$ is discrete, such a map 
$\sigma: \Omega \to \Omega$ is just a permutation 
 of $\Omega$. A regular set isomorphism $\sigma$ induces a unique linear map $T_\sigma$ on the space of measurable functions on $\Omega$ such that 
 $T_\sigma(\chi_X) = \chi_{\sigma (X)}$ for all measurable subsets $X$ of $\Omega$  and $T_\sigma(fg) = T_\sigma(f)\,\cdot \, T_\sigma(g)$
 for all measurable functions $f$ and $g$ on $\Omega$.\s

\begin{theorem}\label{7.4}
{\rm (Lamperti)}  Let $(\Omega_1, \mu_1 )$  and $(\Omega_2, \mu_2 )$  be   measure spaces. Suppose that $p\in [1,\infty)$ with $p\neq 2$. 
Then an isometric isomorphism $U$ from $L^p(\Omega_1, \mu_1 )$ to $L^p(\Omega_2, \mu_2 )$ has the form  
\begin{equation}\label{(7.2)}
U : f\mapsto h\,\cdot\, T_{\sigma}f\,,\quad L^p(\Omega_1, \mu_1 )\to L^p(\Omega_2, \mu_2 )\,,
\end{equation}
where $h : \Omega \to \C$ is such that
$$
\int _{\sigma (X)} \lv h \rv ^{\,p} {\dd}\mu_2 = \mu_1 (X)
$$
for each measurable subset $X$ of $\Omega$, and $T_\sigma \in {\mathcal B}(L^p(\Omega_1), L^p(\Omega_2))$ is  induced by a regular set isomorphism $\sigma$.
\qed
\end{theorem}\smallskip

In the case where $L^p(\Omega_1, \mu_1 )= L^p(\Omega_2, \mu_2 ) = \ell^{\,p}$, the function $h: \N \to\C$ is such that $\lv h(i)\rv =1\,\;(i\in\N)$.\s

A Hausdorff topological space $X$ is {\it extremely disconnected\/}  if  the closure of every open set is itself open; 
this is equivalent to requiring that, for every pair $\{U,V\}$ of open sets  in $X$ with $U\cap V =\emptyset$, we have
  $\overline{U}\cap \overline{V} =\emptyset$.  A compact, extremely disconnected space is called a {\it Stonean space\/}. 
 A Stonean space has a basis for its topology consisting of clopen sets.   By definition, a compact space $K$ such that $C(K)$ is isometrically isomorphic 
to the dual of a Banach space is a  {\it hyper-Stonean space\/} (in which case the predual is unique up to isometric isomorphism).  A hyper-Stonean space is Stonean.
For a discussion and further characterizations of Stonean  and hyper-Stonean spaces, see \cite[Theorems 2.5 and 2.9]{DLS2}.

  Let $\Omega$ be a measure space, and consider the Banach space  $ L^{\infty}(\Omega)$. This is a commutative
  $C^*$-algebra for the pointwise product (defined almost every\-where),
 and so this space is   isometrically isomorphic by the Gel'fand transform to $C(\Phi)$ for a certain compact space $\Phi$;
 the identification is an isomorphism of commutative $C^*$-algebras.
Thus, when $\Omega$ is $\sigma$-finite,  the dual space and the second dual space of $ L^1(\Omega)$ are isometrically isomorphic to $C(\Phi)$ and
  $ M(\Phi)$, respectively.  In the case where $\Omega=S$ is discrete, we have  $\Phi = \beta S$,
the Stone-{\v C}ech compactification  of the set $S$; we shall sometimes identify   ${\ell}^{\,\infty}$ with the space $C(\beta \N)$.

Let $K$ be a non-empty, locally compact  space.
We shall  also identify the space $M(K)'=C_0(K)''$ as $C(\WK)$, where $\WK$ is a certain  hyper-Stonean space, \label{hyper-Stonean} 
called the {\it hyper-Stonean envelope\/}  of $K$ in \cite{DLS2};  in particular, $\WK$ is compact and extremely disconnected.  Thus we are identifying 
 $M(K)''$ with $M(\WK)$. For further details of $\WK$ and  these identifications, see  \cite{DLS2}.\medskip

\subsection{Hermitian elements}  We shall require some notions concerned with numerical ranges and hermitian elements of a Banach algebra.

Let $A$ be a unital Banach algebra, with identity $e_A$. Then the  {\it state space\/}  of $A$ is 
$$
S(A) =\{ \lambda \in A' : \lV \lambda \rV = \langle e_A,\,\lambda\rangle =1\}\,.
$$
 Clearly $S(A)$ contains the character space $\Phi_A$, and so is non-empty; it is convex and closed in the weak-$*$  topology. The {\it numerical range} of $a \in A$ is
 $$
 V(A,a) = \{ \langle a, \lambda\rangle : \lambda \in S(A)\}\,.
 $$
 See \cite{BD1, BD2}.  We see that $V(A, a) \supset \sigma(a)$, the spectrum of $a$.
 
 Let $E$ be a Banach space. Then 
$$
\Pi(E) = \{(x, \lambda) \in E\times E': \lV x \rV = \lV \lambda \rV = \langle x,\,\lambda\rangle =1\}\,.$$
Take  $T \in {\B}(E)$.  Then the {\it spatial  numerical  range\/}  of $T$ is 
$$
V(T)= \{\langle Tx,\,\lambda\rangle : (x, \lambda) \in \Pi(E)\}\,.
$$
Clearly, $V(T) \subset V({\B}(E), T)$, and,  in fact, $V({\B}(E), T) =\ \overline{{\rm co}}\,V(T)$ \cite[$\S9$, Theorem 4(i)]{BD1}.
 
\begin{definition}\label{1.14}
 Let $(A, \norm)$ be a unital Banach algebra. Then an element $a\in A$ is {\it hermitian\/} 
 if $\lV \exp({\rm i}ta)\rV =1$ for all $t \in \R$. \end{definition}\s
 
 The following result is basic; see \cite[$\S5$]{BD1} and \cite[Theorem 5.2.6]{FJ} for other equivalences.\s

\begin{proposition}\label{1.15}
{\rm (i)} Let $A$ be a unital Banach algebra, and take $a\in A$. Then $a$  is hermitian if 
and only if $V(A,a) \subset \R$.\s 

{\rm (ii)}  Let $E$ be a Banach space, and take $T \in {\B}(E)$. Then  $T$ is hermitian if and only if $V(T) \subset \R$.\qed
\end{proposition}\s

The following result is close to   \cite[\S29, Theorem 3]{BD2}. Let $K$ be a compact space.  Point evaluation  at $x\in K$ is denoted by $\varepsilon_x$.\s
 
 \begin{theorem}\label{1.16}
Let $K$ be a compact space, and let $T$ be a hermitian operator on $C(K)$. Then there 
exists an element $h\in C_\R(K)$ such  that $Tf=hf\,\;(f\in C(K))$. 
 \end{theorem}  
 
  \begin{proof} We define $h=T(1)\in C(K)$.

Let  $g \in C(K)$ with $\lv g\rv_K=1$, and take $x \in K$ with $\lv g(x)\rv  = 1$.  Then clearly we have  $\langle g,\,\varepsilon_x\rangle \in\Pi(C(K))$, and so 
 $\langle Tg,\,\varepsilon_x\rangle\in V(T) \subset \R$.  In particular, $h\in C_\R(K)$.
 
 Now take $f \in C_\R(K)$ with $\lv f\rv_K=1$, and write $f= f^+-f^{-}$, where $f^+,f^{-} \in  C_\R(K)^+$.  Suppose that $x\in K$ with $f(x)=0$, so that $f^+(x)=0$
and $f(K)\subset \I$. Then $$\lv 1-f^+\rv_K=\lv (1-f^+)(x)\rv = 1\,,$$ and so $T(1-f^+)(x) \in \R$, whence $(Tf^+)(x) \in \R$.  
Similarly, $(Tf^-)(x) \in \R$, and so $(Tf)(x) \in \R$.  Now set
 $$
 v = (1-f^2)^{1/2}\quad{\rm and} \quad  g= v+{\rm i}f\,,
 $$
 so that $v, g \in C(K)$.   Further, $v(x) = g(x) =1$ and $\lv v\rv_K = \lv g\rv_K =1$, and so  we have both $(Tv)(x)\in \R$ and $(Tg)(x)\in \R$.  Since
$Tg = Tv +{\rm i}Tf$, it follows that  $(Tf)(x) =0$.  Thus, by scaling, we see that $(Tf)(x) =0$ whenever $f \in C_\R(K)$ with $f(x)=0$.

Next, take an arbitrary  $f\in C_\R(K)$  and $x \in K$. Then $(f- f(x)1)(x) = 0$, and so $T(f- f(x)1)(x)=0$.  This says that $Tf(x) = h(x)f(x)$. Hence 
$Tf=hf\in C_\R(K)$.

Finally, take $f\in C(K)$, say $f=f_1+{\rm i}f_2$, where $f_1,f_2\in C_\R(K)$. Then $$Tf=T(f_1+{\rm i}f_2) = h(f_1+{\rm i}f_2)=hf\,,$$
as required.
 \end{proof}\s
 
 An elementary argument, given in \cite[\S29]{BD2} gives the following, related result; the result is due to Tam \cite{Tam}. \s
 
  \begin{theorem}\label{1.17}
Suppose that $p\in [1,\infty]$ with $p\neq 2$. Then each hermitian operator on $\ell^{\,p}$ has the form 
 $\alpha \mapsto \beta \alpha$ for some $\beta \in \ell_{\R}^{\,\infty}$. 
 \qed
 \end{theorem}\medskip

\section{Banach lattices}\label{Banach lattices}
 There is a  strong connection between the  old theory of Banach lattices and our new theory of multi-Banach spaces. 
This will be explained in Chapter 4, \S\ref{Lattice multi-norms}. Here  we  recall briefly some basic notions of the theory of Banach lattices; 
for details, see \cite{AA}, \cite[Chapter 4]{AB}, \cite[Volume II]{LT}, 
 \cite{M}, and \cite{Sc}.  In fact, we choose forms of the standard definitions and notations that are most convenient for us.

\subsection{Definitions} Let $(S, \leq)$ be  a partially ordered set. For $x, y \in S$, the {\it order-interval \/} $[x,y]$  is the set
 $\{z\in S : x\leq z \leq y\}$, and a subset $T$ of $S$ is {\it order-bounded\/}  if there
exist $x, y \in S$ such that $T \subset [x,y]$.   A net $(x_\alpha: \alpha \in A)$ in $S$  is {\it order-bounded\/} 
 if $\{x_\alpha : \alpha \in A\}$ is order-bounded. Further, $(x_\alpha: \alpha \in A)$  is
 {\it increasing\/}  (respectively,  {\it decreasing\/})
 if  $x_\alpha \leq x_\beta$ (respectively, $x_\alpha \geq x_\beta$) whenever $\alpha \leq \beta $ in $A$. We write 
$$
x_\alpha \downarrow x \quad {\rm and}\quad x_\alpha \uparrow x
$$
 if $(x_\alpha)$ is a decreasing net in $S$ and $x =\inf\{x_\alpha : \alpha \in A\}$,   or if    $(x_\alpha)$ is an increasing net
 in $S$ and $x =\sup\{x_\alpha : \alpha \in A\}$, respectively.\s

\begin{definition} \label{4.10a}
A partially ordered set   $(S, \leq)$ is a {\it lattice} if, for each pair $\{s,t\}$
 of elements of $S$, there is a supremum, denoted by $s\vee t$, and an infimum, denoted by $s\wedge t$.

 A lattice is {\/\it  Dedekind complete}  (respectively, {\it $\sigma$-Dedekind complete})  if every 
non-empty  {\rm(}respectively, every countable, non-empty{\rm)} subset which is bound\-ed above has a sup\-remum  and every non-empty {\rm(}respectively, 
every countable, non-empty{\rm)} subset which is bounded below has an infimum.
\end{definition}\s

The supremum and infimum of a non-empty subset  $S$ of a lattice $E$ are denoted by $\bigvee S$ and  $\bigwedge S$,  respectively (if they exist).   Suppose that 
$x_0 =\bigvee S$. Then the  family $\mathcal F$ of finite subsets of $S$ forms a directed set when ordered by inclusion; in this case, set 
$x_F = \bigvee \{y: y\in F\}$ for $F \in {\mathcal F}$. Then $\{x_F : F \in {\mathcal F}\}$ is a net and $x_F \uparrow x_0$. 

Let $E$ be a  linear space  over the real field $\R$ such that  $(E, \leq )$ is also a partially ordered set for an order $\leq$.
 Then $E$ is an {\it ordered linear space\/} if the linear space and order structures are
compatible, in the sense that:\s

(i) $x+z \leq y+ z$  whenever $x,y,z \in E$ and $x\leq y\/$; \s

(ii) $\alpha x \leq \alpha y$ whenever  $\alpha \in \R^+$ and  $x,y \in E$ with $x\leq y\/$. \s

\begin{definition}\label{4.10aa}
An ordered linear space $E$ is a {\it Riesz space}  if $(E, \leq )$ is a lattice. \end{definition}\s

Let $(E, \leq )$ be a Riesz space. Then  the operations $(x,y)\mapsto x\vee y$ and $(x,y)\mapsto x\wedge y$ are
the {\it lattice operations}.  A linear subspace $F$ of a Riesz space  $E$ is a {\it sub-lattice\/}   if $x\vee y, x\wedge y \in F$ whenever $x,y\in F$.
The  {\it positive cone\/} of $E$ is  
$$
E^+ = \{ x \in E : x\geq 0\}\,.
$$
The ordering on a Riesz space is determined by $E^+$. For $x \in E$, set
$$
x^+ = x\vee 0\,,\quad  x^- =(-x) \vee 0\,,\quad  \lv x \rv = x \vee (-x)\,;
$$
thus $x^+$, $x^-$, and $\lv x \rv $  are the {\it positive part\/}, the {\it negative part\/},  and the {\it modulus\/} of $x$, respectively. 
 Elements $x$ and $y$ of $E$ are {\it disjoint\/}, written $x \perp y$, if $\lv x \rv \wedge \lv y\rv =0$.  Two subsets 
$S$ and $T$ of $E$ are {\it disjoint\/}, written $S \perp T$, if $x \perp y$ whenever $x \in S$ and $y\in T$.
 
 For  each non-empty set $S$, the space $\R^S$ is a Riesz space with the pointwise lattice operations, and the definitions 
 of $\lv f \rv$, etc., coincide with the   ones given on page \pageref{positive}.

Let $E$ be a Riesz space. Here are some elementary consequences of the above definitions; they hold for all $x,y,z \in E$ and $\alpha \in \R$:
$$
x = x^+ -x^-\,;\quad \lv x \rv = x^+ +x^-\,;\quad \lv \alpha x\rv = \lv \alpha\rv \lv x \rv\,;\quad \lv x+y\rv \leq \lv x \rv + \lv y\rv\,.
$$
\begin{proposition}\label{2.17}
Let $E$ be a Riesz space, and take $x,y,z \in E$. Then:\s

{\rm (i)} $x+y =x\vee y +x\wedge y$;\s

{\rm (ii)} $(x\vee y) +z = (x+z)\vee (y+z)$;\s

{\rm (iii)} $x\perp y$ if and only if $ \lv x \rv \vee \lv y \rv =  \lv x\rv+\lv y\rv $, and then \mbox{$\lv x+y \rv = \lv x\rv+\lv y\rv $};\s

{\rm (iv)} $\lv x \rv \vee \lv y \rv =  (\lv x+y\rv +\lv x-y\rv)/2$.\s

{\rm (v)}  $\alpha x + \beta y \leq x\vee y$ whenever $\alpha, \beta \in \I$ with $\alpha + \beta =1$.
\qed
\end{proposition}\s

An element $e$ in a Riesz space $E$ is an {\it order-unit\/}  if, for each $x\in E$, there exists $\alpha>0$  such that
$\lv x \rv \leq \alpha e$.

A net   $(x_\alpha : \alpha \in A)$ in a Riesz space  $E$ is {\it order-convergent\/}  to $x \in E$ if there 
exists a net   $(y_\alpha: \alpha \in A)$  and $\alpha_0\in A$ such that  $\lv x_\alpha - x\rv \leq y_\alpha\,\;(\alpha\geq \alpha_0)$ 
and  $y_\alpha \downarrow 0$\,;  in this case, the element $x$ is the {\it order-limit\/}  of  $(x_\alpha: \alpha  \in A)$,  and we write  
$$
x =  \olim_{\alpha} x_\alpha\,.
$$
An order-limit is unique.   A net   $(x_\alpha: \alpha\in A)$    is {\it order-null\/}  if
$$
\olim_{\alpha} x_\alpha = 0\,,
$$
and a  subset $T$ of $E$ is {\it order-closed\/}  if $x \in T$ whenever $(x_\alpha:\alpha\in A)$ is a net  in $T$ with
$x =  \olim_\alpha x_\alpha$.   For a discussion of the notion of order-convergence of nets in a Riesz space, see \cite{AS,KW}. 

Let $(E, \leq )$ be a Riesz space.  A subset $S$ of $E$ is {\it solid\/} if $x \in S$
whenever $x \in E$ and $\lv x \rv \leq \lv y\rv$ for some $y \in S\/$;  a solid linear subspace of $E$
is an {\it order-ideal\/} in $E$. Clearly each order-ideal in $E$ is a sub-lattice
 of $E$. Let $F$ be an order-ideal in $E$, and let $\pi : E\to E/F$ be the quotient map. Then the space $E/F$, with positive cone $\pi(E^+)$,
 is a Riesz space.  An order-closed order-ideal in $E$ is a {\it band\/}.  Suppose that $E=E_1\oplus \cdots\oplus E_n$ is 
a direct-sum decomposition, where each of  $E_1, \dots, E_n$ is an order-ideal. Then each  of $E_1, \dots, E_n$  is a band,
 and the decomposition is a {\it  band decomposition\/}. 

It is clear that  a Riesz space $(E, \leq )$ is   Dedekind complete  if every non-empty subset
  which is bounded above has a supremum.

Let $(E, \leq )$ and $(F,\leq)$ be two Riesz spaces. An operator $T\in  {\mathcal L}(E,F)$ is an {\it order-homomorphism\/}  if 
$$
T(x \vee y)=Tx\vee Ty\quad (x,y\in E)\,;
$$
a  bijective order-homomorphism is an {\it order-isomorphism\/}, and then $(E, \leq)$  and  $(F,\leq)$  are  {\it order-isomorphic\/}. We see easily that 
the operator $T$ is an  order-homomorphism if and only if $T(\lv x\rv) = \lv Tx\rv\,\;(x\in E)$.\s
 
\begin{definition}\label{2.3fc}
Let $(E, \leq )$ be a Riesz space.  A norm $\norm$ on $E$ is a {\it lattice norm} if $\,\lV x \rV \leq \lV y \rV$  whenever $x,y \in E$ with 
$\lv x \rv \leq \lv y \rv$.  A {\it normed Riesz space} is a  Riesz space equipped with a lattice norm.   A {\it real Banach lattice}  is a normed Riesz space 
which is a  real Banach space with respect to the norm.  \end{definition}\smallskip
 
 For example,  the spaces $L^{p}_\R(\Omega,\mu)$ for $p\geq 1$  and $L^\infty_\R(\Omega,\mu)$  for a measure space  $(\Omega, \mu)$  and the spaces
 $C_{0,\R}(K)$ for a non-empty, locally compact space $K$ are real Banach lattices  with respect to the pointwise lattice operations.  
In the case where $K$ is compact, the constant function $1$ is an order-unit of $C_\R(K)$.
 
In a normed Riesz space $(E, \norm, \leq)$, we have
$$
\lV x \rV = \lV\,\lv x\rv\, \rV \quad (x \in E)\,;
$$
further, the lattice operations are uniformly continuous, and so the positive cone $E^+$ and   each order-interval
 $[x,y]$ in $E$ are closed in $(E, \norm )\/$.
 
 Let $E$ and $F$ be normed Riesz spaces, and take $T\in  {\mathcal L}(E,F)$. Then $T$ is an {\it order-isometry\/}
 if it is an order-homomorphism and an isometry; if there is such a map which is a bijection, $E$ and $F$ are
 {\it  order-isometric\/}. 
 
 The {\it functional calculus\/}  or {\it Krivine  calculus\/} for a real  Banach lattice $E$ is described in \cite[II, $\S$1.d]{LT}, for example. Indeed,
a function $f: \R^n\to \R$ is {\it homogeneous of degree $1$\,} if 
$$f(\alpha t_1, \dots,\alpha t_n)=\alpha f(t_1,\dots,t_n) \quad(\alpha \in \R^+,\;\,t_1,\dots,t_n \in \R)\,.
$$
The lattice of all such continuous functions is denoted by  ${\mathcal H}_n$. Then, by \cite[Chapter 16]{DJT} or \cite[II, Theorem 1.d.1]{LT},
for each $x_1,\dots,x_n\in E$, there is a unique order-homomorphism  $\tau : {\mathcal H}_n \to E$ such that $\tau(Z_i) =x_i\,\;(i\in\N_n)$. 
 In particular, for $x_1,\dots,x_n\in E$, we  can define
 $$
 \left(\sum_{i=1}^n\lv x_i\rv^p\right)^{1/p} \in E
 $$
for each $p\geq 1$. \s 
 
\subsection{Complexifications} 
Suppose that $(E_\R, \norm)$ is a real Banach lattice.  Then we make the following definitions. 
Take  $z \in E$, say $z = x + {\rm i}y$, where $x,y \in E_\R$, and first  define the  {\it modulus\/}  $\lv z \rv \in E^+$ of $z$ by 
\begin{equation}\label{(2.3a)}
\lv z \rv = \bigvee\{ x \cos \theta + y \sin \theta : 0\leq \theta \leq 2\pi\}= \left( \lv x\rv^2 +\lv y\rv^2\right)^{1/2}\,.
\end{equation} 
We see that, for  $\alpha \in \C$ and $z,w \in E$, we have: $\lv z \rv =0 $  if and only if $z=0\/$; 
 $\lv\alpha z\rv = \lv \alpha\rv \lv z \rv\/$; $\lv z + w \rv \leq \lv z\rv + \lv w\rv$.  Next, define 
$$
\lV z \rV = \lV\,\lv z \rv\,\rV\quad (z \in E)\,.
$$
Then $\norm$ is a norm  on $E$, and $(E, \norm)$ is a Banach space.  In fact, we have
$$
\frac{1}{2}(\lV x \rV + \lV y\rV) \leq \lV z \rV \leq \lV x \rV + \lV y\rV\quad (z =x+{\rm i}y \in E)\,.
$$
 For details of these remarks, see \cite[\S3.2]{AA}, \cite[Chapter II, \S11]{Sc}, and \cite{Zaa}.  

The above  complexification of a real Banach lattice is defined to be a (complex)  {\it  Banach lattice\/}  \cite[\S 3.2]{AA}. We denote such a Banach lattice by  $E$ or
$(E, \norm, \leq)$, although, strictly, the order $\leq$ is only defined on the real part, $E_\R$, of $E$.  
For example, the spaces $L^{p}(\Omega,\mu)$ for $p\geq 1$  and $L^\infty (\Omega,\mu)$  for a measure space  $(\Omega, \mu)$  and the spaces
 $C_0(K)$ for a non-empty, locally compact space $K$ are   Banach lattices which are the complexifications of the analogous real Banach lattices. 

 We write $E^+$ for $E_\R^+$, and set  $E^+_{[1]} = \{x \in E^+ : \lV x \rV \leq 1\}$.  For $v\in E^+$, we set
$$
\Delta_v = \{ z\in E: \lv z \rv \leq v\}\,.
$$

Let $E$ be a Banach lattice. Again, elements $z$ and $w$ of $E$ are {\it disjoint\/},  written $z\perp w$, if $\lv z \rv \wedge \lv w\rv =0$, and  so
 $z \perp w$ if and only if  $\lv z \rv \vee \lv w\rv = \lv z \rv + \lv w \rv$. In this case, take  $z=x+{\rm i}y$ and $w= u+{\rm i}v$ in $E$. Then we see that 
\begin{eqnarray*}
\lv z +w\rv &=& \bigvee\{(x+u) \cos \theta + (y+v) \sin \theta\}\\
&=& \bigvee\{\lv x\cos \theta + y\sin \theta\rv \vee \lv u \cos \theta + v\sin \theta\rv\}\\
&=& \bigvee \{\lv x\cos \theta + y\sin \theta\rv\}\vee\bigvee\{\lv u\cos \theta + v\sin \theta\rv\}\,, 
\end{eqnarray*}
 where we are  taking suprema over $\theta \in [0, 2\pi]$, and so 
  \begin{equation}\label{(2.4b)}
  \lv z +w\rv  = \lv z \rv \vee  \lv w \rv = \lv z \rv + \lv w \rv\,.
  \end{equation}
  A sequence  $(z_i)$ in $E$  is  {\it pairwise-disjoint\/}  if  $z_i\perp z_j$ for $i,j\in\N$ with $i\neq j$.  

Two subsets  $S$ and $T$ of $E$ are {\it disjoint\/}, written $S \perp T$,  if $z \perp w$  whenever $z \in S$ and $w\in T$. The {\it disjoint complement\/} $S^{\perp}$
 of a non-empty subset $S$ of $E$ is defined by
$$
 S^{\perp} = \{ w \in E :   w \perp z \,\;(z\in S)\}\,.
$$
Note  that $S \cap S^{\perp} \subset  \{0\}$. 

Let $(E, \norm, \leq)$ be a Banach lattice.  A subset $S$ is {\it order-bounded\/}  if
 $\{\lv z \rv : z \in S\}$ is  order-bounded in $E_\R$; this holds if and only if there exists $x\in E^+$ with  
 $\lv z \rv\leq x\,\;(z \in S)$.  Similarly, we define {\it solid subsets}, {\it order-closed subsets},
  {\it order-ideals\/}, and {\it bands\/}. It is easy to see that a  subset of $E$ is a subspace (respectively, an order-ideal)  if and only if it has the form 
$V \oplus  {\rm i}V$, where $V$ is a real subspace (respectively, order-ideal)  in $E_\R$. 

The smallest band containing a subset $A$ of $E$ is denoted by $B(A)$,  and we also set $B_x = B(\{x\})$ for $x \in E$; 
such a set is a {\it principal band}.  An element $x\in E^+$ is a {\it weak order unit\/} if $B_x= E$.
A band $B$ in $E$ is a {\it projection band\/}  if there exists a projection $P \in {\B}(E)$  with $P(E)= B$ and 
$0\leq Px \leq x \,\;(x \in E^+)$, and then $E = B\oplus_\perp B^{\perp}$.

It follows from (c) and (d) of \cite[pp. 259--260]{AB} that every separable Banach lattice contains a weak order unit. 
In particular, the Banach lattices $L^p(\Omega, \mu)$ contain a weak order unit whenever $p\in [1,\infty)$ and the measure space is $\sigma$-finite.

 Let $E = E_1 \oplus \cdots \oplus E_n$ be a direct sum decomposition of a Banach lattice $E$.  This decomposition is
a {\it band decomposition\/}  if each of $E_1,\dots, E_n$ is a band, or equivalently,
if $E_i\perp E_j$ whenever $i,j \in \N_n$ and $i\neq j$. We then write 
$$
E = E_1 \oplus_{\perp} \cdots \oplus_{\perp} E_n\,.
$$
In this case, each of $E_1,\dots, E_n$ is a projection band, and, using \cite[Theorem 1.34]{AA}, 
 each $P_i : E \to E_i$ is a contraction with 
 \begin{equation}\label{(2.4ac)}
 \lv   P_i x  \rv = P_i(\lv x\rv)\leq \lv x \rv \quad (x\in E,\,i\in\N_n)\,.
 \end{equation}
 Further, 
   \begin{equation}\label{(2.4ae)}
 \sum_{i=1}^n\lv x_i\rv = \lv \sum_{i=1}^n x_i\rv\quad (x_i \in E_i,\, i\in \N_n)\,.
  \end{equation}
Indeed, set  $x = \sum_{i=1}^n x_i$. Then
$\sum_{i=1}^n\lv x_i\rv =  \sum_{i=1}^n\lv P_ix\rv = \sum_{i=1}^n P_i(\lv x\rv )$.

Suppose that $E = E_1 \oplus_{\perp} \cdots \oplus_{\perp} E_n$, and take  $z_i\in E_i$ for $i\in\N_n$. Then  we have
  \begin{equation}\label{(2.4aa)}
\lv  z_1 + \cdots + z_n \rv =   \lv z_1\rv \vee \cdots \vee \lv z_n\rv = \lv  z_1 \rv+ \cdots + \lv z_n \rv \,,
\end{equation}
and so  
   \begin{equation}\label{(2.4ab)}
\lV z_1 + \cdots + z_n\rV= \lV\,  \lv z_1\rv \vee \cdots \vee \lv z_n\rv \,\rV= \lV\,\lv  z_1 \rv+ \cdots + \lv z_n \rv\,\rV \,.
\end{equation}
 
Definitions are carried over from real Banach lattices to Banach lattices in the obvious way; for example, a Banach lattice $E$
 is {\it Dedekind complete\/}  if $(E_{\R}, \leq)$ is a Dedekind complete real Banach lattice. 
 
In general, a Banach lattice is not necessarily Dedekind complete. Indeed, the Banach lattices $L^{p}(\Omega)$  are always Dedekind 
complete \cite[Example (v), p.~9]{M}, but the Banach lattice $C(K)$  is Dedekind complete if and only if 
the compact space $K$ is  Stonean  \cite[Proposition 4.2.29(i)]{D}, \cite[II, Proposition 1.a.4(ii)]{LT},
 \cite[Proposition 2.1.4]{M}; $C(K)$ is $\sigma$-Dedekind complete if and only if  
 $K$ is basically disconnected \cite[II, Proposition 1.a.4(i)]{LT}.  A  simple example of a $\sigma$-Dedekind complete space of the form $C(K)$ 
which is not Dedekind complete is the subspace 
of $\ell^{\,\infty}(S)$, for $S$ an uncountable set, spanned  by the constant functions and the functions with countable support.
 
We shall use the following theorem of F.\ Riesz; see  \cite[Theorems 3.8 and  3.13]{AB}, \cite[Theorem 1.2.9 and Proposiiton 1.2.11]{M},
 and \cite[Chapter II, $\S2$]{Sc}.\smallskip

\begin{proposition}\label{2.3fb}
 {\rm (i)}  Every   band   in a Dedekind complete Riesz space is  a projection band. \s

{\rm (ii)}   Every  principal band in a $\sigma$-Dedekind complete Riesz space is  a projection band. \qed
\end{proposition}\smallskip

Suppose that $E$ and $F$ are Banach lattices. For each $T \in {\mathcal B}(E_{\R},F_{\R})$, we see that  $$\lV T \rV \leq \lV T_{\C}\rV \leq 2\lV T\rV\,,$$ and so
 $T_{\C}\in {\B}(E,F)$. Clearly each bounded linear operator from $E$ to $F$ has the form $S + {\rm i}T$, where $S,T \in {\mathcal B}(E_{\R},F_{\R})$.  
  \medskip

\subsection{Continuity, boundedness and completeness}\label{Continuity, boundedness and completeness}  We first  define
 two  properties related to order of the norm on a Banach lattice.\s

\begin{definition}\label{2.3k}
Let $(E, \norm)$ be a Banach lattice.  The norm $\norm$ is {\it order-continuous}  if $\lV x_\alpha \rV \downarrow 0$ whenever $(x_\alpha)$ is a net in $E$ such that 
$x_\alpha \downarrow 0$.  The norm $\norm$ is {\it $\sigma$-order-continuous}
 if $\lV x_n \rV \downarrow 0$ whenever $(x_n)$ is a sequence in $E$ such that $x_n \downarrow 0$.\end{definition}\smallskip

Characterizations of order-continuous Banach lattices are given in \cite[\S 2.3]{AA}, \cite[$\S$12]{AB}, and \cite[\S 2.4]{M}. For example,
 the spaces $L^{p}(\Omega)$ for $p\geq 1$ and Banach lattices which are reflexive as Banach spaces have order-continuous norms,
 but the norm $\lv\,\cdot\,\rv_K$ in $C(K)$ is order-continuous only if $K$ is finite;  the uniform norm on $c_{\,0}$ is order-continuous.  
Each Banach lattice with an order-continuous norm is Dedekind complete. The uniform norm on the Banach lattice  $C(\I)$ is not a $\sigma$-order-continuous norm; 
however, the uniform norm on  the space $C([0, \omega_1])$  is $\sigma$-order-continuous, but not order-continuous.  Suppose that $K$ is Stonean and infinite. Then  $C(K)$
is Dedekind complete, but the norm is not order-continuous.\medskip
 
 Our final definitions in this area are  the following.
  The terms `monotonically complete' and   `Nakano property',  are  defined in \cite[Definition 2.4.18(iii)]{M}, in \cite{Wi2},
and in \cite[Definition 14.10]{AB},
but we have not seen the term `monotonically bounded' in the literature.\smallskip

\begin{definition}\label{2.3j}
Let $(E, \norm)$ be a Banach lattice.  Then:\s 

{\rm (i)} $E$ is  {\it monotonically bounded} if every increasing net in   $E^+_{[1]}$ is bounded above;\s 

{\rm (ii)} $E$ is  {\it monotonically complete} if every  increasing net in $E^+_{[1]}$ has a supremum; \s

{\rm (iii)} $E$  has the {\it  weak Nakano property}  if there 
is a constant $K\geq 1$ such that, for every increasing, order-bounded net $(x_\alpha: \alpha \in A)$ in  $E_{\R}$ and every $\varepsilon >0$, the set 
$\{x_\alpha: \alpha \in A\}$  has an upper bound $u\in E_{\R}$ such that  $\lV u\rV \leq K\sup_{\alpha \in A}  \lV x_\alpha\rV + \varepsilon$;\s

{\rm (iv)} $E$  has the {\it  weak $\sigma$-Nakano property}  if there is a constant $K\geq 1$ such that,
 for every increasing, order-bounded  sequence  $(x_n: n \in \N)$ in  $E_{\R}$ and every $\varepsilon >0$, the set $\{x_n: n \in \N\}$ 
has an upper bound $u\in E_{\R}$ such that  $\lV u\rV \leq K\sup_{n \in \N}  \lV x_n\rV + \varepsilon$;\s

{\rm (v)} $E$  has the {\it Nakano property}  if it has the weak Nakano property with $K=1$.
\end{definition}\s

Trivially, every monotonically complete Banach lattice is mono\-tonically bounded and Dedekind complete. A Banach lattice with an order-continuous norm 
has the Nakano property. We note the following result which  is essentially \cite[Proposition 2.4.19]{M}.\s

\begin{proposition}\label{2.3q}
 A monotonically bounded Banach lattice has the weak Nakano property.\qed
\end{proposition}\s

A Banach lattice is said to be  a {\it KB-space}  if  it is   monotonically complete and has an order-continuous norm 
\cite[p.~89]{AA2}. Thus every KB-space is Dedekind complete, monotonically bounded, and has the Nakano property. The $L^p$ spaces for $p\geq 1$ are examples of 
KB-spaces. \label{KB}  

The  Banach lattice $c_{\,0}$ is Dedekind complete and has the Nakano property, but it  is not monotonically bounded because  the increasing sequence 
$(\delta_1 +\cdots+\delta_n: n\in \N)$   in  $(c_{\/0,\R})_{[1]}$  has no upper bound, and hence $c_{\,0}$ is not monotonically complete.

Let $K$ be a compact space. Then the Banach lattice $C(K)$ is mono\-tonically complete if and only if it is Dedekind complete (if and only if $K$ is Stonean), and so 
the  Banach lattice $\ell^{\,\infty} \cong C(\beta \N)$ is monotonically complete, but its norm is not order-continuous; $C(K)$ is always monotonically bounded; $C(K)$  
has the Nakano property whenever $K$ is Stonean.  The Banach lattice $M(K)$ is monotonically complete. \s

\begin{example} \label{2.41a}
{\rm For $K\geq 1$,  the Banach lattice  $(\ell^{\,\infty}, \norm_K)$,  where  $\norm_K$ is given by 
$$\lV (\alpha_n)\rV = \lv (\alpha_n)\rv_\N + K\limsup_{n\to \infty}\lv \alpha_n\rv\quad ((\alpha_n) \in \ell^{\,\infty})\,,
$$
 is monotonically complete and has the weak Nakano property, but not the Nakano property whenever $K>1$. The Banach lattice 
$\ell^{\,\infty}((\ell^{\,\infty}, \norm_K): K\in\N)$  is  Dede\-kind complete, but it does not have the weak $\sigma$-Nakano property.\qed}
\end{example}\s 
 
A Dedekind-complete lattice has the Nakano property if and only if the norm is a {\it Fatou norm},  in the sense of \cite[p.\ 65]{AA}  and \cite[Definition 2.4.18]{M}. 
In \cite{AA2} and \cite{AW2}, a norm $\norm$ on a Banach lattice $E$ is said to be a {\it Levi norm\/}  if $(E, \norm)$ is monotonically complete. \smallskip
 
\subsection{Positive, regular, and order-bounded operators}  Let $E$ and $F$ be real Banach lattices, and take  $S,T \in {\mathcal L}(E,F)$. We define
$$
 S \leq T\quad {\rm if}\quad Sx \leq Tx\quad(x \in E^+)\,.
$$
Clearly, $({\mathcal L}(E,F),\leq)$ is an ordered linear space.\s

\begin{definition} \label{2.41}
Let $E$ and $F$ be real Banach lattices, and consider   $T \in {\mathcal L}(E,F)$. Then:\s

{\rm (i)} $T$ is {\it positive\/}  if $T\geq 0\,$;\s

{\rm (ii)} $T$  is {\it regular\/}  if $T= T_1-T_2$, where $T_1$ and $T_2$ are positive operators;\s

{\rm (iii)}  $T$  is {\it order-bounded\/}  if  $T(B)$ is an order-bounded subset of $F$ for each order-bounded subset $B$ of $E$.
\end{definition}\s

The set of  positive operators from $E$ to $F$ is closed under addition and multiplication by $\alpha \in \R^+$, and so it  is a {\it cone\/},  denoted by
 ${\mathcal L}(E,F)^+$. 

 The book \cite{AB} is devoted to positive operators. 
 
 We shall (at least implicitly) use a basic theorem of Kantorovic  \cite[Theorem 1.7]{AB}: 
each additive map $T:E^+\to F^+$ extends uniquely to a positive operator from $E$ to $F$, and the unique extension $T$ satisfies 
$$
Tx= T(x^+)-T(x^-)\quad (x \in E)\,.
$$
Thus a positive operator $T$ has been specified as soon as we know that $T:E^+\to F^+$ is additive.

Let $E$ be a $\sigma$-Dedekind complete Banach lattice. Then, for each $v\in E^+$,  a projection $P_v$ is defined by first  setting
 \begin{equation}\label{(2.3e)}
P_v(x)=\bigvee \{nv\wedge x: n\in\N\}\quad(x\in E^+)\,,
\end{equation}
and then extending $P_v$ by linearity to the whole of $E$; see \cite[II, p.\ 8]{LT}.  In this case,
  the map $P_v:E\to E$ is a positive linear projection with  $\lV P_v\rV\leq 1$ for each $v\in E^+$. 
Note that $P_{\lv x \rv}(x) = x\,\;(x\in E)$. In the special case where 
$E$ is $L^p(\Omega)$ for $p\in [1,\infty]$ and  a measure space $\Omega$, the map $P_v$ is just multiplication by the
 characteristic function\label{characteristic} of the set $\{t\in \Omega:\ v(t)\neq 0\}$. 

The space of  all regular operators from $E$ to $F$ is  denoted by  ${\mathcal L}_r(E,F)$.  We see immediately that  
$({\mathcal L}_r(E,F), \leq)$ is an ordered linear subspace of $({\mathcal L}(E,F), \leq)$, with positive cone ${\mathcal L}(E,F)^+$.

The space of all order-bounded operators from $E$ to $F$ is denoted by ${\mathcal L}_b(E,F)$. 
Clearly $({\mathcal L}_b(E,F), \leq)$ is  an ordered linear subspace of \mbox{$ ({\mathcal L}(E,F),\leq)$} and
$$
{\mathcal L}(E,F)^+\subset {\mathcal L}_r(E,F)\subset  {\mathcal L}_b(E,F)\subset  {\mathcal L}(E,F)\,.
$$
Each order-bounded linear operator is continuous \cite[p.\  22]{AA}, and so
  ${\mathcal L}_b(E,F)\subset {\mathcal B}(E,F)$.  For this reason, we denote  ${\mathcal L}(E,F)^+$, ${\mathcal L}_r(E,F)$, and 
${\mathcal L}_b(E,F)$ by ${\mathcal B}(E,F)^+$, $ {\mathcal B}_r(E,F)$, and ${\mathcal B}_b(E,F)$, respectively. 

 Now suppose that $E$ and $F$ are Banach lattices.  In the case where $T:E_\R \to F_\R$ is a positive operator, we have $\lV T_\C \rV =\lV T\rV$ 
(but this is not necessarily true for all regular  operators $T$ \cite[Exercise 9 of \S 3.2]{AA}).   We shall use the following observation. 
Take  $T \in {\mathcal B}(E,F)^+$. Then
 \begin{equation}\label{(2.3f)}
 \lV T \rV =\sup\{ \lV Tx\rV : x\in E^+,\,\lV x \rV \leq 1\}\,.
 \end{equation}

An operator  $S + {\rm i}T \in {\B}(E,F)$  is {\it regular}  or {\it order-bounded} or  {\it order-isometric} if both $S$ and $T$ are  regular or  order-bounded 
or  order-isometric, respectively.   Again, each order-bounded operator is continuous,  and so we denote the spaces of all positive, all  regular,  and all order-bounded
 operators from $E$ to $F$ by ${\mathcal B}(E,F)^+$, ${\mathcal B}_r(E,F)$, and ${\mathcal B}_b(E,F)$, respectively. Thus we have
$$
{\mathcal B}(E,F)^+\subset {\mathcal B}_r(E,F)\subset  {\mathcal B}_b(E,F)\subset  {\mathcal B}(E,F)\,.
$$
We write  ${\mathcal B}_r(E)$ and $ {\mathcal B}_b(E)$  for ${\mathcal B}_r(E,E)$ and $ {\mathcal B}_b(E,E)$, respectively.

An operator $T \in {\mathcal B}(E,F)$ is {\it order-continuous}  if  $ Tx = \olim_{\alpha}T(x_\alpha)$ in $F$ whenever
 $x =\olim_{\alpha} x_\alpha$ in $E$. By \cite[Theorem 2.1]{AS}, each  such operator is order-bounded. 
 
 The following result is based on  \cite[$\S3$]{Wi1}.\s
 
 \begin{proposition}\label{6.13d}
  Let $E$ and $F$ be Banach lattices. Then, for each $T\in {\B}_b(E,F)$, there exists $c>0$ such that, for each $v \in E^+$, there exists $w\in F^+$
 with $T(\Delta_v)\subset \Delta_w$ and $\lV w\rV \leq c\lV v \rV$.
 \end{proposition}
 
 \begin{proof}  Assume towards a contradiction that no such constant $c$ exists.   For each $n\in\N$, there exists $v_n \in E^+$ with
 $\lV v_n\rV =1/2^n$  such that $\lV w\rV \geq n$ whenever $w\in F^+$ has the property that $\lv Tx\rv \leq  w$ for each $x\in E$
 with $\lv x \rv \leq v_n$.  Take $$v =\sum_{n=1}^{\infty}v_n \in E^+\,.
$$  Then there exists $w_0 \in F^+$ such that $\lv Tx\rv \leq w_0$ whenever 
$x\in E$ with $\lv x \rv \leq v$. For each $n\in\N$, we have $v_n\leq v$, and so $\lv Tx\rv \leq  w_0$ whenever $\lv x \rv \leq v_n$,
 whence $\lV w_0\rV \geq n$. This is the required contradiction.
 \end{proof}\s
 
 \begin{definition}\label{6.13e}
 Let $E$ and $F$ be Banach lattices, and let $T\in {\B}_b(E,F)$. Then the infimum of the constants $c$ such that, for each $v \in E^+$, there exists $w\in F^+$
 with $T(\Delta_v)\subset \Delta_w$ and $\lV w\rV \leq c\lV v \rV$, is denoted by $\lV T \rV_b$. 
\end{definition}\s 

Now consider  $T \in {\mathcal B}_r(E,F)$. The following definition is given in \cite[Exercise 2.2.E2]{M}.\s

\begin{definition}\label{6.13a}
 Let $E$  and $F$ be Banach lattices. For $T \in  {\mathcal B}_r(E,F)$, set
$$
 \lV T\rV_r  = \inf\{ \lV S \rV : S\in {\mathcal B}(E,F)^+,\, \lv Tz\rv \leq S(\lv z\rv) \,\;(z \in E)\}\,.
$$
\end{definition}\smallskip

  \begin{proposition}\label{6.13f}
Let $E$ and $F$ be Banach lattices. Then:\s

{\rm (i)}  $\norm_b$ is a norm on the space ${\B}_b(E,F)$ such that 
$\lV T \rV_b\geq \lV T \rV\,\;(T\in {\B}_b(E,F))$,  and  $({\B}_b(E,F), \norm_b)$ is a Banach space;\s

{\rm (ii)}  $\norm_r$ is a norm on ${\B}_r(E,F)$ with
$$
\lV T \rV_r  \geq \lV T \rV_b\geq \lV T \rV \quad(T \in {\mathcal B}_r(E,F))\,,
$$
and  $({\B}_r(E,F), \norm_r)$ is a Banach space.\s
 \qed
  \end{proposition}\s

If ${\mathcal B}_r(E,F) = {\mathcal B}_b(E,F)$, then the norms $\norm_r $ and  $\norm_b $ are equivalent  on ${\mathcal B}_r(E,F)$,
 but  examples in \cite{Wi1}  shows that the norms are not necessarily equal in this case, and that, in general, the norms are not necessarily
 equivalent on ${\mathcal B}_r(E,F)$; Example 4.1 of \cite{Wi1} exhibits Banach lattices $E$ and $F$ and a compact, order-bounded operator $V: E\to F$
 which is not even in the $\norm_b$-closure of ${\B}_r(E,F)$.
 Examples with ${\mathcal B}_r(E,F)\subsetneq  {\mathcal B}_b(E,F)$ and with ${\mathcal B}_b(E,F)\subsetneq  {\mathcal B}(E,F)$
 are given in \cite[Examples 1.11 and 15.1]{AB}. An example given in \cite[$\S2$]{Wi1} shows that there may be operators in ${\mathcal B}_b(E,F)$ 
 that are not even in the $\norm$-closure of ${\mathcal B}_r(E,F)$. 

The three clauses of the following theorem are taken from \cite{AMS}, from   \cite[Theorem 3.9]{AA} and  \cite[Theorem 15.3]{AB},  and from \cite{AV}, respectively.\s
 
\begin{theorem}\label{1.27}
{\rm (i)} Let $K$ be a compact space with weight smaller than the smallest strongly inaccessible cardinal.  Then
${\mathcal B}_r(C(K)) = {\mathcal B}(C(K))$  if and only if $K$ is Stonean. \s

 {\rm (ii)} Let $\Omega$ be a measure space. Then ${\mathcal B}_r(L^1(\Omega))= {\mathcal B}(L^1(\Omega))$ and, further, 
 $$\lV T\rV_r=\lV T \rV \quad (T\in {\mathcal B}(L^1(\Omega)))\,.
$$

 {\rm (iii)} Let $\Omega$ be a measure space and take $p$ with $1< p< \infty$ such that $L^p(\Omega)$ is infinite-dimensional. 
 Then ${\mathcal B}_r(L^p(\Omega))$ is not dense in  ${\mathcal B}(L^p(\Omega))$, and $\norm_r$ and $\norm$ are not equivalent on ${\mathcal B}_r(L^p(\Omega))$.  \qed
 \end{theorem}\s

Let $T\in {\B}(E,F)^+$. Then
\begin{equation}\label{(1.3a)}
\lV T \rV_{b} = \lV T \rV_r =\lV T \rV\,.
\end{equation}

We shall use the following standard theorem of F.\ Riesz and Kantorovitch; see 
\cite[Theorems 1.16, 1.32,  3.24, 3.25]{AA}, \cite[Theorems 1.10 and 1.13]{AB},  \cite[Propositions 1.3.6 and 2.2.6]{M}, and  \cite[Chapter 4, $\S1$]{Sc}.\s

\begin{theorem}\label{6.13}
 Let $E$  and $F$ be real Banach lattices, with $F$ Dede\-kind complete. Then  ${\mathcal B}_r(E,F)=
 {\mathcal B}_b(E,F)$  is a Dede\-kind complete real Banach lattice for the lattice operations defined for $T \in  {\mathcal B}_r(E,F)$ and $x \in E^+$ by
$$
T^+(x) =\sup\{Ty: y\in[0,x]\}\,, \quad T^-(x) =\sup\{-Ty: y\in[0,x]\}\,.
$$  

Let $T_1,\dots, T_n \in {\mathcal B}_r(E,F)$ and $x \in E^+$. Then 
\begin{equation}\label{(1.3b)}
(T_1\vee \cdots\vee T_n)(x) = \bigvee\left\{\sum_{i=1}^n T_ix_i: x_i\in E^+,\,x_1+\cdots + x_n=x\right\}\,.
\end{equation}

Let $E$  and $F$ be  Banach lattices, with $F$ Dede\-kind complete. Then ${\mathcal B}_r(E,F)= {\mathcal B}_b(E,F)$ is a Dede\-kind complete Banach lattice, and 
$$ 
 \lv T \rv (u) =\sup\{\lv  Tz\rv: \lv z\rv \leq u\}\quad (u\in E^+)\,.
$$ Further,  $\lV T\rV_r = \lV \,\lv T \rv\,\rV$ and $\lv Tx\rv\leq \lv T\rv(\lv z\rv)\;\, (z \in E)$ 
for   $T \in  {\mathcal B}_r(E,F)$.\qed
\end{theorem}\s

\subsection{The Banach algebra ${\B}_r(E)$} The following result is clear.\s

\begin{theorem}\label{6.15}
Let $E$ be a  Banach lattice. Then  $({\mathcal B}_r(E), \norm_r)$   and  $({\mathcal B}_b(E), \norm_b)$   are unital Banach algebras. \qed
\end{theorem}\s 

There appears to be surprisingly little about the Banach algebra ${\mathcal B}_r(E)$ in the literature; for example, it is not mentioned in \cite{D}.
  There seems to be no mention of the Banach algebra ${\mathcal B}_b(E)$ at all. \s 

\begin{definition}\label{6.15b}
Let $E$ be a  Banach lattice,   and take  $T \in {\mathcal B}_b(E)$. The {\it  order-spectrum\/},  $\sigma_o(T)$,   of $T$ is the spectrum of $T$ with respect
 to the Banach algebra $({\mathcal B}_b(E), \norm_b)$. The corresponding {\it order-spectral radius} is denoted by $\nu_o(a)$.
\end{definition}\smallskip

Of course,  $\sigma_o(T)\supset \sigma(T)$  and $\nu_o(T)\geq \nu(T)$ for each $T \in {\mathcal B}_b(E)$.

For a discussion of $\sigma_o(T)$  and $\nu_o(a)$, see \cite[$\S$7.4]{AA} and \cite[\S4.5]{M}; in the latter source, and elsewhere, the order-spectrum is defined 
for $T \in {\mathcal B}_r(E)$ with respect to the Banach algebra ${\mathcal B}_r(E)$.

\begin{example}\label{6.15c}
{\rm  Let $E$ be the Banach lattice  $L^2(\T)$, so that $E$ is  monotonically complete  with order-continuous norm.

An example of Arendt \cite{Ar}  exhibits a positive, compact operator $T \in  {\mathcal K}(E)\cap {\B}_r(E)$
 (so that $\sigma(T)\subset \R$ is countable)  such that $\sigma_o(T)$ contains the unit circle $\T$. The operator has the form
$$
T_\mu : f\mapsto \mu\,\star\, f\,, \quad  L^2(\T)\to L^2(\T)\,,
$$
where $\mu$ is a certain singular measure on $\T$.  Note the interesting fact that
$$
\sigma_o(T_\mu) = \sigma_{M(\T)}(\mu)\supsetneq \sigma(T_\mu)\,.
$$
It  follows that  there are compact operators on $L^2(\T)$ which are not regular.

An example of Ando, which is discussed in \cite[Example 7.36]{AA} and \cite[p.\ 306]{M}, exhibits a Dedekind complete
 Banach lattice $E$ with order-continuous norm  and an operator $T \in {\B}_r(E)$ such that $\nu_o(T)> \nu(T)$.} \qed
 \end{example}\medskip

\subsection{Dual Banach lattices} Let $E$ be a real Banach lattice, with dual space $E'$.  Then $E'$ is ordered by the requirement that $\lambda \in E'$ 
belongs to $(E')^+$ if and only if $\langle x,\,\lambda\rangle \geq 0\,\;(x \in E^+)$, and then  $E'$ becomes a real Banach lattice  with respect
 to the following definitions of $\lambda\vee \mu$ and $\lambda\wedge \mu$   for $\lambda, \mu \in E'$. 
In fact, $\lambda\vee \mu$ and $\lambda\wedge \mu$ are defined for $x\in E^+$ by
\begin{equation}\label{(1.3)}\left\{
\begin{array}{rcl}
\langle x,\,\lambda\vee \mu\rangle
&=&
 \sup \{\langle y,\,\lambda\rangle + \langle z,\,\mu\rangle: y,z \in E^+,\,y+z=x\}\,,
\\
\langle x,\,\lambda\wedge \mu\rangle
&=&
 \inf \{\langle y,\,\lambda\rangle + \langle z,\,\mu\rangle: y,z \in E^+,\,y+z=x\}\,,
\end{array}\right.
\end{equation}
and then $\lambda\vee \mu$ and $\lambda\wedge \mu$   are extended to $E'$.
The dual of a Banach lattice $E$ is also a Banach lattice; this is the {\it dual Banach lattice} of $E$.

Let $E$ be a real Banach lattice, and take $x \in E^+$ and  $\lambda  \in E'$. Then we have
$$
\langle x,\,\lambda^+\rangle = \sup\{\langle y,\,\lambda\rangle : 0\leq y\leq x\}\,.
$$

Let $E$ be a   Banach lattice. We note that 
\begin{equation}\label{(2.3c)}
\lv\langle z\,,\lambda\rangle\rv \leq \langle \lv z\rv,\,\lv \lambda\rv\rangle\quad (z\in E,\,\lambda \in E')\,;
\end{equation}
this is easily checked.

Let $(\lambda_\alpha : \alpha \in A)$ be a net in $E'$, where $E$ is a real Banach lattice, and suppose that
 $\lambda_\alpha \uparrow \lambda\in (E')^+$. Define  $\mu(x) = \lim_\alpha  \langle x,\,\lambda_\alpha\rangle\,\;(x\in E)$. Then 
$\mu$ is a positive linear functional on $E$, and so $\mu \in E'$;  $\lambda_\alpha \leq \mu \leq \lambda\,\;(\alpha\in A)$, and so 
$\mu = \lambda$. It follows that 
\begin{equation}\label{(2.4ad)}
\langle x,\,\lambda_\alpha\rangle \uparrow \langle x,\,\lambda \rangle\quad (x\in E^+)\,.
\end{equation}

A dual Banach lattice $E'$ is  monotonically complete and has the Nakano property; $E'$ is always Dedekind complete, and so every band in $E'$ is a projection band.

 For example, let $(\Omega, \mu)$ be a measure space, and take $E= L^{p}(\Omega, \mu)$, where $p\geq 1$, in the case where $E' = L^q(\Omega, \mu)$,
 where $q$ is the conjugate index to $p$.   Then the dual lattice operations on $E'$ coincide with the given lattice operations on $L^q(\Omega, \mu)$.
 
   Let $K$ be a non-empty, locally compact space. Then $M(K)=C_0(K)'$ is a dual Banach lattice, and 
$$
(\mu \vee \nu)(S) = \sup\{\mu (S_1) + \nu (S_2) \}\,,\quad 
(\mu \wedge \nu)(S) = \inf\{\mu (S_1) + \nu (S_2)  
$$
for $\mu,\nu \in M_\R(K)$ and  a  measurable subset $S$ of $K$, where the supremum and infimum are taken over all ordered
 partitions $(S_1,S_2)$ of $S$.  Let $\mu, \nu \in M(K)$. Then  $\mu\perp \nu$ in the Banach lattice $M(K)$ if and only if $\lv \mu\rv \wedge \lv \nu\rv =0$, so that
 $\mu$ and $\nu$ are mutually singular in the classical sense of measures.  We  see that the following are equivalent: \s

(a) $\mu\perp \nu$\,; \s

(b) $\lV \mu \rV + \lV \nu\rV =  \lV \mu + \nu\rV  =  \lV \mu - \nu\rV$\,;\s

(c)  $\lV\,\lv \mu\rv\,  + \lv \nu\rv\,\rV = \lV\, \lv \mu\rv \vee \lv \nu\rv \,\rV$\,.\s

\noindent For example, $M(K) = M_d(K)\oplus_\perp M_c(K)$ is a band\label{band} decomposition.  

We shall use the following proposition.\s 

\begin{proposition}\label{2.3gf}
Let $E$ be a  Banach lattice, and take $x \in E^+$, $\lambda \in E'$, and $\varepsilon >0$. Then there exists $z\in E$ such that   
$$
\lv z \rv\leq x\quad {\rm and}\quad\langle z,\,\lambda\rangle > \langle x,\,\lv \lambda\rv\rangle-\varepsilon\,.
$$
\end{proposition}

\begin{proof} We write $\lambda =\mu+{\rm i}\nu$, where $\mu,\nu\in (E_\R)'$. By the definition, we have  
$$
\lv \lambda  \rv = \bigvee\{ \mu  \cos \theta + \nu \sin \theta : 0\leq \theta \leq 2\pi\}\,,
$$
and so there exist $\theta_1,\dots,\theta_n \in [0, 2\pi]$  such that 
$$
\langle x,\,(\mu  \cos \theta_1 + \nu \sin \theta_1)\vee \cdots\vee (\mu \cos \theta_n + \nu \sin \theta_n)\rangle
> \langle x,\,\lv \lambda\rv\rangle-\varepsilon\,.
$$
By extending the definition in (\ref{(1.3)}), we see that there exist $u_1,\dots, u_n\in E^+$ such that $u_1+\cdots+  u_n =x$ and 
$$
\langle u_1,\,\mu  \cos \theta_1 + \nu \sin \theta_1 \rangle +\cdots +\langle u_n,\, \mu  \cos \theta_n + \nu \sin \theta_n\rangle
> \langle x,\,\lv \lambda\rv\rangle-\varepsilon\,.
$$
Thus 
\begin{equation}\label{(2.18)}
\sum_{j=1}^n\langle (\cos \theta_j)u_j,\,\mu  \rangle + \sum_{j=1}^n\langle (\sin \theta_j)u_j,\,\nu  \rangle
 > \langle x,\,\lv \lambda\rv\rangle-\varepsilon\,.
\end{equation}
  Set
$$
w = \sum_{j=1}^n(\cos \theta_j- {\rm i}\sin \theta_j)u_j\in E\,.
$$
Then equation (\ref{(2.18)}) states that ${\Re}\,\langle w,\,\lambda\rangle > \langle x,\,\lv \lambda\rv\rangle-\varepsilon$, and so 
 $\lv \langle w,\,\lambda\rangle\rv > \langle x,\,\lv \lambda\rv\rangle-\varepsilon$. For each $\theta \in[0,2\pi]$, we have 
 $$
 \sum_{j=1}^n(\cos \theta \cos \theta_j - \sin \theta \sin \theta_j)u_j = \sum_{j=1}^n\cos (\theta + \theta_j)u_j\,,
 $$
and hence  
$$
\lv w \rv =\sup\left\{\sum_{j=1}^n\cos (\theta + \theta_j)u_j : 0\leq \theta \leq 2\pi\right\} \leq \sum_{j=1}^n u_j=x\,. 
$$
 Finally, set $z = \zeta w$, where $\zeta \in \T$ is chosen to be such that $\zeta  \langle w,\,\lambda\rangle= \lv \langle w,\,\lambda\rangle\rv$. 
Then $\lv z\rv =\lv w \rv \leq x$  and $\langle z,\,\lambda\rangle > \langle x,\,\lv \lambda\rv\rangle-\varepsilon$, as required.
\end{proof}\s

Let $E = E_1 \oplus_{\perp} \cdots \oplus_{\perp} E_n$ be a band decomposition of a Banach lattice $E$. Then the  corresponding decomposition
 of $E'$ is a band decomposition, so that 
   \begin{equation}\label{(2.4af)}
 E' = E'_1 \oplus_{\perp} \cdots \oplus_{\perp} E'_n\,.  
   \end{equation}
However, in general, it is not true that  every band decomposition of $E'$ arises in this way.
\medskip

\subsection{$AL$ and $AM$ spaces}  We now define some special types of Banach lattices.\s

\begin{definition}\label{2.3ga}
A real Banach lattice $(E, \norm)$ is:  an {\it  $AL$-space} if 
$$
\lV x + y \rV = \lV x \rV + \lV y\rV \quad {\rm whenever}\quad x,y \in E^+\quad {\rm with}\quad x\wedge y = 0\,;
$$
 an  {\it  $AL_p$-space} (for  $p\geq 1\/$) if  
$$
\lV x + y \rV^p = \lV x \rV^p + \lV y\rV^p \quad {\rm whenever}\quad x,y \in E^+\quad {\rm with}\quad x\wedge y = 0\,;
$$
and an   {\it $AM$-space} if  
$$
\lV x \vee y \rV =\max \{\lV x\rV,\lV y\rV\} \quad {\rm whenever}\quad x,y \in E^+\quad {\rm with}\quad x\wedge y = 0\,.
$$
A Banach lattice is an {\it $AL$-space} or an  {\it  $AL_p$-space} or an {\it $AM$-space} if $E_\R$ has the appropriate property.
\end{definition}\s

For example, each space of the form $L^p (\Omega,\mu)$, where $(\Omega,\mu)$ is a measure space,  is an $AL_p$-space, and each space
 $C_{0}(K)$, where $K$ is a non-empty, locally compact space,  is an $AM$-space.   

Let $E$ be a Banach lattice. Then $E$ is an $AL$-space  if and only if  
\begin{equation}\label{(2.19)}
\lV x+y\rV =\lV x \rV + \lV y\rV\quad (x,y\in E^+)\,,
\end{equation}
and an  $AM$-space  if and only if  
\begin{equation}\label{(2.19a)}
\lV x\vee y\rV =\max\{\lV x \rV, \lV y\rV\}\quad (x,y\in E^+)\,.
\end{equation}

The following duality result is \cite[Theorem 12.22]{AB}, for example.\s

\begin{theorem}\label{2.3gb}
Let $E$ be a  Banach lattice, with dual Banach  lattice $E'$.  Then $E$ is an $AL$-space  if and only if $E'$ is an $AM$-space, and $E$ is an 
$AM$-space if and only if $E'$ is an $AL$-space.\qed
\end{theorem}\smallskip

The following central representation theorem is proved in   \cite[Theorems 3.5 and 3.6]{AA},
 \cite[Theorems 12.26 and 12.28]{AB}, and \cite[II. $\S$1.b]{LT}. We shall call it `{\it Kakutani's theorem\/}'; 
 detailed attributions for the various statements are given in \cite{AA}.\s  

\begin{theorem}\label{2.3gc}
{\rm (i)} Take $p\geq 1$. A  Banach lattice  is an $AL_p$-space if 
and only if it is order-isometric to a  Banach lattice of the form $L^p(\Omega,\mu)$, where $(\Omega,\mu)$ is a measure space, and hence each 
 $AL_p$-space  has an order-continuous norm  and  is Dedekind complete.\s

{\rm (ii)} A Banach lattice  is an $AM$-space if and only if it is order-isometric to a closed sub-lattice of a space 
$C(K)$, where $K$ is a compact space. \qed
\end{theorem}\smallskip

\begin{corollary}\label{2.3gd}
Let  $(\Omega, \mu)$ be  a measure space.  Then there is an order-isomorphism $\theta$ from  the dual
 space of $L^1(\Omega, \mu)$ onto $C(K)$ for some compact space $K$, and the restriction of $\theta$ to $L^{\infty}(\Omega, \mu)$  is the 
Gel'fand identification of $L^{\infty}(\Omega, \mu)$ with a $C^*$-subalgebra of  $C(K)$.\qed
 \end{corollary}\s
 
\begin{corollary}\label{2.3ge}
Let  $K$ be  a non-empty, locally compact space. Then  $M(K)$ is  order-isometric to  the space $L^1(\Omega,\mu)$ for some  measure space $(\Omega,\mu)$.\qed
 \end{corollary}\s
 
 We also mention  a related result from \cite{Wi2}. Let $E$ be a Banach lattice.  Then $E$ is an $AM$-space with the Nakano property if and
 only if $E$ is order-isometric to $C_0(K)$ for some locally compact space $K$. 
\medskip

\section{Summary}

\noindent  In Chapter 2, we shall begin with our axiomatic definitions of multi-normed spaces and of their relatives, the dual multi-normed spaces; 
we shall obtain some immediate consequences and some characterizations. In particular, we shall show that, of course, the concept of a 
`dual multi-normed space' is dual to that of  `multi-normed space'.  We shall give alternative characterizations of {\mn} spaces in
 terms of matrices and of tensor products, and we shall show that  our notion of a {\mn} space coincides with that of spaces
 satisfying `condition (P)' of Pisier.

 In Chapter 3, we shall  give the first  examples of multi-normed spaces. These  are the minimum and the maximum multi-norms associated 
with a fixed normed space $E$.  The latter notion  leads to  a  sequence $(\varphi_n^{\,\max}(E))$ that is intrinsic to $E$. 
We shall relate this sequence to some  known sequences connected with the theory of absolutely summing operators; the background involving
 $p$-summing operators  will be reviewed.   We shall give various characterizations of the maximum multi-norm, and then  
 calculate the sequence  $(\varphi_n^{\,\max}(E))$  for  a variety of examples, including the spaces $\ell^{\,p}$.

  In Chapter 4,  we shall give several  specific examples of multi-norms, including the  $(p,q)$-multi-norm based on an arbitrary normed space,
the Hilbert multi-norm based on a Hilbert space, and the standard $q\,$-multi-norm based on 
  $L^{p}(\Omega)$ for $1\leq p\leq q$.  We shall compare these multi-norms, and determine in some cases when they are mutually equiv\-alent.
 This chapter concludes with the definition of the lattice multi-norm based on a Banach lattice; there is a representation theorem
 that shows that every multi-normed space is a  sub-multi-normed space of such an example. 
 
 In Chapter 5, we shall extend our theory to cover some multi-topological linear spaces, and shall discuss the notion of multi-convergence 
in these spaces, concentrating on the case of multi-convergence in multi-normed spaces.

In Chapter 6, we shall develop a theory of multi-bounded subsets of a multi-normed space and of multi-bounded and multi-continuous linear
 operators between multi-Banach spaces based on $E$ and $F$   that is parallel
 to the classical theory of continuous and  bounded linear operators between the Banach spaces $E$ and $F$. For a monotonically bounded Banach lattice,
 a subset is multi-bounded with respect to the lattice multi-norm if and only if it is order-bounded. The space  of multi-bounded operators
  ${\mathcal M}(E,F)$ is a  Banach operator algebra in   ${\B}(E,F)$, and can be given a natural multi-normed structure.  Examples show that 
sometimes ${\mathcal M}(E,F)$ coincides with ${\B}(E,F)$, but can coincide with ${\mathcal N}(E,F)$, the space of nuclear operators from 
$E$ to $F$.  The multi-normed space based on  ${\mathcal M}(E,F)$ is identified for various classes of Banach lattices.

In Chapter 7, our aim is to find a reasonable theory of `multi-dual spaces': we require a multi-norm based on $E'$, given a multi-norm based on a normed space $E$.
We shall  achieve this by first establishing a theory of direct sum decompositions of a normed space $E$ with respect to a multi-norm based on $E$,
 and then by using the duals of these decompositions  to generate a multi-norm based on $E'$.
\medskip

\section{History and acknowledgements}

\noindent This work was commenced in 2005 when Maksim Polyakov, from Moscow,  was a Marie--Curie Research Fellow  at the University of Leeds.

 Our motivation at that time was to seek to resolve some questions left open in \cite{DP}. In particular, we were concerned with the
 following question; for the definitions of the terms used, see \cite{DP}.   Let $G$  be a locally compact  group,
and let $L^1(G)$ be the group  algebra  of $G$. For each $p \in [1,\infty]$, the Banach space 
$L^{p}(G)$ is a Banach left  $L^1(G)$-module  in a natural way. We would like to know when these modules are injective in the 
appropriate category.  For $p=\infty$, this holds  for each locally compact group $G$ \cite[Theorem 2.4]{DP}; 
for $p=1$, this holds if and only if $G$ is discrete and  amenable  \cite[Theorem 4.9]{DP}.  
  Now suppose that $1< p< \infty$. Then $L^{p}(G)$  is a dual Banach left  $L^1(G)$-module, and so it follows from now standard results that
 $L^{p}(G)$ is injective whenever $G$ is an  amenable group. We conjectured that the converse  is true.
 In \cite[Theorem 5.12]{DP}, we proved that, in the case where $G$ is discrete and $\ell^{\,p}(G)$ is injective for some $p\in (1,\infty)$, 
the group $G$ is at least `pseudo-amenable'.  No example of a  pseudo-amenable  group which is not amenable is known; 
since such a group cannot contain the free group on two generators, there are very few candidates for such a group.  In fact this conjectured 
result has now been proved, and will be established (with other results) in \cite{DDPR1}.

We realised that  the above question, and other related questions, can be reformulated in the language of what we call `multi-Banach algebras', 
and we began to develop a theory of such algebras.  This required a substantial background in a new theory  of `multi-normed spaces'; this new
 theory came to life in its own right, and  it seems to be a useful framework in which many important concepts of functional analysis can be expressed, 
often generalizing known ideas to a wider situation.

Tragically,  Maksim Polyakov died in Moscow in January 2006  when this project had just been commenced.  I pay great tribute to
 this fine mathematician and colleague, and especially to his original ideas which underlie this work.

In due course, the project was continued by myself. Eventually it became apparent that the preliminary work on  multi-normed spaces 
 was so considerable that there should be one memoir  devoted just to this topic; this is the present work. Thus this work was developed 
 with particular applications in mind, but these applications will  not be discussed  here.  The  subsequent papers  \cite{DDPR1} and  \cite{DDPR2}
 will develop  a theory of   multi-normed spaces, with particular application to the theory of modules over the group algebras $L^1(G)$, 
where $G$ is a locally compact group; I anticipate a future paper on `multi-Banach algebras'.
\medskip

I acknowledge with thanks the financial support of the  original Marie--Curie International Fellowship, awarded for 2005-2007.  I also acknowledge
 with thanks the financial support of EPSRC grant EP/H019405/1 that enabled Hung Le Pham  to come to the University of Leeds  for  three months 
in 2010, during which time we discussed the present manuscript and its successors \cite{DDPR1,DDPR2}.

I am  very grateful  to  Matthew Daws (Leeds),  Mohammad Moslehian (Mashhad), Hung Le Pham (Wellington),  Paul Ramsden (Leeds), and 
Mar\-zieh Shamsi Yousefi (Teh\-eran) for   careful readings of various drafts of this work,   for pointing out  some errors, and for suggesting 
some changes and additions.  I am also very grateful to Graham Jameson (Lancaster), to Nigel Kalton (Columbia), to
Michael Elliott,  Stanislav Shkarin and Anthony Wickstead (Belfast), and  to Volker Runde  and  Vladimir Troitsky (Edmonton) for valuable background 
information  and references.  \medskip \medskip \medskip 

\noindent H.\ G.\ D., Lancaster, September, 2011

\chapter{The axioms and some  consequences}

\noindent We shall  now commence our study of multi-norms.\smallskip

\section{The axioms}

\subsection{Multi-norms} We begin with  our definition of a multi-norm. \smallskip

\begin{definition} \label{1.1}
Let $(E, \norm )$ be a complex (respectively, real) normed space, and take   $n\in \N$.  A {\it multi-norm of level $n$}  on  $\{E^k : k\in \N_n\}$ is a sequence  
 $(\norm_k) =(\norm_k : k\in \N_n)$
 such that $\norm_k$ is a norm on $E^k$ for each $k\in \N_n$, such that  $\lV x \rV_1 = \lV x \rV$ for each $x\in E$ {\rm (}so that 
$\norm_1$ is the {\rm initial norm}{\rm)},  and such that the following Axioms {\rm (A1)--(A4)} are satisfied for each $k\in \N_n$ with 
$k\geq 2$:\smallskip

{\rm (A1)} for each $\sigma \in  {\mathfrak S}_k$ and $x \in E^k$, we have 
$$
\lV A_\sigma(x) \rV_k =\lV x \rV_k\,;$$

{\rm (A2)} for each $\alpha_1,\dots,\alpha_k \in \C$  {\rm(}respectively,  each  $\alpha_1,\dots,\alpha_k \in \R${\rm)} and $x \in E^k $, we have 
$$\lV M_\alpha(x)\rV_k\leq {(\max_{i\in \N_k}}\lv\alpha_i\rv)\lV x \rV_k\,;
$$

{\rm (A3)} for each $x_1,\dots,x_{k-1} \in E$, we have
$$\lV (x_1,\dots, x_{k-1},0)\rV_{k}= \lV (x_1,\dots, x_{k-1})\rV_{k-1}\,;$$

{\rm (A4)} for each $x_1,\dots,x_{k-1}\in E$, we have
  $$\,\lV (x_1,\dots, x_{k-2}, x_{k-1},x_{k-1})\rV_{k}= \lV (x_1,\dots, x_{k-2},x_{k-1})\rV_{k-1}\,.$$

\noindent In this case, $((E^k, \norm_k) : k\in \N_n)$ is a {\rm multi-normed space of level $n$}.

A {\it multi-norm}  on  $\{E^k : k\in \N\}$  is a sequence
$$
(\norm_k) =(\norm_k : k\in \N)
$$
 such that $(\norm_k : k\in \N_n)$ is  a multi-norm of level $n$ for each $n \in \N$. In this case,   $((E^n, \norm_n) : n\in \N)$ is a  {\it multi-normed space}.
\end{definition}\smallskip

We shall sometimes say   that $(\norm_k : k\in \N)$ is a multi-norm  {\it based on $E$}.\s

 Let $(E, \norm )$ be a normed space.
  Then Axiom (A1) says that $A_\sigma$ is an isometry on $(E^k, \norm_k)$  whenever
 $\sigma \in {\mathfrak S}_k$, and Axiom (A2) says that $\lV M_\alpha \rV \leq 1$
 whenever $\alpha \in \overline{\D}^k$, where we regard $ M_\alpha$ as a bounded linear operator on $(E^k, \norm_k)$; in fact,
$$
\lV M_\alpha \rV =\max_{i\in \N_k}\lv \alpha_i\rv\quad (\alpha =(\alpha_1,\dots,\alpha_n)\in \C^n)\,.
$$

Note that Axioms (A1) and (A4) together say precisely that, for each $n \in \N$, the value of $\lV (x_1,\dots,x_n)\rV_n$ 
depends on only  the set $\{x_1,\dots,x_n\}$.

\subsection{Dual multi-norms} We shall also have some occasion to refer to a dual concept to that  of a multi-norm. We give the definition
 just in the case where the index set is $\N$, but there is also an obvious definition of `dual multi-normed space of level $n$'. 
The justification of the term `dual multi-normed space' will be apparent in $\S5$ of this chapter.\smallskip

\begin{definition} \label{1.1b}
Let $(E, \norm )$ be a normed space.  A {\it dual multi-norm}   on  $\{E^k : k\in \N\}$  is a sequence
$(\norm_k) =(\norm_k : k\in \N)$  such that $\norm_k$ is a norm on $E^k$ for each $k\in \N$, such that  $\lV x \rV_1 = \lV x \rV$ for each $x\in E$, and such that the 
  Axioms {\rm (A1), (A2), (A3)} and the following modified form of Axiom {\rm (A4)} are satisfied for each $k\in \N$ with $k\geq 2$:\smallskip

{\rm (B4)}  for each $x_1,\dots,x_{k-1}\in E$, we have
$$
\lV (x_1,\dots, x_{k-2}, x_{k-1},x_{k-1})\rV_{k}= \lV (x_1,\dots, x_{k-2}, 2x_{k-1})\rV_{k-1}\,.
$$
\noindent In this case, we say that $((E^k, \norm_k) : k\in \N)$ is a {\it dual multi-normed space}.
\end{definition}\smallskip

We sometimes say, in the above situation,  that $(\norm_k : k\in \N)$ is a dual multi-norm {\it based on $E$}.

Suppose that the normed spaces $(E, \norm)$ and $(E^{\/2}, \norm_2)$  satisfy just the case $k=2$ of both Axioms (A4) and (B4). Then
 $$\lV x\rV =\lV (x,x)\rV_2 = 2\lV x\rV \quad (x\in E)\,,
$$
 and so $E=\{0\}$.  Thus we should stress that a dual multi-normed space  is {\bf not} a multi-normed space unless $E=\{0\}$. \s
 
\subsection{Independence of the axioms}

\noindent It is natural to ask whether  the   four Axioms (A1)--(A4) are independent.  We give examples to show that this is indeed the case.\smallskip

\begin{example}\label{1.3}
{\rm Let $(E, \norm )$ be a non-zero normed space. 

We set $\lV x\rV_1 =\lV x\rV\,\;(x\in E)$, and,
for each $n\in\N$ with $n\geq 2$, set
$$
\lV (x_1,\dots, x_n)\rV_n = \max\left\{\lV x_1 \rV, \frac{\lV x_2 \rV}{2},\dots, \frac{\lV x_n \rV}{2}\right\}
 \quad ((x_1,\dots, x_n)\in E^n)\,.
$$
Then it is immediately checked that  $\norm_n$ is a norm on $E^n$ for each $n\in \N$, and that   $(\norm_n)$ is a sequence that satisfies 
Axioms (A2), (A3), and (A4) for each $n\in \N$.  However, take $x\in E$ with $\lV x \rV = 1$. Then $\lV (2x,3x)\rV_2 = 2$, but
 $\lV (3x,2x)\rV_2 = 3$, and so $\norm_2$ does not satisfy Axiom (A1).}\qed
\end{example}\smallskip

\begin{example}\label{1.4}
{\rm In this example, we work with $E =\C$. 

For $z\in \C$, we set $\lV z \rV_1 = \lv z \rv$. Next, for $(z,w ) \in \C^{\,2}$, set
$$
r((z,w)) = \frac{1}{2}(\lv z-w\rv + \lv z+w\rv)\,.
$$
Then $r$ is a norm on $\C^{\,2}$. Further, $r((z,z)) =  r((z,0)) = \lv z \rv\,\;(z\in \C)$ and also
$$r((z,w))= r((w,z))\geq \max\{\lv z \rv, \lv w \rv\}\quad ((z,w ) \in \C^{\,2})\,.
$$

Finally, for $n\in \N$ with $n\geq 2$, set
$$\lV (z_1,\dots, z_n)\rV_n=  \max\{ r((z_i,z_j)): i,j \in \N_n\}\quad ((z_1,\dots, z_n)\in \C^n)\,,
$$
so that $\lV (z,w)\rV_2 = r((z,w))\,\; ((z,w ) \in \C^{\,2})$ and$$\lV (z_1,\dots, z_n)\rV_n \geq \max_{i\in\N_n}\lv z_i\rv\quad((z_1,\dots, z_n)\in \C^n)\,.
$$
It follows easily that $\norm_n$ is a norm on $\C^n$  and that the sequence $(\norm_n)$ satisfies Axioms (A1), (A3), and (A4) for each $n\in \N$.

However we {\it claim\/} that $\norm_2$ does not satisfy Axiom (A2).  Indeed,
$$\lV (1,{\rm i})\rV_2 =\frac{1}{2}(\lv 1-{\rm i}\rv + \lv 1+{\rm i}\rv)
= \sqrt{2} > 1 = \lV (1,1)\rV_2\,,
$$
giving the claim.

Here is a similar example involving real spaces.  Let $E = \R$, and define
$$
\lV (x_1, \dots,x_n)\rV_n = \max\{\max_{i\in \N_n} \lv x_i\rv,\, \max_{i,j\in\N_n} \lv x_i - x_j\rv\}
$$
for $n\in\N$ and  $x_1, \dots,x_n \in \R$.  Then $(\norm_n : n\in\N)$ satisfies Axioms (A1), (A3), and (A4), but Axiom (A2) fails because
 $\lV (1,1)\rV_2 =1$, whilst $\lV (1,-1)\rV_2 =2$.}\qed
\end{example}\smallskip

We now consider the independence of Axiom (A3). The following example shows that for multi-norms of level $2$, (A3) 
is indeed independent of the other axioms.  However we shall see below that Axiom (A3) follows from Axioms (A1), (A2), and (A4) 
for multi-norms on the whole family $\{E^n : n\in \N\}$.\smallskip

\begin{example}\label{1.5}
{\rm  We again take $E=\C$, and set $\lV z \rV_1 = \lv z \rv\,\;(z\in \C)$. Set
$$
\lV (z,w)\rV_2 = \frac{1}{2}(\lv z \rv + \lv w \rv)\quad (z,w\in \C)\,.
$$
Then $\norm_2$ is a norm on $\C^{\,2}$, and $\norm_2$ satisfies Axioms (A1), (A2), and (A4) for $n=2$. 
 However $\lV (1,0)\rV_2 =1/2< 1 = \lV 1 \rV_1$, and so Axiom (A3) does not hold.}\qed
\end{example}\smallskip

\begin{example}\label{1.2}
{\rm  Let $(E, \norm )$ be a non-zero normed space. For each $n\in\N$, set
$$
\lV (x_1,\dots, x_n)\rV_n= \left(\sum_{j=1}^n\lV x_j\rV^{\,p}\right)^{1/p}\quad ((x_1,\dots, x_n)\in E^n)\,,
$$
where $p\geq 1$. Then it is immediately checked that, for each $p$, the function   $\norm_n$ is a norm on $E^n$, and that   $(\norm_n)$ is a sequence that
 satisfies Axioms (A1), (A2), and (A3) for each $n\in \N$,  but $\norm_2$ does not satisfy  Axiom (A4).
 
We note that the sequence  $(\norm_n : n\in\N)$ satisfies Axiom (B4) if and only if $p=1$; in  this latter case, $(\norm_n : n\in\N)$  is a  dual multi-norm.
}\qed
\end{example}\smallskip

We shall now show that Axiom (A3) follows from the other axioms in the case where we have norms on the whole  family $\{E^n : n\in\N\}$.\smallskip

\begin{proposition}\label{2.5}
Let $(E, \norm )$ be a normed space.   Let $(\norm_n : n\in \N)$  be a sequence such that $\norm_n$ is a norm on $E^n$ for each $n\in \N$, such that 
 $\lV x \rV_1 = \lV x \rV$ for each $x\in E$, and such that Axioms {\rm (A1)}, {\rm (A2)}, and {\rm (A4)} are satisfied for each $n\in \N$.  
Then  $(\norm_n: n\in \N)$ is a multi-norm on $\{E^n : n\in \N\}$.
\end{proposition}

\begin{proof} We must show that Axiom (A3) holds.

 Let $n\in \N$, and take $x= (x_i)\in E^n$, say $\lV x\rV_n =1$.  Set
$$
\alpha = \lV (x_1,\dots, x_n, 0)\rV_{n+1}\,,
$$
so that $0< \alpha \leq 1$ by (A2) and (A4).  For each $k\in\N$,  we see that $x^{[k+1]} \in E^{(k+1)n}$  and that $\lV x^{[k+1]} \rV _{(k+1)n} =1$
by (A1) and (A4). For $i\in \N_{n+1}$, let  $B_i$ be the subset
 $$
\{(i-1)k+1, \dots, ik\}
$$  of $\N_{(n+1)k}$, and let $Q_{B_i}$ be the projection onto the complement of $B_i$; by (A1) and (A4), we have
 $\lV Q_{B_i}(x^{[k+1]})\rV_{(k+1)n} =\alpha$. Further,
$$
\sum_{i=1}^{k+1}Q_{B_i}\left(x^{[k+1]}\right) = kx^{[k+1]}\,,
$$
and so
$$
k = k\lV x^{[k+1]} \rV _{(k+1)n} \leq  \sum_{i=1}^{k+1}\lV Q_{B_i}\left(x^{[k+1]}\right)\rV_{(k+1)n} =(k+1)\alpha\,,
$$
whence $\alpha \geq k/(k+1)$.  This holds for each $k \in \N$, and so $\alpha = 1$.

The result follows.
\end{proof}

Stanislav Shkarin has pointed out that Axiom (A3) also follows from Axioms (A1), (A2), and (B4), imposed on  the family $\{(E^n, \norm_n) :n\in\N\}$.\medskip

\section{Elementary consequences of the axioms}

\noindent   The following are immediate consequences of the axioms for multi-normed and dual multi-normed spaces.

\subsection{Results for special-norms}  A sequence $(\norm_k) =(\norm_k : k\in \N)$ such that $\norm_k$ is a norm on $E^k$ for each
 $k\in \N$, such that  $\lV x \rV_1 = \lV x \rV$ for each $x\in E$, and such that just the   Axioms  (A1), (A2), and  (A3) are satisfied
 is called a {\it special-norm}  in \cite{R1}, and $((E^n, \norm_n) : n\in \N)$ is then a {\it special-normed space}.  
Thus multi-norms and dual multi-norms are examples of special-norms.

Initially in this subsection, we suppose that $(E, \norm )$ is  a complex  normed space, that $n\in \N$, 
 and that $(\norm_k  : k\in \N_n)$  is a  special-norm. Thus our first results  apply to both multi-normed spaces and to dual multi-normed spaces of level $n$.

Trivial  modifications give entirely similar results when $(E, \norm )$ is  a real normed space.\smallskip

\begin{lemma} \label{2.0a}
Let $k\in\N_{n}$,  $x_1,\dots,x_k \in E$, and  $\zeta_1,\dots,\zeta_k \in \T$. Then
$$
\lV (\zeta_1x_1,\dots,\zeta_kx_k)\rV_k = \lV (x_1,\dots,x_k)\rV_k\,.
$$
\end{lemma}

\begin{proof}  This is immediate from (A2).\end{proof}\smallskip

\begin{lemma} \label{2.1}
Let $k\in\N_{n-1}$ and  $x_1,\dots,x_{k+1} \in E$. Then
$$
\lV (x_1,\dots, x_k)\rV_k \leq \lV (x_1,\dots, x_k, x_{k+1})\rV_{k+1}\,.
$$
\end{lemma}

\begin{proof} We have
\begin{eqnarray*}
 \lV (x_1,\dots, x_k)\rV_k &= &\lV (x_1,\dots, x_k,0)\rV_{k+1}\;\, \qquad \mbox {\rm by (A3)} \\
 &\leq& \lV(x_1,\dots,x_k, x_{k+1})\rV_{k+1} \quad \mbox{\rm by (A2)}\,,
\end{eqnarray*}giving the result.
\end{proof}\smallskip

\begin{lemma} \label{2.1a}
Let $j,k \in \N $ with $j+k\leq n$  and $x_1,\dots, x_j,y_1,\dots , y_k\in E$. Then
$$
\lV (x_1,\dots, x_j,y_1,\dots , y_k)\rV_{j+k} \leq \lV (x_1,\dots, x_j)\rV_{j}+\lV (y_1,\dots , y_k)\rV_{k}\,.
$$
\end{lemma}

\begin{proof} This is immediate from Axiom (A3).
\end{proof}\smallskip

\begin{lemma} \label{2.2}
Let $k\in \N_n$ and $\,x_1,\dots,x_{k} \in E$. Then
$$
\max_{i\in\N_k}\lV x_i\rV \leq \lV (x_1,\dots, x_k)\rV_k \leq \sum_{i=1}^k \lV x_i\rV\leq k\max_{i\in\N_k}\lV x_i\rV\,.
$$
 \end{lemma}

\begin{proof}  Set $x =(x_i)$. For $i\in\N_k$, we have
$\lV x_i\rV = \lV (0, \dots, 0, x_i, 0,\dots,0)\rV_k \leq\lV x  \rV_k$ by (A1), (A2), and (A3),  and so the stated inequalities follow.
\end{proof}\smallskip

It follows that any two special-norms on $\{E^k : k\in \N_n\}$  define the same topology on the space $E^k$
 for each $k\in \N_n\/$; the topology is the product topology.\smallskip

\begin{corollary} \label{2.3}
Suppose that $(E, \norm )$ is  a Banach space. Then the normed space $(E^k, \norm_k )$ is  a Banach space for each $k\in \N_n$.\qed
 \end{corollary}
\smallskip

As we remarked, the above results apply to both multi-normed spaces and to dual multi-normed spaces, and so, in the light of the above corollary, 
the following definition is reasonable.\smallskip

\begin{definition}\label{2.4}
Let   $((E^n, \norm_n) : n\in \N)$ be a  multi-normed space (respectively, dual multi-normed space)  for which
 $(E, \norm)$ is a Banach space.  Then $((E^n, \norm_n) : n\in \N)$ is a {\it multi-Banach space}  (respectively, {\it dual multi-Banach space}).
\end{definition}\s

More generally, we can refer to {\it special-Banach spaces\/}.
\medskip

\subsection{Results for multi-norms} We now give some   elementary lemmas that suppose, further, that the sequence $(\norm_k : k\in \N_n)$  
also satisfies Axiom (A4), and hence that  $((E^k, \norm_k): k\in\N_n)$ is a multi-normed space of level $n$, where $n\in\N$. \smallskip

\begin{lemma} \label{2.0}
Let $k\in\N_{n}$ and $x\in E$. Then $\lV (x,\dots, x)\rV_k = \lV x \rV$.
\end{lemma}

\begin{proof}  This is immediate from (A4).\end{proof}\smallskip

\begin{lemma} \label{2.1b}
Let  $j,k \in \N_n$  and $x_1,\dots,x_{j}, y_1,\dots, y_k\in E$ be such that   $\{x_1,\dots, x_j\}$ is a subset of $  \{y_1, \dots, y_k\}$. Then 
$$
 \lV (x_1,\dots, x_j)\rV_j \leq  \lV (y_1,\dots, y_k)\rV_k\,.
$$\end{lemma}

\begin{proof} By Axioms (A1) and (A4), we may suppose that $j\leq k$ and that $x_i =y_i\,\;(i\in \N_j)$.  Now the result follows from Lemma \ref{2.1}.
\end{proof}\smallskip

\begin{lemma} \label{2.1c}
Let $k\in \{2,\dots,n\}$  and $x_1,\dots, x_k \in E$.  Take $\alpha,\beta \in \I$ with $\alpha + \beta =1$, and set $x = \alpha x_{k-1}+\beta x_k$. Then
$$
\lV (x_1, \dots, x_{k-2}, x, x)\rV_k \leq \lV (x_1, \dots, x_{k-2}, x_{k-1}, x_k)\rV_k\,.
$$
\end{lemma}

\begin{proof} Set $y= (x_1, \dots, x_{k-2})$ and $A = \lV (y, x_{k-1}, x_k)\rV_k$. Then 
$$
(y, x, x) = \alpha^2 (y, x_{k-1}, x_{k-1}) + \alpha\beta (y, x_{k-1}, x_k) + \alpha\beta (y, x_k, x_{k-1})+ \beta^2(y, x_k, x_k)\,.
$$
But $\lV (y, x_k, x_{k-1})\rV_k =A$ by (A1). Also $\lV (y, x_{k-1}, x_{k-1})\rV_k \leq  A$ and $\lV (y, x_k, x_k)\rV_k \leq  A$ by Lemma \ref{2.1b}.  Hence 
$$
\lV (x_1, \dots, x_{k-2}, x, x)\rV_k \leq (\alpha + \beta)^2A =A\,,
$$
giving the result.
\end{proof}\smallskip

The following {\it inequality-of-roots\/} will be useful later.\smallskip

\begin{proposition}\label{2.5a}
Let   $((E^n, \norm_n) : n\in \N)$ be a  multi-normed space, and take $k \in \N$. Set $\zeta_k = \exp\left({2\pi{\rm i}}/{k}\right)$.
Then 
\begin{equation}\label
{(2.1)}
\lV (x_1,\dots, x_k)\rV_k \leq \frac{1}{k}\sum_{j=1}^k \lV \sum_{m=1}^k \zeta_k^{jm}x_m\rV\quad (x_1,\dots, x_k \in E)\,.
\end{equation}
\end{proposition}

\begin{proof} We write $\zeta$ for $\zeta_k$.

First note that
$$
x_{\ell} = \frac{1}{k}\sum_{m=1}^k\sum_{j=1}^k \zeta^{j(m-\ell )}x_m\quad (\ell \in \N_k)
$$
because $\sum_{j=1}^k \zeta^{j(m-\ell )}= 0$ when $m \neq \ell$ and  $\sum_{j=1}^k \zeta^{j(m-\ell )}= k$ when $m=\ell$.  Thus
$$
\lV (x_1,\dots, x_k)\rV_k \leq \frac{1}{k} \sum_{j=1}^k
\lV \left(\sum_{m=1}^k \zeta^{j(m-1 )}x_m, \dots ,\sum_{m=1}^k \zeta^{j(m-k)}x_m\right)\rV_k\,.
$$
For $j \in \N_k$,  set $y_j = \sum_{m=1}^k\zeta^{jm}x_m\,\;(j\in \N_k)$.  Then
$$
\lV \left(\sum_{m=1}^k \zeta^{j(m-1 )}x_m, \dots ,\sum_{m=1}^k \zeta^{j(m-k)}x_m\right)\rV_k
=\lV (\zeta^{-j}y_j,\dots, \zeta^{-kj}y_j)\rV_k \,.
$$
But $\lV (\zeta^{-j}y_j,\dots, \zeta^{-kj}y_j)\rV_k  = \lV y_j \rV\,\;(j\in \N_k)$ by (A2)
 and Lemma \ref{2.0}, and  so
inequality (\ref{(2.1)}) follows.\end{proof}\smallskip

\begin{corollary}\label{2.1bz}
Let $E=\ell^{\,r}$, where $r\geq 1$, and let $(\norm_n : n\in\N)$ be a multi-norm based on $E$. Then 
$$
\lV (\delta_1,\dots,\delta_k)\rV_k  \leq k^{1/r}\quad (k\in\N)\,.
$$
\end{corollary}

\begin{proof}  In this case, 
$$
\lV \sum_{m=1}^k \zeta_k^{jm}\delta_m\rV = \lV \left(\zeta_k^j, \dots,\zeta_k^{kj}\right)\rV_{\ell^{\,r}} =k^{1/r}
$$
for each $j\in\N_k$, and so the result follows from the proposition
\end{proof}\s

\subsection{Results for dual multi-norms}  We now have some elementary lemmas about dual multi-normed spaces. In the remainder of this section,
 we suppose that $(E, \norm )$ is  a normed space   and that $((E^k, \norm_k) : k\in \N)$ is a dual multi-normed space, 
and so the sequence $(\norm_k  : k\in \N)$ satisfies Axioms (A1)--(A3) and Axiom (B4). \smallskip

\begin{lemma} \label{2.10}
Let $k \in \N$ and $x_1,\dots,x_k \in E$. Then
$$
\lV (x_1,\dots,x_{k-2}, x_{k-1}+x_k)\rV_{k-1} \leq \lV (x_1,\dots,x_{k-2}, x_{k-1}, x_k)\rV_{k}\,.
$$
\end{lemma}

\begin{proof}  We have
\begin{eqnarray*}
\lV (x_1,x_{k-2}, x_{k-1}+x_k)\rV_{k-1} 
&=&
\lV (x_1,\dots,x_{k-2}, (x_{k-1}+x_k)/2, (x_{k-1}+x_k)/2 \rV_{k}
 \quad \mbox{\rm by  (B4)}\\
&=&
\frac{1}{2}\lV (x_1,\dots,x_{k-2}, x_{k-1}, x_k) + (x_1,\dots,x_{k-2}, x_k, x_{k-1})\rV_k \\
&\leq&
\lV (x_1,\dots,x_{k-2}, x_{k-1}, x_k)\rV_{k}\quad\mbox{\rm by (A1)}\,,
\end{eqnarray*} as required.
\end{proof}\smallskip
 
\begin{lemma} \label{2.10a}
Let $k \in \N$ and $x_1,\dots,x_k \in E$. Then
$$
\sup\{\lV\zeta_1x_1+\cdots+ \zeta_kx_k \rV : \zeta_1,\dots,\zeta_k \in\T\}\leq \lV (x_1,\dots,x_k)\rV_k\,.
$$
\end{lemma}

\begin{proof}  This follows from Lemmas \ref{2.0a} and \ref{2.10}.
\end{proof}\smallskip

\begin{lemma} \label{2.11}
Let $m,n \in \N$ with $m \leq n$, let   $x \in E^n$, and let $y \in E^m$ be a coagulation of $x$. Then $\lV y\rV_m \leq \lV x\rV_n$.
\end{lemma}

\begin{proof}  This follows from Lemma \ref{2.10}.
\end{proof}\smallskip

\begin{lemma} \label{2.12}
Let $k\in \N$,   $\alpha_1,\dots,\alpha_k \in \C$, and   $x \in E$. Then
$$
\lV (\alpha_1x, \dots,\alpha_kx)\rV_k = \left(\sum_{j=1}^k\lv\alpha_j\rv\right)\lV x \rV\,.
$$
\end{lemma}

\begin{proof} By Lemma \ref{2.2}, we have
$$
\lV (\alpha_1x, \dots,\alpha_kx)\rV_k \leq \left(
\sum_{j=1}^k\lv\alpha_j\rv\right)\lV x \rV\,.
$$
But also
\begin{eqnarray*}
\lV (\alpha_1x, \dots,\alpha_kx)\rV_k &=&
 \lV (\lv\alpha_1\rv x, \dots,\lv \alpha_k\rv x)\rV_k
\quad \mbox{\rm by Lemma \ref{2.0a}}\\
&\geq &\lV \sum_{j=1}^k\lv\alpha_j\rv x\rV= \left(\sum_{j=1}^k\lv\alpha_j\rv\right)\lV x \rV\qquad\quad \mbox{\rm by Lemma \ref{2.10}}\,.\\
\end{eqnarray*}The result follows.
\end{proof}\s 

\subsection{The family of multi-norms} We first have an elementary result.\s

\begin{proposition}\label{2.5d}
Let $(E, \norm)$ be a normed space. Take  $n\in \N$, and let   $( \norm^1_k : k\in\N_n)$  and $(  \norm^2_k : k\in\N_n)$  be  two multi-norms of level $n$ on the family
$\{E^k: k\in\N_n\}$. For $k\in \N_n$ and $x_1,\dots,x_k \in E$, set
$$
\lV (x_1,\dots,x_k)\rV_k = \max\left\{\lV (x_1,\dots,x_k)\rV^1_k, \lV (x_1,\dots,x_k)\rV^2_k \right\}\,.
$$
Then $((E^k, \norm _k): k\in\N_n)$ is a multi-normed space  of level $n$.
\end{proposition}

\begin{proof}  This is immediately checked.\end{proof}\smallskip

 We now define a family of multi-norms.\s

\begin{definition}\label{2.5e}
 Let $(E, \norm)$ be a normed space. Then ${\mathcal E}_E$ is the family of all multi-norms based on $E$. Let $( \norm^1_k : k\in\N)$ and $( \norm^2_k : k\in\N)$
  belong to ${\mathcal E}_E$. Then  
$$
( \norm^1_k : k\in\N)\leq  ( \norm^2_k : k\in\N)
$$
 if
$$\lV (x_1,\dots,x_k)\rV^1_k\leq  \lV (x_1,\dots,x_k)\rV^2_k\quad (x_1,\dots,x_k \in E,\,k\in \N)\,.
$$
Further, the multi-norm $( \norm^2_k : k\in\N)$ {\it dominates\/}  the  multi-norm $( \norm^1_k : k\in\N)$, written 
$$
( \norm^1_k : k\in\N)\preccurlyeq ( \norm^2_k : k\in\N)\,,
$$
  if there is a constant $C > 0$  such that
\begin{equation}\label{(2.5)}
\lV (x_1,\dots,x_k)\rV^1_k\leq  C\lV (x_1,\dots,x_k)\rV^2_k\quad (x_1,\dots,x_k \in E,\,k\in \N)\,.
\end{equation}
The two multi-norms $( \norm^1_k : k\in\N)$ and $( \norm^2_k : k\in\N)$  are {\it \/equivalent}, written
$$
( \norm^1_k : k\in\N)\cong ( \norm^2_k : k\in\N)\,,
$$
 if each dominates the other.
  \end{definition}

It is clear that $({\mathcal E}_E, \leq )$  is  a partially ordered set; by Proposition \ref{2.5d}, 
each pair of elements has an upper bound. We shall see in  Proposition \ref{4.10b} that  $({\mathcal E}_E, \leq )$ is  a Dedekind-complete lattice. 

 There is an entirely similar ordering of, and notion of equivalence for, the family of dual multi-norms on $\{E^k: k\in\N_n\}$.

In \cite{DDPR2}, we shall explore when various specific multi-norms are mutually  equivalent, and sometimes calculate  the best constant $C$
  in equation (\ref{(2.5)}).\s

\subsection{Standard constructions}

\noindent We now give some standard constructions that generate
 new multi-normed spaces from old ones.  Analogous constructions also generate new dual multi-normed spaces.\smallskip

Let   $((E^n, \norm_n) : n\in \N)$ be a  multi-normed space, and let $F$ be a closed linear subspace of $E$.
 For $n\in\N$ and $x_1, \dots, x_n \in E$, define
$$
\lV (x_1 + F,\dots, x_n +F)\rV_n = \inf\{\lV (y_1,\dots,y_n)\rV_n: y_i\in x_i+F \,\;(i\in \N_n)\}\,,
$$
so that $\norm_n$ is a norm  on $(E/F)^n$.\smallskip

\begin{proposition} \label{2.4e}
Let   $((E^n, \norm_n) : n\in \N)$ be a  multi-normed space.\smallskip

{\rm (i)}  Let $F$ be a  linear subspace of $E$. Then $((F^n, \norm_n) : n\in \N)$ is a  multi-normed space.\smallskip

{\rm (ii)}  Let $F$ be a closed linear subspace of $E$. Then  $(((E/F)^n, \norm_n) : n\in \N)$ is a  multi-normed space.
\end{proposition}

\begin{proof} These are easily checked; to show that each norm $\norm_n$ on $(E/F)^n$ satisfies (A4), we use Lemma \ref{2.1c}.\end{proof}\smallskip

We say that $((F^n, \norm_n) : n\in \N)$ and $(((E/F)^n, \norm_n) : n\in \N)$ are  a {\it multi-normed subspace\/} 
 and a {\it multi-normed quotient space\/}, respect\-ively,  of the multi-normed space  $((E^n, \norm_n) : n\in \N)$.\smallskip

\begin{proposition} \label{3.2ac}
Let $F$ be a $1$-complemented subspace of a normed space $E$, and suppose that $(\norm_n : n\in\N)$ is a multi-norm 
on $\{F^n : n\in\N\}$. Then there is a multi-norm $(\Norm_n : n\in\N)$ on $\{E^n : n\in\N\}$ such that $((F^n, \norm_n): n\in\N)$ is a
{\mn}  subspace of $((E^n, \Norm_n): n\in\N)$.
\end{proposition}

\begin{proof}  Let $P: E\to F$ be a projection onto $F$ with $\lV P\rV =1$, and set 
$$
\LV (x_1,\dots,x_n)\RV_n =\max\{\lV x_1\rV, \dots, \lV x_n\rV, \lV (Px_1,\dots,Px_n)\rV_n\}
$$
for $x_1,\dots,x_n \in E$. Then the sequence $(\Norm_n : n\in\N)$ has the required properties, as is easily checked.
\end{proof}\s

\begin{proposition} \label{2.4a}
Let   $((E^n, \norm_n) : n\in \N)$ be a  multi-normed space, and let $k \in \N$. Set 
$F = E^k$ and $\norm_F  = \norm_k$. Then $(F, \norm_F)$ is a normed space, and $((F^n, \norm_{nk}) : n\in \N)$ is a  multi-normed space.
\end{proposition}

\begin{proof}  Let $y_1,\dots, y_n\in F$, say $y_i = (x_{i,1},\dots, x_{i,k}) \,\;(i \in \N_n)$.  Then
$$
\lV (y_1,\dots, y_n)\rV_n = \lV (x_{1,1},\dots, x_{1,k}, \dots, x_{n,1},\dots, x_{n,k}) \rV_{nk}\,,
$$
and   $((F^n, \norm_{nk}) : n\in \N)$ is clearly a  multi-normed space.
\end{proof}\medskip

  Let   $\{((E^n_\alpha, \norm^\alpha_n) : n\in \N): \alpha \in A\}$ be a family of  multi-normed spaces, defined for each $\alpha$ in a non-empty index set $A$ (perhaps finite).   Then we consider the following spaces.

First, for $n\in \N$ and $(x^1_\alpha), \dots,  (x^n_\alpha) \in \ell^{\,\infty}(E_\alpha)$, set
$$
\lV ((x^1_\alpha), \dots,  (x^n_\alpha)) \rV_n = \sup\left\{\lV (x^1_\alpha, \dots, x^n_\alpha)\rV^\alpha_n :\alpha \in A\right\}\,.
$$

\begin{proposition} \label{2.4b}
The space  $((\ell^{\,\infty}(E_\alpha)^n, \norm_n): n\in\N)$ is a multi-normed space.\end{proposition}

 \begin{proof}  This is   immediately checked.\end{proof}\

Take $p$ with $1\leq p< \infty$.
 For $n\in \N$ and $(x^1_\alpha), \dots,  (x^n_\alpha) \in \ell^{\,p}(E_\alpha)$, we define
$$
\lV ((x^1_\alpha), \dots,  (x^n_\alpha)) \rV_n = \left(\sum_\alpha \left(\lV (x^1_\alpha, \dots, x^n_\alpha) \rV_n^\alpha\right)^{\,p}\right)^{1/p}\,.
$$

\begin{proposition} \label{2.4c}
The space  $((\ell^{\,p}(E_\alpha))^n, \norm_n): n\in\N)$ is a multi-normed space.\end{proposition}

 \begin{proof} We must show that  $\lV ((x^1_\alpha), \dots,(x^n_\alpha)) \rV_n$, as defined above, is finite in each case. Indeed,
$$
\left(\sum_\alpha \left(\lV (x^1_\alpha, \dots,x^n_\alpha) \rV_n^\alpha\right)^{\,p}\right)^{1/p}
 \leq
\left(\sum_\alpha \left(\lV x^1_\alpha\rV^\alpha + \cdots + \lV x^n_\alpha  \rV_n^\alpha\right)^{\,p}\right)^{1/p}
$$
by Lemma \ref{2.2}, and so, by Minkowski's inequality,
$$
\lV ((x^1_\alpha), \dots,  (x^n_\alpha)) \rV_n \leq
\left(\sum_\alpha \left( \lV x^1_\alpha\rV^\alpha\right)^{\,p}   \right)^{1/p}
+ \cdots +
\left(\sum_\alpha \left( \lV x^n_\alpha\rV^\alpha\right)^{\,p}   \right)^{1/p}\,,
$$
and the right-hand side is finite.

The triangle inequality for $\norm_n$ also follows from Minkowski's inequality, and the remainder is easy to check.
\end{proof}\smallskip

In particular, let $((E^n, \norm_n) : n\in \N)$ and $((F^n, \norm_n) : n\in \N)$  be  multi-normed spaces.
  Set $G = E\oplus F$.  For $n \in \N$, define  $\norm_n$ on $G^n$ by taking
 $\lV (x_1+y_1, \dots, x_n+y_n) \rV_n $ to be either
$$
\max\{\lV (x_1, \dots, x_n )\rV_n, \lV (y_1, \dots,y_n) \rV_n\}
\quad\!
{\rm  or}
\quad\!
\lV (x_1, \dots, x_n) \rV_n + \lV (y_1, \dots,y_n )\rV_n
 $$
for $x_1, \dots, x_n\in E$ and $y_1, \dots,y_n \in F$. Then
 $((G^n, \norm_n) : n\in \N)$ is a  multi-normed space, denoted by
$$ (((E\oplus_\infty F)^n, \norm_n) : n\in \N)\quad {\rm or} \quad (((E\oplus_1 F)^n, \norm_n) : n\in \N)\,,
$$
 respectively.\medskip

\section{Theorems on duality}\label{Theorems on duality}

\noindent
  In this section, we shall justify the term `dual multi-normed space'.\s

\subsection{Special-normed spaces}  Let  $(E, \norm)$ be a normed space,   let $k\in\N$, and let $\norm_k$ be any norm on the space $E^k$.
As before,  the dual norm on the space $(E')^k$ is denoted by $\norm_k'$, so that, explicitly, 
$$
 \lV (\lambda_1,\dots,\lambda_k)\rV_k' = \sup\left\{\lv \sum_{j=1}^k
 \langle x_j,\,\lambda_j\rangle\rv :  \lV (x_1,\dots,x_k)\rV_k \leq 1\right\}
 $$
for $\lambda_1,\dots,\lambda_k\in E'$,  taking the supremum over $ x_1,\dots,x_k\in E$.

Now let $((E^k, \norm_k): k\in\N)$ be a special-normed space.
Then it follows from Lemma \ref{2.2} and Axiom (A3) that each norm $\norm_k$ satisfies  the equations
(\ref{(1.5)}) and (\ref{(1.6)}) (with $\norm_k$ for $\Norm$), and so $((E^k)', \norm_k')$ is linearly homeomorphic 
to $(E')^k$ (with the product topology from $E'$).  Thus we have defined a sequence $(\norm_k' : k\in\N)$
 such that $\norm_k'$ is a norm on $(E')^k$ for each $k\in \N$.  Clearly $\lV \lambda\rV_1' = \lV \lambda\rV' $
 for each $\lambda \in E'$.\smallskip
 
 \begin{proposition}\label{2.13b}
 Let $((E^k, \norm_k): k\in\N)$ be a special-normed space.  Then  it also holds that $(((E')^k, \norm_k'): k\in\N)$  is a special-Banach space.
 \end{proposition}

\begin{proof}  It is clear that Axioms (A1) and (A2) for   $((E^k, \norm_k): k\in\N)$ imply, respectively,
 that  (A1) and (A2) hold for $(((E')^k, \norm_k'): k\in\N)$.

Take $k\geq 2$ and $\lambda_1,\dots,\lambda_{k-1}\in E'$.  For each $x_1,\dots,x_k \in E$, it follows from Lemma \ref{2.1}
that  $\lV(x_1,\dots,x_{k-1})\rV_{k-1}\leq \lV(x_1,\dots,x_{k-1}, x_k)\rV_{k}$, and so
$$
\lV (\lambda_1,\dots,\lambda_{k-1},0)\rV_k' \geq \lV(\lambda_1,\dots,\lambda_{k-1})\rV_{k-1}'\,.
$$
Thus $(\norm_k' : k\in\N)$  satisfies  (A3).
 \end{proof}\s
 
\subsection{Multi-normed and dual multi-normed spaces}
 
 We now establish the duality that we are seeking.  Throughout, $(E,\norm)$ and  $(F,\norm)$ are normed spaces.\s

\begin{theorem}\label{2.13}
Let $((E^k, \norm_k): k\in\N)$ be a multi-normed space.  Then  $$(((E')^k, \norm_k'): k\in\N)$$  is a dual multi-Banach space.\end{theorem}

\begin{proof}  By Proposition \ref{2.13b}, it suffices to show that $(((E')^k, \norm_k'): k\in\N)$ satisfies  (B4).

Fix $ \lambda_1,\dots,\lambda_{k-1} \in E'$, and set
$$
A = \lV (\lambda_1,\dots,\lambda_{k-2}, \lambda_{k-1}, \lambda_{k-1})\rV'_k\,,\quad
B  = \lV (\lambda_1,\dots,\lambda_{k-2}, 2\lambda_{k-1})\rV'_{k-1}\,.
$$
Take $\varepsilon >0$.

 First choose  $(x_1,\dots,x_k) \in (E^k, \norm_k)_{[1]}$ with
$$
\lv  \sum_{j=1}^{k-2}\langle x_j,\,\lambda_j\rangle +  \langle x_{k-1},\,\lambda_{k-1}\rangle + \langle x_{k},\,\lambda_{k-1}\rangle \rv > A - \varepsilon\,.
$$
Set $x =(x_{k-1}+x_k)/2$, so that it follows from  Lemma \ref{2.1c} and (A4) that  we have
$(x_1,\dots,x_{k-2}, x)  \in (E^{k-1}, \norm_{k-1})_{[1]}$, and hence
$$
B\geq\lv  \sum_{j=1}^{k-2}\langle x_j,\,\lambda_j\rangle + \langle x,\,2\lambda_{k-1}\rangle\rv 
= \lv  \sum_{j=1}^{k-2}\langle x_j,\,\lambda_j\rangle +  \langle x_{k-1},\,\lambda_{k-1}\rangle +
 \langle x_{k},\,\lambda_{k-1}\rangle \rv > A - \varepsilon\,.
$$

Second, choose  $(x_1,\dots,x_{k-1}) \in (E^{k-1}, \norm_{k-1})_{[1]}$ with
$$
\lv  \sum_{j=1}^{k-2}\langle x_j,\,\lambda_j\rangle + \langle x_{k-1},\,2\lambda_{k-1}\rangle\rv> B -\varepsilon\,.
$$
Then  $(x_1,\dots,x_{k-1}, x_{k-1})\in (E^{k},\norm_k)_{[1]}$ by (A4), and so
$$
A \geq \lv  \sum_{j=1}^{k-2}\langle x_j,\,\lambda_j\rangle + \langle x_{k-1},\,2\lambda_{k-1}\rangle\rv > B-\varepsilon\,.
$$

The above two inequalities hold for each $\varepsilon >0$, and so
 {$A=B$.} 

Thus the sequence $(\norm_k' : k\in\N)$  satisfies Axiom (B4), and hence  we have shown that 
  $(((E')^k, \norm_k'): k\in\N)$  is a dual multi-Banach space.
\end{proof}\smallskip

\begin{definition}\label{2.13a}
Let $((E^k, \norm_k): k\in\N)$ be a multi-normed space.  Then  $$(((E')^k, \norm_k'): k\in\N)$$  is the 
{\it dual multi-Banach space} of the space  $((E^k, \norm_k): k\in\N)$.\end{definition}\s

\begin{theorem}\label{2.14}
Let $((F^k, \norm_k): k\in\N)$ be a dual multi-normed space.  Then  $$(((F')^k, \norm_k'): k\in\N)$$  is a  multi-Banach space.\end{theorem}

\begin{proof} It suffices to show that   $(((E')^k, \norm_k'): k\in\N)$ satisfies Axiom (A4).

Fix $ \lambda_1,\dots,\lambda_{k-1} \in F'$, and set
$$
A = \lV (\lambda_1,\dots,\lambda_{k-2}, \lambda_{k-1}, \lambda_{k-1})\rV'_k\,,\quad
B  = \lV (\lambda_1,\dots,\lambda_{k-2},  \lambda_{k-1})\rV'_{k-1}\,.
$$
Take $\varepsilon >0$.

 First choose  $(x_1,\dots,x_k) \in (F^k, \norm_k)_{[1]}$ with
$$
\lv \sum_{j=1}^{k-1}\langle x_j,\,\lambda_j\rangle + \langle x_k,\,\lambda_{k-1}\rangle \rv > A -\varepsilon\,.
$$
Then $(x_1,\dots,x_{k-2}, x_{k-1} +x_k) \in (F^{k-1},\norm_{k-1})_{[1]}$ by Lemma \ref{2.10}, and so
$$
B \geq \lv \sum_{j=1}^{k-2}\langle x_j,\,\lambda_j\rangle +\langle x_{k-1}+x_k,\,\lambda_{k-1}\rangle\rv> A -\varepsilon\,.
$$

Second, choose  $(x_1,\dots,x_{k-1}) \in (F^{k-1}, \norm_{k-1})_{[1]}$ with
$$
\lv \sum_{j=1}^{k-1}\langle x_j,\,\lambda_j\rangle \rv> B-\varepsilon\,.
$$
Then  $(x_1,\dots,x_{k-1}, 0) \in (F^{k}, \norm_k)_{[1]}$ by (A3), and so $A> B-\varepsilon$.

It follows that $A=B$, and so the sequence $(\norm_k' : k\in\N)$  satisfies Axiom  (A4). Thus $(((F')^k, \norm_k'): k\in\N)$ 
 is a  multi-Banach space.\end{proof}\smallskip

Let   $((E^k, \norm_k): k\in\N)$ be a multi-normed space.  Then, for each $k\in\N$,   the norm  on $(E'')^k$ which is the dual norm 
to $\norm_k'$ on   $(E')^k$  is temporarily denoted by  $\norm_k''$. It is clear from Theorems \ref{2.13} and \ref{2.14} that
 \mbox{$(((E'')^k, \norm_k''): k\in\N)$} is  a multi-Banach  space.  Of course the embedding of each space $(E^k, \norm_k)$ into $((E'')^k, \norm_k'')$
 is an isometry of normed spaces, and so we can write  $\norm_k $ consistently for $\norm_k''$ on $(E^k)''$.  Thus we have the 
following conclusion.\smallskip

\begin{theorem}\label{2.15}
Let $((E^k, \norm_k): k\in\N)$ be a multi-normed space. Then $$((E^k, \norm_k): k\in\N)$$ is a multi-normed subspace of the multi-Banach 
space $(((E'')^k, \norm_k): k\in\N)$.\qed
\end{theorem}
\medskip

\section{Reformulations of the axioms}

\noindent  In this section, we shall give some  reformulations of the axioms for a multi-normed space $((E^n, \norm_n) : n\in \N)$.\s

\subsection {Multi-norms and matrices}  Again, let $E$ be a linear space, and suppose that $m, n\in \N$. We have remarked that    ${\mathbb M}_{\,m,n}$ acts 
as a map from $E^n$ to $E^m$ in the obvious way; in particular,   $E^n$ is a left ${\mathbb M}_{\,n}$-module.  Our reformulation requires 
these actions to be `Banach' actions, so that, for each $m, n\in\N$, we have
$$
\lV a\,\cdot\,x \rV_m \leq \lV a\rV \lV x \rV_n\quad (x \in E^n,\,a \in {\mathbb M}_{\,m,n})\,,
$$
where we recall that $\lV a \rV$ is an abbreviation of $\lV a : \ell_n^{\,\infty}\to \ell_m^{\,\infty}\rV$. In particular, 
$E^n$ is a Banach left ${\mathbb M}_{\,n}$-module.  See \cite{D} for  a discussion of  the theory of Banach left $A$-modules over a Banach algebra $A$.

We first give some preliminary notions.  Let $m,n\in\N$, and let
$$
a = (a_{ij})\in   {\mathbb M}_{\,m,n}\,.
$$
Then $a$ is a {\it row-special\/} matrix if, for each $i\in \N_m$,  there is at most one non-zero term, 
say $a_{i,j(i)}$, in the
$i^{\rm th.}$ row, the term $a_{i,j(i)}$ being in the $j(i)^{\rm th.}$ column.

We {\it claim\/}  that each $ a= (a_{ij})\in   {\mathbb M}_{\,m,n}$ can be written as
$$
a = \sum_{r=1}^k a_r\,,
$$
 where $a_1, \dots, a_k$ are row-special matrices in $ {\mathbb M}_{\,m,n}$ and
$$
\lV a \rV = \sum_{r=1}^k\lV  a_r\rV\,.
$$

To prove  this claim, we may suppose that $a\neq 0$. For each $i \in \N_m$ such that the $i^{\rm th.}$ row of $a$ is non-zero, choose
 $j(i) \in \N_n$ to be the maximum number $j\in \N_n$ such that $a_{ij}\neq 0$, and set
$$
c_i =a_{i,j(i)}\quad(i\in\N_n)\,,
$$
 taking $c_i =0$ when the $i^{\rm th.}$ row of $a$ is zero. Then choose $i_0\in\N_n$ such that
$$
\lv c_{i_0}\rv= \min\{ \lv c_i \rv : c_i \neq 0,\,i\in\N_m\}\,.
$$
  Finally, define a matrix $b \in {\mathbb M}_{m,n}$
by setting
$$
b_{i,j(i)}= \frac{c_i}{\lv c_i\rv }\lv c_{i_0}\rv \quad (i\in\N_m),
$$
(with $b_{i,j} =0\,\;(j\in \N_n)$ whenever   the $i^{\rm th.}$ row of $a$ is zero), and setting $b_{r,s} = 0$ whenever $(r,s) \neq (i,j(i))$ 
for any $i\in \N_n$.  The matrix $b$ is row-special.  Further, we can see from equation (\ref{(2.3)}) that $\lV b \rV = \lv c_{i_0}\rv $.
 The coefficients of the matrix $a-b$ are the same as those of $a$, save that, for each $i\in \N_n$ for which  the $i^{\rm th.}$ row of $a$ 
is non-zero,  the  coefficient $a_{i,j(i)}$ has been replaced by
$$
 a_{i,j(i)}\left( 1 - \frac{\lv c_{i_0}\rv}{\lv c_i \rv}\right)= c_i\left( 1 - \frac{\lv c_{i_0}\rv}{\lv c_i \rv}\right)\,,
$$
 and so $\sum_{j=1}^n \lv a_{ij}\rv$ is replaced by $\sum_{j=1}^n \lv a_{ij}\rv - \lv c_{i_0}\rv \geq 0$, and $ a_{i_0,j(i_0)}$ becomes 
$0$. Note that no zero term in the matrix $(a_{ij})$ is changed. It follows immediately that  $\lV a-b\rV = \lV a \rV -  \lv c_{i_0}\rv $, and so
$\lV a-b\rV + \lV b\rV =\lV a \rV$.

We continue to decompose $a-b$ in a similar way; after at most $mn$ steps, the process must terminate, and then we have the claimed 
representation of the matrix $a$.\smallskip

\begin{theorem}\label{2.5b}
Let  $(E, \norm)$ be a normed space, and take  $N \in\N$. Suppose that, for each $ n\in \N_N$, $ \norm_n$ is a  norm  on the space  $E^n$  and, further,  
 that $ \lV x \rV_1 =\lV x \rV\,\;(x \in E)$.  Then the following are equivalent:\smallskip
 
{\rm (a)} $(\norm_n : n\in \N_N)$ is a multi-norm of level $N$ on $\{E^n : n\in\N_N\}\,$;  \smallskip

{\rm (b)}
$\lV a\,\cdot\,x\rV_m \leq \lV a \rV \lV x \rV_n$ for each row-special matrix $a \in {\mathbb M}_{\,m,n}$,  each  $x \in E^n$, and each $m,n\in \N_N\,$;\smallskip

{\rm (c)}
$\lV a\,\cdot\,x\rV_m \leq \lV a \rV \lV x \rV_n$ for each   $a \in {\mathbb M}_{\,m,n}$, each  $x \in E^n$, and  each $m,n\in \N_N$.
\end{theorem}

\begin{proof}  (a) $\Rightarrow$ (b)  Suppose that $(\norm_n : n\in \N_N)$ is a multi-norm of level $N$  on the family $\{E^n : n\in\N_N\}$, and 
let $a$ be  a row-special matrix, of the form specified above. Then, for each $x \in E^n$, we have  the following, where we take $a_{i,j(i)}=0$
 when the $i^{\rm th.}$ row of $a$ is zero:
\begin{eqnarray*}
\lV a\,\cdot\,x\rV_m\!
 &=& \!\lV (a_{1,j(1)}x_{j(1)}, \dots, a_{m,j(m)}x_{j(m)})\rV_m\\
 &\leq &\! \max\left\{\lv a_{1,j(1)}\rv, \dots, \lv a_{m,j(m)}\rv\right\}\lV(x_{1}, \dots,x_{n})\rV_n\quad\mbox{\rm by Lemma \ref{2.1b}}\\
&=& \! \lV a \rV \lV x \rV_n \quad \mbox{\rm by equation (\ref{(2.3)})}\,,
\end{eqnarray*}
and so (b) holds.\smallskip

(b) $\Rightarrow$ (c)   Let  $a \in {\mathbb M}_{\,m,n}$, where $m,n\in \N_N\,$. Then  $a = \sum_{r=1}^k a_r$,
 where $a_1, \dots a_k$ are row-special matrices in $ {\mathbb M}_{\,m,n}$ and $\lV a \rV = \sum_{r=1}^k\lV  a_r\rV$, as in the decomposition given above. 
For  each $x \in E^n$, we have 
$$
\lV a\,\cdot\,x\rV_m   \leq   \lV a_1\,\cdot\,x\rV_n + \cdots + \lV a_k\,\cdot\,x\rV_n  \leq   
(\lV a_1\rV +\cdots + \lV a_k\rV)\lV x \rV_n = \lV a \rV \lV x \rV_n \,,
$$
 as required.\smallskip

(c) $\Rightarrow$ (b)  This is immediate. \smallskip

(b) $\Rightarrow$ (a)
 We must show that Axioms (A1)--(A4) of Definition \ref{1.1} are satisfied.
Let $k\in\N_N$ with $k \geq 2$.

Let  $x\in  E^k$. By taking $a$ to be, first, a suitable matrix  in $ {\mathbb M}_k$  with exactly one non-zero term equal to $1$ 
in each row, so that $a$ corresponds to a given permutation in ${\mathfrak S}_k$, and, second, a diagonal matrix with diagonal terms 
$\alpha_1, \dots, \alpha_k \in \C$, we see that (A1) and (A2) follow immediately from   (b).

Now take $x_1, \dots, x_{k-1} \in E$,  and take $a\in {\mathbb M}_{k, k-1}$  to be the row-special matrix
$$
\left[
\begin{array}{cccc}
1&0&\dots  &0\\
0&1&\dots&0\\
\vdots&\vdots&\ddots&\vdots\\
0&0&\dots&1\\
0&0&\dots&0\\
\end{array}
\right]\,.$$
It follows from (b) that  $\lV (x_1,\dots, x_{k-1},0)\rV_{k}\leq  \lV (x_1,\dots, x_{k-1})\rV_{k-1}$.  Similarly, we see that 
$\lV (x_1,\dots, x_{k-1})\rV_{k-1}\leq  \lV (x_1,\dots, x_{k-1},0)\rV_{k}$, and so (A3) holds.

Finally, take  $a\in {\mathbb M}_{k}$  to be the row-special matrix
$$
\left[
\begin{array}{ccccc}
1 &  \cdots &0& 0&0\\
\vdots& \ddots &\vdots & \vdots &\vdots\\
0&\cdots &1&0&0\\
0 &  \cdots &0& 1&0\\
0 &  \cdots &0& 1&0\\
\end{array}
\right]\,.$$
Then $\lV a \rV =1$, and it follows from (b), (A2), and (A3) that
$$
\lV (x_1, \dots, x_{k-1}, x_{k-1})\rV_k \leq \lV (x_1, \dots, x_{k-1}, 0)\rV_k = \lV (x_1, \dots, x_{k-1})\rV_{k-1}\,.
$$
 Similarly,  $\lV (x_1, \dots, x_{k-1})\rV_{k-1}\leq \lV (x_1, \dots, x_{k-1}, x_{k-1})\rV_k$, and so (A4) holds.

We have shown that $(\norm_n : n\in \N_N)$ is a multi-norm of level $N$ on the family $\{E^n : n\in\N_N\}$, giving (a)  \end{proof}\smallskip

\subsection{Dual multi-norms and matrices}  Let $m,n\in\N$, and let $a = (a_{ij})\in   {\mathbb M}_{\,m,n}$.
Then $a$ is a {\it column-special\/} matrix if, for each $j\in \N_n$,  there is at most one non-zero term 
in the  $j^{\rm th.}$  column.  Clearly the transpose of a row-special matrix is a column-special matrix, and vice versa.

 We {\it claim\/}  that each $ a= (a_{ij})\in   {\mathbb M}_{\,m,n}$ can be written as
$$
a = \sum_{r=1}^k a_r\,,
$$
 where $a_1, \dots, a_k$ are column-special matrices in $ {\mathbb M}_{\,m,n}$ and $\lV a \rV = \sum_{r=1}^k\lV  a_r\rV$,
 where  now $\lV a \rV$ is an abbreviation of $\lV a : \ell_n^{\,1}\to \ell_m^{\,1}\rV$.  This claim follows from an earlier remark by taking transposes.

The following theorem   can be proved by a similar argument to that in Theorem \ref{2.5b}. Indeed, the proof uses Lemma \ref{2.11} 
and the above decomp\-osition of matrices. For details, see \cite[Theorem 4.6.4]{R1}.\smallskip

\begin{theorem}\label{2.5c}
Let  $(E, \norm)$ be a normed space, and take  $N \in\N$. Suppose that, for each $ n\in \N_N$, $ \norm_n$ is a  norm  on the spaces $E^n$  and, further,  
 that $ \lV x \rV_1 =\lV x \rV\,\;(x \in E)$.  Then the following are equivalent:\smallskip

{\rm (a)} $(\norm_n : n\in \N_N)$ is a dual multi-norm  of level $N$ on  $\{E^n : n\in\N_N\}$;  \smallskip

{\rm (b)} $\lV a\,\cdot\,x\rV_m \leq \lV a : \ell_n^{\,1}\to \ell_m^{\,1}\rV \lV x \rV_n$ for each column-special  
$a \in {\mathbb M}_{\,m,n}$, each $x \in E^n$, and each $m,n\in \N_N\,$;\smallskip

{\rm (c)}  $\lV a\,\cdot\,x\rV_m \leq \lV a : \ell_n^{\,1}\to \ell_m^{\,1}\rV \lV x \rV_n$ for each  $a \in {\mathbb M}_{\,m,n}$, each $x \in E^n$, and
  each $m,n\in \N_N$.\qed
\end{theorem}\smallskip

As remarked in \cite{R1}, the above two characterizations of multi-normed spaces and of dual multi-normed spaces together give  an alternative
 proof of Theorems \ref{2.13} and \ref{2.14}.\s
 
\subsection{Generalizations}\label{gen}

Consideration of Theorems \ref{2.5b} and \ref{2.5c} suggest a further generalization of the notions of multi-norms and dual multi-norms. 
 The following is    \cite[Definition 4.3.1]{R1}.\s

\begin{definition}\label{2.7a}
Let $(E, \norm )$ be a normed space, and take $p\in [1,\infty]$.  A {\it type-$p$ multi-norm}  on 
 $\{E^n : n\in \N\}$  is a sequence $(\norm_n : n\in \N)$
 such that $\norm_n$ is a norm on $E^n$ for each $n\in \N$, such that  $\lV x \rV_1 = \lV x \rV$ for each $x\in E$, and such that
$$
\lV a\,\cdot\,x\rV_m \leq \lV a : \ell_n^{\,p}\to \ell_m^{\,p}\rV \lV x \rV_n
$$
 for each matrix $a \in {\mathbb M}_{\,m,n}$, each $x \in E^n$, and  each $m,n\in \N$.
\end{definition}\s

Thus a multi-norm is a type-$\infty$ multi-norm and a dual multi-norm is a type-$1$ multi-norm in the sense of the above definition.  A type-$p$ multi-norm
is a special-norm in the above sense. \s
 
 For example, fix $p\in [1,\infty]$,   let $E =\C$, and take the $\ell^{\,p}$-norm on $E^n$ for each $n\in\N$. 
Then we obtain a type-$p$ multi-norm.  Further, a short calculation involving the matrices 
$$ 
\left[\begin{array}{cc}
1 & 1\\
0& 0\end{array}\right]\quad {\rm and}\quad \left[\begin{array}{cc}
0 & 1\\
0& 1\end{array}\right]\quad{\rm in }\quad {\mathbb M}_2\,,
$$
shows that this example is not a type-$q$ multi-norm for any $q \in [1,\infty]$ save for $q= p$. Thus the classes prescribed by type-$p$ multi-norms
 are distinct for different values of $p$.\s

\begin{example}\label{3.9}
{\rm Let $E$ be a Banach space, and take $p\in [1,\infty]$.  For $n\in\N$, define
$$
\lV (x_1, \dots,x_n)\rV_n = \left(\sum_{i=1}^n \lV x_i\rV^p\right)^{1/p}\quad (x_1, \dots,x_n \in E)\,,
$$
and consider the sequence $(\norm_n: n\in \N)$. In the case where $p=1$, we obtain a dual multi-norm, and in the case where $p=\infty$, 
we obtain a multi-norm based on $E$.  Now take $p\in (1,\infty)$. Then it follows from \cite[\S4]{Kw} that $(\norm_n: n\in \N)$ 
is a type-$p$ multi-norm if and only if $E$ is  isometrically isomorphic to a subspace of a quotient of an  $L^p$-space.} 
\qed
\end{example}\s
 
  The following is    \cite[Lemmas 4.3.2 and 4.3.3]{R1}.\s
  
  \begin{proposition}\label{2.7c}
  Let $E$ be a normed space.  Suppose that $(\norm_n : n\in \N)$ is a  type-$p$ multi-norm on $\{E^n : n\in \N\}$, where  $p\in [1,\infty]$.
  Take  $n\in \N$ and $\alpha, \beta \in \C$. Then
  $$
  \lV (x_1, \dots, x_{n-1}, \alpha x_n, \beta x_n)\rV_{n+1} = \lV (x_1, \dots, x_{n-1}, \gamma x_n\rV_n
  $$
  for $x_1, \dots, x_n\in E$, where $\gamma = (\lv \alpha\rv^p+ \lv \beta\rv^p)^{1/p}$. In particular,
  $$
  \lV (x_1, \dots, x_{n-1},   x_n, x_n)\rV_{n+1}= \lV (x_1, \dots, x_{n-1}, 2^{1/p} x_n)\rV_n 
$$
  for $x_1, \dots, x_n\in E$. \qed
\end{proposition}\s

The following result, from \cite{R1}, generalizes those of $\S\ref{Theorems on duality}$.\s

\begin{theorem}\label{2.7b}
Let $E$ be a normed space, and take $p\in [1,\infty]$.  Then the dual of a type-$p$ multi-norm on  $\{E^n : n\in \N\}$ is a type-$q$ multi-norm on
  $\{(E')^n : n\in \N\}$, where $q$ is the conjugate index to $p$.\qed
\end{theorem}\s

\subsection{Sequential norms}  Let $E$ be a Banach space.  A somewhat similar notion to that of our multi-norms has already been defined; these are sequential norms on
the  family $\{E^n : n\in \N\}$;  these norms were first defined and extensively studied in \cite{L}, and their definition and basic properties are summarized in
\cite{LNR}.

Indeed, a {\it sequential norm\/} on $\{E^n : n\in\N\}$ is a sequence $(\norm_n : n\in \N)$
 such that $\norm_n$ is a norm on $E^n$ for  each $n\in \N$, such that  $\lV x \rV_1 = \lV x \rV$ for each $x\in E$, and such that the following axioms
are satisfied for each $m, n\in \N$: \smallskip

   (L1) $\lV (x_1,\dots, x_{n},0)\rV_{n+1}=  \lV (x_1,\dots, x_{n})\rV_{n}\,\;(\,x_1,\dots,x_{n} \in E)\,;$\smallskip

 (L2) $$\lV (x_1,\dots, x_{m},y_{1},\dots, y_n)\rV_{m+n}^2= \lV (x_1,\dots, x_{m})\rV_{m}^2+ \lV (y_1,\dots, y_{n})\rV_{n}^2$$   whenever
  $x_1,\dots,x_{m}, y_{1},\dots, y_n \in E\,;$\smallskip

 (L3) $ \lV a\,\cdot\, x \rV_m\leq \lV a : \ell_n^{\,2}\to \ell_m^{\,2}\rV\lV x \rV_n\,\; (x \in E^n, \, a \in {\mathbb M}_{m,n})$.\smallskip
   
The space $E$ together with the sequential  norm  $(\norm_n : n\in \N)$ is called an {\it operator sequence space\/}  over $E$.

It is clear that a  sequential norm is a type-$2$ multi-norm, and so  it satisfies our axioms (A1), (A2), and (A3).
The above example, with $p=2$, gives a sequential norm which is not a type-$q$ multi-norm for any $q \in [1,\infty]$ save for $q= 2$.
  On the other hand, a multi-norm satisfies   (L1), but it need not satisfy (L2). For example, let 
 $E=\C$, and consider the multi-norm specified by 
$$
\lV (\alpha,\beta)\rV_2 = \lv  \alpha \rv + \lv  \beta \rv \quad (\alpha, \beta \in \C)\,.
$$
This is rarely equal to $\left (\lv  \alpha \rv^2 + \lv  \beta \rv ^2\right)^{1/2}$, as required by (L2).
In fact, a multi-norm never satisfies (L3) (unless $E =\{0\}$).  For take $x\in E$ with $\lV x\rV =1$, and take
$$ a= 
\left[\begin{array}{cc}
1 & 1\\
0& 0\end{array}\right]\in {\mathbb M}_2\,,
$$
so that $\lV (x,x)\rV_2 =1$ by Lemma \ref{2.1b}  and hence $\lV a \rV =\sqrt{2}$, but 
 $\lV a\,\cdot  x \rV_2 = \lV (2x,0)\rV_2 =2$ by (A3).  Thus (L3) fails.\medskip

\subsection{Multi-norms and tensor norms}  The following definition and theorem (with a proof) will be given in \cite{DDPR1}.\s

\begin{definition}\label{2.38}
Let $(E, \norm)$ be a normed space.  Then a norm $\norm$ on the tensor product $c_{\,0}\otimes E$ is a {\it  $c_{\,0}$-norm} if 
$$
\lV \delta_1\otimes x\rV = \lV x \rV\quad(x \in E)
$$
 and if $T\otimes I_E$ is a bounded linear operator on $(c_{\,0}\otimes E, \norm)$ with 
$$
\lV T\otimes I_E\rV \leq \lV T\rV\quad (T \in {\mathcal K}(c_{\,0}))\,.
$$
\end{definition}\s

In fact, each such $c_{\,0}$-norm $\norm$ is a reasonable cross-norm, and so we have
 $$\lV z\rV_\varepsilon \leq \lV z\rV\leq \lV z\rV_\pi\quad(z\in c_{\,0}\otimes E)\,.
$$ 
It will also be noted in \cite{DDPR1} that, for each such $c_{\,0}$-norm $\norm$ on $c_{\,0}\otimes E$, we have
\begin{equation}\label{(2.2)}
 \lV T\otimes I_E\rV = \lV T\rV\quad (T \in {\mathcal B}(c_{\,0}))\,.
\end{equation}

Let $\norm$ be a $c_{\,0}$-norm on a space $c_{\,0}\otimes E$, and take $n\in\N$. We define
\begin{equation}\label{(2.13)}
\lV (x_1, \dots,x_n)\rV_n = \lV \sum_{j=1}^n\delta_j\otimes x_j\rV\quad (x_1,\dots, x_n \in E)\,.
\end{equation}

For example, the injective   norm $\norm_\varepsilon$ on the tensor product $c_{\,0}\otimes E$ is such that  
$$
\lV \sum_{j=1}^n\delta_j\otimes x_j\rV_\varepsilon  =\max_{i\in\N_n}\lV x_i\rV\quad (x_1,\dots,x_n \in E)
$$
for each $n\in\N$, and it is easily seen that  $\norm_\varepsilon$ is a $c_{\,0}$-norm.  It is also easily seen that the projective   norm $\norm_\pi$ is a $c_{\,0}$-norm.
Thus $\norm_\varepsilon$ and $\norm_\pi$  are the minimum and maximum $c_{\,0}$-norms on $c_{\,0}\otimes E$, respectively.

\begin{theorem}\label{2.38a}
 Let $E$ be a normed space. Then the family ${\mathcal E}_E$ of multi-norms based on
 $E$ corresponds bijectively to the family of $c_{\,0}$-norms on $c_{\,0}\otimes E$ via the above correspondence.\qed
\end{theorem}\s

In fact, a more general theorem will be proved in \cite[Theorem 3.5]{DDPR1}.   There is a similar characterization of 
dual multi-norms; one replaces `$c_{\,0}$' by `$\ell^{\,1}$'; see \cite{DDPR1}.\s 

Let $E$ be a normed space. Then we have seen that there are two complementary approaches to the theory of {\mn} spaces: the `coordinate approach' involving
sequences $(\norm_n: n\in\N)$ of norms, where $\norm_n$  is a norm on $E^n$ for each $n\in\N$, and the `non-coordinate approach' 
involving norms on the tensor product $c_{\,0}\otimes E$. An analogous  contrast appears in the  well-known theory of
 {\it operator space theory\/}, or {\it quantum functional analysis\/}. 
The  `coordinate approach' to this  theory involves sequences  $(\norm_n: n\in\N)$ of norms, 
where $\norm_n$  is a norm on ${\M}_n(E)$ for each $n\in\N$; the complementary `non-coordinate approach' 
 involves norms on ${\mathcal F}(L)\otimes E$, where 
${\mathcal F}(L)$ denotes the space of finite-rank operators on a fixed separable Hilbert space $L$. The former approach predominates in the
 works \cite{BM,ER,Pa,Pis2}, for example;  the latter approach predominates in the monograph \cite{Hel2} of Helemskii,  and the Introduction to \cite{Hel2}
 contains a clear discussion of the contrasting strengths of the two approaches. 
 We give some brief details of the two approaches.\s
 
  \begin{definition}\label{8.1}
Let $E$ be a linear space, and consider an assignment of norms $\norm_n$ on ${\mathbb M}_{\,n}(E)$ for each $n\in\N$; these norms are called the {\it matrix norms}.
  An {\it  abstract operator space} on $E$  is a sequence $(\norm_n: n\in\N )$ of matrix norms such that:\smallskip

{\rm (M1)} 
 $\lV \alpha v \beta\rV_n \leq \lV \alpha \rV \lV v\rV_m \lV \beta \rV$ for  
 $m,n \in \N$, $\alpha \in {\mathbb M}_{n,m}$, $\beta \in {\mathbb M}_{m,n}$, and $ v \in  {\mathbb M}_{m}(E).$\s

{\rm (M2)}
$\lV  v\oplus w \rV _{m+n} = \max\{\lV v \rV_m, \lV w\rV_n\}$  for  $m,n \in \N$,   $v \in {\mathbb M}_{m}(E)$, and  $w\in {\mathbb M}_{n}(E)\,.$ 
\end{definition}\s

The following definition is taken from \cite{Hel2}.  We set ${\mathcal F} = {\mathcal F}(L)$ for a fixed Hilbert space $L$,
 and note that    ${\mathcal F}\otimes E$ is a bimodule (with operations denoted by $\,\cdot\,$) over ${\B}(L)$.\s

  \begin{definition}\label{8.1a}
Let $E$ be a linear space.  Then a {\it quantum  norm} on $E$ is a norm $\norm$ on  ${\mathcal F}\otimes E$ satisfying 
the following two conditions:\s

{\rm (R1)} $\lV T\,\cdot\,u\rV, \lV u\,\cdot\,T\rV \leq \lV T \rV\lV u\rV$ whenever $T\in {\B}(L)$ and $u \in {\mathcal F}\otimes E$\,;\s

{\rm (R2)}  whenever   $u,v \in {\mathcal F}\otimes E$ and there exist self-adjoint projections $P, Q \in {\B}(L)$ 
with  $P\,\cdot\,u\,\cdot\,P = u$, with  $Q\,\cdot\,v\,\cdot\,Q = v$, and with  $PQ=0$, then  $\lV u+v\rV \leq \max\{\lV u\rV, \lV v\rV\}$.
\end{definition}\s
 
It is shown in \cite{Hel2} that the family of quantum norms on $E$ corresponds bijectively to the abstract operator space on $E$ described in  
Definition  \ref{8.1}.\s

Given an axiomatic theory one often wishes to find a `concrete representation' of the objects defined by the theory.  For example,
 the Gel'fand--Naimark theory  gives a concrete representation of each abstractly-defined 
$C^*$-algebra as a self-adjoint, norm-closed subalgebra of the $C^*$-algebra
 ${\B}(H)$ for some Hilbert space $H$.  The concrete representation of an abstract operator space is 
{\it Ruan's theorem\/},  which represents each such system  as a closed subspace of ${\B}(H)$ 
for some Hilbert space $H$, the matricial norms being recovered in a canonical way.\s

After  a first draft of this work was completed, the late Professor Nigel Kalton pointed out the memoir \cite{MN} of  Marcolini Nhani; 
I am deeply grateful for this reference and for some valuable comments.

In fact, let $E$ be a Banach space. Then a norm $\norm$ on $c_{\,0}\otimes E$ satisfies `{\it condition\/} (P)'  of \cite[\S2, p.~12]{MN} if
$$
\lV (T\otimes I_E)(z)\rV \leq \lV T \rV\lV z\rV\quad (z \in c_{\,0}\otimes E,\,T\in {\B}(c_{\,0}))\,.
$$ 
It is clear from our remarks that such norms are exactly the $c_{\,0}$-norms  of Definition \ref{2.38},  and so the definition  of a multi-normed space
 corresponds to the theory in \cite{MN} of norms on $c_{\,0}\otimes E$ satisfying property (P). 

 As remarked  in \cite{MN}, this theory is a form of  `commutative counterpart' to that of operator space theory. Indeed, 
we obtain the Axiom (P) by replacing ${\mathcal F}$ by $c_{\,0}$ in the   axiom  (R1).  However, our theory has no analogue of Axiom (R2), so,
 in that sense, it is more general.

The analogue of Ruan's theorem is Pisier's  theorem,  given as  Th{\'e}or{\`e}me 2.1 in \cite{MN}; we shall describe this result in Theorem \ref{2.3p}, below. 
\medskip

\chapter{The minimum and maximum   multi-norms}

\noindent  In this chapter, we shall  first define a `rate-of-growth'  sequence $(\varphi_n(E))$  for each multi-normed space $((E^n, \norm_n) : n\in \N)$, 
and then define two important examples of multi-norms for an arbitrary normed space $E$: these are the minimum and the maximum multi-norms.  
  We shall investigate the duals of these multi-norms and the sequence  $(\varphi^{\max}_n(E))$   corresponding to the maximum multi-norm, 
and relate them to $p$-summing constants.\medskip

\section{An associated sequence}\label{associated sequence}

\begin{definition} \label{3.3}
Let $((E^n, \norm_n) : n\in \N)$ be  a  multi-normed space.  For each $n\in \N$, set 
$$
\varphi_n(E) = \sup\{ \lV (x_1,\dots,x_n)\rV_n : x_1,\dots,x_n \in E_{[1]}\}\,.
$$
The sequence $(\varphi_n(E): n\in\N)$ is the {\it rate of growth}  sequence for the multi-normed space.
\end{definition}\smallskip

Note that $\varphi_n(E)$ is not intrinsic to the initial normed space $E$; it depends on the multi-norm, and so, strictly, we should write 
$(\varphi_n(E^n, \norm_n))$ instead of  $(\varphi_n(E))$.  

Suppose that two  multi-norms are  equivalent.   
Then  their rate of growth sequences are similar.  
However, the converse to this is not true; see Proposition \ref{4.1aj}, below.

Clearly $(\varphi_n(E): n\in\N)$ is an  increasing sequence in $\R$ for each multi-normed space $((E^n, \norm_n) : n\in \N)$,
 and it  follows from Lemma \ref{2.2}  that
$$
1\leq \varphi_n(E)\leq n\quad (n\in \N)
$$
and from Lemma \ref{2.1a} that
$$
\varphi_{m+n}(E) \leq \varphi_m(E)+ \varphi_n(E)\quad (m,n\in \N)\,.
$$ 

Let $F$ be a subspace of a normed space $E$, so  that  $((F^n, \norm_n) : n\in \N)$ is a  multi-normed subspace
of a {\mn} space $((E^n, \norm_n) : n\in \N)$. Then clearly we have $\varphi_n(F) \leq \varphi_n(E)\,\;(n\in\N)$. \medskip

\section{The minimum multi-norm}

\subsection{Definitions}  We first define the most obvious multi-norm.\smallskip

\begin{definition} \label{3.1}
 Let $(E, \norm )$ be a normed space. For $k\in\N$, define  $\norm_k^{\min}$  on $E^k$ by
$$
\lV (x_1,\dots, x_k)\rV_k^{\min} =\max_{i\in\N_k}\lV x_i\rV\quad (x_1,\dots,x_k  \in E)\,.
$$
\end{definition}\smallskip

It is immediate that  $\norm_k^{\min}$ is a norm on $E^k$ for each $k\in \N$,  and then that, for each $n\in\N$, the sequence
 $(\norm_k^{\min}: k\in\N_n)$ is a multi-norm of level $n$.  It follows that  $((E^k,\norm_k^{\min}): k\in\N)$ is  a multi-normed space. \smallskip

\begin{definition} \label{3.1a}
 Let $(E, \norm )$ be a normed space. For each $n\in \N$, the  sequence $$(\norm_k^{\min}: k\in\N_n)$$ is the  {\it minimum multi-norm of level\/} $n$.
The  sequence $(\norm_n^{\min} : n\in \N)$ is the {\it minimum multi-norm\/}.  The {\it rate of growth\/} of this multi-norm is  $(\varphi_n^{\min}(E) :n\in\N)$.
\end{definition}\smallskip

It follows immediately from this example  that there is indeed a multi-norm based on  each normed space $(E, \norm )$. The terminology `minimum' 
is justified by Lemma \ref{2.2}, given above. The minimum multi-norm 
corresponds to the injective norm on the tensor product  $c_{\,0}\otimes E$ via the correspondence of Chapter 2, \S6.4; see \cite{DDPR1}.

Let   $({\mathcal E}_E, \leq)$  be the partially ordered family of multi-norms based on  $E$, as in Definition \ref{2.5e}.
  Then it  is clear that the minimum multi-norm  is the minimum element in   $({\mathcal E}_E, \leq)$.   

More generally, take $n\in \N$, and let $((E^k, \norm_k) : k\in \N_n)$ be  a  multi-normed space of level $n$ on $\{E^k :k\in\N_n\}$. For $m > n$, define
$$
\lV (x_1, \dots, x_m)\rV_m  \!=\!\max\{ \lV (y_1, \dots, y_n)\rV_n :y_1, \dots, y_n \in \{x_1, \dots, x_m\}\} 
$$
for $x_1, \dots, x_m \in E$. Then $((E^m, \norm_m) : m\in \N)$ is   a  multi-normed space.  Thus a multi-norm of level $n$ can be 
extended to a multi-norm, in an obvious sense.\s

The following result is immediate. \s

\begin{proposition}\label{3.2ad}
Let $E$ be a normed space, and  let  $(\norm_n :n\in\N)$ be a multi-norm  based on $E$. 
Then  $(\norm_n :n\in\N)$ is equal to the  minimum multi-norm   if and only if  $\varphi_n(E)=1\,\;(n\in \N)$,
 and it is equivalent to the minimum multi-norm   if and only if  $(\varphi_n(E): n\in \N)$ is bounded. \qed 
\end{proposition}\s

Let $E$ be a normed space, let $n\in \N$, and let $((E^k, \norm_k) : k\in \N_n)$ be  a  multi-normed space of level $n$. Extend this 
multi-norm to the multi-normed space  $((E^m, \norm_m) : m\in \N)$, as above.   Then clearly $\varphi_m(E) = \varphi_n(E)\quad(m\geq n)$.
Thus there are multi-norms which are equivalent to the minimum multi-norm, but are not equal to it, whenever $\varphi_2(E) >1$.  \smallskip
 
Let $(E, \norm)$ be a normed space.  As noted above, there is a  similar ordering of dual multi-norms on the family
 $\{E^n:n\in\N\}$. As   in Example \ref{1.2}, the sequence $(\norm_n : n\in\N)$, where
$$
\lV (x_1,\dots, x_n)\rV_n= \sum_{j=1}^n\lV x_j\rV\quad (x_1,\dots, x_n \in E)\,,
$$
is a dual multi-norm on $\{E^n: n\in\N\}$. It follows from Lemma \ref{2.2} that this sequence  $(\norm_n : n\in\N)$ is the {\it maximum dual multi-norm}
 on  $\{E^n: n\in\N\}$.  \s

Let $E$ be a normed space. It  is easily seen  that the dual of the minimum multi-norm on
 $\{E^n :n\in\N\}$ is  $(\norm_n': n\in\N_n)$, where $\norm_n'$ is  defined  by
$$
\lV (\lambda_1, \dots,\lambda_n)\rV_n' = \sum_{j=1}^n\lV \lambda_j\rV\quad (\lambda_1,\dots,\lambda_n  \in E')\,,
$$
and that the dual of the maximum dual multi-norm on $\{E^n :n\in\N\}$ is the sequence $(\norm_n': n\in\N_n)$, where 
$$
\lV (\lambda_1, \dots,\lambda_n)\rV_n' =\max\{\lV \lambda_1\rV, \dots, \lV\lambda_n\rV\}\quad (\lambda_1,\dots,\lambda_n  \in E')\,,
$$
and so, by Lemma \ref{2.2}, which applies to dual multi-norms,  the  following result is immediate.\s

\begin{proposition} \label{3.2a}
 Let $E$ be a normed space, and take $n\in\N$. Then:\s

{\rm (i)} the dual of the minimum multi-norm  on $\{E^n:n\in\N\}$ is the  maximum dual multi-norm on $\{(E')^n:n\in\N\}\,$;\s

 {\rm (ii)}  the dual of the maximum dual multi-norm  on $\{E^n:n\in\N\}$  is the  minimum multi-norm on $\{(E')^n:n\in\N\}\,$;\s
 
  {\rm (iii)}  the second dual of the minimum multi-norm  on $\{E^n:n\in\N\}$ is the  minimum multi-norm on $\{(E'')^n:n\in\N\}\,$.\qed
\end{proposition}\medskip

\subsection{Finite-dimensional spaces}  We show  the uniqueness of multi-norms based on finite-dimensional normed spaces.\s

\begin{proposition} \label{3.2}
Let $n\in \N$. Then the  minimum multi-norm of level $n$ is the unique multi-norm of level $n$ on $\{\C^k : k\in \N_n\}$. \end{proposition}

\begin{proof} Let $(\norm_k : k\in \N_n)$ be a   multi-norm of level $n$ on  the family  $\{\C^k : k\in \N_n\}$. Take $k\in \N_n$. 
By Lemma \ref{2.1b}, we have $\lV (1,\dots, 1)\rV_k =1$.
Now take $(\alpha_1, \dots, \alpha_k) \in \C^k$. By (A2), we have
$$
\lV (\alpha_1, \dots, \alpha_k)\rV_k \leq  (\max_{i\in \N_k}\lv\alpha_i\rv)\lV (1,\dots, 1)\rV_k = \max_{i\in \N_k}\lv\alpha_i\rv\,,
$$
and, by Lemma \ref{2.2},  $\max_{i\in \N_k}\lv\alpha_i\rv \leq \lV (\alpha_1, \dots, \alpha_k)\rV_k$. Thus
$$
\lV (\alpha_1, \dots, \alpha_k)\rV_k = \max_{i\in \N_k}\lv\alpha_i\rv\,,
$$
 giving the result.\end{proof}\s

\begin{proposition} \label{3.2b}
Let $((E^n, \norm_n) : n\in \N)$ be  a  multi-normed space such that $E$ is
 finite-dimensional. Then  $(\norm_n: n\in\N)$  is  equivalent   to the minimum multi-norm.
\end{proposition}

\begin{proof} Suppose that $\dim E= m$, and  take $\{e_1, \dots, e_m\}$ to be a basis of $E$; we may suppose that  $\lV e_j\rV=1\,\;(j\in\N_n)$. Set 
$e = (e_1, \dots, e_m) \in E^m$, so that $\lV e\rV_m \leq m$.

 There exists a constant $C>0$ such that each $x\in E$ can be written uniquely as $x=\sum_{j=1}^m\alpha_je_j$, with $\sum_{j=1}^m\lv\alpha_j\rv \leq C\lV x \rV$.

Now take $n\in\N$ and $x_1,\dots,x_n\in E_{[1]}$, say $x_i = \sum_{j=1}^m\alpha_{i,j}e_j$ for $i\in\N_n$. Then $\sum_{j=1}^m\lv\alpha_{i,j}\rv \leq C\,\;(i\in\N_n)$.
Set $a =(\alpha_{i,j}) \in {\mathbb M}_{n,m}$, so that  
$$
\lV a: \ell_m^{\,\infty}\to \ell_n^{\,\infty}\rV = \max_{i\in\N_n}\sum_{j=1}^m\lv\alpha_{i,j}\rv \leq C\,.
$$
Then, using Theorem \ref{2.5b}, (a) $\Rightarrow$ (c),  we have
$$
\lV (x_1,\dots,x_n)\rV_n = \lV a\,\cdot\,e\rV_n \leq \lV a: \ell_m^{\,\infty}\to \ell_n^{\,\infty}\rV\/\lV e\rV_m  \leq Cm\,.
$$
Thus $\varphi_n(E)\leq Cm\,\;(n\in\N)$.  By Proposition \ref{3.2ad}, $(\norm_n: n\in\N)$  is  equivalent   to the minimum multi-norm.  
\end{proof}\medskip

\section{The maximum multi-norm}

\noindent Let $E$ be a normed space. The multi-norm based on $E$ to be defined in this section, whilst natural,
 is much more interesting than the minimum multi-norm.  \smallskip

\subsection{Existence of the maximum multi-norm}  We first show that there is a maximum multi-norm.

\begin{definition} \label{3.4}
Let $(E, \norm )$ be  a normed space,  take $n\in\N$, and suppose that
$$(\LV\,\cdot\,\RV_k: k\in\N_n)$$ is a multi-norm of level $n$ on $\{E^k : k\in \N_n\}$.  Then 
$(\LV\,\cdot\,\RV_k : k\in \N_n)$  is the {\rm  maximum  multi-norm of level} $n$ if
$$
\lV (x_1,\dots, x_k)\rV _k \leq \LV (x_1,\dots, x_k)\RV_k\quad (x_1,\dots, x_k \in E,\, k\in \N_n)
$$
whenever $(\norm_k : k\in \N_n)$ is a multi-norm of level $n$ on $\{E^k : k\in \N_n\}$. \end{definition}\smallskip

We define the  {\it  maximum multi-norm} on the family $\{E^n : n\in \N\}$ similarly.\smallskip

Let $n \in \N$. Then it is easy to see that there is a maximum multi-norm  of level $n$ on   $\{E^k : k\in \N_n\}$
 for each normed space $(E, \norm)$.  Indeed let
$$
\{(\norm_k^{\alpha} : k\in \N_n) :  \alpha \in A\}
$$
 be the (non-empty) family of all multi-norms  of level $n$ on   $\{E^k : k\in \N_n\}$, and, for $k \in \N_n$,  set
$$
\LV (x_1,\dots, x_k)\RV_k =\sup_{\alpha \in A}\lV (x_1,\dots, x_k)\rV_k^{\alpha}\quad (x_1,\dots, x_k \in E)\,.
$$
It follows from Lemma \ref{2.2} that the supremum is finite in each case, and then it  is easily checked that the sequence 
$(\LV\,\cdot\,\RV_k: k\in\N_n)$ is  a multi-norm of level $n$ on $\{E^k : k\in \N_n\}$, and hence $(\LV\,\cdot\,\RV_k: k\in\N_n)$ 
is  the  maximum multi-norm of level $n$ on $\{E^k : k\in \N_n\}$. Similarly this applies to multi-norms themselves.\s

\begin{definition} \label{3.4c}
Let $(E, \norm )$ be  a normed space.   We  write 
$$
(\norm^{\max}_n :n\in\N)
$$
 for  the maximum multi-norm on $\{E^n : n\in \N\}$. 
\end{definition}\smallskip

Suppose that $m,n \in\N$ with $m\leq  n$, and let $(\norm_k^{\max}  :k\in\N_n)$ be the maximum multi-norm of level $n$ on 
$\{E^k : k\in\N_n\}$. Then it is immediate that  $(\norm^{\max}_k  :k\in\N_m)$ is the maximum multi-norm of level $m$ on $\{E^k : k\in\N_m\}$.

Let   $({\mathcal E}_E, \leq)$  be the partially ordered family of multi-norms on  the family
${\{E^n: n\in\N}\}$ for a normed space $E$, as in  Definition \ref{2.5e}. It is clear that the maximum multi-norm 
is the maximum element in   $({\mathcal E}_E, \leq)$. The maximum multi-norm  corresponds to the projective norm on 
the tensor product  $c_{\,0}\otimes E$ via the correspondence of  Chapter 2, \S6.4; see \cite{DDPR1}.\s
  
  \begin{proposition}\label{4.10b}
Let $E$ be a normed space. Then   $({\mathcal E}_E, \leq)$ is a Dede\-kind complete lattice.
\end{proposition}

\begin{proof}  We know  that $({\mathcal E}_E, \leq)$ has a maximum and a minimum element.
  By Proposition \ref{2.5d}, the maximum of each pair  of elements ${\mathcal E}_E$ belongs to ${\mathcal E}_E$.
 It is now routine to check that the pointwise supremum of a non-empty set in ${\mathcal E}_E$ is the supremum of the set.

To see that each non-empty subset $S$ in ${\mathcal E}_E$ has an infimum, consider the set $T$
 of multi-norms that lie under every element of $S$. This  set   $T$  has a supremum, and this supremum is the infimum of $S$.
\end{proof}

Similarly, the family of dual multi-norms on  ${\{E^n: n\in\N}\}$ is a Dedekind complete lattice.\s    
 
\subsection{The sequence $(\varphi_n^{\,\max}(E)$)}  We now define a key sequence assoc\-iated to each  normed space $E$.\s

\begin{definition} \label{3.4ca}
For $n\in \N$, set
$$
\varphi_n^{\,\max}(E) = \sup\left\{ \lV  (x_1,\dots,x_n)\rV^{\max}_n : x_1,\dots,x_n \in E_{[1]}\right\}\,.
$$
\end{definition}\s 

Thus the sequence $(\varphi_n^{\,\max}(E): n\in\N)$ is now intrinsic to the normed space  $(E, \norm)$; it is the  {\it maximum rate of growth\/} of any  multi-norm
on $\{E^n: n\in\N\}$. We find it to be  interesting   to calculate this sequence  for an arbitrary normed space 
$E$ and for a variety of  examples; we shall give some explicit calculations  later.

 Let $E$ be a normed space with $\dim E \geq n$. Then
\begin{equation}\label{(3.1c)}
\varphi_n^{\,\max}(E) \leq \sup\{\varphi_n^{\,\max}(F) : \dim F \leq n\}\,,
\end{equation}
where the supremum is taken over all  subspaces $F$ of $E$ with   $\dim F \leq n$.  We shall see in Example \ref{3.7je}  that we can have strict inequality in 
(\ref{(3.1c)}). \s

\begin{theorem} \label{3.3a}
 Let $E$ and $F$ be Banach spaces, and let $G$ be a $\lambda$-complemented subspace of $E$ with  $G$ linearly homeomorphic  to $F$.   Then
  $$
\varphi_n^{\,\max}(E) \geq  \varphi_n^{\,\max}(F)/d(F,G)\lambda\quad (n\in\N)\,.
$$
\end{theorem}

\begin{proof} There is a projection $P: E \to G$ with $\lV P \rV \leq  \lambda$.

Set $C = d(F,G)$, and take  $\varepsilon> 0$. Then there is  a bijection $T \in {\B}(F,G)$ with  $
\lV T \rV \lV T^{-1}\rV < C + \varepsilon$.

Let $n\in \N$. Then there are elements $y_1, \dots, y_n \in F_{[1]}$ and  a  multi-norm $(\norm_k : k\in\N)$ on $\{F^k: k\in\N\}$ such that
$$
\lV (y_1, \dots, y_n )\rV_n > \varphi_n^{\,\max}(F) -\varepsilon\,.
$$

  Set $Q= T^{-1}\,\circ\,P \in {\mathcal B}(E,F)$, so that $$\lV Q\rV \lV T\rV \leq (C+\varepsilon) \lV P \rV\leq (C+\varepsilon) \lambda\,,
$$
 and then set
$$
\LV (x_1,\dots, x_k )\RV_k = \max\{\lV x_1\rV, \dots, \lV x_k\rV, \lV (Qx_1,\dots, Qx_k )\rV_k/\lV Q \rV \}
$$
for  each $k\in\N$ and  $x_1,\dots, x_k \in E$, so that $\LV x\RV_1 = \lV x \rV\,\;(x \in E)$. Then we quickly see that 
$(\LV \,\cdot\,\RV_k:k\in\N)$ is a multi-norm on the family $\{E^k: k\in\N\}$.

For $j\in \N_n$,   set $z_j = Ty_j/\lV T \rV \in G_{[1]}$, so that $Qz_j = y_j/\lV T\rV$. Then
$$
\LV (z_1,\dots, z_n )\RV_n  \geq   \frac{\lV (y_1, \dots,y_n)\rV_n}{\lV Q \rV \lV T\rV} \\
\geq
 \frac {\varphi_n^{\,\max}(F) -\varepsilon}{(C + \varepsilon)\lambda}\,.
$$
Thus $\varphi_n^{\,\max}(E) \geq (\varphi_n^{\,\max}(F) -\varepsilon)/(C+\varepsilon)\lambda$.
This holds true  for each $\varepsilon >0$, and so the result follows.
\end{proof}\smallskip

\begin{corollary} \label{3.3ae}
 Let $E$ be a Banach space, and let $F$ be a  $\lambda_F$-com\-plemented subspace of $E$.  Then
$$
\vspace{-\baselineskip}   \varphi_n^{\,\max}(F) \leq \lambda_F \varphi_n^{\,\max}(E)\quad(n\in\N)\,.
$$
\hspace*{\stretch{1}}\qed
\end{corollary}\smallskip

\begin{corollary} \label{3.3ad}
 Let $E$ and $F$ be two linearly homeo\-morphic Banach spaces. Then
$$
\vspace{-\baselineskip}\varphi_n^{\,\max}(F) \leq d(E,F)\varphi_n^{\,\max}(E)\quad(n\in\N)\,.
$$
\hspace*{\stretch{1}}\qed
\end{corollary}\smallskip

The above corollary shows that, when we are seeking to calculate the sequence $(\varphi_n^{\,\max}(E): n\in\N)$ for a normed space $E$ of dimension $n$, 
we may suppose that  we have  $(\ell^{\,1}_n)_{[1]} \subset E_{[1]} \subset  (\ell^{\,\infty}_n)_{[1]}$   because $E$ is isometrically isomorphic to a normed space
$F$ with this additional property.\medskip

\section{Summing norms}
\subsection{Introduction}   We shall see below that the calculation  of the $(\varphi_n^{\,\max}(E):n\in\N)$
for certain  normed spaces $(E, \norm)$ involves some summing operators and  $p\,$--\,summing norms. For this reason, we make some preliminary remarks on
 these norms.  For much more information, including considerable history, see \cite{DJT, Ga,J,Ry,T-J, W}, for example.  
Some remarks that we make will not actually  be used, and are given to establish some background.

The first definition slightly extends \cite[p.\ 24]{J}.\s

\begin{definition}\label{3.12aa}
Let $E$ be a normed space, let $x_1,\dots,x_n \in E$, and take $p\geq 1$. Then 
$$
\mu_{p,n}(x_1,\dots,x_n) = \sup\left\{\left(\sum_{j=1}^n\lv \langle x_j,\lambda\rangle\rv^{\,p} \right)^{1/p}: \lambda \in E'_{[1]}\right\}\,.
$$
Then $\mu_{p,n}$ is the  {\it  weak $p$--summing norm\/}  on $E^n$.
\end{definition}\s

Let $E$ be a normed space,  and take $p\geq 1$ and $n\in\N$. 
We see that  $\mu_{p,n}(x_1,\dots,x_n)\leq 1$ if and only if  $\lV (\langle x_j,\lambda\rangle: j\in\N_n)\rV_{\ell_n^{\,p}} \leq 1$ for each $\lambda \in E'_{[1]}$.
 It is clear that each $\mu_{p,n}$ is a norm on the space $E^n$, and indeed $(E^n, \mu_{p,n})$ is a Banach space whenever $E$ is a Banach space.
 We shall write   $\ell^{\,p}_n(E)^{w}$ for the space $(E^n, \mu_{p,n})$.

The sequences $x=(x_j)$ for which there is a constant $C\geq 0$  such that   $$\mu_{p,n}(x_1,\dots,x_n)\leq C\quad(n\in\N)$$ are the 
{\it weakly $p$-summable sequences} in $E$  \cite[p.\ 134]{Ry};  the least such constant $C$ is a norm  on the space of these sequences. These norms are 
 denoted by $\norm_p^{weak}$ in \cite[p.\ 32]{DJT} and by  $\norm_p^w$  in \cite[(6.4)]{Ry}. We shall write 
  $\ell^{\,p}(E)^{w}$ for the space   of weakly  $p$-summable sequences in $E$.
   
Clearly $\mu_{p,n}(0,\dots,0,x_j,0,\dots,0) =\lV x_j\rV$ and
\begin{equation}\label{(3.2a)}
\max\{\lV x_i \rV: i\in\N_n\}\leq  \mu_{p,n}(x_1,\dots,x_n) \leq \left(\sum_{j=1}^n\lV x_j\rV^{\,p}\right)^{1/p}
\end{equation}
  for $x_1,\dots,x_n\in E$, and so $\mu_{p,n}$ satisfies equations (\ref{(1.5)}) and (\ref{(1.6)}). Now let $T \in {\B}(E)$. Then clearly
\begin{equation}\label{(3.2aa)}
\mu_{p,n}(Tx_1,\dots,Tx_n)\leq \lV T\rV\mu_{p,n}(x_1,\dots,x_n)\quad (x_1,\dots, x_n\in E)\,.
\end{equation}
  Further, we have
$$
\mu_{p,n}(x_1,\dots,x_n) \geq \mu_{q,n}(x_1,\dots,x_n)\quad (x_1,\dots,x_n \in E)
$$
 whenever $1\leq p\leq q$.\s
 
 \begin{theorem}\label{3.20}
Let $E$ be a normed space, and take $p\geq 1$. Then $(\mu_{p,n}:n\in\N)$ is a type-$p$ multi-norm.\end{theorem}\s

\begin{proof}  It is easily checked that $(\mu_{p,n}:n\in\N)$ satisfies Axioms (A1)--(A3), and so the sequence $(\mu_{p,n}:n\in\N)$ is a special-norm.

Take $m,n\in \N$, and then take $x=(x_1,\dots, x_n)\in E^n$ and $a \in {\M}_{m,n}$. Set $y=a\,\cdot\,x$ so that $y\in  E^m$, 
and consider  $\lambda \in E'_{[1]}$;  we write
$$u_\lambda = (\langle x_j,\,\lambda\rangle: j\in\N_n) \in \ell^{\,p}_n\quad {\rm and}\quad  v_\lambda = (\langle y_i,\,\lambda\rangle  :i\in\N_m)\in \ell^{\,p}_m\,.
$$
Then $ v_\lambda =a\,\cdot\,u_\lambda$, and so
$$
\lV  v_\lambda \rV_{\ell_m^{\,p}} \leq \lV a : \ell_n^{\,p} \to \ell_m^{\,p}\rV \lV  u_\lambda \rV_{\ell_n^{\,p}}\,. 
$$
It follows that $\mu_{p,m}(y) \leq \lV a : \ell_n^{\,p} \to \ell_m^{\,p} \rV\mu_{p,n}(x)$, and hence   $(\mu_{p,n}:n\in\N)$ is a type-$p$ multi-norm.
\end{proof}\s

It follows from Proposition \ref{2.7c} that 
  $$
\mu_{p,n+1} (x_1, \dots, x_{n-1}, \alpha x_n, \beta x_n)  = \mu_{p,n}(x_1, \dots, x_{n-1}, \gamma x_n) 
  $$
  for $x_1, \dots, x_n\in E$ and $n\in\N$, where $\gamma = (\lv \alpha\rv^p+ \lv \beta\rv^p)^{1/p}$.  
   
Suppose that $F$ is a   subspace of  a normed space  $E$, and  take elements $x_1,\dots,x_n \in F$.
  Then, by the Hahn--Banach theorem, the value of $\mu_{p,n}(x_1,\dots,x_n)$ is the  same, whether it be  evaluated with respect to either  $F$ or $E$.  
In part\-icular, the restriction of the weak $p$--summing norm defined on $(E'')^n$ to the subspace $E^n$ agrees with the 
weak $p$--summing norm defined on this space.

Let $E$ be a normed space, and take $p>1$; the  conjugate index to $p$ is denoted by $q$. By \cite[p.\ 26]{J} and \cite[(6.4)]{Ry}, 
it follows that, for each  $n\in\N$ and  $x_1,\dots,x_n \in E$, we have
\begin{equation}\label{(3.10f)}
\mu_{p,n}(x_1,\dots,x_n) =\sup\left\{\lV \sum_{j=1}^n \zeta_j x_j\rV : \sum_{j=1}^n\lv \zeta_j\rv^q \leq 1\right\}\,.
\end{equation}
(Here, and later, we think of $E$ as a complex normed space and take  $\zeta_1,\dots,\zeta_n \in \C$; in the case where $E$ is
 a real normed space, we must take  $\zeta_1,\dots,\zeta_n \in \R$.) Similarly, by \cite[2.2]{J}, we have 
\begin{equation}\label{(3.10e)}
\mu_{1,n}(x_1,\dots,x_n)=\sup\left\{\lV \sum_{j=1}^n \zeta_j x_j\rV : \zeta_1,\dots,\zeta_n \in \T\right\}\,.
\end{equation}
(In the real case, the numbers $\zeta_1,\dots,\zeta_n$ range over the finite set $\{\pm1\}$.) 

Take $p\geq 1$ with conjugate index $q$. For each $x =(x_1,\dots,x_n)\in E^n$, define
$$
T_x : (\zeta_1,\dots,\zeta_n)\mapsto \sum_{j=1}^n\zeta_jx_j\,,\quad \ell_n^{\,q}\to E\,.
$$
Then $T_x \in {\B}(\ell_n^{\,q},E)$,  and it follows from  (\ref{(3.10f)}) and  (\ref{(3.10e)})  that  $\mu_{p,n}(x_1,\dots,x_n) = \lV T_x\rV$.
 Further, the map $x\mapsto T_x,\,\;(E^n, \mu_{p,n}) \to {\B}(\ell_n^{\,q},E)$, is an isometric isomorphism, and so, as in 
 \cite[Proposition 2.2]{DJT},
\begin{equation}\label{(3.10c)}
(E^n, \mu_{p,n}) = \ell^{\,p}_n(E)^{w}\cong {\B}(\ell_n^{\,q},E) \cong  (\ell^{\,p}_n \otimes E, \norm_\pi)'\,.
\end{equation}

 Let $E$ be a normed space, and take $p\geq 1$. By \cite[p.\ 26]{J}, we have 
\begin{equation}\label{(3.10)}
\mu_{p,n}(\lambda_1,\dots,\lambda_n) = \sup\left\{\left(\sum_{j=1}^n\lv \langle x ,\lambda_j\rangle\rv^{\,p}\right)^{1/p} : x\in E_{[1]}\right\}
\end{equation}
for $\lambda_1,\dots,\lambda_n \in E'$. In particular, 
\begin{equation}\label{(3.10fa)}
\mu_{1,n}(\lambda_1,\dots,\lambda_n) = \sup\left\{\sum_{j=1}^n\lv \langle x ,\lambda_j\rangle\rv   : x\in E_{[1]}\right\}\,.
\end{equation} \smallskip

\begin{proposition}\label{3.16}
Let $E$ be a normed space, and take $p\geq 1$ and $n\in\N$. Then the weak $p$-summing norm on 
$(E'')^n$ is  the second dual of the  weak $p$-summing norm on $E^n$.\end{proposition}

\begin{proof}  We write $\mu_{p,n}$ and $\nu_{p,n}$ for the weak $p$-summing norms on $E^n$ and $(E'')^n$, respect\-ively. 
 Take $\Lambda=(\Lambda_1,\dots,\Lambda_n) \in (E'')^n$.  By (\ref{(3.10)}),
$$
\nu_{p,n}(\Lambda) = \sup\left\{\left(\sum_{j=1}^n\lv \langle \Lambda_j, \lambda\rangle\rv^p\right)^{1/p} : \lambda \in (E')^n_{[1]}\right\}\,.
$$
Take  a net $(x_\alpha)$ in $E^n$ such that $x_\alpha\to \Lambda$ in the weak-$*$ topology on $(E'')^n$.  For each $\lambda \in (E')^n$, we have
$\langle x_\alpha,\,\lambda\rangle\to \langle\Lambda,\,\lambda\rangle$,  and so $\nu_{p,n} (x_\alpha) \to \nu (\Lambda)$. 
Since $\nu_{p,n}\mid E^n =\mu_{p,n}$, it follows that $\nu_{p,n}(\Lambda) = \mu''_{p,n}(\Lambda)$. This gives the result.
\end{proof}\s

Recall that  we take  $\gamma =\ \max\{\lv \alpha\rv, \lv \beta\rv\}$ for $\gamma = (\lv \alpha\rv^q +  \lv \beta\rv^q)^{1/q}$ in the special case where  $q=\infty$.\s

\begin{proposition}\label{3.13}
Let $E$ be a normed space. Take $n\in \N$, $p\geq 1$, and  $\alpha,\beta \in \C$, and set
 $\gamma = (\lv \alpha\rv^q +  \lv \beta\rv^q)^{1/q}$, where $q$ is the conjugate index to $p$. Then
 $$
\mu_{p,n}(x_1,\dots,x_{n-1},\alpha x_n+ \beta x_{n+1}) \leq \mu_{p,n+1}(x_1,\dots,x_{n-1}, \gamma x_n,\gamma x_{n+1}) 
 $$
 for each $x_1,\dots,x_{n+1} \in E$.
\end{proposition}

\begin{proof}  Take $\lambda  \in E'$. Then we have $\lv \langle  \alpha x_n + \beta x_{n+1},\,\lambda\rangle\rv \leq 
\gamma\left(\lv \langle x_n,\,\lambda\rangle\rv^{\,p} + \lv \langle x_{n+1},\,\lambda\rangle\rv^{\,p}\right)^{1/p}$
by H{\"o}lder's inequality,  and so  
$$
\lv \langle \alpha x_n + \beta x_{n+1},\,\lambda\rangle\rv^{\,p} \leq 
\lv \langle  \gamma x_n,\,\lambda\rangle\rv^{\,p} + \lv \langle  \gamma x_{n+1},\,\lambda\rangle\rv^{\,p}\,. 
$$
The result follows from equation (\ref{(3.10)}).\end{proof}\s

\begin{theorem}\label{3.10b}
Let $(E, \norm)$ be a normed space. Then  $(\mu_{1,n} : n\in \N)$
 is a dual multi-norm on $\{E^n: n\in\N\}$, and
\begin{equation}\label{(3.10d)}
 \mu_{1,n}(x_1,\dots,x_n) \leq \lV (x_1,\dots,x_n)\rV_n\quad (x_1,\dots,x_n \in E)
\end{equation}
whenever $(\norm_n : n\in\N)$ is a dual multi-norm on  $\{E^n : n\in\N\}$.
\end{theorem}

\begin{proof}  It is immediate that   $( \mu_{1,n} : n\in \N)$ also satisfies Axiom  (B4), and so $(\mu_{1,n} : n\in \N)
$ is a dual multi-norm.  Inequality (\ref{(3.10d)}) follows from (\ref{(3.10e)}) and Lemma \ref{2.10a}.
\end{proof}\smallskip

Thus $(\mu_{1,n}: n\in \N)$ is the minimum dual multi-norm on   $\{E^n: n\in\N\}$.\smallskip

Clause (i) of the following proposition concerning a specific normed space  is given in \cite[2.6]{J}; clause (ii) follows because, 
for each measure space $\Omega$, the dual space to $L^1(\Omega)$ is  order-isometric to a space  $C(K)$ for some compact space $K$,
 as noted in Corollary \ref{2.3gd}.\smallskip

\begin{proposition}\label{3.10ba}
Let $n\in\N$  and  take $p\geq 1$.\s

{\rm (i)} Let $K$ be a compact space.  Then
$$ 
\mu_{p,n}(f_1,\dots,f_n) = \lv \sum_{i=1}^n \lv f_i\rv^{\,p}\rv_K^{1/p}\quad (f_1,\dots, f_n\in C(K))\,.
$$

{\rm (ii)}  Let  $\Omega$ be a measure space. Then
$$\vspace{-\baselineskip}
\mu_{p,n}(\lambda_1,\dots,\lambda_n) = \lV \sum_{i=1}^n \lv \lambda_i\rv^{\,p} \rV^{1/p}\quad (\lambda_1,\dots, \lambda_n\in L^1(\Omega)')\,.
$$
\hspace*{\stretch{1}}\qed
\end{proposition}\medskip

\subsection{Summing constants}
 The following definition of certain important constants is given explicitly in \cite{DJ}, extending one in  
 \cite[p.\ 56]{DJT}, \cite[$\S$16.3]{Ga}, \cite[p.\ 33]{J}, and \cite[\S6.3]{Ry}.\smallskip

\begin{definition}\label{3.10}
Let $E$ and $F$ be   normed spaces, and take  $n\in\N$  and  $p,q\in [1, \infty)$ with $p\leq q$.   Then  the {\it $(q,p)$-summing constants} of the operator
 $T\in {\B}(E, F)$  are the numbers   
 $$
\pi_{q,p}^{(n)}(T) : = \sup\!\left\{\left(\sum_{j=1}^n \lV Tx_j\rV^{\,q}\right)^{1/q}\! :x_1,\dots, x_n\in E,\, \mu_{p,n}(x_1,\dots, x_n)\leq 1\right\}.
$$
Further,  $\pi_{q,p}^{(n)}(E)= \pi_{q,p}^{(n)}(I_{E})$; these are the {\it $(q,p)$-summing constants} of the normed space $E$.  
We write $\pi_p^{(n)}(T)$ for  $\pi_{p,p}^{(n)}(T)$ and $\pi_p^{(n)}(E)$ for $\pi_{p,p}^{(n)}(E)$.
\end{definition}\smallskip

Let $E$ be a normed space, and take $n\in\N$. For each $p\geq 1$, it follows that   
\begin{equation}\label{(3.10a)}
 \pi_p^{(n)}(E )= \sup\left\{\left(\sum_{j=1}^n \lV x_j\rV^{\,p}\right)^{1/p} : \mu_{p,n}(x_1,\dots, x_n)\leq 1\right\}\,,
\end{equation}
where the supremum is taken over   $x_1,\dots, x_n\in E$. In particular,
\begin{equation}\label{(3.10b)}
 \pi_1^{(n)}(E )= \sup\left\{\sum_{j=1}^n \lV x_j\rV : \lV\sum_{j=1}^n \zeta_jx_j\rV \leq 1\,\;(\zeta_1,\dots,\zeta_n\in \T)\right\}\,,
\end{equation}
where again the supremum is taken over   $x_1,\dots, x_n\in E$.

  Clearly, in  each  case, 
$$
\lV T \rV\leq \pi_{q,p}^{(n)}(T)\leq n\lV T \rV\,,
  $$ 
so that $1\leq \pi_{q,p}^{(n)}(E)\leq n$, and  $(\pi_{q,p}^{(n)}(T): n\in \N)$ is an increasing sequence. Also, 
 each $\pi_{q,p}^{(n)}$ is a norm on ${\B}(E, F)$. Suppose that  $E$ is  a closed subspace of a Banach space $F$. 
Then it is clear that $\pi_{q,p}^{(n)}(E )\leq  \pi_{q,p}^{(n)}(F )$. \smallskip

Let $E$ and $F$ be normed spaces. Then these norms are closely related to the standard {\it $(q,p)\,$--\,summing norms}.
 Indeed, for $T\in {\B}(E, F)$, we set
   $$
 \pi_{q,p}(T) =\sup\{\pi_{q,p}^{(n)}(T): n\in\N\}= \lim_{n\to\infty}\pi_{q,p}^{(n)}(T)\in[0,\infty]\,.
 $$
In the case where $\pi_{q,p}(T) < \infty$, the operator $T$ is said to be ({\it absolutely\/})  {\it $(q,p)\,$--\,summing\,}; the
 set of these operators is denoted by $\Pi_{q,p}(E,F)$.  We shall write $\pi_p(E)$ for $\pi_p(I_E)$,   $\pi_p(T)$ for $\pi_{p,p}(T)$,  and $(\Pi_p(E,F), \pi_p)$
for $(\Pi_{p,p}(E,F), \pi_{p,p})$, etc. It is clear that $(\Pi_{q,p}(E,F), \pi_{q,p})$ is a Banach space whenever $F$ is a Banach space, and 
indeed it is a component of an {\it  operator ideal}.   There are extensive studies of these ideals in \cite{DJT, J, JL, Pie, Ry, T-J, W},  for example.

Elements of  $\Pi_1(E,F)$ are also called the ({\it absolutely}) {\it summ\-ing operators\,};
 they are characterized by the property that the series $\sum_{j=1}^\infty Tx_j$ converges absolutely in $F$ whenever  $\sum_{j=1}^\infty  x_j$
 converges weakly unconditionally in $E$.

We shall use the following results about the norms $\pi_p^{(n)}$. \smallskip

\begin{proposition}\label{3.11}
Let $E$ be a normed space, and take $n\in\N$. Then:\s

{\rm (i)} $\pi_2^{(n)}(T)\leq \pi_1^{(n)}(T)$  and $\pi_2(T)\leq \pi_1(T)$  for each $T\in {\B}(E)\,$;\smallskip

{\rm (ii)} $\pi_2(E)=\sqrt{n}$ whenever $\dim E = n\,$; \smallskip

{\rm (iii)} $\pi_1 (E)\geq \sqrt{n}$ whenever $\dim E \geq n\,$; \smallskip

{\rm (iv)}  $\pi_p^{(n)}(E ) =  \pi_p^{(n)}(E'')$ for each $p\geq 1$. \end{proposition}

\begin{proof} Clause (i) is a small variation of \cite[3.3, p.\ 32]{J}, and (ii) is  \cite[Proposition 5.13, p.\ 62]{J} and \cite[Theorem 4.17]{DJT}.  
Clearly (iii) follows from  (i) and (ii).

Clause  (iv)  is essentially \cite[Proposition 17.4, p.\ 157]{J}; it follows from the Principle of Local Reflexivity, Proposition \ref{1.6}.
\end{proof}\smallskip

There have been studies of the relationship of the numbers $ \pi^{(n)}_{q,p}(T)$, and especially when suitable multiples bound 
 $ \pi_{q,p}(T)$.  For a summary, see \cite[Chapter 4]{T-J}; further results are given in   \cite{DJ} and \cite{J2}.
We shall use the following  result of Szarek  \cite[Theorem 3]{Sz}.\smallskip

\begin{theorem}\label{3.11a}
There is a universal constant $C>0$ such that, for each $n\in\N$, each Banach spaces $E$ and $F$ with $\dim E = n$, and  
each $T \in {\B}(E,F)$, there exists $k\in \N$ with $k\leq n\log n$ such that
$$\vspace{-\baselineskip}\pi_1(T) \leq C\pi_1^{(k)}(T)\,.
$$\hspace*{\stretch{1}}\qed
\end{theorem}

In the following corollary, $C$ is the constant of the above theorem. \smallskip

\begin{corollary}\label{3.11b}
 Let $n\in\N$, and let  $F$ be a normed space such that  $\dim F \geq n$. Then
 $$
\pi_1^{(n)}(F) \geq \frac{1}{C} \sqrt{\left[\frac{n}{\log n} \right]}\,\,.
$$
\end{corollary}

\begin{proof} Set $m = [n/\log n]$,  and take  a subspace $E$ of $F$ with $\dim E = m$; let $T$ be the embedding of $E$ into $F$.

We have $\pi_1^{(n)}(T) = \pi_1^{(n)}(E) \leq \pi_1^{(n)}(F)$,  and  it follows from (i) and (ii) of
 Proposition  \ref{3.11} that $\sqrt{m} =\pi_2(T) \leq \pi_1(T)$.
By  Theorem \ref{3.11a}, there exists $k\in \N$ with $k\leq m\log m$ and  $\pi_1(T) \leq C\pi_1^{(k)}(T)$.  But
$$
m\log m \leq \frac{n}{\log n}\,\cdot\,(\log n -\log\log n) < n\,,
$$
 and so $k\leq n$. Thus $\pi_1^{(k)}(T)\leq \pi_1^{(n)}(T)$.

By combining the various inequalities, we obtain the result.\end{proof}\s

For a similar result involving $\pi_p^{(n)}(F)$ for $p\geq 1$, see \cite{JS}.\medskip

\subsection{Related constants}  We now introduce two constants related to $ \pi_1^{(n)}(E )$    that will be referred to later. 
 Recall that $S_E$ denotes the unit sphere of a normed space $E$.

\begin{definition}\label{3.7c}
Let $E$ be a normed space, and take $n\in\N$. Then
$$
 \overline{\pi}_1^{(n)}(E)= \sup\left\{\left(\sum_{j=1}^n \lV x_j\rV\right) : \lV x_1\rV = \cdots= \lV x_n\rV,\,\mu_{1,n}(x_1,\dots, x_n)\leq 1\right\} 
$$
 and   
$$
c_n(E) = \inf\{\sup \{\lV \zeta_1x_1+ \cdots +\zeta_n x_n\rV : \zeta_1, \dots, \zeta_n \in \T\} : x_1,\dots,x_n\in S_E\}\,.
$$
\end{definition}\smallskip

In particular, $c_1(E)=1$ and  
\begin{equation}\label{(3.2b)}
c_2(E) = \inf\{\sup_{\zeta\in\T}\{\lV x_1 + \zeta x_2\rV\} : x_1,x_2 \in S_E\}\,.
\end{equation}
We see that  $(c_n(E): n\in \N)$ is an increasing sequence in $[1,\infty)$.  Let $n\in\N$. Then it follows from (\ref{(3.10e)})  that 
$$
c_n(E) = \inf\{\mu_{1,n}(x_1, \dots, x_n) :  x_1,\dots, x_n\in S_E\}\,.
$$
Clearly, $\overline{\pi}_1^{(n)}(E) \leq  \pi_1^{(n)}(E)$ and $\overline{\pi}_1^{(n)}(E)\,\cdot\,c_n(E) =n$, and so 
\begin{equation}\label{(3.2ba)}
\pi_1^{(n)}(E)\,\cdot\,c_n(E) \geq n \quad (n\in\N)\,.
\end{equation} \s

We first make a simple remark.\label{simpleremark}  Let $(E,\norm)$ be a normed space.  Suppose that $x_1,\dots, x_n\in E$ are such that
 $$\sup \{\lV \zeta_1x_1+ \cdots +\zeta_n x_n\rV : \zeta_1, \dots, \zeta_n \in \T\} = C 
$$
(so that $C \in \R^+$), and take $t\geq 1$.  Then we {\it claim\/} that 
$$\sup \{\lV \zeta_1tx_1+ \zeta_2x_2+\cdots +\zeta_n x_n\rV : \zeta_1, \dots, \zeta_n \in \T\} \geq C\,.
$$
  Indeed, take $\zeta_1, \dots, \zeta_n \in \T$ with  $\lV y \rV =C$, where 
$y = \zeta_1x_1+ \zeta_2x_2+\cdots +\zeta_n x_n$.  Set $z =   \zeta_1x_1- \zeta_2x_2- \cdots -\zeta_n x_n$, so that $\lV z \rV \leq C$.  Then
$$
2(\zeta_1tx_1+ \cdots +\zeta_n x_n )=  (t+1)y +  (t-1)z\,,
$$
and so
$$
2\lV \zeta_1tx_1+ \cdots +\zeta_n x_n\rV \geq     (t+1)\lV y \rV - (t-1)\lV z \rV \geq    (t+1)C - (t-1)C = 2C\,,
$$
as claimed.  It follows that
\begin{equation}\label{(3.2bb)}
c_n(E) \leq \sup \{\lV \zeta_1t_1x_1+ \cdots +\zeta_n t_nx_n\rV : \zeta_1, \dots, \zeta_n \in \T\}
\end{equation}
for each  $x_1,\dots,x_n\in S_E$  and $t_1,\dots,t_n\geq 1$.\s

\begin{proposition}\label{3.7d}
Let $(E, \norm)$ and $(F, \norm)$ be Banach spaces, and let $G$ be a closed subspace of $E$ with  $G$ linearly homeomorphic  to $F$.   Then 
  $$
c_n(E) \leq c_n(F)d(F,G) \quad (n\in\N)\,.
$$

\end{proposition}

\begin{proof}  Set $C = d(F,G)$, and take  $\varepsilon> 0$. Then there is  a bijection $T \in {\B}(F,G)$ with  $\lV T \rV \lV T^{-1}\rV < C + \varepsilon$.

Let $n\in \N$. Then there are elements $y_1, \dots, y_n \in S_F$  such that
$$
\lV \zeta_1y_1+ \cdots +\zeta_n y_n\rV < c_n(F)+\varepsilon\quad (\zeta_1,\dots, \zeta_n \in\T)\,.
$$ 
 Set $x_j = Ty_j/\lV Ty_j\rV\,\;(j\in\N_n)$. Then  $x_1,\dots,x_n\in S_E$.  For $j\in\N_n$, set  $t_j = \lV T^{-1}\rV \lV Ty_j\rV$,
so that $t_j \geq 1$.  By (\ref{(3.2bb)}), we have
\begin{eqnarray*}
c_n(E) &\leq& \sup \{\lV \zeta_1t_1x_1+ \cdots +\zeta_n t_nx_n\rV : \zeta_1, \dots, \zeta_n \in \T\}\\
&=&
\lV T^{-1}\rV \sup \{\lV \zeta_1Ty_1 + \cdots +\zeta_nTy_n\rV : \zeta_1, \dots, \zeta_n \in \T\}\\
&\leq &
(C + \varepsilon)(c_n(F)+\varepsilon)\,.
\end{eqnarray*}
This holds true for each $\varepsilon >0$, and so the result follows.
\end{proof}\s
 
\subsection{Orlicz property} The following   definition is given in \cite[p.\ 43]{J} and \cite[Remark II.D.7]{W}.\smallskip

\begin{definition}\label{3.12a}
A  Banach space $(E, \norm)$ has  the {\it Orlicz property with constant\/} $C$  if $\pi_{2,1}(E) =C$ is finite, so that
$$
C:= \sup\left\{\left(\sum_{j=1}^n \lV x_j\rV^2\right)^{1/2} : x_1,\dots, x_n\in E,\,\mu_{1,n}(x_1,\dots, x_n)\leq 1\right\}< \infty\,.
$$
\end{definition}\smallskip

Clearly $C\geq 1$ in each case. It is shown in \cite[Corollary 11.17]{DJT} and \cite[p.\ 69]{J} that every Banach space
 `of cotype 2' has the Orlicz property.  We remark that, by \cite[Theorem 14.5]{DJT}, an infinite-dimensional
Banach space $E$ with the Orlicz property is `of cotype $q$' for each $q>2$, but an example of Talagrand \cite{Tal}
 shows that there is a Banach lattice $E$ with  the Orlicz property such that $E$  is not of cotype $2$.\smallskip

\begin{theorem}\label{3.12b}
Let $E$ be a Banach space such that $E$ has  the Orlicz property with constant $C$.  Then
 $$
\pi_1^{(n)}(E)  \leq C\sqrt{n}\,,\quad  \sqrt{n} \leq Cc_n(E)\quad (n\in\N)\,.
$$
\end{theorem}

\begin{proof}  Let $n\in\N$ and $x_1,\dots,x_n\in E$.  Then
$$
\sum_{j=1}^n \lV x_j\rV \leq \sqrt{n}\left(\sum_{j=1}^n \lV x_j\rV^2\right)^{1/2} 
$$
by the Cauchy--Schwarz inequality. Now suppose that  $\mu_{1,n}(x_1,\dots,x_n)\leq 1$. Then it follows from Definition \ref{3.12a} that
$$
\sum_{j=1}^n \lV x_j\rV \leq C\sqrt{n}\,,
$$
and so the  result  follows from equations (\ref{(3.10a)}) and (\ref{(3.2ba)}).
\end{proof}\smallskip

In particular, $(\pi_1^{(n)}(E))  = O(\sqrt{n})$ for each  Banach space $E$ of cotype 2.

The following theorem of Orlicz \cite{Or} can be regarded  as the historical beginning of the study of summing operators; a proof is given in
   \cite[Corollary 11.7(a)]{DJT} and \cite[Theorem II.D.6]{W}.\smallskip

\begin{theorem}\label{3.12c}
Let $(\Omega, \mu)$ be a measure space, and take $q \in [1,2]$.  Then the Banach space $L^q(\Omega,\mu)$ has
 cotype $2$, and hence the Orlicz property.\qed
\end{theorem}\smallskip

The  Orlicz constant associated with the space $\ell^{\,q}$ (for $1\leq q\leq 2$) is denoted by $C_q$. 
We know that $C_2 =1$ \cite[3.25]{J} and that $C_1 \leq \sqrt{2}$ \cite[7.6]{J}.\smallskip

\begin{corollary}\label{3.12d}
Let   $q \in [1,2]$.  Then  $$\pi_1^{(n)}(\ell^{\,q})  \leq C_q\sqrt{n}\quad (n\in\N),.
$$
 In particular, $\pi_1^{(n)}(\ell^{\,2})  \leq  \sqrt{n}\,\; (n\in\N)$.
\end{corollary}\smallskip

\begin{proof} This follows from Theorems \ref{3.12b}  and  \ref{3.12c}.
\end{proof}\smallskip

\subsection{Specific spaces} We shall also use the following specific calculations involving the spaces $\ell^{\,p}$, where $p\geq 1$.  Note that, for $n\in\N$,  always 
$\pi_1^{(n)}(\ell_n^{\,p})\leq  \pi_1^{(n)}(\ell^{\,p})\leq \pi_1(\ell^{\,p})$.\smallskip

\begin{proposition}\label{3.12}
 Let $n\in\N$. Then:\smallskip
 
 {\rm (i)}  for each  $q \in [1,2]$,  we have  $\sqrt{n} \leq \pi_1^{(n)}(\ell_n^{\,q})\,$;\s

{\rm (ii)}   $\sqrt{n}\leq  \pi_1^{(n)}(\ell_n^{\,1})\leq \pi_1^{(n)}(\ell^{\,1})=   \pi_1(\ell^{\,1}_n )\leq \sqrt{2n}$\,; \smallskip

{\rm (iii)} $\sqrt{n} = \pi_1^{(n)}(\ell_n^{\,2})= \pi_1^{(n)}(\ell^{\,2})  \leq \pi_1(\ell_n^{\,2}) \leq (2/\sqrt{\pi}\,)\sqrt{n}$\,;\smallskip

{\rm (iv)} $\pi_1^{(n)}(\ell_n^{\,\infty})=  \pi_1^{(n)}(\ell^{\,\infty})=   n $\,; \smallskip

 {\rm (v)} for each  $q \in [2,\infty)$,  we have
$$
\sqrt{n}\leq   n^{1-1/q}  \leq \pi_1^{(n)} (\ell^{\,q}_n) \leq\pi_1^{(n)} (\ell^{\,q})\,.
$$
\end{proposition}

\begin{proof} (i) Take $\zeta =\exp(2\pi{\rm i}/n)$  and then set $f_i  = (\zeta^i, \zeta^{2i},\dots, \zeta^{ni})$ for $i\in\N_n$. Then we have
 $\lV \zeta_1f_1+\cdots +\zeta_nf_n\rV \leq n^{1/2+1/q}$  by Lemma \ref{4.1af}(ii). But $\sum_{i=1}^n \lV f_i\rV = n^{1+1/q}$,
 and so $\pi_1^{(n)}(\ell_n^{\,q})\geq \sqrt{n}$ by (\ref{(3.10b)}).\s

  (ii) We have $\pi_1^{(n)}(\ell^{\,1})=  \pi_1(\ell^{\,1}_n ) \leq \sqrt{2n}$ by \cite[7.18 and 7.12]{J}.\smallskip

(iii)   We have  $\pi_1^{(n)}(\ell^{\,2}) = \sqrt{n}$ by \cite[3.9]{J} and $\pi_1(\ell_n^{\,2}) \leq (2/\sqrt{\pi}\,)\sqrt{n}$ by \cite[8.10]{J}.\s

 (iv) By taking $x_j = \delta_j\,\;(j\in \N_n)$ in equation (\ref{(3.10b)}), we see that $\pi_1^{(n)}(\ell_n^{\,\infty})\geq n$; 
certainly  $\pi_1^{(n)}(\ell^{\,\infty})\leq n$.\smallskip

(v) Let $q \in [2,\infty)$.  By taking $x_j = \delta_j/n^{1/q}\,\;(j\in \N_n)$ in equation (\ref{(3.10b)}), we see that $\pi_1^{(n)}(\ell_n^{\,q})\geq n^{1-1/q}$.
 \end{proof}\smallskip

 We note that the precise value of $\pi_1(\ell_n^{\,2})$ is given in \cite[8.10]{J}, and that $\pi_1(\ell_n^{\,2})>\sqrt{n}$  
for $n\geq 2\,$; the results are due to Gordon \cite{Gor}. We also remark that the following estimates (and more 
general estimates) are contained in  \cite[Theorem 5]{Gor}; we shall not use the results.   (The results in \cite{Gor} are for real-valued 
spaces, but the analogous results follow for our complex-valued spaces, with a possible change in the implicit constants.) \smallskip

\begin{proposition}\label{3.12e}
{\rm (i)} Take $q$ with  $2\leq q < \infty$. Then $\pi_1^{(n)} (\ell^{\,q}_n)\sim n^{1-1/q}$ as $n\to \infty$\,.\smallskip

{\rm (ii)} Take $q$ with $1\leq q\leq 2$. Then $\pi_1^{(n)} (\ell^{\,q}_n)\sim n^{1/2}$ as $n\to \infty$\,. \qed
 \end{proposition}\smallskip

It would be interesting to find the exact values of $\pi_1^{(n)}(\ell^{\,p}_m)$ for each $m,n \in\N$ and $p\in [1,\infty]$.  
Towards this, take $q$ to be the conjugate index to $p$, and let $\lambda_1,\dots, \lambda_n  \in \ell^{\,q}_m$, say
$
\lambda_j = (\lambda_{1j}, \dots, \lambda_{mj})\;\, (j=1,\dots,n)$. Then set
$
\Lambda = (\lambda_{ij}: i\in\N_m,\,j\in\N_n),
$
an $m\times n$-matrix, so that $\Lambda \in \M_{m,n}$. Following Feng and Tonge in \cite{FT}  
(but replacing their $p$ and $q$ by $u$ and $v$, where $1\leq u,v\leq \infty$), we define
$$
\lv \Lambda\rv_{u,v} = \left(\sum_{j=1}^n\left(\sum_{i=1}^m \lv \lambda_{ij}\rv^u\right)^{v/u}\right)^{1/v}
$$
and
$$
\lV \Lambda\rV_{u,v} =\max\{\lV \Lambda x\rV_v : \lV x \rV_u\leq 1\}\,.
$$
By   (\ref{(3.10)}),
the  condition that $\mu_{1,n}(\lambda_1, \dots,\lambda_n)\leq 1$ is just the condition that
$\lV \Lambda\rV_{p,1}\leq 1$.  The number $\sum_{j=1}^n \lV\lambda_j\rV$ is just  $\lv \Lambda\rv_{q,1}$. Thus $\pi_1^{(n)}(\ell^{\,p}_m)$ 
 is the least constant $d$ such that $$\lv \Lambda\rv_{q,1}\leq d \lV \Lambda\rV_{p,1}$$ for each $\Lambda \in \M_{m,n}$. 
 The determination of such a $d$  is exactly a special case of the question addressed in \cite[Problem 1, (4)]{FT}; unfortunately,
 this is a case that is left open in \cite{FT}.
 
 More generally, Feng and Tonge study in \cite{FT}, for fixed $m,n \in\N$,  the constant
 $$
 d_{m,n}(u,v, r,s) =\sup\{\lv \Lambda\rv_{u,v} : \Lambda \in \M_{m,n},\, \lV \Lambda\rV_{r,s}\leq 1\}\,;
 $$
this number was determined in  the case where $u=v\geq 2$ for  most (but not all) choices of $r,s \in [1,\infty)$.  We see that the above argument shows that 
$$
d_{m,n}(u,v, r,s) = \pi_{v,r'}^{(n)}(I : \ell^{\,s}_m  \to  \ell^{\,u}_m)\,,
$$ 
where $I$ is the identity map and $r'$ is the conjugate index to $r$.
\medskip

\section{Characterizations of the maximum multi-norm}
\subsection{Characterizations in terms of weak summing norms} We now give  some alternative descriptions of the maximum multi-norm; 
these remarks will be used to give some calculations of the maximum rate of growth  for certain Banach spaces $E$.

Let $E$ be a normed space, and take $n \in \N$. Then we set
$$
S_n = \{ (\zeta_1x, \dots, \zeta_nx)\in E^n : \zeta_1, \dots, \zeta_n \in \T,\, x \in S_E\}
$$
and $K_n = \overline{\rm co}(S_n)$, the closed convex hull of $S_n$, so that  $K_n $ is absolutely convex and absorbing. 
Then the Minkowski functional, temporarily called  $p_n$, of $K_n$ is a norm on $E^n$. Since $A_\sigma(K_n) =K_n$ for each $\sigma \in {\mathfrak S}_n$ and 
$M_\alpha(K_n) \subset K_n$  for each $\alpha \in {\overline{\D}}^{\,n}$, the norm $p_n$ satisfies Axioms (A1) and (A2).  Now let $n$ vary in 
$\N$, so that we obtain a sequence $(p_n : n\in \N)$ of norms.  This sequence clearly   satisfies (A3) and (A4), and so $(p_n : n\in\N)$
 is a multi-norm on $\{E^n: n\in\N\}$.  Further, let $(\norm_n :n\in\N)$ be any multi-norm on $\{E^n: n\in\N\}$, and let $B_n$ be the closed unit 
ball of $(E^n, \norm_n)$.  Then we see that  $K_n \subset B_n$, and so $(p_n : n\in \N)$ is the maximum multi-norm on $\{E^n:n\in\N\}$.
 We conclude that the closed unit ball of $(E^n, \norm^{\max}_n)$ is the set $K_n$. 

The first characterization of $\norm_n^{\max}$  follows easily from the  Hahn--Banach theorem; in the proof, we temporarily write $p_n$ for
 $\norm_n^{\max}$,  $q_n$ for the dual norm to $p_n$,  and we write $\mu_{1,n}$ for the weak $1$--summing norm on $(E')^n$.  \s

\begin{theorem}\label{3.4d}
Let $E$ be  a normed space, and take $n\in\N$.  Then
 $$
 \lV (x_1, \dots, x_n)\rV^{\max}_n= \sup \left\{ \lv \sum_{j=1}^n \langle x_j, \,\lambda_j\rangle\rv : 
\mu_{1,n}(\lambda_1,\dots,\lambda_n) \leq 1\right\}
 $$
for each   $x_1, \dots, x_n\in E$, where   the supremum is taken over $\lambda_1,\dots,\lambda_n\in E'$.  Further, the dual of
 $\norm_n^{\max}$  is $\mu_{1,n}$.  \end{theorem}

\begin{proof} Let   $x_1, \dots, x_n\in E$, and set $x = (x_1, \dots, x_n)$.
By the Hahn--Banach theorem,
$$
\lV x\rV^{\max}_n  = \sup\{\lv \langle x,\,\lambda\rangle \rv: q_n(\lambda)\leq 1 \}\,,
$$
where $\langle x,\,\lambda\rangle = \sum_{j=1}^n \langle x_j, \,\lambda_j\rangle$ for $\lambda = (\lambda_1, \dots, \lambda_n) \in (E')^n$, 
as in (\ref{(1.6e)}). However it is clear that $q_n(\lambda) \leq 1$ if and only if $\lv \langle y,\,\lambda\rangle\rv\leq 1\;\,(y \in S_n)$,
 and so $q_n(\lambda)  \leq 1$ if and only if
 $$
\lv \sum_{j=1}^n \langle \zeta_j y, \,\lambda_j\rangle\rv \leq 1\quad
(\zeta_1,\dots, \zeta_n \in \T,\, y \in S_E)\,.
$$
 This latter occurs  if and only if $\sum_{j=1}^n\lv \langle y, \,\lambda_j\rangle\rv \leq 1$ for each $y \in E_{[1]}$.

 Further, $q_n(\lambda) \leq 1$ if and only if
 $$
\lv \sum_{j=1}^n \langle  y, \,\zeta_j\lambda_j\rangle\rv \leq 1\quad(\zeta_1,\dots, \zeta_n \in \T,\, y \in S_E)\,,
$$
 and this occurs if and only if  $\mu_{1,n}(\lambda_1,\dots,\lambda_n) \leq 1$. Hence   $q_n = \mu_{1,n}$.

 The result follows.\end{proof}\smallskip

 Thus we can confirm from Theorem \ref{2.13} that $(((E')^n, \mu_{1,n}): n\in\N)$ is a dual multi-Banach space,
 as already noted in Theorem \ref{3.10b}. For a related result, see Theorem \ref{4.0ka}.\smallskip
 
 \begin{corollary}\label{3.4dd}
Let $E=\ell^{\,r}$, where $r\geq 1$. Then 
$$
\lV (\delta_1,\dots,\delta_n)\rV^{\max}_n  = n^{1/r}\quad (n\in\N)\,.
$$
\end{corollary}

\begin{proof}  By Corollary  \ref{2.1bz}, $\lV (\delta_1,\dots,\delta_n)\rV^{\max}_n  \leq n^{1/r}\,\; (n\in\N)$.

The conjugate index to $r$ is $s$.  Take $\lambda_j=\delta_j\in E'\,\;(j\in\N_n)$. By equation (\ref{(3.10e)}), 
$$
\mu_{1,n}(\delta_1,\dots,\delta_n)=\sup\{ \lV (\zeta_1, \dots,\zeta_n)\rV_{\ell^{\,s}}:\zeta_1,\dots,\zeta_n\in\T\} = n^{1/s}\quad (n\in\N)\,,
$$
and so $\lV (\delta_1,\dots,\delta_n)\rV^{\max}_n\geq n/n^{1/s} = n^{1/r}\,\; (n\in\N)$.
\end{proof}\s

 \begin{corollary}\label{3.4da}
Let $E$ be  a normed  space,  and take  $n\in\N$.  Then
\begin{eqnarray*}
 \varphi_n^{\,\max}(E)
&=&
 \sup\left\{\sum_{j=1}^n \lV \lambda_j\rV :  \mu_{1,n}(\lambda_1,\dots,\lambda_n)\leq 1\right\}\\
&=&
\sup\left\{\sum_{j=1}^n \lV \lambda_j\rV :  \lV \sum_{j=1}^n \zeta_j\lambda_j\rV \leq 1\,\;(\zeta_1, \dots,\zeta_n\in\T)\right\}
\,,
\end{eqnarray*}
where   the supremum is taken over $\lambda_1,\dots,\lambda_n\in E'$,  and so 
$$
\varphi_n^{\,\max}(E) =\pi_1^{(n)}(E')\geq n/c_n(E')\,.
$$
 \end{corollary}

\begin{proof}  Take $\lambda_1,\dots,\lambda_n \in E'$. Then   
$$
\sup \left\{ \lv \sum_{j=1}^n \langle x_j, \,\lambda_j\rangle\rv : x_1,\dots,x_n\in E_{[1]}\right\} =  \sum_{j=1}^n\lV \lambda_j\rV\,,
$$
and so the first equality holds. The final  remark  follow  from equations (\ref{(3.10a)}) and (\ref{(3.2ba)}).
\end{proof}\s 

\begin{corollary}\label{3.4db}
Let $E$ be  a normed  space,  and take  $n\in\N$. Then
$$
\varphi_n^{\,\max}(E') =  \pi_1^{(n)}(E)\quad {\rm and}\quad\varphi_n^{\,\max}(E) = \varphi_n^{\,\max}(E'')\,.
$$
 \end{corollary}

\begin{proof} These  follow  from    Proposition \ref{3.11}(iv) and Corollary \ref{3.4da}.\end{proof}\smallskip

\begin{corollary}\label{3.4dc}
There is a constant $C>0$ such that, for each  $n\in\N$ and each normed space $E$ with $\dim E \geq n$, we have 
$$
\varphi_n^{\,\max}(E) \geq \frac{1}{C} \sqrt{\left[\frac{n}{\log n} \right]}\,\,.
$$
\end{corollary}

\begin{proof}  Since $\dim E \geq n$, we have $\dim E' \geq n$, and so this follows from Corollaries \ref{3.11b} and \ref{3.4da}.
\end{proof}\smallskip

We do not know if the factor `$\log n$' is required in the above theorem; we shall see in Theorem \ref{3.19b} that it
is not required in the case where the space $E$ is infinite-dimensional.\medskip

\subsection{The dual of the minimum dual multi-norm}    Let $(E, \norm )$ be  a normed space.  Then we have seen that
  $(\mu_{1,n}: n\in\N)$ is the minimum  dual multi-norm  on $\{E^n:n\in\N\}$, and so
   $(\mu'_{1,n}: n\in\N)$ is a  multi-norm on $\{(E')^n:n\in\N\}$. We ask if it is the maximum multi-norm? To see that this is the case,
 take $\lambda_1,\dots, \lambda_n \in E'$, and set $\lambda = (\lambda_1,\dots, \lambda_n)$.  By Theorem \ref{3.4d},  we have 
 $$
 \lV \lambda \rV_n^{\max} =   \sup \left\{ \lv \sum_{j=1}^n \langle \Lambda_j, \,\Lambda_j\rangle\rv : \Lambda_1,\dots, \Lambda_n\in E'', \,
\mu_{1,n}(\Lambda_1,\dots,\Lambda_n) \leq 1\right\}\,.
 $$
On the other hand, we have
$$
\mu_{1,n}'(\lambda) =  \sup \left\{ \lv \sum_{j=1}^n \langle x_j, \,\Lambda_j\rangle\rv : x_1,\dots, x_n\in E, \,
\mu_{1,n}(x_1,\dots,x_n) \leq 1\right\}\,,
$$
where we recall that the restriction of $\mu_{1,n}$ defined  on $(E'')^n$ to   $E^n$ is   just $\mu_{1,n}$ defined on $E^n$.
 Clearly,  $\mu_{1,n}'(\lambda) \leq  \lV \lambda \rV_n^{\max}$.
 The reverse  inequality  follows from the Principle of Local Reflexivity.\s

 \begin{theorem}\label{3.4e}
Let $E$ be  a normed space, and take  $n\in \N$. For each $\lambda \in (E')^n$, we have $\mu_{1,n}'(\lambda) =  \lV \lambda \rV_n^{\max}$.
 \end{theorem}

 \begin{proof} Take $\varepsilon >0$.   Then there exists  $\Lambda =(\Lambda_1,\dots, \Lambda_n) \in (E'')^n$ with  $ \mu_{1,n}(\Lambda)\leq 1$ and
 $$
 \lv \sum_{j=1}^n \langle \Lambda_j, \lambda_j\rangle\rv \geq \lV \lambda \rV_n^{\max}  - \varepsilon\,.
 $$
 Set $X = \lin\{\Lambda_1,\dots, \Lambda_n\}$ and $Y = \lin\{\lambda_1,\dots,\lambda_n\}$, so that $X$ and $Y$ are 
 finite-dimen\-sional subspaces of $E''$ and $E'$, respectively.
 By the Principle of Local Reflexivity, Theorem  \ref{1.6}, there is an injective, bounded linear map $S : X \to E$  with $\lV S\rV  < 1 + \varepsilon$ and with
$\langle S(\Lambda_j),\,\lambda_j\rangle =   \langle  \Lambda_j,\,\lambda_j\rangle\;\, (j\in\N_n)$.
 Set $x = (S(\Lambda_1), \dots,S(\Lambda_n)) \in E^n$. Then it follows from (\ref{(3.2aa)}) that $\mu_{1,n}(x) \leq (1 + \varepsilon)\mu_{1,n}(\Lambda)$    
 (with $T$ taken to be  $S: X \to S(X)$),  and so $\mu_{1,n}(x) \leq 1 + \varepsilon$. Now we have
$$
\mu_{1,n}'(\lambda)  \geq   \frac{1}{1 + \varepsilon}\lv \sum_{j=1}^n \langle S(\Lambda_j),\,\lambda_j\rangle\rv =
\frac{1}{1 + \varepsilon}\lv \sum_{j=1}^n \langle  \Lambda_j ,\,\lambda_j\rangle\rv\geq \frac{1}{1 + \varepsilon}(\lV \lambda \rV_n^{\max} - \varepsilon)\,.
$$
 This holds true for each  $\varepsilon >0$, and so $\mu_{1,n}'(\lambda)\geq \lV \lambda \rV_n^{\max}$. 

Thus $\lV \lambda \rV_n^{\max}=\mu_{1,n}'(\lambda)$, as required.
 \end{proof}\s 

\begin{theorem}\label{3.4f}
Let $E$ be  a normed space. Then $(\mu_{1,n}' : n\in \N)$  is the maximum multi-norm  on the family $\{(E')^n : n\in\N\}$.\qed
\end{theorem}\s

In summary, we have the following.  Let $E$ be  a normed space. Then the minimum and maximum multi-norms based on $E$ are $(\lV \,\cdot\,\rV^{\min}_n : n\in\N)$ and 
$(\lV \,\cdot\,\rV^{\max}_n : n\in\N)$, respectively. The dual of these multi-norms are the maximum and minimum dual multi-norms, respectively,  on the family 
$\{(E')^n: n\in\N\}$, and the latter  is exactly the multi-norm $(\mu_{1,n}: n\in\N)$.
Combining these remarks, we have the following consequence.\smallskip

\begin{corollary}\label{3.4g}
Let $E$ be  a normed space. Then the second dual of the maximum multi-norm $(\lV \,\cdot\,\rV^{\max}_n : n\in\N)$ based on $E$
is  the maximum multi-norm based on $E''$.\qed
\end{corollary}\medskip

\subsection{Characterizations in terms of projective  norms}  Our second characterization of the maximum multi-norm
 involves a projective norm.\s

\begin{definition}\label{3.5}
Let $E$ be a linear space. A  subset $S$ of $E$ is {\it one-dimensional} if  $S \subset \C x$ for some $x \in E$. 
 A family $\{y_1,\dots, y_m\}$ in $E$ has an {\it elementary representation} if there exist  $n \in \N$ and   $x_{ij} \in E$ for $i \in \N_m$ and $j\in\N_n$  with
$$
y_i = \sum_{j=1}^n x_{ij}\quad (i \in \N_m)
$$
and such that  $\{x_{ij} :i \in\N_m\}$ of $E$ is one-dimensional for each $j\in \N_n$.
\end{definition}\smallskip

Each family $\{y_1,\dots, y_m\}$  in the linear space $E$ has at least one elementary representation.
Indeed, each such family has a representation of the form
\begin{equation}\label{(3.1b)}
y_i = \sum_{j=1}^n \alpha_{ij}x_j\quad (i \in \N_m)\,,
\end{equation}
where $n\in \N$, $\alpha_{ij} \in \C\,\;(i \in \N_m, \,j\in\N_n)$, and $x_j \in E\,\;(j\in\N_n)$ have the property that  $\lV x_1\rV = \cdots = \lV x_n\rV =1$.

Let  $((E^n, \norm_n) : n\in \N)$ be  a  multi-normed space.  Take  $k\in\N$, and suppose that
  $\{x_1,\dots, x_k\}$ is a one-dimensional set in $E$. Then clearly
$$
\lV (x_1, \dots, x_k)\rV_k =\max\{\lV x_1\rV, \dots, \lV x_k\rV\}\,.
$$
Now let $\{y_1,\dots, y_m\}$ be a family in $E$   with the   elementary represent\-ation of equation (\ref{(3.1b)}).  Then
\begin{eqnarray*}
\lV (y_1, \dots, y_m)\rV_m &=& \lV \sum_{j=1}^n (\alpha_{1j}x_j, \dots,\alpha_{mj}x_j)\rV_m   
 \leq   \sum_{j=1}^n \lV (\alpha_{1j}x_j, \dots, \alpha_{mj}x_j)\rV_m \\ 
&=&\sum_{j=1}^n \max\{\lv \alpha_{ij}\rv : i\in\N_m\}\,,
\end{eqnarray*}
and so 
\begin{equation}\label{(3.1a)}
\lV (y_1, \dots, y_m)\rV_m  \leq  \LV (y_1, \dots, y_m)\RV_m\,,
\end{equation}
where
\begin{equation}\label{(3.1aa)}
\LV (y_1, \dots, y_m)\RV_m =\inf\left \{\sum_{j=1}^n \max\{\lv \alpha_{ij}\rv : i\in\N_m\}\right\}
\end{equation}
and the infimum is taken over all elementary representations as specified in equation (\ref{(3.1b)}) of the family $\{y_1,\dots, y_m\}$.\smallskip

\begin{theorem}\label{3.6}
  Let $E$ be a normed space. Then the above sequence $(\LV \,\cdot\,\RV_n : n\in\N)$  is the maximum 
multi-norm  on  $\{E^n : n \in\N\}$, and, for each $m\in\N$, we have 
$$
 \varphi_m^{\max}(E)= \sup\left\{\inf\left\{ \sum_{j=1}^n \max\{\lv \alpha_{ij}\rv : i\in\N_m\}\right\}: y_1, \dots, y_m \in E_{[1]}\right\}\,,
$$
where the infimum is taken over all elementary representations of the form $$y_i = \sum_{j=1}^n \alpha_{ij}x_j\quad (i\in\N_m)$$ 
for which  $n\in \N$, $\alpha_{ij} \in \C\,\;(i \in \N_m, \,j\in\N_n)$, and $x_j \in E_{[1]}\,\;(j\in\N_n)$. \end{theorem}

\begin{proof}  It is clear from equation (\ref{(3.1a)}) that it is sufficient  to show that  $(\LV \,\cdot\,\RV_n : n\in\N)$ 
 is a multi-norm on $\{E^n : n \in\N\}$.  However it is easily checked that $\LV \,\cdot\,\RV_n$ is a norm on $E^n$ for each $n\in\N$, that
$\LV \,\cdot\, \RV_1$ is the initial norm on $E$, and that Axioms (A1), (A2), and (A4) are satisfied.  It follows that  $(\LV \,\cdot\,\RV_n : n\in\N)$  
is indeed  a multi-norm.\end{proof}\smallskip

We can re-express  the above evaluation of $\norm^{\max}_m$  as follows.  In the statement,  $\pi$ denotes the projective norm on the space $\ell^{\,\infty}_m\otimes E$.  
More  general versions of the following theorem will be given  in \cite{DDPR1}. \smallskip

\begin{theorem}\label{3.8a}
Let $E$ be a normed space, and take $m\in\N$. Then
$$(E^m, \norm^{\max}_m) \cong (\ell^{\,\infty}_m\otimes E, \norm_\pi)\,.
$$
\end{theorem}

\begin{proof}  Let $m \in\N$, and take $\{\delta_1, \dots,\delta_m\}$ to be the standard basis of $\ell^{\,\infty}_m$.  Then the map
$$
T : (y_1, \dots,y_n)\mapsto \sum_{i=1}^m \delta_i\otimes y_i\,,\quad E^m \to  \ell^{\,\infty}_n\otimes E\,,
$$
is a linear bijection.  Let $y_i = \sum_{j=1}^n \alpha_{ij}x_j$ be the elementary representation of $y_i$ for $i \in \N_m$, as in (\ref{(3.1b)}), 
where $\lV x_1\rV =\cdots = \lV x_n\rV =1$, and set $z = T(y_1,\dots,y_m)$. Then
$$
z= \sum_{i=1}^m\left(\sum_{j=1}^n \alpha_{ij}\delta_i\right)\otimes x_j\,,
$$
and every representation of $z$ as an element of $\ell^{\,\infty}_m\otimes E$ has this form.  By (\ref{(3.1aa)}), we have   
$$
\lV z \rV_\pi =\inf\left\{\sum_{j=1}^n \max\{\lv\alpha_{ij}\rv : i\in\N_m\right\}=\lV(y_1, \dots,y_m)\rV^{\max}_m\,.
$$
This shows that $T$ is an isometry.
\end{proof}\smallskip
 
The following is related to equation (\ref{(3.10c)}).\smallskip

\begin{corollary}\label{3.8b}
Let $E$ be a normed space. Then
$$
((E')^n, \mu_{1,n}) = \ell^{\,p}_n(E')^{w}\cong {\B}(E,\ell^{\,1}_n)\cong {\B}(\ell_n^{\,\infty},E')\quad(n\in\N)\,.
$$
\end{corollary}

\begin{proof} Let $n\in\N$.  By Theorem \ref{3.4d}, the dual space to $(E^n, \norm_n^{\max})$ is $((E')^n, \mu_{1,n})$.
 By (\ref{(1.5a)}), the dual  space  of  $(\ell^{\,\infty}_n\otimes E, \norm_\pi)$ is the Banach space 
${\B}(E,\ell^{\,1}_n)\cong {\B}(\ell_n^{\,\infty},E')$.\end{proof}
\medskip

\section{The function $\varphi_n^{\,\max}$ for some examples}  

\noindent We shall  calculate the value  of $\varphi_n^{\,\max}(E)$  for  some standard Banach spaces $E$;
sometimes we shall use elementary means, even if more general theorems are available. \s

\subsection{The spaces $\ell^{\,p}$}   
 In the following examples, $p\in [1,\infty]$, and $q$ is the conjugate index to $p$.  Take $n\in\N$.  Then $\ell_n^{\,p}$ is $1$-complemented in
 $\ell^{\,p}$, and so it follows from Corollary \ref{3.3ae} that $\varphi_n^{\,\max}(\ell^{\,p}) \geq \varphi_n^{\,\max}(\ell_n^{\,p})$.\smallskip

\begin{example}\label{3.7e}
{\rm Let $n\in\N$. Then we  have $\pi_1^{(n)}(\ell_n^{\,\infty})=  \pi_1^{(n)}(\ell^{\,\infty})=   n$ 
by   Prop\-osition  \ref{3.12}(iv), and so, by Corollary \ref{3.4da},
$$
 \varphi_n^{\,\max}(\ell_n^{\,1}) = \varphi_n^{\,\max}(\ell^{\,1}) =n\,. 
$$
The maximum multi-norm on the family $\{(\ell^{\,1})^n : n\in\N\}$ will be calculated in Theorem \ref{4.1ad}.}\qed
\end{example}\smallskip

\begin{example}\label{3.7f}
{\rm Let $n\in\N$, and  take $q>  1$.  Set $F= \ell_n^{\,q}$. By the choice  $\lambda_j =\delta_j\in S_F$ for
 $j\in\N_n$,  we see that   $c_n(\ell^{\,q}_n)\leq n^{1/q}$. Now take  $p> 1$. Then $(\ell^{\,p}_n)' =\ell^{\,q}_n$,  whence   
$$
\varphi_n^{\,\max}(\ell^{\,p})\geq \varphi_n^{\,\max}(\ell^{\,p}_n)  \geq n/n^{1/q} 
$$
by Corollary \ref{3.4da}, and so $\varphi_n^{\,\max}(\ell^{\,p}) \geq n^{1/p}$.}\qed
\end{example}\smallskip

We make a trivial preliminary remark: for $\zeta \in \T$ and $q\geq 1$, we have
\begin{equation}\label{(3.8)}
\lv 1+ \zeta\rv^q+ \lv 1- \zeta\rv^q \leq \max\{2\,\cdot\,2^{q/2},2^q\}\,.
\end{equation}

\begin{example}\label{3.7i}
 {\rm  Let $F =  \ell^{\,q}_2$, where $q\geq 1$. We choose $$\lambda_1 = (1, 1)/2^{1/q}\quad{\rm  and}\quad 
\lambda_2 = (1, -1)/2^{1/q}\,,
$$
 so that $\lambda_1,\lambda_2 \in S_F$.  By (\ref{(3.8)}),  
 $\sup\{\lV \zeta_1\lambda_1 + \zeta_2 \lambda_2 \rV :\zeta_1,\zeta_2\in\T\}\leq \max\{\sqrt{2}, 2^{1/p}\}$.

Now suppose that $p\geq 2$, so that $2^{1/p}\leq \sqrt{2}$. Then $c_2(F) \leq \sqrt{2}$, and so,  by Corollary  \ref{3.4da}, we have
\begin{equation}\label{(3.7)}
\varphi_2^{\,\max}(\ell^{\,p})\geq \varphi_2^{\max}(\,\ell_2^{\,p})\geq\sqrt{2}\,.
\end{equation}
By Examples \ref{3.7e} and \ref{3.7f}, this inequality also holds for  $p\in [1,2]$, and so (\ref{(3.7)})  holds for all $p\geq 1$.
}\qed
\end{example}\smallskip

\begin{example}\label{3.7j}
{\rm Let $n\in\N$. We have noted in Proposition \ref{3.12}(iii) the equality    $$ \pi_1^{(n)}(\ell_n^{\,2})= \pi_1^{(n)}(\ell^{\,2}) = \sqrt{n}\,,
$$ 
and hence, by Corollary \ref{3.4da}, we have 
$$
\varphi_n^{\,\max}(\ell^{\,2})=  \varphi_n^{\,\max}(\ell_n^{\,2})= \sqrt{n}\,.
$$
We wish to obtain this result directly from our definitions.

 Let $E =  \ell^{\,2}$,  so that $E' = E$, and we write $F$ for $E'$; the usual inner product on $E$ is denoted by
 $[\,\cdot\,,\,\cdot\,]$.  Let $(\norm_n : n\in \N)$ be any multi-norm on $\{E^n: n\in\N\}$, and take  $n \in\N$. 
For $x_1,\dots, x_n \in E_{[1]}$ and $\zeta =\exp({2\pi{\rm i}}/{n})$, we  have
$$
\sum_{j=1}^n \lV \sum_{m=1}^n \zeta^{jm}x_m\rV^2 = \sum_{j=1}^n  \sum_{m=1}^n [\zeta^{jm}x_m,\zeta^{jm}x_m] = \sum_{j=1}^n  \sum_{m=1}^k \lV x_m\rV^2\leq k^2\,,
$$
and so, by H{\"o}lder's inequality, we have
$$
\sum_{j=1}^n \lV \sum_{m=1}^n \zeta^{jm}x_m\rV \leq k^{1/2}\left(\sum_{j=1}^n \lV \sum_{m=1}^n \zeta^{jm}x_m\rV^2\right)^{1/2}\,.
$$
Hence
$$
\frac{1}{n}\sum_{j=1}^n \lV \sum_{m=1}^n \zeta^{jm}x_m\rV \leq n^{1/2}\,.
$$
It follows from the   Proposition \ref{2.5a}  that $\lV (x_1,\dots, x_n)\rV_n \leq n^{1/2}$,
 and thus we have  $\varphi_n^{\,\max}(E) \leq n^{1/2}$.

By Example \ref{3.7f}, $\varphi_n^{\,\max}(E) \geq n^{1/2}$, and so $\varphi_n^{\,\max}(\ell^{\,2}) = n^{1/2}$.

It now follows from Corollary \ref{3.4da}  that $c_n(\ell^{\,2}) \geq n^{1/2}$, and so, by Example \ref{3.7f}, we have 
$c_n(\ell^{\,2})=c_n(\ell_n^{\,2}) = n^{1/2}$.\qed}
\end{example}\smallskip

\begin{example}\label{3.7jc}
{\rm Let $n\in\N$,  and take $F= \ell^{\,q}_n$, where $q\in [1,2]$.  

Let  $\zeta =\exp({2\pi{\rm i}}/{n})$, and then set
$$
\lambda_j = \frac{1}{n^{1/q}}(\zeta^j, \zeta^{2j}, \dots, \zeta^{nj})\in S_F\quad (j\in\N_n)\,,
$$
 For each $\zeta_1, \dots, \zeta_n \in \T$, we have  $  \lV \zeta_1\lambda_1+ \cdots +\zeta_n\lambda_n\rV \leq \sqrt{n}$
 by Lemma \ref{4.1af}(ii), and so $c_n(\ell^{\,q})\leq c_n(\ell^{\,q}_n)\leq \sqrt{n}$.

Now take $p$ with  $2\leq p<\infty$,  so that $q \in (1, 2]$.  Set $E= \ell^{\,p}_n$ and $F = E' =\ell^{\,q}_n$.
By Corollaries \ref{3.3ae} and \ref{3.4da},    $\varphi_n^{\,\max}(\ell^{\,p}) \geq \varphi_n^{\,\max}(\ell^{\,p}_n) \geq \sqrt{n}$.
By Corollaries \ref{3.12d} and   \ref{3.4da},  $\varphi_n^{\,\max}(\ell^{\,p})\leq C_q\sqrt{n}$, where $C_q$ is the Orlicz constant  for $\ell^{\,q}$,
and so, again by  Corollary \ref{3.4da}, $c_n(\ell^{\,q})\geq \sqrt{n}/C_q$.

In particular, we have shown that
$$
\sqrt{n} \leq \varphi_n^{\,\max}(\ell^{\,p}_n) \leq  \varphi_n^{\,\max}(\ell^{\,p})\leq C_q\sqrt{n}\quad (n\in\N)
$$
 whenever $2\leq p<\infty$.}\qed
 \end{example}\smallskip

\begin{example}\label{3.7jd}
{\rm Let $n\in\N$. As in Example \ref{3.7jc}, $c_n(\ell^{\,1}_n)\leq \sqrt{n}$, and so, by Corollary \ref{3.4da}, 
 $\varphi_n^{\,\max}(\ell^{\,\infty}_n) \geq \sqrt{n}$.  Thus it follows from Proposition \ref{3.12}(ii) and Corollary \ref{3.4db} that
$$
\vspace{-\baselineskip}\sqrt{n} \leq \varphi_n^{\,\max}(\ell^{\,\infty}_n) \leq  \varphi_n^{\,\max}(\ell^{\,\infty})\leq  \sqrt{2n}\,.
$$
\hspace*{\stretch{1}}\qed}  
\end{example}\smallskip

The above two results are in accord with the estimates of Gordon given in Proposition \ref{3.12e}.\s

\begin{example}\label{3.7je}
{\rm This example shows that strict inequality can arise in equation (\ref{(3.1c)}). 

Indeed, take $n\in\N$, and consider   $E= \ell^{\,\infty}$, so that
 $\varphi_n^{\,\max}(E)\leq  \sqrt{2n}$  by Example \ref{3.7jd}. By \cite[Theorem 2.5.7]{AK}, each separable Banach space is isometrically isomorphic 
to a closed subspace of $\ell^{\,\infty}$, and so we can regard $ \ell^{\,1}$ as a closed subspace of $E$. However, by Example \ref{3.7e}, we know that 
$\varphi_n^{\,\max}(\ell_n^{\,1}) = \varphi_n^{\,\max}(\ell^{\,1}) =n$.   Thus $F:= \ell_n^{\,1}$ is a closed subspace  of $E$ with $\dim F =  n$ and
 $$\varphi_n^{\,\max}(E) \leq \sqrt{2n} < n = \varphi_n^{\,\max}(F)$$ for $n\geq 3$. 
}\qed 
\end{example}\smallskip

The next result refers to  the Banach--Mazur distance $d(F, \ell^{\,2}_n)$ for a
normed space $F$  with $\dim F = n$. \s

\begin{proposition}\label{3.7jb}
Let $E$ be a Banach space. Then 
$$
\varphi_n^{\,\max}(E) \leq \sqrt{n}\,\sup\{d(F, \ell_n^{\,2}):  F \subset E,\,\dim F =n\}\quad(n\in\N)\,.
$$
\end{proposition}

\begin{proof}  This follows from equation (\ref{(3.1c)}), Corollary \ref{3.3ad}, and Example \ref{3.7j}.
\end{proof}\s

\begin{example}
\label{3.7ja}
{\rm  Let $p\in [1,\infty]$, and take $n\in\N$.  By \cite[Corollary III.B.9]{W}, we have  
\begin{equation}\label{(3.2c)}
d(F, \ell^{\,2}_n) \leq n^{\lv \frac{1}{p}- \frac{1}{2}\rv}\quad (n\in\N) 
\end{equation}
whenever $F$ is a subspace of $\ell^{\,p}$ with $\dim F =n$.

Now suppose that  $p\in [1,2]$.  By (\ref{(3.2c)}), $d(F, \ell^{\,2}_n)\leq n^{1/p-1/2}$
 whenever $F$ is a subspace of $\ell^{\,p}$ with $\dim F = n$, and so $\varphi_n^{\,\max}(\ell^{\,p}) \leq n^{1/p}$  by Proposition \ref{3.7jb}.
By Example \ref{3.7f}, $\varphi_n^{\,\max}(\ell^{\,p}_n) \geq n^{1/p}$, and so we see that
$$
\varphi_n^{\,\max}(\ell^{\,p})=  \varphi_n^{\,\max}(\ell^{\,p}_n)  = n^{1/p}\quad (n\in\N)\,.
$$
This is a sharpening of the result of Gordon contained in Proposition \ref{3.12e}.

It now follows from Corollary \ref{3.4da}   and Example \ref{3.7f} that we have 
$$
c_n(\ell^{\,q}) = c_n(\ell_n^{\,q}) = n^{1/q}\quad (n\in\N)
$$
 whenever $q\geq 2$.}\qed 
\end{example}\smallskip

We summarize some results of this section; again, $q$ is the conjugate index to $p\in [1,\infty]$ and $C_q$ is the Orlicz constant
  for $\ell^{\,q}$, where $q \in [1,2]$.\smallskip

 \begin{theorem} \label{3.7k}
Let $n\in\N$. Then:\s
 
{\rm (i)} for   $p\in [1,2]$, we have $\varphi_n^{\,\max}(\ell^{\,p})= \varphi_n^{\,\max}(\ell^{\,p}_n)  = n^{1/p}$\,;\smallskip

{\rm (ii)}  for  $p\in [2,\infty]$, we have $\sqrt{n} \leq \varphi_n^{\,\max}(\ell^{\,p}_n) \leq  \varphi_n^{\,\max}(\ell^{\,p})\leq C_q\sqrt{n}$\,.\qed
\end{theorem}\medskip

\subsection{The spaces $L^{p}$}  We now consider, more briefly, spaces denoted by  $L^{p}:= L^{p}(\Omega, \mu)$ for a measure space $(\Omega, \mu)$. 
 Throughout, we shall suppose that $L^{p}$ is  infinite dimensional, and so, for each $n\in\N$, there exist pairwise-disjoint, measurable 
subsets $X_1,\dots,X_n$ of $\Omega$ with   $0< \mu(X_i)<\infty\,\;(i\in\N_n)$; we may suppose that $\Omega$ is $\sigma$-finite.  
We shall determine the rate of growth of the sequence $(\varphi_n^{\,\max}(L^{p}): n\in\N)$.\s

  \begin{theorem} \label{3.14}
Let $n\in\N$. Then:\s
  
{\rm (i)} for   $p\in [1,2]$, we have  $\varphi_n^{\,\max}(L^{p}) = n^{1/p} \,\;(n\in\N)$;\s

{\rm (ii)}  for  $p\in [2,\infty]$, we have  $\varphi_n^{\,\max}(L^{p}) \sim \sqrt{n}$ as  $n\to \infty$.
\end{theorem}

 \begin{proof} Take $p\in [1,\infty]$, with conjugate index  $q$.  Fix  $n\in \N$, and take  measurable subsets $X_1,\dots,X_n$ of 
$\Omega$ with $0< \mu(X_i)< \infty\,\;(i\in\N_n)$.
 For each $i\in\N_n$,  set $\chi_i= \chi_{X_i}/\mu(X_i)^{1/q}$ when $q< \infty$ and $\chi_i= \chi_{X_i}$ when $q=\infty$,  so that $\lV \chi_i\rV =1$ 
in $L^q= (L^{p})'$ for each $p\in [1,\infty]$.  Clearly,
$$
\lV \zeta_1 \chi_1 +\cdots + \zeta_n \chi_n\rV_{L^q} = \lV (\zeta_1,\dots,\zeta_n)\rV_{\ell^{\,q}}\quad (\zeta_1,\dots,\zeta_n\in \C)\,.
$$

It follows  immediately  that $c_n(L^q) \leq n^{1/q}$ when $q\geq 2$ and  $c_n(L^q) \leq \sqrt{n}$  when $q\in [1,2]$. By Corollary \ref{3.4da}, 
$\varphi_n^{\,\max}(L^{p}) \geq n^{1/p}$ when $p\in [1,2]$ and $\varphi_n^{\,\max}(L^{p}) \geq \sqrt{n}$ when $p\in [2,\infty]$.

Again by \cite[Corollary III.B.9]{W}, we have 
$$
d(F, \ell^{\,2}_n) \leq n^{\lv \frac{1}{p}- \frac{1}{2}\rv}\quad (n\in\N) 
$$
whenever $F$ is a subspace of $L^p$ with $\dim F =n$.

For $p\in [1,2]$, it follows from Proposition \ref{3.7jb} that   $\varphi_n^{\,\max}(L^{p}) \leq n^{1/p}$, and thus we have shown that 
 $\varphi_n^{\,\max}(L^{p}) =n^{1/p}$.

For $p\in [2,\infty]$, $L^q$ has the Orlicz property, and so $\pi_1^{(n)}(L^q)  \leq C\sqrt{n}$   for a constant $C>0$. By Corollary \ref{3.4da}, 
 $\varphi_n^{\,\max}(L^{p}) \leq C\sqrt{n}$, and so we have $\varphi_n^{\,\max}(L^{p})\sim \sqrt{n}$.
\end{proof}\medskip

\subsection{The spaces $C(K)$}  The calculation of the maximum rate of growth of the spaces  $C(K)$ is rather easy.\s

\begin{theorem}\label{3.15}
Let $K$ be an infinite, compact space. For each  $n\in\N$, we have
$$
\sqrt{n} \leq \varphi_n^{\,\max}(C(K)) \leq \sqrt{2n}\,.
$$
\end{theorem}

\begin{proof}   Take $n\in\N$. There exist non-zero functions $f_1,\dots, f_n \in C(K)^+$ such that $f_if_j=0$ for $i,j \in\N_n$ with $i\neq j$.
The map $$
(\zeta_1,\dots,\zeta_n)\mapsto \sum_{j=1}^n\zeta_jf_j,\quad  \ell_n^{\,\infty} \to C(K)\,,
$$
 is an isometry onto a closed subspace of $C(K)$, and so, by Example \ref{3.7jd}, we have 
$\varphi_n^{\,\max}(C(K))\leq \varphi_n^{\,\max}(\ell_n^{\,\infty})\leq \sqrt{2n}$.

There exist non-zero $\mu_1,\dots,\mu_n\in M(K)^+$ with pairwise-disjoint supports. The map 
$$(\zeta_1,\dots,\zeta_n)\mapsto \sum_{j=1}^n\zeta_j\mu_j\,,\quad \ell_n^{\,1} \to M(K)\,,
$$
 is an isometry onto a closed  subspace of $M(K)$, and so $c_n(M(K)) \leq c_n(\ell_n^{\,1})$ by Proposition \ref{3.7d}. By Example \ref{3.7jd}, 
$c_n(\ell_n^{\,1}) \leq \sqrt{n}$, and so, by Corollary \ref{3.4da}, $\varphi_n^{\,\max}(C(K))\geq \sqrt{n}$.

The result follows. \end{proof}\medskip

\subsection{A lower bound for $\varphi_n^{\,\max}(E)$}
We  shall now establish  that $\varphi_n^{\,\max}(E)\geq \sqrt{n}$ for each $n\in \N$ and  each infinite-dimensional  Banach space $E$ 
({\it cf.} Corollary \ref{3.4dc}). Since $\varphi_n^{\,\max}(\ell^{\,2}) =  \sqrt{n}\,\;(n \in\N)$, this is the best-possible lower bound. For this, we
shall use the following  famous {\it theorem of  Dvoretzky\/}, sometimes called the  {\it theorem on almost spherical sections\/};  for proofs and discussions,
 see \cite[\S12.3]{AK}, \cite[Chapter 19]{DJT},  or \cite[Chapter 4]{Pis1}.\smallskip

\begin{theorem}\label{3.19a}
For each $n \in\N$ and $\varepsilon > 0$, there exists \mbox{$m = m(n, \varepsilon)$} in $ \N$ such that, for each normed space $F$ with $\dim F\geq m$, there is 
an $n$-dimensional subspace $L$ of  $F$ such that $d(L, \ell^{\,2}_n) < 1 + \varepsilon$.\qed
\end{theorem}\s

 \begin{theorem}\label{3.19b}
 Let $E$ be an infinite-dimensional normed space. Then  
$$
\varphi_n^{\,\max}(E)\geq \sqrt{n}\quad(n\in\N)\,.
$$ 
 \end{theorem}

 \begin{proof}  Fix $n\in\N$, and take $\varepsilon > 0$.   

By Theorem \ref{3.19a},  there is an $n$-dimensional subspace $L$ of $E'$ with  $d(L, \ell^{\,2}_n)  < 1 + \varepsilon$. 
By Proposition \ref{3.7d}, $$c_n(E') \leq c_n(\ell^{\,2}_n) d(L, \ell^{\,2}_n)\,.
$$
As in Example \ref{3.7j}, $c_n(\ell^{\,2}_n) = \sqrt{n}$. Thus $c_n(E') \leq (1+\varepsilon)\sqrt{n}$.
This holds true for each $\varepsilon > 0$, and so $c_n(E') \leq \sqrt{n}$.

By Corollary \ref{3.4da}, $\varphi_n^{\,\max}(E)\geq \sqrt{n}$.
 \end{proof}\s
 
 \begin{corollary}\label{3.19c}
 Let $E$ be an infinite-dimensional normed space. Then the maximum multi-norm is not equivalent  to the minimum multi-norm.\qed
 \end{corollary}
\medskip

\chapter{Specific   examples of multi-norms}

\noindent  In this  chapter, we shall give some specific examples of multi-normed spaces.
\smallskip

\section{The $(p,q)$-multi-norm}

\subsection{Definition}   Let $(E,\norm) $ be a normed space, and  take $p,q$ such that   $1\leq p,q< \infty$.  Again we shall sometimes write 
$p'$ and $q'$ for the conjugate indices of $p$ and $q$, respectively.

 For each $n \in \N$ and $x=(x_1,\dots, x_n) \in E^n$, we define
\begin{equation}\label{(4.0d)}
\lV x \rV^{(p,q)}_n=\sup \left\{  \left(\sum_{i=1}^n\lv \langle x_i,\,\lambda_i\rangle\rv^{q}\right)^{1/q}:  
  \mu_{p,n}(\lambda_1,\dots,\lambda_n)\leq 1\right\}\,,
\end{equation}
taking the supremum over $\lambda_1,\ldots,\lambda_n  \in E' $.  It is clear that $\norm^{(p,q)}_n$  is a norm on $E^n$.

It is convenient for calculations to see that, for 
a constant $C\geq 0$, we have   $\lV x \rV^{(p,q)}_n\leq C$ if and only if
\begin{equation}\label{(4.0a)}
 \left(\sum_{i=1}^n\lv \langle x_i,\,\lambda_i\rangle\rv^{q}\right)^{1/q} \leq
 C\sup\left\{ \left(\sum_{i=1}^n\lv \langle y,\,\lambda_i\rangle\rv^{p}\right)^{1/p} : y\in E_{[1]}\right\} 
\end{equation}
for all $\lambda_1,\ldots,\lambda_n  \in E' $; this is immediate from (\ref{(3.10)}).

\begin{theorem}\label{4.0}
Let $E$ be a normed space. Suppose that  $1\leq p\leq q< \infty$. Then the sequence   $(\norm_n^{(p,q)}: n\in \N)$ is a multi-norm based on $E$.
\end{theorem}

\begin{proof}
It is clear that  $(\norm_n^{(p,q)}: n\in \N)$  satisfies Axioms (A1)--(A3); we shall verify that the sequence satisfies Axiom (A4). 
 
Take  $n\in \N$, let $x_1,\dots, x_n \in E$, and set   $$x=(x_1,\dots, x_{n-1},x_n,x_n) \in E^{n+1}\,.$$   By Lemma \ref{2.1}, it suffices to show that
$\lV x\rV^{(p,q)}_{n+1}\leq \lV (x_1,\dots, x_n )\rV^{(p,q)}_n$.

Take  $\varepsilon >0$.  Then there exist elements $\lambda_1,\dots, \lambda_{n+1} \in E'$ such that  
 $$\mu_{p,n+1}(\lambda_1,\dots,\lambda_{n+1})\leq 1$$  and such that
$$
\left(\sum_{i=1}^{n-1}\lv\langle x_i,\,\lambda_i\rangle\rv^{q}+\lv\langle x_n,\,\lambda_n\rangle\rv^{q} 
+\lv\langle x_n,\,\lambda_{n+1}\rangle\rv^{q}\right)^{1/q}> 
\lV x\rV^{(p,q)}_{n+1}-\varepsilon\,.
$$
Since $(\ell^{\,q}_2)' = \ell^{\,q'}_2$, there exist $\alpha, \beta \in \C$  with $\lv \alpha\rv^{q'} +\lv \beta\rv^{q'}\leq 1$ and
 $$
 \lv\langle x_n,\,\lambda_n\rangle\rv^{q} +\lv\langle x_n,\,\lambda_{n+1}\rangle\rv^{q} = \langle x_n,\,\alpha \lambda_n+\beta\lambda_{n+1}\rangle^q\,.
 $$
Set $\gamma=\lv \alpha\rv^{p'} +\lv \beta\rv^{p'}\,$; since $q'\leq p'$, we have $\gamma\leq 1$. By Proposition \ref{3.13},
$$
\mu_{p,n}(\lambda_1,\dots, \lambda_{n-1},\alpha\lambda_n+\beta\lambda_{n+1}) \leq 
 \mu_{p,{n+1}}(\lambda_1,\dots, \lambda_{n-1},\gamma\lambda_n,\gamma\lambda_{n+1})\,, 
 $$
 and so, since $\mu_{p,n+1}$ satisfies (A2), 
 $$
\mu_{p,n}(\lambda_1,\dots, \lambda_{n-1},\alpha\lambda_n+\beta\lambda_{n+1}) 
\leq  \max\{1, \gamma\}\mu_{p,n+1}(\lambda_1,\dots,\lambda_{n+1})\leq 1\,.
$$
Hence
\begin{eqnarray*}
\lV (x_1,\dots, x_n )\rV^{(p,q)}_n& \geq &
\left(\sum_{i=1}^{n-1}\lv\langle x_i,\,\lambda_i\rangle\rv^{q}+ \langle x_n,\,\alpha \lambda_n+\beta\lambda_{n+1}\rangle^q\right)^{1/q}\\
&=&\left(\sum_{i=1}^{n-1}\lv\langle x_i,\,\lambda_i\rangle\rv^{q}+\lv\langle x_n,\,\lambda_n\rangle\rv^{q} 
+\lv\langle x_n,\,\lambda_{n+1}\rangle\rv^{q}\right)^{1/q} > \lV x\rV^{(p,q)}_{n+1}-\varepsilon\,.
\end{eqnarray*}
 This holds true for each $\varepsilon>0$, and so the result follows.
\end{proof}\s

\begin{definition}\label{4.0a}
Let $E$ be a normed space, and take $p,q\in \R$ such  that  $1\leq p\leq q< \infty$.   Then $(\norm_n^{(p,q)}: n\in \N)$  is   the {\it  $(p,q)$-multi-norm
 based on  $E$}. The rate of growth of this multi-norm is denoted by $(\varphi_n^{(p,q)}(E): n\in \N)$.
\end{definition}\s

Let $E$ be a normed  space, take $1\leq p\leq q< \infty$, and take $x_1,\dots, x_n\in E$.   Suppose that $F$ is a closed subspace of $E$ with $x_1,\dots, x_n\in E$.  
Then the value of $\lV (x_1,\dots, x_n )\rV_n^{(p,q)}$  might depend on the space $F$ to which $x_1,\dots, x_n$ belong.
 To indicate this, we (tempor\-arily) write $(\norm_{n,F}^{(p,q)})$ for the $(p,q)$-multi-norm based on  $F$.\s

\begin{proposition}\label{4.0aa}
Let $E$ be a normed space, let $F$ be a closed subspace of $E$, and suppose that  $1\leq p\leq q< \infty$.  Let $n\in\N$  and    $x = (x_1,\dots, x_n)\in F^n$. 
Then $\lV x\rV_{n,F}^{(p,q)} \geq  \lV x\rV_{n,E}^{(p,q)}$. In the case where  $F$ is $1$-comp\-lemented in $E$,
 $\lV x\rV_{n,F}^{(p,q)} =  \lV x\rV_{n,E}^{(p,q)}$.
\end{proposition}

\begin{proof}  Take $\lambda_1,\dots,\lambda_n \in E'$. By equation (\ref{(3.10)}), $\mu_{p,n}(\lambda_1|F,\dots,\lambda_n |F) \leq \mu_{p,n}(\lambda_1,\dots,\lambda_n)$,
and so $\lV x\rV_{n,F}^{(p,q)} \geq  \lV x\rV_{n,E}^{(p,q)}$.
 
 Now suppose that $F$ is $1$-comp\-lemented in $E$, so that there is a projection $P:E\to F$ with $\lV P \rV =1$. For $\lambda_1,\dots,\lambda_n \in F'$, 
we have 
$$
\lv \langle y,\,P'\lambda_j\rangle\rv = \lv \langle Py,\, \lambda_j\rangle\rv\quad (y \in E_{[1]})\,.
$$ 
Since $Py \in F_{[1]}$, it follows from (\ref{(3.10)}) that $\lV x\rV_{n,F}^{(p,q)} \leq  \lV x\rV_{n,E}^{(p,q)}$. 
 Hence  $\lV x\rV_{n,F}^{(p,q)} =  \lV x\rV_{n,E}^{(p,q)}$.
\end{proof}\s

The following result is a generalization of Corollary \ref{3.4da}; it follows by the same argument.\s

\begin{theorem}\label{4.0ka}
Let $E$ be a normed space. Suppose that  $1\leq p\leq q< \infty$ and  $n\in\N$. Then $\varphi_n^{(p,q)}(E)= \pi_{q,p}^{(n)}(E')$. \qed
\end{theorem}\s

Indeed,  it is explained in \cite{DDPR2} that the $(p,q)$-multi-norm based on a normed space $E$ corresponds via the correspondence of $\S2.4.5$
to the norm induced  on the space $c_{\,0}\otimes E$ by embedding $c_{\,0}\otimes E$ into $\Pi_{q,p}(E',c_{\,0})\/$. The $(p,p)$-multi-norm  corresponds to the {\it Chevet--Saphar\/}
  norm, $d_p$, on the tensor  product $c_{\,0}\otimes E$;  for a discussion of the Chevet--Saphar norm and related  norms 
on tensor products, see \cite{DF}  and \cite[\S6.2]{Ry}.\medskip

\subsection{Relations between $(p,q)$-multi-norms} Take $n\in\N$. Clearly, for each fixed $p\geq 1$ and $q_1\geq q_2\geq p$,
 we have $\norm^{(p,q_1)}_n\leq  \norm^{(p,q_2)}_n$, 
and, for each fixed $q\geq 1$ and  $p_1\leq p_2 \leq q$, we have  $\norm^{(p_1,q )}_n\leq  \norm^{(p_2,q)}_n$.  
In fact,  $\norm^{(p,p)}$ is also a  decreasing function of  $p$ on the interval $[1,\infty)$; this is not immediately obvious, but is given by the  
following calculation, which is essentially that of page 134 of \cite{Ry}. A more general result is given in \cite[Theorem 10.4]{DJT}.
Thus  the maximum among these norms is $\norm_n^{(1,1)}$. \s 

\begin{theorem}\label{4.0k}
Let $E$ be a normed space, and suppose that  $1\leq p \leq q< \infty$. Then 
$$
\lV x \rV^{(p,p)}_n \geq \lV x \rV^{(q,q)}_n\quad(x \in E^n)
$$
for each $n\in \N$.
\end{theorem}

\begin{proof} We may suppose that $p<q$. 

Take $n\in\N$ and $x =(x_1,\dots,x_n) \in E^n$, and set $C= \lV x \rV^{(p,p)}_n$.  Then
$$
A:= \left(\sum_{i=1}^n\lv \langle x_i,\,\lambda_i\rangle\rv^{q}\right)^{1/p} = 
\left(\sum_{i=1}^n\lv \langle x_i,\,\alpha_i\lambda_i\rangle\rv^{p}\right)^{1/p}\,,
$$
where $\alpha_i = \lv \langle x_i,\,\lambda_i\rangle\rv^{(q-p)/p}$ for $i\in\N_n$.  By equations (\ref{(3.10)}) and (\ref{(4.0a)}), we have 
$$
\left( \sum_{i=1}^n\lv \langle x_i,\,\alpha_i\lambda_i\rangle\rv^{p}\right)^{1/p}  \leq
C\sup\left\{\left( \sum_{i=1}^n\lv \langle y,\,\alpha_i\lambda_i\rangle\rv^{p}\right)^{1/p} : y\in E_{[1]}\right\}\,.
$$
However
$$
 \sum_{i=1}^n\lv \langle y,\,\alpha_i\lambda_i\rangle\rv^{p}  =
 \sum_{i=1}^n\lv \langle x_i,\,\lambda_i\rangle\rv^{q-p}\lv \langle y,\,\lambda_i\rangle\rv^{p}\,.
$$
By H{\"o}lder's inequality  with conjugate exponents $q/(q-p)$ and $q/p$, the right-hand side of the above equation is at most 
$$
\left(\sum_{i=1}^n\lv \langle x_i,\,\lambda_i\rangle\rv^q\right)^{(q-p)/q}\,\cdot\,\left(\sum_{i=1}^n\lv \langle y,\,\lambda_i\rangle\rv^q\right)^{p/q}\,.
$$
Hence we have
$$
A \leq C\sup \left\{\left(\sum_{i=1}^n\lv \langle x_i,\,\lambda_i\rangle\rv^q\right)^{(q-p)/pq}\,\cdot\,
\left(\sum_{i=1}^n\lv \langle y,\,\lambda_i\rangle\rv^q\right)^{1/q} :y\in E_{[1]}\right\}\,.
$$
Note that $(1/p) - (q-p)/pq =1/q$,
and so, by an appropriate division, we see that
$$
\left(\sum_{i=1}^n\lv \langle x_i,\,\lambda_i\rangle\rv^{q}\right)^{1/q}\leq
 C\sup \left\{\left(\sum_{i=1}^n\lv \langle y,\,\lambda_i\rangle\rv^q\right)^{1/q} :y\in E_{[1]}\right\}\,.
$$
Thus $\lV x \rV^{(q,q)}_n \leq C = \lV x \rV^{(p,p)}_n$, as required.
\end{proof}\s

By Theorem \ref{3.4d}, we have the following result.\s

\begin{theorem}\label{4.0d}
Let $E$ be a normed space. Then  $\norm_n^{(1,1)} = \norm_n^{\max}$ for $n\in\N$, and so 
$(\norm_n^{(1,1)}: n\in \N)$ is the maximum multi-norm based on $E$.\qed
\end{theorem}\s 

The  relations between the multi-norms $(\lV x \rV^{(p,p)}_n:n\in\N)$ can be illustrated in the following diagram, where the arrows indicate
increasing multi-norms in the ordering $\leq$:

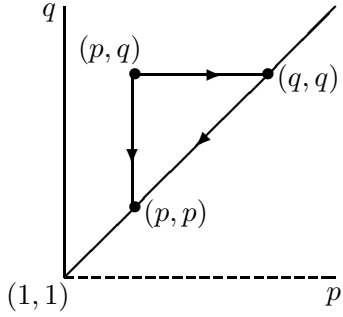
\setlength{\unitlength}{0.6cm} 
\begin{picture}(8,8)
\thicklines 
\multiput(1,1)(0.3,0){19}{\line(1,0){0.2}}\put(6.7,1){\line(1,0){0.3}} 
\put(1,1){\line(0,1){6}}
\put(1,1){\line(1,1){6}}
\put(6.8,0.5){$p$}\put(0.5,6.8){$q$} \put(-0.3,0.4){$(1,1)$}
\put(1.3,5.9){$(p,q)$} \put(2.4,5.35){$\bullet$}
\put(5.7,5.25){$(q,q)$} \put(5.35,5.35){$\bullet$}
\put(2.75,2.25){$(p,p)$} \put(2.4,2.4){$\bullet$}
\put(2.5,5.5){\vector(1,0){2}}\put(4.5,5.5){\line(1,0){1}}
\put(2.5,5.5){\vector(0,-1){2}}\put(2.5,3.5){\line(0,-1){1}}
\put(4,4){\vector(-1,-1){0.1}}
 \end{picture}

\begin{proposition}\label{4.0e}
Let $E$ be a normed space, and suppose that  $1\leq p \leq q< \infty$. Then $\varphi_n^{(p,q)}(E)\leq n^{1/q}$ for $n\in\N$.
\end{proposition}

\begin{proof}  Consider $\lambda_1,\dots,\lambda_n\in E'$ with $\mu_{p,n}(\lambda_1,\dots,\lambda_n) \leq 1$. Then $\lV \lambda_i\rV \leq 1\,\;(n\in\N)$. 
 Now take $x_1,\dots,x_n\in E_{[1]}$. Then $\lv \langle x_i,\,\lambda_i\rangle\rv\leq 1\,\;(i\in\N_n)$, and so 
$\lV(x_1,\dots,x_n)\rV_n^{(p,q)} \leq n^{1/q}$. The result follows. 
\end{proof}\s

\begin{example}\label{4.4}
{\rm  Let $E = \ell^{\,r}$, where $r\geq 1$, so that $E' =  \ell^{\,s}$, where $s=r'$.  

Fix $p,q$ with $1 \leq p\leq q< \infty$, and take  $n\in\N$. We shall calculate $\lV f\rV^{(p,q)}_n$, where $f = (\delta_1,\dots, \delta_n)\in E^n$. Set $u=p'$.

 By Proposition \ref{4.0e}, $\lV f\rV^{(p,q)}_n     \leq n^{1/q}$.

Now consider the choice $\lambda_i =\delta_i\,\;(i\in\N_n)$, and set $\zeta= (\zeta_1, \dots, \zeta_n)\in \C^n$. Then
$\mu_{p,n}(\lambda_1,\dots,\lambda_n) =\sup\{\lV \zeta\rV_{\ell^{\,s}}: \lV \zeta\rV_{\ell^{\,u}}\leq 1\}$.   In the case where $u\leq s$, i.e., $p\geq r$, we have 
$\lV \zeta\rV_{\ell^{\,s}}\leq \lV \zeta\rV_{\ell^{\,u}}$, and so $\mu_{p,n}(\lambda_1,\dots,\lambda_n)\leq 1$. Hence $\lV f\rV^{(p,q)}_n \geq n^{1/q}$.  

This implies that 
$$
\lV (\delta_1,\dots, \delta_n)\rV^{(p,q)}_n = n^{1/q}\quad{\rm whenever}\quad p\geq r\,.
$$
A similar calculation gives the same conclusion in the case where  $r=1$.

We conclude that two  multi -norms $(\norm_n^{(p_1,q_1)})$ and $(\norm_n^{(p_2,q_2)})$  based on $\ell^{\,r}$ are not equivalent whenever $p_1,p_2 \geq r$ 
and $q_1\neq q_2$.

It follows that 
$$
\lV (\delta_1,\dots, \delta_n)\rV^{(p,q)}_n \leq  n^{1/q}\quad{\rm whenever}\quad q\geq r\,.
$$
However, we know from Corollary \ref{3.4dd} that $$\lV (\delta_1,\dots,\delta_n)\rV^{\max}_n  = n^{1/r}\quad(n\in\N)\,,
$$
 and so the multi-norm 
$(\norm_n^{(p,q)})$ is not equivalent to $(\norm_n^{\max})$ whenever $q>r$.   Further, $\lV (\delta_1,\dots, \delta_n)\rV^{(p,p)}_n = n^{1/r} \,\;(n\in\N)$ whenever
$p\in [1,r]$.}\qed 
\end{example}\smallskip

The general question of the equivalence of the two multi-norms 
$$(\norm_n^{(p_1,q_1)}: n\in\N)\quad  {\rm and} \quad (\norm_n^{(p_2,q_2)}: n\in\N)$$
 on the spaces $L^r(\Omega)$ will be addressed in \cite{DDPR2}.\s

\begin{theorem}\label{4.0f}
Let $E$  and $F$ be isomorphic Banach spaces such that  $d(E,F)\leq C$, and suppose that  $1\leq p\leq q< \infty$.  Then 
$$
\varphi_n^{(p,q)}(E) \leq C \varphi_n^{(p,q)}(F)\quad(n\in\N)\,.
$$
\end{theorem}

\begin{proof} Take $\varepsilon >0$. Then there exists a linear bijection $T:E\to F$ with $\lV T\rV  < C + \varepsilon$ and 
$\lV S \rV =1$, where $S=T^{-1}:F\to E$.
  We have $\lV S'\rV =1$.
  
  Take $n\in\N$ and $x_1,\dots,x_n\in E_{[1]}$. Suppose that $\lambda =(\lambda_1,\dots,\lambda_n)\in (E')^n$ with 
$\mu_{p,n}(\lambda) \leq 1$. By (\ref{(3.2aa)}),  $\mu_{p,n}(S'\lambda_1, \dots, S' \lambda_n)  \leq 1$, and so
 $$
  \left(\sum_{i=1}^n\lv \langle x_i,\,\lambda_i\rangle\rv^{q}\right)^{1/q} =  \left(\sum_{i=1}^n\lv \langle Tx_i,\,S'\lambda_i\rangle\rv^{q}\right)^{1/q}
\leq (C+\varepsilon) \varphi_n^{(p,q)}(F)\,.
 $$
 Thus $\varphi_n^{(p,q)}(E) \leq (C+\varepsilon) \varphi_n^{(p,q)}(F)$. This holds true for each $\varepsilon>0$, and so the result follows.
\end{proof}\s

\subsection{Duality theory}  Let $E$ be a normed  space, and take $p,q\in [1,\infty)$. For   $n\in\N$ and
  $\lambda=(\lambda_1,\dots,\lambda_n)\in (E')^n$, the formula for  $\lV \lambda\rV_n^{(p,q)}$ is
 $$
\lV \lambda  \rV^{(p,q)}_n=\sup \left\{  \left(\sum_{i=1}^n\lv \langle \Lambda_i,\,\lambda_i\rangle\rv^{q}\right)^{1/q}:  
  \mu_{p,n}(\Lambda_1,\dots,\Lambda_n)\leq 1\right\}\,,
$$
taking the supremum over $ \Lambda_1,\ldots,\Lambda_n  \in E''$.  In fact, there is a simpler formula for $\lV \lambda\rV_n^{(p,q)}$; 
the proof, from the Principle of Local Reflexivity,  of the following proposition is almost identical to that of Theorem \ref{3.4e}, and is omitted.\s

\begin{proposition}\label{4.0b}
Let $E$ be a normed space, and  take $p,q\in [1,\infty)$. For each  $n\in\N$ and  $\lambda \in (E')^n$, we have 
 $$
\lV \lambda  \rV^{(p,q)}_n=\sup \left\{  \left(\sum_{i=1}^n\lv \langle x_i,\,\lambda_i\rangle\rv^{q}\right)^{1/q}:  
 \mu_{p,n}(x_1,\dots,x_n)\leq 1\right\}\,,
$$
taking the supremum over $x_1,\dots,x_n  \in E$. \qed
\end{proposition} \medskip

\subsection{The dual of the $(p,q)$-multi-norm}\label{The dual of the $(p,q)$-multi-norm}
In this section, we shall determine the dual of the the multi-norm $(\norm_n^{(p,q)}: n\in\N)$, based  on $(E')^n$, following remarks of Paul Ramsden.

Let $E$ be a Banach space, and  fix $r,s $ with  $1\leq r< \infty$ and $1< s\leq \infty$. The conjugate index to $r$ is $r'$.
For each $n \in \N$ and $x = (x_1,\dots, x_n) \in E^n$,  we set
$$
\LV x \RV^{(r,s)}_n=\inf\left\{\sum_{k=1}^m\lV \alpha_k\rV_{s}\,\cdot\,\mu_{r,n}(y_k)\right\}\,,
$$ 
where the infimum is taken over all representations 
$$
x= \sum_{k=1}^m M_{\alpha_k}(y_k) 
$$
for which $\alpha_1,\dots,\alpha_m \in \C^n$,  $y_1, \dots, y_m \in E^n$, and $m\in\N$.  It is clear that $\Norm^{(r,s)}_n$  is a norm on $E^n$. 
 
The following is `dual' to the proof of Theorem \ref{4.0}, and will also follow from Theorem \ref{4.0l}, below, and so the direct proof is omitted

\begin{theorem}\label{4.0g}
Let $E$ be a normed space, and take $r,s \in[1,\infty]$ with   $1<s\leq r'\leq \infty$. Then the sequence  $(\Norm_n^{(r,s)}: n\in \N)$ is a dual multi-norm based on  
$E$.\qed
\end{theorem}\s

\begin{definition}\label{4.0h}
Let $E$ be a normed space, and take $r,s \in[1,\infty]$ with   $1<s\leq r'\leq \infty$.   Then $(\Norm_n^{(r,s)}: n\in \N)$ is  
 the {\it  $(r,s)$-dual multi-norm based on  $E$}. \end{definition}\s

Let $E$ be a normed space, and take $p,q$ such that $1\leq p,q<\infty$. For $n\in\N$ and $x =(x_1,\dots,x_n) \in E^n$, define an embedding 
$$
\nu_E(x) : (\lambda_1,\dots,\lambda_n)\mapsto (\langle x_1,\,\lambda_1\rangle, \dots, \langle x_n,\,\lambda_n\rangle)\,, 
 \quad  \ell^{\,p}_n(E')^{w}  \to  \ell_n^{\,q}\,.
$$
Then $\nu_E(x):  \ell^{\,p}_n(E')^{w} \to  \ell_n^{\,q}$ is a bounded linear map, and  we have $\lV \nu_E(x)\rV = \lV x \rV^{(p,q)}_n$, and so
 we have an isometric embedding
\begin{equation}\label{(4.0aa)}
 \nu_E :   (E^n, \norm^{(p,q)}_n)  \to {\B}( \ell^{\,p}_n(E')^{w},\ell_n^{\,q})\,.
\end{equation}

Now take  $r,s$ with $1\leq r< \infty$ and $1< s\leq \infty$. Then there is a continuous linear 
surjection  $\theta_E: \ell^{\,r}_n(E)^{w}\projectivetensor \ell^{\, s}_n\rightarrow E^n$  such that 
$$
\theta_E(x\otimes \alpha)= (\alpha_1x_1,\ldots,\alpha_nx_n) 
$$
whenever  $x=(x_1,\ldots,x_n) \in E^n$ and $ \alpha =(\alpha_1,\dots, \alpha_n)\in \ell^{\, s}_n$. Thus there is an isometric isomorphism of Banach spaces
$$
(\ell^{\,r}_n(E)^{w}\projectivetensor \ell^{\, s}_n)/\ker \theta_E\cong \left( E^n, \Norm_n^{(r,s)}\right)\,.
$$
 
\begin{theorem}\label{4.0l}
Let $E$ be a Banach space, and take $p,q,r,s$ such that   $1\leq p, q, r< \infty$  and $1< s\leq \infty$. Then there are isometric isomorphisms:\s
 
{\rm (i)} $\left( E^n, \norm _n^{(p,q)}\right)'\cong \left((E')^n, \Norm_n^{(p,q')}\right)\/$; \smallskip

{\rm (ii)} $\left( E^n, \Norm_n^{(r,s)}\right)'\cong \left(({E}')^n, \norm_n^{(r,s')}\right)$.
\end{theorem}
 
 \begin{proof} (i) It is easily checked that the following diagram commutes:
\begin{equation*}
\SelectTips{eu}{12}\xymatrix{(\ell^{\,p}_n(E')^{w}\projectivetensor\ell^{\, q'}_n)'' \ar[r]^-{  \nu_{E}'}  & (E')^n  \\ 
                             \ell^{\,p}_n(E')^{w}\projectivetensor\ell^{\, q'}_n\,.\ar[u] \ar[ur]_{\theta_{E'}}& ~}
\end{equation*}
Hence we have isometric isomorphisms of Banach spaces
\begin{align*}
\left( E^n, \norm_n^{(p,q)}\right)'&\cong  (\ell^{\,p}_n(E')^{w}\projectivetensor\ell^{\, q'}_n)''/\ker  \nu_E'  \\
&\cong (\ell^{\,p}_n(E')^{w}\projectivetensor \ell^{\, q'}_n)/\ker \theta_{E'}\cong \left((E')^n, \Norm_n^{(p,q')}\right)\,.
\end{align*}

(ii) Similarly, the following diagram commutes:
\begin{equation*}
\SelectTips{eu}{12}\xymatrix{(E')^n \ar[r]^-{\theta_{E}'}  \ar[dr]_{\nu_{E'}} & {\B}(\ell^{\,r}_n(E)^{w}, \ell^{\, s'}_n)  \\ 
            ~ & {\B}(\ell^{\,r}_n(E'')^{w}, \ell^{\, s'}_n)\,. \ar[u]_{j: T\mapsto T|E^n}}
\end{equation*}
Hence there is an isometric isomorphism
$$
\left( E^n, \Norm_n^{(r,s)}\right)'\cong {\rm im} \,\theta_{E}'  ={\rm im} (j\circ \nu_{E'})\,.
$$
By Proposition \ref{4.0b}, there is an isometric isomorphism 
\begin{equation*}
{\rm im} (j\circ \nu_{E'})\cong {\rm im}\, \nu_{E'}\cong \left((E')^n, \norm_n^{(r,s')}\right)\,, 
\end{equation*}
and so the result follows.
 \end{proof}\s
 
 Thus the dual of the multi-norm  $( \norm _n^{(p,q)}: n\in\N)$ based on $E$ is the dual multi-norm $(\Norm_n^{(p,q')}:n\in\N)$ based on $E'$.

The following corollary resolves a `second dual question'  for the $(p,q)$-multi-norm  (defined when $1\leq p\leq q < \infty$).\s
 
 \begin{corollary}\label{4.0m}
 Let $E$ be a Banach space, and take $p,q\in [1,\infty)$. Then  
$$\vspace{-\baselineskip}
 ( E^n, \norm _n^{(p,q)})''\cong ( (E'')^n, \norm _n^{(p,q)})\,.
$$\hspace*{\stretch{1}}\qed
\end{corollary}
\medskip

\subsection{Multi-norms on Hilbert spaces}\label{Multi-norms on Hilbert spaces} We now consider an example which involves Hilbert  spaces.   It will
lead to an alternative description of the $(2,2)$-multi-norm  based on a Hilbert space.

Let $(H, \norm)$ be a Hilbert space.  (Basic facts about Hilbert spaces were recalled in $\S$\ref{Hilbert spaces}.) 
  For each family ${\bf H} = \{H_1, \dots , H_n\}$, where $n\in\N$ and each $H_j$ is a closed subspace of $H$ and 
 $H = H_1\oplus_{\perp} \cdots \oplus_{\perp} H_n$, set 
\begin{equation}\label{(4.3a)}
r_{\bf H}(( x_1,\dots, x_n)) = \left ( \lV P_1x_1\rV^2 + \cdots + \lV P_nx_n\rV^2\right)^{1/2} 
= \lV P_1x_1 +\cdots + P_nx_n\rV
\end{equation}
for $x_1,\dots,x_n  \in H$, where  $P_i : H \to H_i$ for $i\in\N_n$ is the  orthogonal projection, and then set 
\begin{equation}\label{(4.3)}
\lV (x_1,\dots,x_n)\rV_n^{H}  = \sup_{\bf H}r_{\bf H}( (x_1,\dots,x_n ))\quad (x_1,\dots,x_n  \in H)\,,
\end{equation}
where the supremum is taken over all such families $\bf H$. (We allow the possibility that $H_j=\{0\}$ and $P_j =0$ for some $j\in\N_n$.)

The following result is easily  checked.\s

\begin{theorem}\label{4.13}
Let $H$ be a Hilbert space. Then $(\norm_n^H: n\in\N)$ is a multi-norm on the family $\{H^n: n\in\N\}$.\qed
\end{theorem}\s

For example, let $H=\ell^{\,2}$, and take  $n\in\N$ and $\beta_1, \dots, \beta_n \in \C$.  Then
$$
\lV  (\beta_1\delta_1, \dots, \beta_n\delta_n)\rV^H_n = \left(\sum_{j=1}^n\beta_j^2\right)^{1/2}\,.
$$
 
\begin{definition}\label{4.3b}
Let $(H, \norm )$ be a Hilbert  space.  Then the {\it Hilbert multi-norm} based on $H$  is the multi-norm $(\lV \,\cdot\,\rV^H_n : n\in\N)$  defined above. 
The rate of growth of this multi-norm is denoted  by $(\varphi_n^H(H) : n\in\N)$.\end{definition}\smallskip

The following results are based on remarks of Hung Le Pham.\s

\begin{proposition}\label{4.2a}
  Let $H$ be a Hilbert space,   take $n\in \N$, and  $x_1,\dots,x_n\in H$. Then
\begin{equation}\label{(4.3d)}
\lV(x_1,\dots, x_n)\rV^H_n = \sup\{\lv \alpha_1[e_1,x_1] + \cdots + \alpha_n[e_n,x_n]\rv\}\,,
\end{equation}
taking the supremum  over orthonormal sets $\{e_1,\dots,e_n\}$ in $H$ and  
 $(\alpha_1,\dots,\alpha_n) \in (\ell^{\,2}_n)_{[1]}$.
\end{proposition}

\begin{proof} Set  $A = \lV(x_1,\dots, x_n)\rV^H_n$ and  $B= \sup\{\lv \alpha_1[e_1,x_1] + \cdots + \alpha_n[e_n,x_n]\rv\}$.

Given $\varepsilon>0$, let  $\{P_1,\dots,P_n\}$ be an orthogonal family of projections such that 
$\lV P_1x_1\rV^2 + \cdots + \lV P_nx_n\rV^2> A^2 -\varepsilon$.
It follows from  (\ref{(4.3c)}) that there is an ortho\-normal set $\{e_1,\dots,e_n\}$ in $  H$ such that
$$
\lV P_1x_1\rV^2 + \cdots + \lV P_nx_n\rV^2 = [e_1,x_1]^2 + \cdots + [e_n,x_n]^2\,.
$$
Set $\alpha_j = [e_j,x_j]/(A+\varepsilon)\,\;(j\in\N_n)$. Then
$$
\sum_{j=1}^n \lv \alpha_j\rv^2 \leq 1\quad {\rm and}\quad \sum_{j=1}^n \lv \alpha_j[e_j,x_j]\rv> \frac{A^2-\varepsilon}{A + \varepsilon}\,.
$$
Thus $B\geq (A^2-\varepsilon)/(A + \varepsilon)$. Since this holds  true for each $\varepsilon>0$, we have $B\geq A$.

Conversely, given an orthonormal set $\{e_1,\dots,e_n\}$ in $  H$, there is an orthogonal decomp\-osition 
$H = H_1\oplus_{\perp} \cdots \oplus_{\perp} H_n$ such that  $e_j\in H_j\,\;(j\in\N_n)$, and then 
$$\lv [e_j,x]\rv \leq \lV P_jx\rV\quad (x\in H,\,j\in\N_n)\,.
$$
Take $\alpha_1,\dots,\alpha_n\in \C$ such that  $\sum_{j=1}^n \lv \alpha_j\rv^2 \leq 1$. Then
$$
 \sum_{j=1}^n \lv \alpha_j[e_j,x_j]\rv \leq \sum_{j=1}^n \lv \alpha_j\rv \lV P_jx\rV \leq \left(\sum_{j=1}^n \lV P_jx\rV^2\right)^{1/2}\leq A
$$
by the Cauchy--Schwarz inequality.  Hence $B\leq A$.

The result follows.
\end{proof}\s

In the following result, $D_n$ denotes the family  of all orthonormal  $n$-tuples of  elements  in $H$, 
and the closure of ${\rm co}(D_n)$ is taken in the weak-$*$ topology on $H^n$.\s

\begin{proposition}\label{4.14}
 Take $n\in\N$, and let $H$ be a Hilbert space with $\dim H \geq n$. Then the closed unit ball of $(H^n, \mu_{2,n})$ is equal to $\overline{\rm co}(D_n)$.
\end{proposition}

\begin{proof}  We write  $B_n$ for  $(H^n, \mu_{2,n})_{[1]}$, and we identify $H $ with\ $\ell^{\,2}(I)$, 
for an index set $I$; we may suppose that  $\N_n$ is a subset of $I$.

It is clear from equation (\ref{(3.10f)}) that $D_n\subset B_n$, and so $\overline{\rm co}(D_n)\subset B_n$.

For the converse, take $x =(x_1,\dots,x_n)\in B_n$.    Define  $S : H \to H$ by setting 
$$
S(\delta_i) = x_i\quad (i\in\N_n)\,,\quad S(\delta_i) = 0\quad (i\in I \setminus \N_n)\,.
$$
Since $\mu_{2,n}(x) \leq 1$, we see that $\lV S\rV \leq 1$, and so   $S \in \overline{\rm co}({\mathcal U}({\B}(H))$ by the Russo--Dye theorem. 
Hence $(x_1,\dots,x_n) = (S(\delta_1),\dots,S(\delta_n))\in \overline{\rm co}(D_n)$, as required.
\end{proof}\s

\begin{theorem}\label{4.15}
Let $H$ be an infinite-dimensional Hilbert space. Then 
$$
\lV (x_1,\dots,x_n)\rV_n^{H} = \lV (x_1,\dots,x_n)\rV_n^{(2,2)}\quad (x_1,\dots,x_n\in H)
$$
for each $n\in\N$.
\end{theorem}

\begin{proof} This follows from the two previous propositions.\end{proof}\s

Thus the Hilbert multi-norm and the  $(2,2)$-multi-norm on $\ell^{\,2}$ are equal. It is natural to ask if these multi-norms 
are also equal to the maximum  multi-norm on $\ell^{\,2}$; in fact, they are  \label{Hmultinorm}equivalent to the maximum  multi-norm, but not equal to it \cite{DDPR2}.

A more general version of following result will be proved in \cite{DDPR2}.\s

\begin{theorem}\label{4.16}
Let $H$ be an infinite-dimensional Hilbert space.  Then the following multi-norms based on $H$ are mutually equivalent: \s

{\rm (a)}  the Hilbert multi-norm $(\norm_n^H)$;\s

 {\rm (b)} the maximum multi-norm $(\norm_n^{\max})$; \s

{\rm (c)}  the $(p,p)$-multi-norm  $(\norm_n^{(p,p)})$ for $p\in [1,2]$.\s

\noindent For the above multi-norms, the rate of growth is equivalent to $(\sqrt{n}:  n\in\N)$. 

Further, the $(p,p)$-multi-norm  and the $(q,q)$-multi-norms  based on $H$ are not equivalent whenever $p\neq q$ and $\max\{p,q\}>2$.\qed
\end{theorem} 
\medskip

\section{Standard $q$-multi-norms}\label{Standard $q$-multi-norms}
\noindent We shall now construct some multi-norms based on the Banach spaces $L^{p}(\Omega, \mu)$ and $M(K)$.  
We begin with the spaces $L^{p}(\Omega, \mu)$.\s

\subsection{Definition}
 Let $(\Omega, \mu)$  be a measure space.  For each $p\in [1,\infty)$, we consider the Banach space $E= L^{p}(\Omega, \mu)$, with the norm
$$
\lV f \rV =   \left (\int_\Omega \lv f \rv ^{\,p}\right)^{1/p}
=  \left (\int_\Omega \lv f \rv ^{\,p}{\dd}\mu\right)^{1/p} \quad (f \in E)\,,
$$
as in $\S$\ref{Standard Banach spaces}.  For a measurable subset $X$ of $\Omega$, we write $r_X$ for the seminorm on $E$ specified by
$$
r_X(f) = \lV f\chi_X\rV = \left (\int_X \lv f \rv ^{\,p}\right)^{1/p}\quad (f \in E)\,,
$$
where we again suppress in the notation the dependence on $p$.  (We  take $r_X(f)=0$ when $X=\emptyset$.)

Now take $q \geq p\,$; we shall define a multi-norm based on $E$ that depends on $q$.

Take $n\in \N$.  For each ordered partition  ${\bf X} = (X_1, \dots , X_n)$ of  $\Omega$  into meas\-urable subsets 
and each $f_1,\dots,f_n\in E$, we set
\begin{eqnarray*}
r_{\bf X}((f_1,\dots,f_n))
& = &
\left( r_{X_1}( f_1) ^q+\cdots + r_{X_n}(f_n)^q\right )^{1/q}\\
& = & \left(
\left(\int_{X_1}\lv f_1\rv^{\,p}\right)^{q/p} + \cdots + \left(\int_{X_n}\lv f_n\rv^{\,p}\right)^{q/p}\right)^{1/q}\,,
\end{eqnarray*}
so that $r_{\bf X}$ is a seminorm on $E^n$ and
$$
r_{\bf X}((f_1,\dots,f_n))\leq \left( \lV f_1\rV^q +\cdots + \lV f_n\rV^q\right)^{1/q}\,.
$$
Finally, we define
\begin{equation}\label{(4.1a)}
\lV (f_1,\dots,f_n)\rV_n^{[q]}   = \sup_{\bf X}r_{\bf X}((f_1,\dots,f_n))\quad (f_1,\dots,f_n\in E) \,,
\end{equation}
where the supremum is taken over all such ordered partitions  ${\bf X}$.   Then $\norm_n^{[q]} $ is a norm on $E^n$.

In the case where $q=p$, we have
\begin{equation}\label{(4.1ad)}
\lV (f_1,\dots,f_n)\rV_n^{[p]}   = \sup_{\bf X}\lV f_1\mid X_1 +\cdots+ f_n\mid X_n\rV \quad (f_1,\dots,f_n\in E) \,.
\end{equation}

In the case where $q\geq p$ and $f_1, \dots, f_n \in E$ have disjoint support, we have
\begin{equation}\label{(4.1aa)}
\lV (f_1,\dots,f_n)\rV_n^{[q]}   = \left( \lV f_1\rV^q +\cdots + \lV f_n\rV^q\right)^{1/q}\,;
\end{equation}
if, further, $q=p$, then
\begin{equation}\label{(4.1ab)}
\lV (f_1,\dots,f_n)\rV_n^{[p]}   = \lV f_1+ \cdots + f_n\rV\,.
\end{equation}

It is easily checked that that $(\norm_n^{[q]} : n\in\N)$ is a multi-norm based on  $E$: indeed, Axioms (A1), (A2),
 and (A3) are immediate, and Axiom (A4)  follows because
$$
(\alpha^{\,p} + \beta ^{\,p})^{1/p} \geq   (\alpha^q + \beta ^{\,q})^{1/q}\quad (\alpha, \beta \in \R^+) 
$$
whenever $p\leq q$. Further, for each $n\in\N$, we have
\begin{equation}\label{(4.1ac)}
\lV (f_1,\dots,f_n)\rV_n^{[q]} \leq \left( \lV f_1\rV^q +\cdots + \lV f_n\rV^q\right)^{1/q}\quad (f_1,\dots,f_n\in E)\,.
\end{equation}

\begin{definition}\label{4.1a}
 Let $\Omega$  be a measure space, and take $p\geq 1$.   Then, for each $q\geq p$,  the {\it standard $q\,$-multi-norm} based on $L^{p}(\Omega)$ is the
 multi-norm $(\norm_n^{[q]} :n\in\N)$. The rate of growth of this multi-norm is denoted by $(\varphi_n^{[q]}(L^{p}(\Omega)): n\in \N)$.
\end{definition}\smallskip

  At this point, it appears that the definition of the  standard $q\,$-multi-norm based 
on $L^{p}(\Omega)$ depends on the concrete representation of $L^{p}(\Omega)$ as a Banach space of functions. We would wish that, 
 if $L^{p}(\Omega_1)$  and $L^{p}(\Omega_2)$ are isometrically order-isomorphic  Banach lattices, then the corresponding standard $q\,$-multi-norms
 based on $L^{p}(\Omega_{\,1})$  and on $L^{p}(\Omega_{\,2})$  are equal. We shall see in Theorem \ref{4.31}  that this is indeed the case; see also Theorem \ref{6.38}.

It follows from (\ref{(4.1ac)}) that $\varphi_n^{[q]}(L^{p}(\Omega)) \leq n^{1/q}$.\s

We may consider  these multi-norms   $(\norm_n^{[q]}  : n\in \N)$  as a function of $q$ when $q\in [p, \infty)$;   clearly, for each $n \in\N$, 
the norms  $\norm_n^{[q]} $  decrease  as $q$ increases, and so the maximum multi-norm among these multi-norms is $(\norm_n^{[p]}  : n\in \N)$. 

There is an equivalent way of defining the norm $\lV (f_1,\dots,f_n)\rV_n^{[q]} $ for $f_1,\dots,f_n \in   L^{p}(\Omega)$ 
in the special case where $q=p$.  Indeed, set $f =\lv f_1 \rv \vee \dots \vee \lv f_n \rv$, so that
$$
f(x)  =\max \{\lv f_1(x) \rv,  \dots, \lv f_n(x) \rv\} \quad (x \in \Omega)\,.
$$
Then  we see immediately that 
\begin{equation}\label{(4.1b)}
\lV (f_1,\dots,f_n)\rV_n^{[p]} = \lV f \rV =
 \left(\int_\Omega (\lv f_1 \rv \vee \dots \vee \lv f_n \rv)^{\,p}\right)^{1/p}\,.
\end{equation}
In particular, in the case where $E=\ell^{\,p}$, we have
\begin{equation}\label{(4.1c)}
\lV (f_1,\dots,f_n)\rV_n^{[p]}  = \left(\sum_{j=1}^\infty (\lv f_1(j) \rv \vee \cdots \vee \lv f_n(j) \rv)^{\,p}\right)^{1/p}\,.
\end{equation}
[To see that the formula $\lV (f_1,\dots,f_n)\rV_n^{[q]}=  \lV f \rV$    is correct only  when $q =p$, consider the case where
 $X_1$ and $X_2$ are disjoint subsets of $\N$ of
 cardinalities $m$ and $n$, respectively, and let $f_j$ be the characteristic function of $X_j$ for $j=1,2$.   
By  (\ref{(4.1a)}),
$$
\left(\lV (f_1,f_2)\rV_2^{[q]}\right)^q = m^{q/p} + n^{q/p}\,
$$ 
whereas $\lV f \rV ^q =  (m+n)^{q/p}$, and we have $m^{q/p} + n^{q/p} = (m+n)^{q/p}$ for all $m,n \in \N$ if and only if $q=p$.]
\smallskip

Suppose that $1\leq p\leq q< \infty$, and set $E =\ell^{\,p}$.    Take $n\in\N$,    and  consider the elements $\delta_1,\dots, \delta_n \in E_{[1]}$.  Let
 ${\bf X}= (X_1, \dots , X_n)$  be an ordered partition of $\N$; suppose, in fact, that  $i\in X_i \,\;(i\in \N_n)$.  For each $q\geq p$, we  have
$r_{\bf X}((\delta_1,\dots,\delta_n))= n^{1/q}$, and so $\lV (\delta_1,\dots, \delta_n)\rV_n^{[q]} \geq n^{1/q}$.
It follows that
\begin{equation}\label{(4.12)}
 \varphi_n^{\,[q]}(\ell^{\,p})= n^{1/q}\quad (n\in\N)\,.
\end{equation}
In particular, taking $q=p$, we see that $\varphi_n^{\max}(\ell^{\,p})\geq \varphi_n^{\,[p]}(\ell^{\,p}) =n^{1/p}$ for $n\in \N$, so recovering a result 
of Example \ref{3.7f}.

 Let $n\in\N$, and  let $(\alpha_i)$ be a fixed element of $\C^n$.  Set  $x_i =\alpha_i\delta_i\,\;(i\in\N_n)$. Then we now have
\begin{equation}\label{(4.2)}
\lV (x_1, \dots, x_n)\rV_n^{[q]} = \left(\lv \alpha_1\rv^q+ \dots +\lv \alpha_n\rv^q\right)^{1/q}\quad (n\in\N)\,.
\end{equation}
Thus $\left(E^n, \norm_n^{[q]}\right)$ contains ${\ell}^{\,q}_n $ as a closed subspace.

 There does not seem to be an accessible, explicit  formula for the dual of the standard $q\,$-multi-norm based on $L^{\,p}(\Omega)$ in  the general case 
where $q\geq p$. Let  $(\Norm_n^{[s]}:n\in\N)$  denote the dual multi-norm, based on $L^r(\Omega)$,  to the standard $q\,$-multi-norm based on $L^{\,p}(\Omega)$; 
here $r$ and $s$ are the conjugate indices  to $p$ and $q$, respectively, so that we have $1< s\leq r < \infty$. Then we have an estimate
$$
 \LV (\lambda_1,\dots,\lambda_n)\RV_n^{[s]} \leq  \inf_{\bf X}\left\{\sum_{k=1}^n\left(\sum_{j=1}^n \lV \lambda_{j+k-1}\mid X_j\rV_{\ell^r}^s \right)^{1/s}\right\} 
 $$
 for $\lambda_1,\dots,\lambda_n\in L^r(\Omega)$ and $n\in\N$, where the infimum is taken over all ordered partitions ${\bf X} = (X_1, \dots , X_n)$ of  $\Omega$ 
 into meas\-urable subsets.   For   the special case where $q=p$, see Example 4.47, below; unfortunately, the above estimate does not give the `correct' 
value even in this special case.\medskip

\subsection{A comparison of multi-norms} Suppose that $1\leq p\leq q< \infty$. We have defined the $(p,q)$-multi-norm  $(\norm_n^{(p,q)}: n\in \N)$ 
and the standard $q\,$-multi-norm $(\norm_n^{[q]}  : n\in \N)$ based on $E:= L^{p}(\Omega)$, where $\Omega$ is a measure space.  We shall now show that  
$$(\norm_n^{[q]}  : n\in \N) \leq (\norm_n^{(p,q)}: n\in \N)$$ in ${\mathcal E}_E$ in the notation of Definition \ref{2.5e}.\s

\begin{theorem}\label{4.1b}
Let $(\Omega,\mu)$ be a measure space, and suppose  that  $1\leq p\leq q< \infty$. Then 
$$
\lV(f_1,\dots,f_n)\rV_n^{[q]} \leq  \lV(f_1,\dots,f_n)\rV_n^{(p,q)}\quad (f_1.\dots,f_n \in L^{p}(\Omega,\mu),\,n\in\N)\,.
$$
\end{theorem}

\begin{proof}  We set $r=p'$, the conjugate index to $p$.
Take $n\in\N$ and $f_1.\dots,f_n  \in L^{p}(\Omega)$, and then suppose that  ${\bf X} =(X_1,\dots,X_n)$ is an ordered  partition of $\Omega$. 
 There exist elements $\lambda_1,\dots,\lambda_n \in L^r(\Omega)$ such that $\supp \lambda_i \subset X_i$,  such that $\lV \lambda_i\rV_{L^r} =1$,  and such that   
we have  $\langle f_i, \lambda_i\rangle = \lV f_i\mid X_i\rV_{L^{p}}$ for $i\in\N_n$.  For each $\zeta_1,\dots, \zeta_n \in \C$, we have
 $$
 \lV \sum_{i=1}^n\zeta_i\lambda_i\rV_{L^r} = \left(\sum_{i=1}^n\lv \zeta_i\rv^r\right)^{1/r}\,,
 $$
and so, by (\ref{(3.10f)}),  
$$
\mu_{p,n}(\lambda_1,\dots,\lambda_n) =\sup\left\{\lV \sum_{i=1}^n\zeta_i\lambda_i\rV_{L^r} : \sum_{i=1}^n\lv \zeta_i\rv^r \leq 1\right\}\leq 1\,.
$$
Thus
$$
r_{\bf X}(( f_1,\dots,f_n)) =\left(\sum_{i=1}^n \lV f_i\mid X_i\rV_{L^{p}}^q\right)^{1/q} = \left(\sum_{i=1}^n
\langle f_i, \lambda_i\rangle^q\right)^{1/q} \leq  \lV(f_1,\dots,f_n)\rV_n^{(p,q)}\,.
$$
This holds for each ordered partition ${\bf X}$ of $\Omega$, and so the result follows.
\end{proof}\s

\subsection{Maximality} The following result was pointed out by Paul Ramsden; a more general version will be given in Theorem \ref{2.3n}(i), below.\smallskip

\begin{theorem}\label{4.1ad}
 Let $\Omega$  be a measure space. Then the standard  $1$-multi-norm  and the maximum multi-norm  based on $L^1(\Omega)$ are equal. 
\end{theorem}

\begin{proof} Set $E= L^1(\Omega)$.
Fix $n \in\N$,   take $f_1,\dots,f_n \in E$,  and  set $f = \lv f_1\rv \vee\cdots \vee \lv f_n\rv$ in $E$. 
For  $\lambda_1, \dots,\lambda_n \in E'$, it follows  from   Proposition \ref{3.10ba}(ii) that
\begin{eqnarray*}
\lv \sum_{j=1}^n\langle f_j,\,\lambda_j\rangle\rv &\leq & \sum_{j=1}^n\lv \langle f_j,\,\lambda_j\rangle\rv
\leq  \sum_{j=1}^n \langle \lv f_j\rv,\,
\lv \lambda_j\rv\rangle
\leq \left\langle f,\,\sum_{j=1}^n\lv \lambda_j\rv\right\rangle\\
&\leq& \lV f \rV \lV \sum_{j=1}^n\lv \lambda_j\rv\,\rV = \lV (f_1,\dots,f_n )\rV_n^{[1]} \mu_{1,n}(\lambda_1, \dots,\lambda_n)\,.
\end{eqnarray*}
Hence  $ \lV (f_1,\dots,f_n )\rV_n^{\max} \leq\lV (f_1,\dots,f_n )\rV_n^{[1]}$ by Theorem \ref{3.4d},  giving  the result.
\end{proof}\smallskip

\begin{proposition}\label{4.12}
  Suppose that $1\leq p<q<\infty$.  Then the  $(p,q)$-multi-norm is not equivalent to the maximum multi-norm on $\ell^{\,p}$.
\end{proposition} 

\begin{proof}  By Proposition \ref{4.0e},  $\varphi_n^{\,(1,q)}(\ell^{\,1})\leq n^{1/q}\,\;(n\in\N)$.

Suppose that $p\in [1,2]$.  By Theorem \ref{3.7k}(i), $\varphi_n^{\,\max}(\ell^{\,p}) = n^{1/p}\,\;(n\in\N)$. Since there is no constant $C>0$  such that
 $n^{1/p}\leq Cn^{1/q}\,\;(n\in\N)$, the two multi-norms are not equivalent.
 
 Suppose that $p\in [2, \infty)$.  By Theorem \ref{3.7k}(ii), $\varphi_n^{\,\max}(\ell^{\,p}) \sim  n^{1/2}$. Since there is no constant $C>0$  such that
 $n^{1/2}\leq Cn^{1/q}\,\;(n\in\N)$, the two multi-norms are not equivalent.
\end{proof}\medskip

\subsection{Equality of two multi-norms on $L^1(\Omega)$}  The first result of this section is similar to that given in Proposition \ref{4.14}.  

Let $(\Omega, \mu)$ be a  measure space, and set $E =L^1(\Omega,\mu)$.  Then there is a compact space $K$
such that $E'$ is order-isometric to $C(K)$; $F:= L^{\infty}(\Omega,\mu)$ is a $C^*$-subalgebra of $C(K)$.
  For $n\in\N$,  the weak-$*$ topology on $(E')^n$ as the dual of $E^n$ is 
denoted by $\sigma_n$.  In the following result, $\overline{{\rm co}}(S)$ denotes the $\sigma_n$-closure of the convex hull of a subset $S$ of $(E')^n$.

For each $n\in\N$, let $D_n$ be the set of elements $(\lambda_1,\dots,\lambda_n)$ in $F^n$  such that the subsets
 $\supp \lambda_1,\dots,\supp \lambda_n$ of $\Omega$ are pairwise disjoint.  Since $D_n$ is balanced, ${\rm co}(D_n)$ is also balanced, and hence absolutely convex.\s
 
 \begin{lemma}\label{4.1ag}
Let  $n\in\N$. Then  $(F^n, \mu_{1,n})_{[1]} = \overline{{\rm co}}(D_n)$.
\end{lemma}

\begin{proof} Write  $B_n$  for the closed unit ball $(F^n, \mu_{1,n})_{[1]}$.  Clearly we have $D_n \subset B_n$, and so $\overline{{\rm co}}(D_n)\subset B_n$.

Assume towards a contradiction that there exists  $$\lambda = (\lambda_1,\dots,\lambda_n) \in B_n\setminus \overline{{\rm co}}(D_n)\,.
$$
By the Hahn--Banach separation theorem, there exists  $f=(f_1,\dots,f_n) \in E^n$
such that $\sum_{i=1}^n\langle f_i,\,\lambda_j\rangle >1$, but $\lv\sum_{i=1}^n\langle f_i,\,\mu_i\rangle \rv \leq 1\,\;(\mu_1,\dots,\mu_n \in D_n)$.
By the definition of the standard $1$-multi-norm $(\norm_n^{[1]})$ on $E$,  we have
$$
\lV f \rV_n^{[1]} = \sup\left\{\sum_{i=1}^n \lV f_i\mid X_i \rV : {\bf X} = (X_1,\dots,X_n)   \right\}\,,
$$
 where the supremum is taken over all ordered partitions $\bf X$ of $\Omega$. For $i\in\N_n$, we have   $\lV f_i\mid X_i \rV = 
\sup\{\lv \langle f_i,\,\mu_i\rangle\rv: \mu_i\in  L^1(X_i)'_{[1]}\} $, and so 
 $$
 \lV f \rV_n^{[1]} = \sup\left\{\lv\sum_{i=1}^n\langle f_i,\,\mu_i\rangle \rv : \mu_1,\dots,\mu_n \in D_n\right\} \leq 1\,,
 $$
 whereas $\lV f \rV_n^{(1,1)}\geq \sum_{i=1}^n\langle f_i,\,\lambda_j\rangle >1$. However $\lV f \rV_n^{[1]}= \lV f \rV_n^{\max}$
 by Theorem \ref{4.1ad}, and so
 $\lV f \rV_n^{\max}<  \lV f \rV_n^{[1]}$, a contradiction. 
\end{proof}\s
 
\begin{theorem}\label{4.1ah}
 Let $\Omega$  be a measure space, and take $q\geq 1$. Then the standard  $q\,$-multi-norm  and the $(1,q)$-multi-norm based on $L^1(\Omega)$ are equal. 
\end{theorem}

\begin{proof} Set $E= L^1(\Omega,\mu)$, and   take $n\in\N$ and $f =(f_1,\dots,f_n) \in E^n $.   By replacing
  $\Omega$ by $\bigcup_{i=1}^n \supp f_i$, we may suppose that
$\Omega$ is $\sigma$-finite and hence that $F =E'$ in the notation of Lemma \ref{4.1ag}.
Then 
$$
\lV f \rV_n^{(1,q)} = \sup \left\{\left(\sum_{i=1}^n\lv \langle f_i,\,\lambda_i\rangle \rv^q\right)^{1/q}:  
  \mu_{p,n}(\lambda_1,\dots,\lambda_n)\leq 1\right\} 
$$
and 
$$
\lV f \rV_n^{[q]} = \sup \left\{\left(\sum_{i=1}^n\lv \langle f_i,\,\lambda_i\rangle \rv^q\right)^{1/q}:  (\lambda_1,\dots,\lambda_n) \in D_n\right\}\,,
$$
taking the supremum over all $\lambda_1,\dots,\lambda_n \in E'$ in each case.  By Lemma \ref{4.1ag}, the two suprema are equal.
\end{proof}\medskip

\subsection{Equivalence of multi-norms on $\ell^{\,p}$}  We now ask when various multi-norms based on the spaces $\ell^{\,p}$ are equivalent.

  Take $p,q$ such that  $1\leq p\leq q< \infty$, and set  $E= \ell^{\,p}$. Then we know  that
$$
\lV f \rV_n^{[q]} \leq \lV f \rV_n^{(p,q)} \leq \lV f \rV_n^{\max}\quad (f \in E^n)
$$
for each $n\in \N$.  We ask whether   $(\norm_n^{[q]}) $ is   equivalent to $(\norm_n^{(p,q)}) $, and whether $(\norm_n^{(p,q)}) $  is equivalent to 
 $(\norm_n^{\max}) $, where each multi-norm is based on $\ell^{\,p}$.
 
 First suppose that $p=1$. Then  we  saw in Theorem \ref{4.1ad} that the answer to both these questions  is `yes' when also $q=1$ (with equality of norms).  
In the case where $q>1$,  the $(1,q)$-multi-norm is not equivalent to the maximum multi-norm by Proposition \ref{4.12}. However,  by Theorem \ref{4.1ah}, 
 $\lV f \rV_n^{[q]} = \lV f \rV_n^{(1,q)}$ for $f \in (\ell^{\,1})^n$, $n\in\N$, and all $q\geq 1$. Thus we have complete answers when $p=1$, 
 and so  we shall now consider the case where  $p>1$.   

We shall  show first that $(\norm_n^{[q]}) $ is not equivalent to $(\norm_n^{(p,q)}) $  on $\ell^{\,p}$ in certain cases when  $p>1$.\s

\begin{theorem}\label{4.1ae}
Take $p,q$ such that  $1<  p\leq q< \infty$.  Suppose that either $2\leq p\leq q$ or that $1< p <2$ and $p\leq q < p/(2-p)$. Then  the multi-norms 
$(\norm_n^{[q]}  : n\in \N)$ and $(\norm_n^{(p,q)}  : n\in \N)$ based on $\ell^{\,p}$ are not equivalent.
\end{theorem}

\begin{proof}  The conjugate index to $p$ is denoted by $r$.

Assume towards a contradiction that the two multi-norms are equivalent, so that there exists $C>0$ such that 
$$
\lV (f_1, \dots, f_k)\rV_k^{(p,q)} \leq C\lV (f_1, \dots, f_k)\rV_k^{[q]} 
$$
for each $k\in\N$ and each $f_1, \dots, f_k \in \ell^{\,p}$.

Fix $k\in\N$. For $i\in \N_k$, take 
$$
f_i= \sum_{j=1}^k\zeta^{-ij}\delta_j= (\zeta^{-i}, \zeta^{-2i},\dots,\zeta^{-ki},0,0, \dots)\in \ell^{\,p}\,,
$$
where $\zeta =\exp(2\pi{\rm i}/k)$, and set $f=(f_1,\dots,f_k)$. 

For each ordered partition ${\bf X} = (X_1,\dots, X_k)$ of $\N_k$, we have 
$$
  r_{\bf X}((f_1, \dots, f_k)) \leq \left(\lv X_1\rv^{q/p} + \cdots + \lv X_k\rv^{q/p}\right)^{1/q} \leq k^{1/p}\,,
$$
and so $\lV f\rV_k^{[q]} = k^{1/p}$.

 Now take $\lambda=(\lambda_1,\dots,\lambda_k)$, where 
 $$
 \lambda_i= \sum_{j=1}^k\zeta^{ij}\delta_j= (\zeta^{i}, \zeta^{2i},\dots,\zeta^{ki},0,0, \dots)\in \ell^{\,r}\,.
$$
  As in Lemma \ref{4.1af}, we set $z_i  = \sum_{j=1}^k \zeta_j\zeta^{ij}\;\, (i\in\N_k)$,  so that 
$$
\lV \sum_{i=1}^k\zeta_i\lambda_i\rV_{\ell^{\,r}} = \left(\sum_{i=1}^k\lv z_i\rv^r\right)^{1/r}\,.
$$
It follows from  (\ref{(3.10f)})  that  
  $$
  \mu_{2,k}(\lambda)= \sup\left\{\left(\sum_{i=1}^k\lv z_i\rv^r \right)^{1/r} : \sum_{i=1}^k\lv \zeta_i\rv^2\leq 1\right\} \,.
$$

In the case where $2\leq p\leq q$, we have $\mu_{p,k}(\lambda)\leq \mu_{2,k}(\lambda)$, and so,  by Lemma \ref{4.1af}(i), 
 $\mu_{p,k}(\lambda) \leq k^{1/r}$. Hence
$$
\lV f\rV_k^{(p,q)}\geq \frac{1}{k^{1/r}}\left(\sum_{i=1}^k \lv\langle f_i,\,\lambda_i\rangle \rv^{\,q}\right)^{1/q} 
= \frac{1}{k^{1/r} }(k\,\cdot\,k^{\,q})^{1/q} =k^{1/p+1/q}\,.
$$ 
 We conclude that $k^{1/p+1/q} \leq Ck^{1/p}$ for each  $k\in\N$, a contradiction.  

In the case where $1<p<2$, so that $r> 2$, it follows from equation (\ref{(3.2d)}) that 
$$
\left(\sum_{i=1}^k\lv \zeta_i\rv^2\right)^{1/2}\leq k^{ 1/2-1/r }
$$
 whenever $\sum_{i=1}^k\lv \zeta_i\rv^r\leq 1$, and so, using Lemma \ref{4.1af}(i) again,
$$ 
 \lV \sum_{i=1}^k\zeta_i\lambda_i\rV_{\ell^{\,r}} \leq  \lV \sum_{i=1}^k\zeta_i\lambda_i\rV_{\ell^{\,2}}
= \left(\sum_{i=1}^k\lv z_i\rv^2 \right)^{1/2}\leq k^{1/2}\,\cdot\, k^{1/2-1/r}=k^{1/p}\,.
$$ 
 Thus $\mu_{p,k}(\lambda)\leq k^{1/p}$, and so 
$$
\lV f\rV_k^{(p,p)}\geq \frac{1}{k^{1/p}}\left(\sum_{i=1}^k \lv\langle f_i,\,\lambda_i\rangle \rv^{\,q}\right)^{1/q} 
= \frac{1}{k^{1/p}}(k\,\cdot\,k^{\,q})^{1/q} =k^{ 1+1/q-1/p}\,.
$$
  We conclude that $k^{ 1+1/q-1/p}\leq Ck^{1/p}$  for each   $k\in\N$. Thus $ 1 + 1/q\leq 2/p$, and so $q\geq p/(2-p)$, again a contradiction of an hypothesis.

Thus  the two multi-norms are not equivalent in the cases stated.
\end{proof}\s

We do not know if the   two multi-norms are   equivalent in the case where $1< p <2$ and  $ q \geq p/(2-p)$. This point, and more general ones,
 will be discussed in \cite{DDPR2}.\s

\begin{corollary}\label{4.1ai}
Let  $p\geq 1$. Then  the two multi-norms $$(\norm_n^{[p]}  : n\in \N)\quad {\rm  and}\quad (\norm_n^{\max}  : n\in \N)$$
based  on $\ell^{\,p}$ are  equivalent if and only if $p=1$.\qed
\end{corollary}

We noted in $\S\ref{associated sequence}$ that the rates of growth of two equivalent multi-norms are similar. The next result, taken together with 
Corollary \ref{4.1ai}, shows that the converse statement is not true.\s

\begin{proposition}\label{4.1aj}
Take $p\geq 1$ and $n\in\N$.  Then:\s     

{\rm (i)} $\varphi_n^{[p]}(\ell^{\,p}) =n^{1/p}$;  \s

{\rm (ii)}  $\varphi_n^{\max}(\ell^{\,p}) =n^{1/p}$  when $p\in [1,2]$  and $\varphi_n^{\max}(\ell^{\,p}) \sim \sqrt{n}$  when $p\in [2,\infty)$.\s  

\noindent Thus, for $p\in (1,2]$, we have $(\varphi_n^{[p]}(\ell^{\,p}))\sim (\varphi_n^{\max}(\ell^{\,p}))$, but the multi-norms  $(\norm_n^{[p]})$ 
 and $(\norm_n^{\max})$ based on $\ell^{\,p}$ are not equivalent.
\end{proposition}

\begin{proof}  This follows from equation (\ref{(4.12)}), Theorem \ref{3.7k}, and Corollary \ref{4.1ai}.
\end{proof}\s

There remains the question whether the two multi-norms $(\norm_n^{(p,p)}) $  and   $(\norm_n^{\max}) $ based on $\ell^{\,p}$ are equivalent. 
  We know from  Theorem \ref{4.1ah} that they are equivalent in the case where $p=1$, and, as we remarked in Theorem \ref{4.16}, they are  equivalent
 in the case where $p=2$. The question for other values of $p$ will be resolved  in \cite{DDPR2}.
\medskip

\subsection{The spaces $M(K)$}\label{The spaces $M(K)$}   Throughout this section, $K$ is a non-empty, 
locally compact space.     
 For $q\geq 1$, we shall define the {\it standard $q\,$-multi-norm\/}  based on $M(K)$ in essentially the same way as above.\s

Take $q\geq 1$. For each ordered partition  ${\bf X} = (X_1, \dots , X_n)$ of  $K$  into (Borel) measurable subsets  and each $\mu_1,\dots,\mu_n\in M(K)$, we set
$$
r_{\bf X}((\mu_1,\dots,\mu_n)) = \left( \lV \mu_1\mid X_1\rV^q+\cdots + \lV\mu_n\mid X_n\rV^q\right)^{1/q}\,,
$$
so that $r_{\bf X}$ is a seminorm on $M(K)^n$ and
$$
r_{\bf X}((\mu_1,\dots,\mu_n))\leq \left( \lV \mu_1\rV^q +\cdots + \lV \mu_n\rV^q\right)^{1/q}\quad (\mu_1,\dots,\mu_n \in M(K))\,.
$$

Finally, we define
$$
\lV (\mu_1,\dots,\mu_n)\rV_n^{[q]}  = \sup_{\bf X}r_{\bf X}((\mu_1,\dots,\mu_n))\quad (\mu_1,\dots,\mu_n\in M(K)) \,,
$$
where the supremum is taken over all such ordered partitions  ${\bf X}$.   Then $\norm_n^{[q]}$ is a norm on $M(K)^n$, and it is again easily checked 
 that $(\norm_n^{[q]}: n\in\N)$ is a multi-norm on $\{M(K)^n : n\in \N\}$.\s
 
 \begin{definition}\label{4.1c}
Let $K$ be a non-empty, locally compact space. For each $q \geq 1$,  the {\it standard $q\,$-multi-norm} based on
 $M(K)$ is the multi-norm $(\norm_n^{[q]} :n\in\N)$, with  rate of growth  $(\varphi_n^{[q]}(M(K)):n\in\N)$. 
\end{definition}\smallskip

We shall see in Theorem \ref{6.38} that the standard $q\,$-multi-norm on a space of the form $M(K)$ is  a property of the Banach space $M(K)$.\s

 \begin{theorem}\label{4.1d}
Let $K$ be a non-empty, locally compact space. Then the  standard $1$-multi-norm $(\norm_n^{[1]}: n\in\N)$ based on $M(K)$ is given by 
\begin{equation}\label{(4.1ba)}
\lV (\mu_1,\dots,\mu_n)\rV_n^{[1]} = \lV\, \lv \mu_1\rv \vee  \cdots\vee \lv \mu_n\rv \,\rV\quad (\mu_1, \dots,\mu_n \in M(K))\,.
\end{equation}
\end{theorem}
 
\begin{proof} Take $\mu_1, \dots,\mu_n \in M(K)$, and set $\mu= \lv \mu_1\rv \vee  \cdots\vee \lv \mu_n\rv\in M(K)$.

  For each ordered partition  ${\bf X} = (X_1, \dots , X_n)$ of  $K$, we have
$$
\lV \mu_1\mid X_1\rV+\cdots + \lV\mu_n\mid X_n\rV =\sum_{i=1}^n\lv \mu_i\rv(X_i) \leq \sum_{i=1}^n\mu(X_i) = \lV \mu\rV\,.
$$
Thus $\lV (\mu_1,\dots,\mu_n)\rV_n^{[1]} \leq \lV\, \lv \mu_1\rv \vee  \cdots\vee \lv \mu_n\rv \,\rV$.

For the opposite inequality, we shall show that, for each $n\geq 2$ and $\mu_1, \dots,\mu_n \in M(K)$, there is  an ordered partition 
${\bf X} = (X_1, \dots , X_n)$ of $K$ such that  $$\lV \mu\rV = \lV \mu_1\mid X_1\rV+\cdots + \lV\mu_n\mid X_n\rV\,.
$$

Consider first the case where $n=2$ and $\mu_1,\mu_2\in M(K)$. Let $P= X_1 $ and $N=X_2$ be the measurable subsets of $K$ associated
  with $\lv \mu_1\rv- \lv \mu_2\rv$ in the Hahn decomposition. (See page \pageref{Hahndecomp}.) Then
\begin{eqnarray*}
 \lV\, \lv \mu_1\rv \vee  \lv \mu_2\rv \,\rV &=& (\lv \mu_1\rv\vee \lv \mu_2\rv)(X_1) + (\lv \mu_1\rv\vee \lv \mu_2\rv)(X_2)\\&=&  
\lv \mu_1\rv (X_1) + \lv \mu_2\rv (X_2)  
= \lV \mu_1\mid X_1\rV + \lV \mu_2\mid X_2\rV\,,
\end{eqnarray*}
and so $(X_1,X_2)$ is the required partition.  

The result for a general $n\in\N$  follows by an easy induction.
\end{proof}\s
 
We shall see in Theorem \ref{2.3n}(i) that the  standard $1$-multi-norm based on $M(K)$ is the maximum multi-norm on $M(K)$.\s

Recall that the topology of a Stonean  space has a basis consisting of clopen subsets; the   space  $\beta \N$ is a Stonean  space.\s

 \begin{proposition}\label{4.1e}
Let $K$ be a Stonean  space, and take $q\geq 1$. Then, for  each $n\in\N$ and $\mu_1,\dots,\mu_n\in M(K)$, we have 
 $$ 
\lV (\mu_1,\dots,\mu_n)\rV_n^{[q]}  = \sup \left( \lV \mu_1\mid K_1\rV^q+\cdots + \lV\mu_n\mid K_n\rV^q\right)^{1/q}\,,
 $$
taking the supremum over all ordered partitions $(K_1,\dots,K_n)$ of $K$ into clopen subspaces.
 \end{proposition}
 
 \begin{proof}  Clearly, $$\lV (\mu_1,\dots,\mu_n)\rV_n^{[q]}  \geq \left( \lV \mu_1\mid K_1\rV^q+\cdots + \lV\mu_n\mid K_n\rV^q\right)^{1/q} $$ for each 
such ordered partition $(K_1,\dots,K_n)$.

Now fix  $\varepsilon>0$, and choose an ordered  partition  ${\bf X} = (X_1, \dots , X_n)$ of  $K$  into   measurable subsets such that 
$$ 
r_{\bf X}((\mu_1,\dots,\mu_n))> \lV (\mu_1,\dots,\mu_n)\rV_n^{[q]}-\varepsilon\,.
$$
 Set $\mu = |\mu_1|+\cdots+|\mu_n|$. Since $\mu$ is regular, there exists a family  $\{L_1,\dots, L_n\}$ of clopen subsets of
 $K$ such that  $\mu(L_i\Delta X_i)<\varepsilon\;\,(i\in\N_n)$. Set  $K_1=L_1$ and  $K_i= L_i\setminus(L_1\cup\cdots \cup L_{i-1})$ for 
$i=2,\dots,n$, so that $(K_1,\dots,K_n) $ is an ordered  partition of $K$ into clopen subspaces. Then 
$$
\mu(K_i\Delta X_i)<\varepsilon+\sum_{j=1}^{i-1}\mu(L_i\cap L_j)<2n\varepsilon\quad (i=2,\dots,n)\,,
$$
 where in the last inequality we also use the fact that $L_i\cap L_j=\emptyset$ when $j<i$. Thus  we see that
$$
r_{\bf X}((\mu_1,\dots,\mu_n))<  \left(\sum_{i=1}^n\left(\lV\mu_i\mid K_i\rV+2n\varepsilon\right)^q\right)^{1/q}\,,
$$
and hence  that
$$
\lV (\mu_1,\dots,\mu_n)\rV_n^{[q]}<  \left(\sum_{i=1}^n\lV\mu_i\mid K_i\rV^q\right)^{1/q}+O(\varepsilon)\quad {\rm as}\;\;\varepsilon \searrow 0\,.
$$ 
 
The result follows. 
\end{proof}\medskip

\subsection{The Schauder multi-norm}\label{Schauder multi-norm}  We now give an example related to the standard $p\,$-multi-norm on $\ell^{\,p}$.

 Let $(E, \norm )$ be a Banach space. A series $\sum_{n=1}^\infty x_n$ in $E$ is said to 
 {\it converge  unconditionally\/} if the series  $\sum_{n=1}^\infty \varepsilon_nx_n$ 
converges in $E$ whenever $\varepsilon _n \in \{1,-1\}\,\;(n\in \N)$.  This is equivalent to the requirement that
 $\sum_{n=1}^{\infty} x_{\sigma(n)}$ converges in $E$ for each $\sigma \in {\mathfrak S}_{\N}$.

 Now suppose that $E$ has a Schauder basis  $\{e_n : n\in \N\}$, so that each $x \in E$ has a unique expansion in the form
$$
x =  \sum_{n=1}^\infty \alpha_ne_n\,,
$$
where $\alpha_n \in \C \;\,(n\in\N)$.  The basis $\{e_n : n\in \N\}$ is an {\it unconditional basis\/}  
if, for each $x \in E$, the corresponding series  $\sum_{n=1}^\infty \alpha_ne_n$ converges unconditionally. The standard basis of 
${\ell}^{\,p}$ (for $p\geq 1$) and of $c_{\,0}$ is unconditional in the appropriate Banach space.   We note that the Banach spaces 
$L^{p}(\I )$ have an unconditional basis whenever $p>1$, but that the Banach spaces $L^1(\I )$ and $C(\I)$ do not have an unconditional basis.

For details of these and related results about unconditional bases, see \cite[\S 3.1]{AK}, \cite[I, \S 1.c]{LT}, or \cite[\S II.D]{W}, for example.

We now define
$$
\LV \sum_{n=1}^\infty \alpha_ne_n \RV = \sup\left\{\lV \sum_{n=1}^\infty \alpha_n\beta _n e_n \rV: \lv \beta_n \rv \leq 1\,\;(n\in \N)\right\}\,.
$$
As in \cite[I, p.~19]{LT},  $\Norm$ is a norm on $E$ such that
$$
\lV x \rV \leq \LV x \RV \leq C \lV x \rV \quad (x\in E)
$$
for some constant $C\geq 1$. The original norm is {\it $1$-unconditional\/} if
 the modified norm coincides with the original one.   In the case where $E = {\ell}^{\,p}$ for $p\geq 1$, the usual  norm is $1$-unconditional.

Now suppose that $\norm$ is a $1$-unconditional norm on $E$. For each  non-empty subset $S$ of $\N$,   define
$$
P_S : \sum_{n=1}^\infty \alpha_ne_n \mapsto \sum_{n\in S} \alpha_ne_n\,,\quad E \to E\,,
$$
so that $\lV P_S \rV =1$. Let ${\bf S}= (S_1, \dots , S_n)$  be an ordered  partition of $\N$, say into infinite subsets of $\N$, and define
$$
r_{\bf S}((x_1,\dots, x_n)) =  \lV P_{S_1}(x_1) + \cdots + P_{S_n}(x_n)\rV
\quad (x_1,\dots,x_n  \in E)\,,
$$
and then set   
$$
\lV (x_1,\dots,x_n)\rV_n  = \sup_{\bf S}r_{\bf S}((x_1,\dots,x_n))\quad (x_1,\dots,x_n  \in E)\,,
$$
where the supremum is taken over all such ordered partitions  $\bf S$. It is again easily checked that   $(\norm_n : n\in\N )$ is a multi-norm on 
$\{E^n : n\in \N\}$.

In the case where $E = {\ell}^{\,p}$, these norms are exactly the standard $p\,$-multi-norms on ${\ell}^{\,p}$ of $\S4.2.1$.
\s

\begin{definition}\label{4.3a}
Let $(E, \norm )$ be a Banach space with a $1$-unconditional norm.  Then the {\it Schauder multi-norm}
 on $\{E^n: n\in\N\}$ is the multi-norm defined above. \end{definition}\smallskip

In particular, let $E = \C^k$ for some $k \in \N$, and let $\norm$  be a norm on $E$ such that
$$
\lV (\zeta_1z_1, \dots , \zeta_kz_k) \rV
= \lV (z_1, \dots , z_k) \rV\quad(z_1, \dots , z_k \in \C,\, \zeta_1, \dots, \zeta_k \in \T)\,.
$$
Then we can generate a Schauder multi-norm on $\{E^n: n\in\N\}$.
\medskip

\subsection{Abstract $q$-multi-norms}  We now give a  more abstract version of the standard $q\,$-multi-norm on the space 
$L^p(\Omega)$, where $\Omega$ is a measure space.  This subsection is based on discussions with Hung Le Pham.

Let $E$ be a $\sigma$-Dedekind complete Banach lattice. Recall from (\ref{(2.3e)}) that,  for each $v\in E^+$, there  is a certain positive linear projection 
$P_v$  with  $\lV P_v\rV\leq 1$.  Now  take $q\geq 1$ and $n\in\N$. For each  $v=(v_1,\ldots,v_n)\in E^n$ with  $|v_i|\wedge |v_j|=0$ for $i,j\in\N_n$ with
 $i\neq j$, set
$$
r_v((x_1,\dots,x_n))=\left(\sum_{i=1}^n\lV P_{|v_i|}x_i\rV^q\right)^{1/q}\quad(x_1,\ldots, x_n\in E)\,.
$$
Next define
$$
\lV (x_1,\ldots,x_n)\rV_n^{[q]} =\sup_v r_v((x_1,\dots,x_n))\quad(x_1,\ldots, x_n\in E)\,,
$$
where the supremum is taken over all  $v=(v_i)\in E^n$ with  $|v_i|\wedge |v_j|=0$ for  $i,j\in\N_n$ with $i\neq j$.
 
Let $n\in\N$, and take $q\geq 1$. Then it  is obvious that $\norm^{[q]}_n$ is a norm on $E^n$. Since $P_{\lv x \rv}(x) = x\,\;(x\in E)$,  we have 
$\lV x\rV_1^{[q]}=\lV x \rV\,\;(x\in E)$.    Moreover, we see that
\begin{equation}\label{(7.1a)}
\lV(x_1,\ldots,x_n)\rV_n^{[q]}=\sup \left(\sum_{i=1}^n\lV y_i\rV^q\right)^{1/q}\quad(x_1,\ldots, x_n\in E)\,,
\end{equation}
where the supremum is taken over $y_1,\ldots,y_n\in E^+$ with  $y_i\leq \lv x_i\rv\,\;(i\in\N_n)$ and $y_i\wedge y_j=0$ for 
 $i,j\in\N_n$ with $i\neq j$. 
 
 The following is clear.  \s
 
 \begin{theorem}\label{4.32}
 Let $(E, \norm)$ be a $\sigma$-Dedekind complete  Banach lattice, and take $q\geq 1$.  Then $(\norm_n^{[q]} :n\in\N)$ is 
 a special-norm;  it is a multi-norm if and only if
$$
\lV x+y\rV^q \geq \lV x\rV^q  + \lV y\rV^q \quad{\rm for}\quad x,y\in E\;\; {\rm with} \;\; \lv x\rv \wedge \lv y \rv =0\,.
$$ 
In the case  where $E$ is an $AL_p$-space and $q\geq p$, $(\norm_n^{[q]} :n\in\N)$ is a multi-norm on $\{E^n: n\in\N\}$.\qed
\end{theorem}\s

 \begin{definition}\label{4.33}
 Let $(E, \norm)$ be an $AL_p$-space, and take $q\geq p$. Then $(\norm_n^{[q]} :n\in\N)$ is the {\it abstract $q\,$-multi-norm based on $E$.} 
 \end{definition}\s

For example,  suppose  that  $p\geq 1$ and $E= L^p(\Omega)$ for  a measure space $\Omega$ and $q\geq p$, or that $E= M(K)$ for a non-empty,
 locally compact space $K$ and $q\geq 1$.  Then the abstract 
$q\,$-multi-norm $(\norm_n^{[q]} :n\in\N)$ is precisely the standard $q\,$-multi-norm of Definition \ref{4.1a} or \ref{4.1c}. 
Thus the following theorem  follows easily from (\ref{(7.1a)}).\s

\begin{theorem}\label{4.31}
Let $E$ be the Banach lattice $L^p(\Omega)$ for  a measure space $\Omega$  and $p\geq 1$, or the Banach lattice $M(K)$ for a non-empty,
 locally compact space $K$. Suppose that $q\geq p$ or $q\geq 1$, respectively.  Then the standard $q\,$-multi-norm on $E$ does not depend on 
the particular realization of $E$ as an $L^p$-space or as a space of  measures; it depends on only the norm and the lattice structures of $E$.  \qed
\end{theorem}\s   

In fact, more can be said.  Let $E$  be an   $AL_p$-space,  and take $q\geq p$.  In Definition \ref{4.33}, we defined the 
abstract $q\,$-multi-norm based on $E$. We shall now show that the abstract  $q\,$-multi-norms
 based on two $AL_p$-spaces which are just isometrically  isomorphic are equal whenever $p\neq 2$. 
 
 The first result is a special case of Theorem \ref{6.39}, given below, but we give a separate short proof.\s

\begin{theorem}\label{6.38}
Let $\Omega$ be a measure space,  and take   $q\geq 1$. Then the standard $q\,$-multi-norm on $L^1(\Omega )$  is determined by the Banach-space structure of
 $L^1(\Omega)$.\end{theorem}
 
\begin{proof}   By Theorem \ref{4.1ah}, the standard  $q\,$-multi-norm  and the $(1,q)$-multi-norm based on $L^1(\Omega)$ are equal. However, the $(1,q)$-multi-norm is 
 determined by the Banach-space structure of $L^1(\Omega)$.
\end{proof}\s

We shall now consider the `second dual question' ({\it cf.} Corollary \ref{4.0m}): we should
 like the second dual of the abstract $q\,$-multi-norm  on $L^p(\Omega)$  or  $M(K)$
to be the abstract $q\,$-multi-norm on the second dual of the respective space. In the case of $L^p(\Omega)$ for $p>1$, this is immediate because 
$L^p(\Omega)$ is then a reflexive Banach space, and so it suffices to consider the spaces $L^1(\Omega)$  and  $M(K)$, which are $AL$-spaces.\s

\begin{theorem}\label{4.34}
Let $E$ be an $AL$-space.   For each $q\geq 1$, the second dual of the abstract $q\,$-multi-norm 
based on $E$ is the abstract $q\,$-multi-norm based on $E''$.
\end{theorem}
   
\begin{proof} The standard $q\,$-multi-norm on the second dual  of an  $AL$-space is the same whether the second dual be considered as a measure space 
or as an $L^1$-space,  and is equal to the abstract $q\,$-multi-norm by Theorem \ref{4.31}; by Theorem \ref{4.1ah}, it is the $(1,q)$-multi-norm, say
 on a space $L^1(\Omega)$.   Also by Theorem \ref{4.1ah}, the standard $q\,$-multi-norm on $E$ is the $(1,q)$-multi-norm. Thus the result follows from Corollary \ref{4.0m}.
\end{proof}\s

We now extend Theorem \ref{6.38} to multi-norms based on $L^p(\Omega )$  when $p\neq 2$.  In the following result $(\norm^{[q]}_n : n\in\N)$ denotes the 
abstract $q\,$-multi-norm based on both $E$ and $F$.\s

\begin{theorem}\label{6.39}
Take $p,q$ with $1 \leq p\leq q< \infty$ and $p\neq 2$. Suppose that  both $E$ and $F$ are $AL_p$-spaces and that $U:E\to F$ is an isometric isomorphism. Then 
$$
U^{(n)}:(E^n, \norm^{[q]}_n)\to (F^n,\norm^{[q]}_n)
$$
is an isometry for each $n\in\N$.
\end{theorem}

\begin{proof}  
By Theorem \ref{2.3gc}(i), we may suppose that $E=L^p(\Omega_{\,1})$ and $F=L^p(\Omega_{\,2})$, where 
$\Omega_{\,1}$ and $\Omega_{\,2}$ are measure spaces.  

Fix $n\in\N$. In our setting, we have
$$
\lV(f_1,\ldots,f_n)\rV^{[q]}_n=\sup \left(\sum_{i=1}^n\lV p_i\,\cdot\, f_i\rV^q\right)^{1/q}\quad(f_1,\dots,f_n\in L^p(\Omega_{\,1}))\,,
$$
where the supremum is taken over the collection, say  $C_{n,E}$, of all tuples $(p_1,\dots, p_n)$ of disjoint projections in $L ^{\infty}(\Omega_{\,1})$
 and $p_i\,\cdot\, f_i$ is the $L ^{\infty}(\Omega_{\,1})$-module product in  $E$;  a similar formula holds for elements in $F^n$.
 
By Lamperti's theorem, Theorem \ref{7.4} (which applies because $p\neq 2$),  we see that $U$ has the form 
$$
U : f\mapsto h\,\cdot\, T_{\sigma}f\,,\quad L^p(\Omega_1)\to L^p(\Omega_2)\,,
$$
where $h : \Omega_2 \to \C$   and $T_\sigma \in {\mathcal B}(L^p(\Omega_1), L^p(\Omega_2))$ is  induced by a regular set isomorphism $\sigma$. 
 Note that $T_\sigma$ extends to a  $*$-isomorphism from the algebra  of all  measurable functions on $\Omega_1$ (modulo null functions), and so 
$T_\sigma$ restricts to a  $*$-isomorphism
from $L ^{\infty}(\Omega_{\,1})$ to $L ^{\infty}(\Omega_{\,2})$.  For each $p\in L ^{\infty}(\Omega_{\,1})$  and $f\in L^p(\Omega_1)$, we have
$$
T_\sigma(p)\,\cdot\,Uf = T_\sigma(p)h\,\cdot\, T_{\sigma}f = hT_\sigma(pf)\,,
$$
and so $U(p\,\cdot\,f) = T_\sigma(p)\,\cdot\,Uf$. Hence, for each $n\in\N$ and $f_1,\dots,f_n\in L^p(\Omega_{\,1})$, we have
\begin{eqnarray*}	
\lV(Uf_1,\dots,Uf_n)\rV^{[q]}_n  &=&
\sup_{\ (q_i)\in C_{n,F}}\left(\sum_{i=1}^n\lV q_i\,\cdot\, Uf_i\rV^q\right)^{1/q}\\
&=&
\sup_{\ (p_i)\in C_{n,E}}\left(\sum_{i=1}^n\lV U(p_i\,\cdot\,f_i)\rV^q\right)^{1/q}\\
&=&
\sup_{\ (p_i)\in C_{n,E}}\left(\sum_{i=1}^n\lV p_i\,\cdot\, f_i\rV^q\right)^{1/q}\\
&=&
 \lV (f_1,\dots,f_n)\rV^{[q]}_n\,,
\end{eqnarray*}
and so $U^{(n)}$ is an isometry, as required.
\end{proof}\s

\begin{theorem}\label{6.39a}
Let $E$ be an $AL_p$-space, where $p\geq 1$ and $p\neq 2$. Then, for each $q\geq p$, the  abstract $q\,$-multi-norms based on $E$ depends on 
only the Banach space $E$, and not on its lattice structure.\qed
\end{theorem}
\medskip

\section{Lattice multi-norms}\label{Lattice multi-norms}
\noindent We now define a `lattice multi-norm' based on a Banach lattice. Basic facts about Banach lattices were recalled in $\S$\ref{Banach lattices}.\s
 
\subsection{Multi-norms and Banach lattices}
 We define a multi-norm and a dual  multi-norm  naturally connected with a Banach lattice. 
 \smallskip

\begin{definition}\label{2.3a}
Let $(E,\norm)$ be a Banach lattice.  For $n\in \N$ and $x_1, \dots, x_n \in E$, set
$$
\lV (x_1,\dots,x_n)\rV_n^{L} = \lV \,\lv x_1\rv\vee \cdots \vee \lv x_1\rv\,\rV\,,\quad \lV (x_1,\dots,x_n)\rV_n^{DL} =
\lV \,\lv x_1\rv+ \cdots + \lv x_1\rv\,\rV\,.
$$ 
\end{definition}\smallskip

\begin{theorem}\label{2.3b}
Let $(E,\norm)$ be a Banach lattice. Then the sequence $(\norm_n^L : n\in \N)$ is a multi-norm based on $E$, and  
 $(\norm_n^{DL} : n\in \N)$ is a dual multi-norm based on $E$
\end{theorem}

\begin{proof} This is immediately checked.\end{proof}\smallskip

\begin{definition}\label{2.3c}
Let $(E,\norm)$ be a Banach lattice.  Then $(\norm_n^L: n\in \N)$ is  the {\it lattice multi-norm}  based  on $\{E^n : n\in\N\}$
 and $(\norm_n^{DL}: n\in \N)$ is   the {\it  dual lattice multi-norm} based  on $\{E^n : n\in\N\}$.  The rate of growth of the lattice
 multi-norm is  denoted by $(\varphi_n^L(E) : n\in\N)$.
\end{definition}\smallskip

\begin{theorem}\label{2.3ce}
Let $(E,\norm)$ be a Banach lattice. Then the dual of the lattice multi-norm on $\{E^n : n\in\N\}$ is the dual lattice multi-norm on $\{(E')^n : n\in\N\}$.
\end{theorem}

\begin{proof} Let $(\norm_n^L : n\in\N)$ be the lattice multi-norm on the family $\{E^n : n\in\N\}$. 

For $n\in\N$,   write $\norm_n'$ for the dual norm to  $ \norm_n^L$ on $(E')^n$.  We must prove that 
\begin{equation}\label{(4.13)}
\lV (\lambda_1,\dots,\lambda_n)\rV'_n = \lV\,\lv \lambda_1\rv+ \cdots +\lv \lambda_n\rv\,\rV
\quad(\lambda_1,\dots,\lambda_n\in E')\,.
\end{equation}
Indeed, take $\lambda_1,\dots,\lambda_n\in E'$, and write $\lambda = \lv \lambda_1\rv+ \cdots +\lv \lambda_n\rv \in E'$.

Suppose that $x_1,\dots, x_n \in E$  with $\lV(x_1,\dots,x_n)\rV_n^L \leq 1$, and set  
$$
x = \lv x_1\rv\vee\cdots\vee \lv x_n\rv\,,
$$
 so that $\lV x \rV \leq 1$.  Using (\ref{(2.3c)}), we see that
\begin{eqnarray*}
\lv \langle (x_1,\dots, x_n),\,(\lambda_1,\dots,\lambda_n)\rangle\rv 
& \leq  & \sum_{j=1}^n \lv \langle   x_j,\, \lambda_j\rangle\rv\leq
 \sum_{j=1}^n \langle \lv x_j\rv,\,\lv\lambda_j\rv\rangle  \leq   \langle x,\,\lambda\rangle\,,
\end{eqnarray*}
and hence  that $\lV (\lambda_1,\dots,\lambda_n)\rV'_n \leq  \lV\lambda\rV$.

Given $\varepsilon >0$, there exists $x \in E^+$ with $\lV x \rV =1$ and $\langle x,\,\lambda\rangle > \lV \lambda\rV -\varepsilon$.
 It follows from Proposition \ref{2.3gf} that, for each $j\in \N_n$, there exists $y_j \in E$ with $\lv y_j\rv \leq x$ and 
 $\langle y_j,\,\lambda\rangle >\langle x,\,\lv \lambda\rv\rangle-\varepsilon$.  We have $\lv y_1\rv\vee \cdots \vee \lv y_n\rv \leq x$, and so
 $$
 \lV (y_1, \dots, y_n)\rV^L_n =\lV \,\lv y_1\rv\vee \cdots \vee \lv y_n\rv\,\rV \leq \lV x \rV \leq 1\,.
 $$
Also,
\begin{eqnarray*}
\lv \langle (y_1,\dots,y_n),\,  (\lambda_1,\dots,  \lambda_n )\rangle\rv &=&
 \lv \sum_{j=1}^n \langle y_j,\,\lambda_j\rangle\rv \geq  \sum_{j=1}^n \langle x,\,\lv\lambda_j\rv\rangle-n\varepsilon\\
&=&  \langle x,\,\lv\lambda\rv\rangle-n\varepsilon > \lV \lambda\rV -(n+1)\varepsilon\,,
\end{eqnarray*}
and so $ \lV (\lambda_1,\dots,\lambda_n)\rV'_n  \geq \lV \lambda\rV- (n+1)\varepsilon$.
This holds true for  each $\varepsilon >0$, and so $ \lV (\lambda_1,\dots,\lambda_n)\rV'_n  \geq \lV \lambda\rV$.
 
Thus equation (\ref{(4.13)}) holds.\end{proof}\smallskip
 
\begin{theorem}\label{2.3cf}
Let $(E,\norm)$ be a Banach lattice. Then the dual of the dual lattice multi-norm on $\{E^n : n\in\N\}$ is the lattice multi-norm on $\{(E')^n : n\in\N\}$.
\end{theorem}

\begin{proof} This is similar to the above proof.   \end{proof}\smallskip

\begin{corollary}\label{2.3cg}
Let $(E,\norm)$ be a Banach lattice. Then the second dual of the   lattice multi-norm on $\{E^n : n\in\N\}$ is the   lattice multi-norm on 
$\{(E'')^n : n\in\N\}$. \qed
\end{corollary}\smallskip
 
\begin{example}\label{2.3ca}
{\rm Let $\Omega$ be a measure space, take $p\geq 1$, and let $E$ be the Banach lattice $L^{p}(\Omega)$.   
Then the corresponding lattice multi-norm $\{(E^n, \norm_n) : n\in \N\}$  is given by
$$
\lV (f_1,\dots,f_n)\rV_n^L =
 \left(\int_\Omega (\lv f_1 \rv \vee \dots \vee \lv f_n \rv)^{\,p}\right)^{1/p}= \lV (f_1,\dots,f_n)\rV_n^{[p]}\,,
$$
where we are using equation (\ref{(4.1b)}). Thus the lattice multi-norm and the  standard $p\,$-multi-norm  based on $E$ coincide.

It follows that  the  dual of the standard $p\,$-multi-norm based on  $L^{p}(\Omega)$ is given by
$$
\LV (\lambda_1,\dots,\lambda_n)\RV_n^{[r]} = \lV \,\lv\lambda_1\rv+ \cdots + \lv\lambda_n\rv\,\rV_{L^{r}(\Omega)}
$$
for $\lambda_1,\dots,\lambda_n\in L^{r}(\Omega)$ and $n\in\N$, where $r=p'$.}\qed 
\end{example}\smallskip

\begin{example}\label{2.3cc}
{\rm Let $K$ be a non-empty, locally compact space, so that the Banach space $(M(K), \norm)$ is a Banach lattice. Then the 
corresponding lattice multi-norm based on $M(K)$  is just the standard $1$-multi-norm; for this, 
see Theorem  \ref{4.1d}.}\qed 
\end{example}\smallskip

\begin{definition}\label{2.3d}
Let $(E,\norm)$ be a Banach lattice.  Then a multi-norm $(\norm_n  : n\in\N)$ on $\{E^n : n\in\N\}$  is {\it compatible with the lattice structure} if, for
 each $n \in\N$, we have
$$
\lV (x_1, \dots, x_n) \rV_n \leq \lV (y_1, \dots, y_n) \rV_n
$$
whenever $\lv x_i\rv \leq \lv y_i\rv$ in $E_{\R}$ for each $i\in \N_n$.  
\end{definition}\smallskip

\begin{proposition}\label{2.3e}
Let $(E,\norm)$ be a Banach lattice. Then the lattice multi-norm is the maximum multi-norm which is compatible with the lattice structure.
 \end{proposition}\smallskip

 \begin{proof}  Certainly the lattice multi-norm $(\norm_n^L: n\in\N)$ is compatible with the lattice structure.  Let $(\norm_n: n\in\N)$
 be any multi-norm which is compatible with the lattice structure.  Take $n\in\N$ and $x_1,\dots, x_n\in E$, and 
set $x = \lv x_1\rv \vee \cdots \vee \lv x_n\rv$. Then
$$
\lV (x_1,\dots, x_n)\rV_n \leq  \lV (x,\dots,x)\rV_n=\lV x \rV= \lV (x_1,\dots, x_n)\rV_n^L\,,
$$
 and so the lattice multi-norm is the maximal norm with this property.
  \end{proof}\s

\begin{proposition}\label{2.3i}
Let $(E,\norm)$ be a Banach lattice, let $n\in\N$, and suppose that $$E = E_1\oplus_{\perp} \cdots \oplus_{\perp} E_n\,.
$$
 Then
$$
\lV (x_1,\dots,x_n)\rV_n^L = \lV \,\lv x_1\rv + \cdots + \lv x_n\rv\,\rV = \lV  x_1 + \cdots +  x_n\rV
$$
whenever $x_j \in E_j$ for $j\in \N_n$.
\end{proposition}\smallskip

 \begin{proof}    This follows  immediately  from equation  (\ref{(2.4ab)}).
  \end{proof}\s
  
  Thus the lattice multi-norm and the dual lattice multi-norm coincide on elements $(x_1,\dots,x_n) \in E^n$ such that $x_j \in E_j$ for $j\in \N_n$.

The following result is easily checked.\s

\begin{proposition}\label{2.3l}
 Let $E$ be a Banach lattice, and let $F$ be a closed subspace which is an order-ideal in $E$.  Then the multi-norm
 defined by the  Banach lattice $E/F$ coincides with the quotient multi-norm.\qed
\end{proposition}\smallskip

There is one circumstance in which we can identify the lattice multi-norm as the  maximum multi-norm.\s

\begin{proposition}\label{2.3m}
Let $(E,\norm)$ be a Banach lattice, and take $n\in\N$. Then
$$
\mu_{1,n}(x_1,\dots,x_n) \leq \lV\, \lv x_1\rv +\cdots + \lv x_n\rv \,\rV \quad (x_1,\dots,x_n\in E)\,.
$$
Further, suppose that $E$ is an $AM$-space. Then
$$ 
\mu_{1,n}(x_1,\dots,x_n) = \lV\, \lv x_1\rv +\cdots + \lv x_n\rv \,\rV \quad (x_1,\dots,x_n\in E)\,,
$$
and the dual $(\mu'_{1,n}:n\in\N)$   is equal to the maximum multi-norm based on $E'$.
\end{proposition}

\begin{proof}  The first part  of the  proposition   follows immediately from equation (\ref{(3.10e)}) (and also from Theorems \ref{3.10b}
 and \ref{2.3b});  see also  \cite[18.4]{J}.

 To show that $(\mu'_{1,n}:n\in\N)$   is equal to the maximum multi-norm based on $E'$, we must show that their respective dual norms are equal on 
  the family $\{(E'')^n:n\in\N\}$.  By Theorem  \ref{3.4d}, the dual of the maximum multi-norm  on $\{(E')^n :  n\in\N\}$ is
  the weak $1$-summing norm on  $\{(E'')^n: n\in\N\}$, and,  by Proposition \ref{3.16}, the latter norm is $\mu''_{1,n}$.     Thus the last clause follows.
 \end{proof}\smallskip

\begin{theorem}\label{2.3n}
Let $(E,\norm)$ be a Banach lattice.\s

{\rm (i)} Suppose that $E$ is an $AL$-space. Then the lattice multi-norm is  the maximum multi-norm  based on $E$.\s
  
  {\rm (ii)}  Suppose that $E$ is an $AM$-space. Then the lattice multi-norm is the minimum multi-norm  based on $E$.
\end{theorem}

\begin{proof} (i) By Theorem \ref{2.3ce}, the dual of the lattice multi-norm based  on $E$ is the dual lattice multi-norm 
 based on $E'$. The dual of the maximum multi-norm  based on $E$ is $({\mu}_{1,n} : n\in\N)$. By Theorem \ref{2.3gb}, $E'$ is an $AM$-space, and so, by
Proposition \ref{2.3m}, the latter two multi-norms are equal on the family $\{(E')^n: n\in\N\}$. Thus the result follows.\s
 
 (ii) Using equation (\ref{(2.19a)}), we see that, 
\begin{eqnarray*}
 \lV(x_1,\dots,x_n)\rV^L_n &=& \lV\,\lv x_1\rv\vee\cdots\vee \lv x_n\rv\,\rV= \max\{\lV\,\lv x_1\rv\,\rV,\dots,\lV\,\lv x_n\rv\,\rV\}\\
& =& \max\{\lV  x_1 \rV,\dots,\lV  x_n \rV\} =  \lV(x_1,\dots,x_n)\rV^{\min}_n 
\end{eqnarray*}
 for each $n\in \N$ and $x_1,\dots,x_n\in E$. Thus  the lattice multi-norm is the minimum multi-norm on $\{E^n: n\in\N\}$.
\end{proof}\s

The following corollary gives a different proof of Theorem \ref{4.1ad}.\s

\begin{corollary}\label{2.3s}
Let $\Omega$ be a measure  space. Then the standard $1$-multi-norm based on $L^1(\Omega)$ is the maximum multi-norm.
\end{corollary}

\begin{proof}  This follows from Example \ref{2.3ca}  and Theorem \ref{2.3n}(i).
\end{proof}\s

\subsection{A representation theorem} The following theorem gives a general {\it representation theorem for multi-normed spaces}. It shows a
 universal property of the lattice multi-norms of this section;
 the result follows from a theorem of  Pisier  stated as \cite[Th{\'e}or{\`e}me 2.1]{MN} and translated into our notation via Theorem \ref{2.38a}.\s
 
\begin{theorem}\label{2.3p}
Let $((E^n, \norm_n) : n\in\N)$ be a  multi-Banach space.  Then there is a Banach lattice $X$ and an isometric embedding 
$J:E\to X$ such that
$$
\lV (Jx_1,\dots, Jx_n)\rV_n^L = \lV (x_1,\dots,x_n)\rV_n\quad(x_1,\dots,x_n\in E)\,.
$$
for each $n\in\N$.\qed
\end{theorem}\s

Thus our multi-normed spaces are the `sous-espace de trellis' of \cite[D{\'e}finition 3.1]{MN}.

As noted in \cite[p.\ 18]{MN}, a lattice multi-norm corresponding to the minimum multi-norm is easily described. 
 Indeed, let $E$ be a Banach space, and set $K= (E'_{[1]}, \sigma(E',E))$, a compact space, so that $C(K)$ is a Banach lattice.  Then the map
 $$J : x\mapsto \iota(x)\mid K\,,\quad E\to C(K)\,,$$ is an isometry, and
 $$
 \lV (Jx_1,\dots,Jx_n)\rV_n^L = \lV (x_1,\dots,x_n)\rV_n^{\min}\quad(x_1,\dots,x_n\in E)\,.
 $$

A description of a  lattice multi-norm corresponding to the maximum multi-norm is also given in \cite[Proposition 3.1]{MN}. Indeed, let
 $E$ be a Banach space, and then set $\Gamma = {\B}(E, \ell^{\,1})_{[1]}$, so that $\ell^{\,\infty}(\Gamma, \ell^{\,1})$ is a Banach lattice.  Then the map
$$J: x\mapsto (Tx : T \in\Gamma)\,,\quad E \to \ell^{\,\infty}(\Gamma, \ell^{\,1})\,,
$$
is an isometry, and 
 $$
 \lV (Jx_1,\dots,Jx_n)\rV_n^L = \lV (x_1,\dots,x_n)\rV_n^{\max}\quad(x_1,\dots,x_n\in E)\,.
 $$
\medskip

\section{Summary}  \noindent We collect here   summary descriptions of the main multi-norms that we have defined, their dual multi-norms, and their rates of growth.\s

\begin{enumerate}

\item  The {\it minimum multi-norm\/} $((E^n,\norm_n^{\min}): n\in\N)$  based on a normed space $E$ is defined by 
$$
\lV (x_1,\dots, x_n)\rV_n^{\min} =\max_{i\in\N_n}\lV x_i\rV\quad (x_1,\dots,x_n  \in E)\,.
$$
The dual multi-norm is the maximum dual multi-norm based on $E'$. The rate of growth of the minimum multi-norm is given by
$\varphi_n^{\min}(E) = 1\,\;(n\in\N)$.\s

\item  The {\it maximum multi-norm\/} based on a normed space $E$ is denoted by  $$((E^n,\norm_n^{\max}): n\in\N)\,.
$$
 The dual multi-norm is $(\mu_{1,n} : n\in \N)$, where 
$\mu_{1,n}$ is the weak $1$-summing norm on $(E')^n$.  The rate of growth of the maximum multi-norm (for $E$ infinite dimensional) satisfies
$$
\sqrt{n}\leq \varphi_n^{\max}(E)= \pi_1^{(n)}(E') \leq n\quad (n\in\N)\,,
$$
and both bounds can be attained.  For example, we have the following. \s
  
 Let $L^p$ be  an infinite-dimensional  measure space. Then: 
$$
\varphi_n^{\,\max}(L^{p}) = n^{1/p} \,\;(n\in\N)\quad {\rm for}\quad  p\in [1,2]\,;
$$ 
$$\varphi_n^{\,\max}(L^{p}) \sim \sqrt{n}\quad {\rm as}  \quad n\to \infty\quad {\rm for}\quad  p\in [2,\infty]\,.
$$

Let $K$ be an infinite compact space. Then:
$$
\sqrt{n} \leq \varphi_n^{\,\max}(C(K)) \leq \sqrt{2n}\quad (n\in\N)\,.
$$
 
\item  Let $E$ be a normed space. For   $1\leq p\leq q< \infty$, the  {\it $(p,q)$-multi-norm\/}  based on   $E$ is denoted by $((E^n,\norm_n^{(p,q)}): n\in\N)$. 
The dual multi-norm based on $E'$ is $((E')^n, \Norm_n^{(p,q')}):n\in\N)$.  The rate of growth of the  $(p,q)$-multi-norm satisfies
$$
 \varphi_n^{(p,q)}(E) =   \pi_{q,p}^{(n)}(E')  \leq n^{1/q}\quad (n\in\N)\,,
$$
and the upper  bound  can be attained.\s

 \item  Fix $p \in [1,\infty)$, and take $q\geq p$. For a measure space $L^p$, the {\it standard $q\,$-multi-norm\/}  based on $L^p$
 is denoted by $(\norm_n^{[q]} :n\in\N)$. We have
 $$
\lV(f_1,\dots,f_n)\rV_n^{[q]} \leq  \lV(f_1,\dots,f_n)\rV_n^{(p,q)}\leq \lV(f_1,\dots,f_n)\rV_n^{\max} 
$$
for  all $f_1.\dots,f_n \in L^{p}$ and $n\in\N$. The rate of growth of the standard $q\,$-multi-norm satisfies
$\varphi_n^{[q]}(L^{p})=n^{1/q}\,\;(n\in\N)$. \s

\item  The {\it Hilbert multi-norm} based on a Hilbert space $H$ is denoted by $(\lV \,\cdot\,\rV^H_n : n\in\N)$.  This multi-norm is equal to the $(2,2)$-multi-norm, 
and is equivalent to the  $(p,p)$-multi-norm for $p\in [1,2]$ and to the maximum multi-norm. The rate of growth of this multi-norm (for infinite-dimensional $H$) 
is given by $\varphi_n^{\min}(H) = \sqrt{n}\,\;(n\in\N)$.\s

\item  The {\it lattice multi-norm} based on a Banach lattice  $E$ is denoted by  $(\norm_n^L: n\in \N)$; it is defined by 
 $$
\lV (x_1,\dots,x_n)\rV_n^{L} = \lV \,\lv x_1\rv\vee \cdots \vee \lv x_1\rv\,\rV\quad (x_1,\dots,x_n \in E,\,n\in\N)\,.
$$
 The dual multi-norm based on $E'$
 is the dual lattice multi-norm  $(\norm_n^{DL}: n\in \N)$; it is defined by  
 $$
\lV (x_1,\dots,x_n)\rV_n^{DL} =\lV \,\lv x_1\rv+ \cdots + \lv x_1\rv\,\rV\quad (x_1,\dots,x_n \in E,\,n\in\N)\,.
$$
  For an $AL$-space, the lattice multi-norm is the maximum multi-norm based on $E$, and, for an $AM$-space,
it is the minimum multi-norm based on $E$. The rate of growth of this multi-norm  is   denoted by  $(\varphi_n^L(E) : n\in\N)$.  
\end{enumerate}
\medskip

\chapter{Multi-topological linear spaces  and multi-norms}

\section{Basic sets}

\subsection{Topological linear spaces}\label{TVS} Let $E$ be a linear space. A {\it local base\/}  of
  $E$ is a family ${\mathfrak B}$ of non-empty,  balanced, absorbing subsets of $E$ such that: \s

(i) for each  $B \in {\mathfrak B}$, there  exists  $C \in {\mathfrak B}$ with $C+C \subset B$; \s

(ii) for each  $B_1,B_2  \in {\mathfrak B}$, there exists $C \in {\mathfrak B}$ with 
$C\subset B_1\cap B_2$;\s

 (iii)  for each $B \in {\mathfrak B}$  and $x \in B$, there  exists  $C \in {\mathfrak B}$ with $x+C \subset B$.  \s

A subset $B$ of a  {\TLS} is {\it bounded\/} if, for each neighbourhood $U$ of $0$ in $E$, there exists
 $\alpha> 0$ with $B \subset \beta U \;\,(\beta > \alpha)$.

Let $E$ be a    {\TLS}. Then $E$ has a local base ${\mathfrak B}$ consisting of all the balanced neighbourhoods of $0$; in this case, 
 each neighbourhood of $0$ contains a member of ${\mathfrak B}$ (and then the open sets of $F$ are precisely the unions of translates
 of members of ${\mathfrak   B}$). Conversely, let ${\mathfrak B}$  be a local base of $E$. Then there is a unique topology
 $\tau$ on $E$ such that $(E,\tau)$ is {\TLS} and ${\mathfrak B}$ is a local base for $\tau$ at $0$. The {\TLS} is Hausdorff 
if and only if $\bigcap\{B: B\in  {\mathfrak B}\}=\{0\}$. 

 For details of these remarks, see \cite{Ru2}, for example.\smallskip

\subsection{Multi-{\TLS}s}\label{multiTVS} Let $E$ be a linear space, and consider the space $E^{\,\N}$, also a linear space;
a generic element of $E^{\,\N}$ is   $x = (x_i) = (x_i : i\in \N)$.
 Define $ \iota : x \mapsto (x), \;\,E \to E^{\,\N}$,  so that $\iota(E)$ is a linear subspace of $E^{\,\N}$.   

For a non-empty subset $S$ of $\N$, we define  $P_S, Q_S$  on $E^{\,\N}$ essentially  as in $\S \ref{Linear operators}$.  We also define $A_\sigma$
  and $M_\alpha$ in ${\mathcal L}(E^{\,\N})$ for $\sigma \in {\mathfrak S}_\N$  and $\alpha = (\alpha_i)\in \overline{\D}^{\N}$ by
$$
A_\sigma((x_i))=(x_{\sigma(i)})\,,\quad M_\alpha ((x_i))=   (\alpha_ix_i) \quad  ((x_i) \in E^{\,\N})\,.
$$
 Finally we define  the {\it amalgamation} $x\amalg y$ of two elements $x= (x_i)$ and $y= (y_i)$  of $E^{\,\N}$ as the element
$$
x\amalg y = (x_1,y_1,x_2,y_2,x_3,y_3,  \dots )
$$
of $E^{\,\N}$. Let $k \in \N$. The amalgamation of $k$ copies of $x \in E^{\,\N}$  is denoted by $x\amalg_k x$,
so that
$$
 x\amalg_k x = (\stackrel{k}{\overbrace{x_1,\dots,x_1}}, \,\stackrel{k}{\overbrace{x_2,\dots,x_2}},\,\dots )\,.
$$

\begin{definition}\label{5.1}
Let $E$ be a linear space, and let $F$ be a linear subspace of $E^{\,\N}$  with $\iota (E) \subset F$.  A subset $B$ of $F$ is {\it basic} if:
\smallskip

{\rm (T1)} $A_\sigma(B) = B$ for each $\sigma \in {\mathfrak S}_\N$\,;\smallskip

{\rm (T2)} $M_\alpha(B) \subset  B$ for  each $\alpha \in \overline{\D}^{\N}$\,;\smallskip

{\rm (T3)} for each $x \in F$,  $x  \in B$  if and only if   $ x\amalg x \in B\/$;\smallskip

{\rm (T4)} for  each $x \in F$,  $x \in B$  if and only if   $  P_{\,\N_n}(x) \in B\;\, (n\in \N)$.\smallskip

\noindent Let  $\tau$ be a topology on $F$.  Then $E$  is a {\it multi-topological linear space} (with respect to 
$(F,\tau)$)   if  $(F, \tau )$ is a Hausdorff {\TLS}   with a local base ${\mathfrak B}$  
consisting of basic sets, each a neighbourhood of $0$.
\end{definition}\smallskip

It may be that $F =E^{\,\N}$ in the above definition, but  we allow greater generality for the sake of future applications.\s

Let $E$ be   a multi-topological linear space with respect to $(F,\tau)$.  For each $x\in E$, we have
$\iota_1(x): = (x,0,0,0, \dots) \in F$, and so $\tau$ induces a topology called $\tau_E$ on $E$ such that a subset $U$ of $E$ belongs to $\tau_E$ if and only if 
$\iota_1(U)$ is relatively $\tau$-open in $F$.  It is clear that  $(E,\tau_E) $ is a {\TLS}.

Let $E$   be  such a  multi-topological linear space, let  $B$ be a basic set in $(F,\tau)$   that is a neighbourhood of
 $0$, and take  $x \in F$. Since the set $B$ is absorbing, there exists $\beta > 0$ such that $x \in \beta B$. It  follows easily from  the definitions that
 $A_\sigma(x) \in F$ for each  $\sigma \in {\mathfrak S}_{\N}$, that $M_\alpha(x) \in F$ for each  $\alpha \in \overline{\D}^{\N}$,
 that $ x\amalg x \in F$, and that $P_{\,\N_n}(x) \in F$ for each $n\in \N$.
\smallskip

\begin{proposition}\label{5.2}
Let $E$ be a multi-topological linear space with respect to $(F,\tau)$, and let  $B$ be a basic set in $F$.\smallskip

{\rm (i)} Take  $x \in F$  and    $k \in \N$.  Then $x\in B$ if and only if $x\amalg_k x \in B$. \smallskip

{\rm (ii)}  Take $x \in F$. Then $x \in B$ if and only if $x\amalg 0 \in B$.\smallskip

{\rm (iii)} Take   $(x_i) \in F$. Then $(x_i) \in B$  if and only if  $(0, x_1,x_2, x_3, \dots) \in B$.\s

{\rm (iv)} Take $x, y \in B$. Then $x\amalg y \in B+B$.\smallskip

{\rm (v)} Take $x =(x_i) \in B$, and let $(k_n)$ be  strictly increasing in $\N$. Then $(x_{k_n}) \in B$. \smallskip

{\rm (vi)}  Take  $x = (x_i) \in B$, and suppose that  $y$ is  a sequence that contains finitely many occurrences of each   $x_i$ in any order. Then $y\in B$.
\end{proposition}

\begin{proof} (i)  Take $n \in \N$ such that $2^j \geq k$. By (T3), $x\in B$ if and only if $x\amalg_{2^j} x \in B$. 
Take $m\geq k$. By (T4), $x\amalg_{m} x \in B$ if and only if  $P_{\,\N_n}(x \amalg _m x ) \in B\,\; (n\in \N)$.  By (T2) and (T1), this 
holds if and only if  $P_{\,\N_n}(x \amalg _k x ) \in B\,\; (n\in \N)$.  By (T4) again, this holds if and only if  $x\amalg_k x \in B$. 
The result follows.\smallskip

(ii)  Suppose that  $x \in B$. Then $x\amalg x \in B$ by (T3), and then  $x\amalg 0 \in B$ by (T2).  Suppose that $x\amalg 0 \in B$. 
 Then it follows from  (T4) that  $P_{\,\N_{2n}}(x\amalg 0) \in B\;\, (n\in \N)$.  By (T1), $P_{\,\N_n}(x) \in B\;\, (n\in \N)$, 
and so $x \in B$ by (T4).\smallskip

 (iii) This is immediate from (T1) and (T4).
\smallskip

(iv)  By (ii), $x\amalg 0, y\amalg 0 \in B$. By (T1), $0\amalg y \in B$. Thus $$x\amalg y = x\amalg 0 + 0\amalg y \in B+B\,.
$$
 
(v)  By (T2), $(0,\dots,0, x_{k_1},0,\dots,0, x_{k_2}, 0, \dots ) \in B$. By (T1), we have $ (x_{k_n})\amalg 0 \in B$. 
 By (ii), $(x_{k_n}) \in B$.\smallskip

(vi)  Suppose that $y$ contains $k_i$ copies of $x_i$ for $i\in \N$.  Take $n \in \N$, and then set  $m = \max \{k_1,\dots,k_n\}$. 
By (i),  $x\amalg_m x \in B$. By (T4), $P_{\,\N_n}(x\amalg_m x) \in B$. By (T2) and (T1), $P_{\,\N_n}(y) \in B$.   But this holds
 for each $n\in\N$, and so $y\in B$ by (T4).\end{proof}
\medskip

\section{Multi-null  sequences}

\subsection{Convergence}\label{Convergence} Let $E$ be a multi-topological linear space.  We can define a notion of convergence in $E$ as follows.\smallskip

\begin{definition}\label{5.3a}
Let $E$ be a multi-topological linear space with respect to $(F,\tau)$ such that $(F,\tau)$ has a local base ${\mathfrak B}$ 
of basic subsets of $F$, and let $(x_i) $ be a sequence in $E$. Then
$$
\Lim_{i\to \infty}x_i = 0\quad {\rm in}\quad E
$$
if, for each $B \in  {\mathfrak B}$, there exists $n_0 \in \N$ such that $(x_n, x_{n+1}, x_{n+2}, \dots ) \in B\;\, (n\geq n_0)$.
Such sequences $(x_i)$ are the {\it multi-null sequences} in $E$.  Further, let $x \in E$. Then
$$
\Lim_{i\to \infty}x_i = x\quad {\rm in}\quad E
$$
if $(x_i-x)$ is a multi-null sequence in $E$;  the sequence $(x_i)$ is {\it multi-convergent to}  $x$.

The collections of multi-convergent and multi-null sequences  in $E$ are denoted by $c_m(E)$ and $c_{m,0}(E)$, respectively. 
\end{definition}\smallskip

Let $E$ be a multi-topological linear space with respect to $(F,\tau)$. Clearly, each multi-null sequence  in $E$ is a null sequence
 in  $(E,\tau_E)$, where $\tau_E$ was described above. Further, let $(x_i)$ be a sequence such that $\lim_{i\to\infty}x_i =0 $ in 
$(E,\tau_E)$. Then there is a subsequence $(x_{k_i})$  of $(x_i)$ such that $\Lim_{i\to \infty}x_{k_i} = 0$.  Let $S$ be a subset of $E$.
One might define   the `multi-closure' of $S$ to be  the set of elements $x$ in $E$ such that there exists a 
multi-null sequence $(x_i)$ contained in $S$ with $\Lim_{i\to \infty}x_i = x$; however, the above remark shows that this multi-closure 
coincides with the closure of $S$ in $(E,\tau_E)$.

The four axioms specified above have an immediate and natural inter\-pretation in terms of this convergence.  Thus: (T1) states that 
each permutation of a multi-null sequence is a multi-null sequence; (T2) states that $M_\alpha (x)$ is a multi-null sequence whenever
 $\alpha = (\alpha_i)$ is a bounded sequence in $\C$ and $x $ is a multi-null sequence;  (T3) states that $x\amalg x$ is a multi-null 
sequence if and only if $x$ is a multi-null sequence. Axiom (T4) is a `Cauchy criterion' for multi-null sequences.  A sequence
$(x_i)  \in E^{\,\N}$ is a {\it multi-Cauchy sequence\/} if, for each for each $B \in  {\mathfrak B}$, 
there exists $n_0 \in \N$ such that
$$
(x_m, x_{m+1},  \dots, x_n, 0,0, \dots  ) \in B\quad  (n\geq m \geq n_0)\,.
$$
By (T4), a sequence is a multi-null sequence if and only if it is a multi-Cauchy sequence.

We shall see shortly  that the notion of a multi-null sequence can depend on the choice of the space $F$.\smallskip

\begin{proposition}\label{5.3b}
Let $E$ be a multi-topological linear space. \smallskip

{\rm (i)}  Each subsequence of   a multi-null sequence in $E$ is itself a multi-null sequence.\smallskip

{\rm (ii)} Let $\alpha, \beta \in \C$, and let $(x_i), (y_i) \in E^{\,\N}$ be such that $$\Lim_{i\to \infty}x_i = x\quad {\rm and}\quad \Lim_{i\to \infty}y_i = y$$ in $E$.
 Then $\Lim_{i\to \infty}(\alpha x_i + \beta y_i) = \alpha x + \beta y$ in $E$.\smallskip

{\rm (iii)}  The collections $c_m(E)$ and $c_{m,0}(E)$ are   linear subspaces of $E^\N$.
\end{proposition}

\begin{proof} These are immediately checked.\end{proof}\smallskip

\subsection{Multi-normed spaces}\label{multi-normed spaces} We now investigate the relation between multi-topo\-logical linear spaces 
and multi-normed spaces.\smallskip

\begin{definition}\label{5.3d} 
Let $((E^n, \norm_n) : n\in \N)$ be  a  multi-normed space, and suppose that $x=(x_i) \in E^{\,\N}$. Then
$$
\Sup x = \Sup (x_i) = \sup\{\lV (x_{k_1}, \dots, x_{k_n})\rV_n : {k_1}, \dots, {k_n}\in \N,\, n\in \N\}\,.
$$\end{definition}\smallskip

In fact, it follows  from (A1), (A4), and Lemma \ref{2.1} that  $(\lV (x_1,x_2, \dots, x_n)\rV_n : n\in\N)$ is an  increasing  sequence
  and that
\begin{equation}\label{(5.1)}
\Sup x =\sup \{ \lV (x_1,x_2, \dots, x_n)\rV_n : n\in \N\} =\lim_{n\to\infty}\lV (x_1,x_2, \dots, x_n)\rV_n\,. 
\end{equation}

Define 
\begin{equation}\label{(5.2)}
F = \{x\in E^{\,\N} : \Sup x < \infty\}\,.
\end{equation}
For  each $\varepsilon > 0$, set
$$
B_\varepsilon = \{ x \in F : \Sup x < \varepsilon\}\,,
$$
and set ${\mathfrak B} = \{B_\varepsilon :   \varepsilon > 0\}$.   \smallskip

\begin{theorem}\label{5.3}
Let $((E^n, \norm_n) : n\in \N)$ be  a  multi-normed space, and let $F$ and  ${\mathfrak B}$ be as above. Then $F$ is  a linear subspace 
of $E^{\,\N}$  with $\iota (E) \subset F$, and $E$   is a multi-topological linear space with respect to $(F,\tau)$, where
  $(F,\tau)$ has ${\mathfrak B}$ as a local base. Further, each set $B_\varepsilon$ is convex and bounded.
\end{theorem}

\begin{proof}  It is clear that $F$ is a linear subspace of $E^{\,\N}$; by Lemma \ref{2.0},  $\iota(E)\subset F$.

We shall show that ${\mathfrak B}$ is a local base at $0$ in $F$.  Given $\varepsilon > 0$, we have
 $B_{\varepsilon/2} +  B_{\varepsilon/2}\subset  B_\varepsilon$. Given   ${\varepsilon_1},{\varepsilon_2} >0$, we have
 $B_{\varepsilon}\subset
B_{\varepsilon_1}\cap B_{\varepsilon_2}$ for $\varepsilon = \min\{  \varepsilon_1, \varepsilon_2\}$.  Given $\varepsilon > 0$
 and $x \in B_\varepsilon$, we have $\Sup x < \varepsilon$, and then $x + B_\eta \subset B_\varepsilon$ for $\eta = \varepsilon -  \Sup x $. 
Thus ${\mathfrak B}$ is a local base at $0$ in $F$, and so ${\mathfrak B}$  defines a topology $\tau$ such that $(F,\tau)$ is a {\TLS}.
 Since $\bigcap\{ B_\varepsilon : \varepsilon > 0\} =\{0\}$, the topology $\tau$ is Hausdorff.

It is clear that Axioms (T1), (T2), and (T4) are satisfied.  Suppose that $x \in B_\varepsilon$, where $\varepsilon > 0$, and take
 ${k_1}, \dots, {k_n}\in \N$. Then
$$
\lV ((x\amalg x)_{k_1}, \dots, (x\amalg x)_{k_n})\rV_n = \lV (x_{j_1}, \dots, x_{j_m})\rV_m
$$
for some $m \in \N_n$ and $j_1,\dots, j_m \in \N$ by (A1) and (A4), and so we have $x\amalg x \in B_\varepsilon$; the converse is 
immediate, and so (T3) is satisfied.  Thus each $B_\varepsilon$ is a basic set in $F$.

Clearly each set $B_\varepsilon$ is convex and bounded.
\end{proof}\smallskip

\begin{definition}\label{5.4}
 The topology $\tau$ defined on $F$ in the above theorem is that {\it specified by} the multi-normed space  $((E^n, \norm_n) : n\in \N)$.
\end{definition}\smallskip

In the future, we shall regard $(F, \tau)$ as the space specified by a multi-normed space  $((E^n, \norm_n) : n\in \N)$  without explicit mention.
We now interpret the concept of `multi-null sequence' in the above situation.\smallskip

\begin{theorem}\label{5.4a}
Let $((E^n, \norm_n) : n\in \N)$ be  a  multi-normed space. Take $(x_i) \in E^{\,\N}$.  Then $(x_i) $ is a multi-null 
sequence in $E$ if and only if, for each $\varepsilon > 0$, there exists $n_0 \in \N$ such that
$$
\sup_{k\in \N}\lV (x_{n+1}, \dots, x_{n+k})\rV_k < \varepsilon\quad (n\geq n_0)\,.
$$
 \end{theorem}

\begin{proof}  This is again immediate.\end{proof}\smallskip

Let $(x_i)$ be a sequence in $E$ with $\Lim_{i\to\infty} x_i =x$. It follows that 
\begin{equation}\label{(5.2a)}
\lim_{n\to \infty}\sup_{k\in \N}\lV (x_{n+1}, \dots, x_{n+k})\rV_k = \lV x \rV\,.
\end{equation}

\begin{example}\label{5.4c}
{\rm Let $(\alpha_i)$ be a fixed element of $\C^{\N}$, and set
$$
x_i =\alpha_i\delta_i\quad (i\in\N)\,.
$$
 
(i) Let $E$  be one of the Banach spaces ${\ell}^{\,p}$ (for $p\geq 1$) or $c_{\,0} $, and  take  $(\norm_n^{\min} :n\in\N)$ to  be the 
minimum multi-norm on $\{E^n : n\in\N\}$. Then it follows immediately  that $(x_i)$ is a multi-null sequence in $E$ if and only if
$\lim_{i\to \infty}\alpha_i =0$, i.e., if and only if $(\alpha_i) \in c_{\,0}$.  This is independent of the choice of the space $E$.\smallskip

(ii) Let $E = {\ell}^{\,p}$ (where  $p\geq 1$), and let  $(\norm_n^{[p]} :n\in\N)$ be the standard $p\,$-multi-norm based on $\{E^n : n\in\N\}$. 
Then it follows from equation (\ref{(4.2)}) that $(x_i)$ is a multi-null sequence in $E$ if and only if
$$
\lim_{n\to \infty}\left(\sum_{i=n}^\infty\lv\alpha_i\rv^{\,p}\right)^{1/p} =0\,,
$$
i.e., if and only if $(\alpha_i) \in {\ell}^{\,p}$. \smallskip

(iii)  We now see, by comparing Examples (i) and (ii), that the   multi-null sequences in a multi-normed space  based on a Banach space $E$ depend
  on the multi-norm  that we are considering.
\qed}\end{example}\smallskip

\begin{proposition}\label{5.4b}
Let $((E^n, \norm_n) : n\in \N)$ be  a  multi-normed space. Then the following are equivalent:\smallskip

{\rm (a)} each null sequence in $(E, \norm)$ is a multi-null sequence;  \smallskip

{\rm (b)} the multi-norm $(\norm_n : n\in\N)$ is equivalent to the minimum multi-norm;\smallskip

{\rm (c)}  there is a topology  $\sigma$ on $E$ such that the  multi-null sequences are precisely the  convergent sequences in $(E,\sigma)$.
\end{proposition}

\begin{proof}  Here $(\varphi_n(E) : n\in\N)$ is the rate of growth sequence for the multi-normed space $((E^n, \norm_n) : n\in \N)$.\s

(a) $\Rightarrow$ (b) Assume towards a contradiction that $\limsup_{n\to \infty} \varphi_n(E) =\infty$.
 Then, for each $n\in\N$, there exists $m_n \in \N$ such that $\varphi_{m_n}(E) > n$, and so there exist 
$x_{1,n},\dots ,x_{m_n,n} \in E_{[1/n]}$ with $\lV (x_{1,n},\dots ,x_{m_n,n})\rV_{m_n} \geq 1$.  The sequence
$$
(x_{1,1},\dots ,x_{m_1,1},x_{1,2},\dots ,x_{m_2,2}, \dots,  x_{1,n},\dots ,x_{m_n,n}, \dots )
$$
is a null sequence in $(E, \norm)$, but it is not  a multi-null sequence. This is a contradiction of (a). 
Thus $(\varphi_n(E) : n\in\N)$ is bounded, and so, by Proposition \ref{3.2ad}, $(\norm_n : n\in\N)$ is equivalent to the minimum multi-norm. \smallskip

(b) $\Rightarrow$ (a)  Suppose that $\sup\{\varphi_n(E) : n\in\N\}\leq C $. Then
$$
\lV (x_{n+1},\dots, x_{n+k})\rV_k \leq C \max\{\lV x_{n+1}\rV, \dots,  \lV x_{n+k}\rV\}\quad(n,k\in \N)\,,
$$
 and so each null sequence in $(E, \norm)$ is a multi-null sequence.\smallskip

(a) $\Rightarrow$ (c)  This is trivial.\smallskip

(c) $\Rightarrow$ (a)  Assume towards a contradiction that (a) fails. Then there is a null sequence  $(x_i)$  
in $(E, \norm)$ such that $(x_i)$   is not  a multi-null sequence.  By (c),  $(x_i)$   is not convergent in $(E,\sigma)$, 
and so there is a $\sigma$-neighbourhood $U$ of $0$ in $E$ and a subsequence $(x_{i_j})$  of $(x_i)$     such that
 $x_{i_j}\not\in U\,\; (j\in \N)$.  There is a subsequence $(y_n)$ of $(x_{i_j})$  such that $\lV y_n\rV \leq 1/n^2\,\;(n\in\N)$, 
and then $(y_n)$ is a multi-null sequence in $E$. However $y_n\not\in U\,\;(n\in \N)$, and so $(y_n)$ is not convergent in $(E,\sigma)$, 
a contradiction of (c).
\end{proof}\smallskip

Let $((E^n, \norm_n) : n\in \N)$ be  a  multi-normed space such that $E$ is finite-dimensional. Then  it follows from 
Proposition \ref{3.2b} that the equivalent conditions of the above prop\-osition are satisfied.

Let $K$ be a compact  space. Then the multi-null sequences in $C(K)$ for the lattice multi-norm based on $C(K)$ are just the usual null sequences.

Recall that a {\TLS} $E$ is a {\it locally convex space\/} if and only if there is a local base
 consisting of convex sets. By \cite[Theorem 1.14(b)]{Ru2}, each neighbourhood of zero in such a space contains a balanced, convex
 neighbourhood of $0$. The following result shows that the topology  of a locally convex space is determined by a class of 
 multi-null sequences.\smallskip

\begin{proposition}\label{5.5}
Let $E$ be a locally convex space.  \smallskip

{\rm (i)}  Let $V$ be a convex, balanced  neighbourhood of $0$ in $E$. Then $V^{\N}$ is a basic subset of $E^{\,\N}$.\smallskip

{\rm (ii)} Let   ${\mathfrak B}$ be  the family of sets in $E^{\,\N}$ of the form $V^{\N}$, where $V$ is a convex, balanced  neighbourhood
 of $\,0$ in $E$.  Then there is a topology  $\tau$ on $E$ such that $E^\N$ is   
a multi-topological linear space with respect to $(E,\tau)$, and  $(E,\tau)$ has ${\mathfrak B}$ as a local base.
\end{proposition}

\begin{proof} (i)  This is immediate.\smallskip

(ii)  It is clear that the specified family ${\mathfrak B}$ is a local base  at $0$ for $E$ consisting of basic sets. There is a unique 
topology $\tau$ on $E$ such that $(E,\tau)$ is {\TLS} and ${\mathfrak B}$ is a local base for $\tau$ at $0$.  The topology $\tau$ is Hausdorff
because  $\bigcap\{B: B\in  {\mathfrak B}\}=\{0\}$. Thus $E^\N$ is a multi-topological linear space with respect to $(E,\tau)$.
\end{proof}\smallskip

We now seek a version for multi-topological linear spaces of {\it Kolmogorov's  theorem\/} for topological 
linear spaces: this states that a {\TLS} $E$  is normable if and only if $0$ has a convex, bounded neighbourhood \cite[Theorem 1.39]{Ru2}.\smallskip

\begin{theorem}\label{5.6}
Let $E$ be   a multi-topological linear space with respect to $(F,\tau)$. Then the topology $\tau$ is specified by a multi-normed space
 if and only if there is a basic set which is a convex, bounded  neighbourhood of $\,0$ in $F$.
\end{theorem}

\begin{proof} Suppose that $\tau$ is specified by a multi-normed space.  Then each set
 $B_\varepsilon$ given above is  a basic subset of $F$ which is a convex, bounded  neighbourhood of $\/0$.

Conversely, suppose that $B$ is a basic subset of $F$ which is a convex, bounded  neighbourhood of $\,0$ in $F$. 
 By \cite[Theorem 1.14(b)]{Ru2}, we may suppose that $B$ is balanced.

Let $n \in \N$, and take $x_1,\dots, x_n \in E$, so that  $(x_1,\dots, x_n, 0, \dots) \in F$.  We define
$$
\lV (x_1, \dots, x_n)\rV_n = p_B((x_1,\dots, x_n, 0, \dots))\quad (x_1,\dots, x_n \in E)\,,
$$
where $p_B$ is the Minkowski functional of $B$. Clearly $\norm_n$ is a seminorm on $E^n$.  Suppose that $(x_1,\dots, x_n, 0, \dots) \neq 0$ in $F$.  
Since $(F, \tau )$  is a  Hausdorff space, there is a neighbourhood $V$ of $0$ in $F$ such that $(x_1,\dots, x_n, 0, \dots) \not \in V$.  Since $B$ is bounded, 
there exists  $\alpha > 0$ such that $B \subset \beta V \,\;(\beta > \alpha)$.  Since $ x \not \in (1/\beta)B\,\;(\beta > \alpha)$, 
 we have 
$$
p_B((x_1,\dots, x_n, 0, \dots))> 1/\alpha > 0\,.
$$ 
Thus $\norm_n$ is a norm on $E^n$.

Set $\lV x \rV = \lV x \rV_1\,\;(x \in E)$. Then $(E, \norm )$ is a normed space.

We shall now  show that  $((E^n, \norm_n) : n\in \N)$ is a  multi-normed space. 

It is immediate that Axioms (A1), (A2), and (A3) are satisfied.  Let $x_1,\dots, x_n \in E$.   
 By Proposition \ref{5.2}(vi), $(x_1,\dots, x_n, 0, \dots)\in B$ if and only if $(x_1,\dots, x_n,x_n, 0, \dots)\in B$, and so Axiom (A4) is satisfied.

Consider the family ${\mathfrak B} = \{ \alpha B : \alpha >0\}$. By \cite[Theorem 1.15(c)]{Ru2},  ${\mathfrak B} $ is a local base
 for the {\TLS} $(F, \tau )$.  Let $\sigma$ be the topology on $F$ defined by the multi-norms $(\norm_n: n\in\N) $ as in Theorem \ref{5.3}, 
and take  $x \in B$.  Then, by (T4), $P_{\,\N_n}(x) \in B \,\;(n\in \N)$, and so, by equation (\ref{(5.1)}), $\Sup x \leq 1$, whence 
 $\tau \subset \sigma$.  Let $x \in F$ with $\Sup x < 1$. Then $x \in B$, and so $\sigma  \subset \tau$.  Thus $\tau =\sigma$. 
 It also follows that $F = \bigcup\{\alpha B : \alpha > 0\}$,  and so, by (T4),  $F$ is exactly the space specified in equation 
(\ref{(5.2)}) in terms of the multi-norms.

This completes the proof.  \end{proof}\smallskip

 Let $((E^n, \norm_n) : n\in \N)$ be  a  multi-normed space. We have seen in Proposition \ref{5.4b} that multi-null sequences 
in $E$ are the null sequences for a topology  on $E$ only in special cases. We generalize this remark.\smallskip

\begin{proposition}\label{5.7}
Let $E$ be a multi-topological linear space with respect to $(F,\tau)$, and suppose that $\tau$ has a countable base of 
neighbourhoods of $0$ in $F$.  Then {\rm either} the  multi-null sequences  in $E$ are exactly the  null sequence in 
 $(E,\tau_E) $, {\rm or} there is no topology  $\sigma$ on $E$ such that the  multi-null sequences  in $E$ are exactly the 
 null sequences in  $(E,\sigma) $.
\end{proposition}

\begin{proof}   We first note the following.  Let $(U_n)$ be a countable base at $0$ for the topology $\tau$ on $F$. 
Then there is a countable base $(V_n)$  at $0$ for the topology $\tau$ on $F$  such  that
$$
V_n \supset   V_{n+1} +  V_{n+2}+ \cdots + V_{n+k}\quad(n,k\in\N)\,.
$$
Indeed set $V_1 =U_1$, and inductively choose $V_n$ to be a neighbourhood of $0$ in $(F, \tau)$ such that
 $V_{n+1} \subset (U_{n+1}\cap V_n)$ and $V_{n+1}+ V_{n+1} \subset V_n$ for each $n \in \N$.

We next note that each null sequence $(x_i)$ in $(E,\tau_E)$ has a subsequence $(x_{i_k})$ which is multi-null.  
Indeed we choose the sequence $(i_k : k\in\N)$ inductively so that, for each $k\in \N$, we have  $i_{k+1} > i_k$ and $(x_{i_k},0,0,\dots ) \in V_k$.  
That $(x_{i_k})$   is a multi-null sequence follows from Axiom (T4).

The result now follows essentially as before.
\end{proof}\s

\subsection{Multi-null sequences and order-convergence}\label{multi-null sequences}
 Let $E$ be a Banach lattice, as in $\S\ref{Banach lattices}$, and let $(x_n)$ be a sequence in $E$.  Recall that $(x_n)$  is order-null
 if and only if there is a sequence $(u_n)$ in  $E^+$ such that $u_n\downarrow 0$  and $\lv x_n\rv \leq u_n \,\;(n\in \N)$.  The lattice multi-norm  
 on $\{E^n : n\in\N\}$  was defined for each $n \in\N$ in Definition \ref{2.3a} by the formula
$$
\lV (x_1, \dots, x_n)\rV_n^L = \lV\, \lv x_1\rv \vee \cdots \vee \lv x_n\rv \,\rV\quad (x_1, \dots, x_n \in E)\,.
$$
We shall consider multi-null sequences with  respect to this multi-norm.\smallskip

\begin{theorem} \label{5.8}
Let $E$ be a Banach lattice.  Then each multi-null sequence in $E$ is order-null in $E$.
\end{theorem}

\begin{proof}  Let $(x_n)$ be a multi-null sequence in $E$.  Then, for  each $k \in \N$, there exists $n_k \in\N$ such that
$$
\lV\, \lv x_{n_k}\rv \vee \lv x_{n_k+1}\rv \vee \cdots \vee \lv x_{n}\rv \,\rV < 2^{-k} \quad (n\geq n_k)\,;
$$
we may suppose that the sequence $(n_k: k\in\N )$ is strictly increasing. Set 
$$
I_k = \{n_k, \dots, n_{k+1}-1\}\subset \N\quad (k\in \N)\,,
$$
 and, for $k \in \N$, define
$$
y_k = \lv x_{n_k}\rv \vee \lv x_{n_k+1}\rv \vee \cdots \vee \lv x_{n_{k+1}-1}\rv\,,
$$
so that $\lV y_k \rV \leq 2^{-k}$ and the series $\sum_{k=n}^\infty y_k$  is convergent in $E$ for each $n \in \N$.  Set
$$
u_n =   \sum_{j=k}^\infty y_j\quad \mbox{\rm for each}\;\, n \in I_k\,.
$$
For $n\in I_k$, we have $\lv x_n \rv \leq y_k \leq u_n$.  Also,  $0\leq u_{n+1} \leq u_n\,\;(n\in\N)$. Suppose that $u\in E$ with
 $0\leq u\leq u_n\,\;(n\in\N)$. Then $0\leq \lV u\rV \leq \lV u_n\rV \leq 2^{-k+1}\;\,(n \geq n_k)$,
and so $u=0$ and $u_n\downarrow 0$. This implies  that $(x_n)$ is an order-null sequence in $E$.
\end{proof}\smallskip

We wish to determine when the converse of the above theorem holds.\smallskip

\begin{theorem} \label{5.9}
 Let $(E, \norm)$ be a Banach lattice.  Then each order-null sequence in $E$ is multi-null in $E$ if and only if the norm is $\sigma$-order-continuous.
\end{theorem}

\begin{proof}  Suppose that each order-null sequence  is multi-null, and let $(x_n)$ be a sequence in $E$ with $x_n \downarrow 0$.
 Then $(x_n)$ is order-null, and hence multi-null.  Certainly this implies that $\lV x_n \rV \downarrow 0$, and so the norm is
 $\sigma$-order-continuous.

Conversely, suppose that the norm is $\sigma$-order-continuous, and let $(x_n)$ be an order-null sequence.  Then there exists a 
sequence $(u_n)$ in $E^+$ with $\lv x_n \rv \leq u_n\,\;(n\in \N)$ and $u_n \downarrow 0$. By hypothesis, we have $\lV u_n \rV \downarrow 0$, and now
$$
\lV\, \lv x_n\rv \vee \cdots \vee \lv x_{n+k}\rv \,\rV_k\leq \lV\,  u_n \vee \cdots \vee u_{n+k}\,\rV  = \lV u_n \rV\quad (n,k\in \N)\,,
$$
so that $\lim_{n\to\infty}\sup_{k\in \N}\lV\, \lv x_n\rv \vee \cdots \vee \lv x_{n+k}\rv \,\rV_k =0$.  Hence $(x_n)$ is multi-null.
\end{proof}\smallskip

For example,  multi-null and order-null   sequences coincide in each Banach lattice  $L^{p}(\Omega)$ for $p\geq 1$ (when this space 
has the lattice multi-norm, which, by Example \ref{2.3ca}, is equal to the standard $p\,$-multi-norm) based on $L^p(\Omega$)
and on the space $C([0,\omega_1])$ (when this space has the minimum multi-norm).

\chapter{Multi-bounded sets and multi-bounded  operators}

\noindent The theory of Banach spaces gains great strength from the facts that, for each Banach spaces $E$ and $F$,
  a linear  operator from $E$ to $F$ is continuous if and only if it is bounded, and that the collection of all bounded 
linear operators from $E$ to $F$ is itself a Banach space.  Our aim in this  chapter is to establish analogous results for multi-normed spaces.
\medskip

\section{Definitions and basic properties}

\noindent We first define multi-bounded sets in multi-{\TLS}s 
(which were defined in Definition \ref{5.1}). \smallskip
 
\subsection{Multi-bounded sets}\label{Multi-bounded sets}

 \begin{definition}\label{6.1}
Let $E$ be    a multi-topological linear space with respect to $(F,\tau)$.  A subset $B$ of $E$ is  {\it multi-bounded} if $B^{\,\N}$ is a bounded set in the {\TLS} 
$(F, \tau)$.
\end{definition}\smallskip

We denote the family of multi-bounded sets in $E$ by ${\mathcal {MB}}(E)$,  suppressing in the notation the role of $F$.

Let $B,C\in {\mathcal {MB}}(E)$ and $\alpha, \beta \in \C$.  Then it is immediate from the defin\-ition that 
$B\cup C, \alpha B + \beta C \in {\mathcal {MB}}(E)$; each compact set is multi-bounded; the absolutely convex hull of a 
multi-bounded set is multi-bounded.\smallskip

\begin{proposition}\label{6.4}
Let $((E^n, \norm_n) : n\in \N)$ be  a  multi-normed space, and let $B$ be a subset of $E$.  Then $B$ is multi-bounded  in $E$ if and only if
$$
\sup\{ \lV (x_1,\dots,x_n)\rV_n : x_1,\dots, x_n \in B,\, n\in \N\}
<\infty\,.
$$
\end{proposition}\smallskip

\begin{proof} This is immediate from our earlier results.\end{proof}\smallskip

\begin{corollary}\label{6.4d}
Let $E$ be a normed space, and consider two multi-norms based on $E$ such that the multi-norms  are equivalent. Then the families of multi-bounded sets with
 respect to the two multi-norms are equal.\qed
\end{corollary}\s

\begin{definition}\label{6.4a}
 Let $((E^n, \norm_n) : n\in \N)$ be  a  multi-normed space, and let $B\in {\mathcal {MB}}(E)$.  Then 
 $$
 c_B = \sup\{ \lV (x_1,\dots,x_n)\rV_n : x_1,\dots, x_n \in B,\, n\in \N\}\,;
 $$
 $c_B$ is the {\it multi-bound} of a multi-bounded  set $B$.
\end{definition}\smallskip

 \begin{proposition}\label{6.4aa}
  Let $((E^n, \norm_n) : n\in \N)$ be  a  multi-normed space.\s

{\rm (i)}  A finite subset  $B= \{ x_1,\cdots,x_k\}$ in $E$  is  multi-bounded , with  $c_B= \lV (x_1,\dots,x_k)\rV_k$.\s

{\rm (ii)} Suppose that $B\subset E$ is multi-bounded. Then   $C: = {\rm aco\/}(B)$ is  multi-bounded, with $c_{\/C} =c_B$.
\end{proposition}

\begin{proof} (i) This is immediate from Lemma \ref{2.1b}. \s

(ii)  Take $y_1,\dots,y_m \in C$. Then clearly there exist  $n\in \N$,   $a =(\alpha_{ij}) \in {\M}_{m,n}$, and $x = (x_1,\dots,x_n) \in B$ such that
$$
\sum_{j=1}^n \lv \alpha_{ij}\rv \leq 1\quad{\rm and}\quad  y_i= \sum_{j=1}^n \alpha_{ij}x_j
$$
for $i\in \N_m$. By (\ref{(2.3)}), $\lV a : \ell^{\,\infty}_n\to \ell^{\,\infty}_m \rV \leq 1$, and so, by Theorem \ref{2.5b}, (a)$\Rightarrow$(c),
we have $\lV (y_1,\dots,y_m )\rV_m = \lV a\,\cdot x\rV_m \leq \lV x\rV_n \leq c_B$, and so $c_{\/C} \leq c_B$.  Thus $c_{\/C} =c_B$.
\end{proof}\s

Let $((E^n, \norm_n) : n\in \N)$ be  a  multi-normed space, and let $(x_n)$ be a sequence in $E$. Then we see that the set
 $\{x_n:n\in\N\}$ is multi-bounded  if and only if
$$
\sup_{n\in\N} \lV (x_1,\dots, x_n)\rV_n = \lim_{n\to \infty}\lV (x_1,\dots, x_n)\rV_n <\infty\,;
$$
in this case, $(x_n)$ is a {\it multi-bounded  sequence}. It follows from (\ref{(5.2a)}) that  each multi-convergent sequence in $E$ is multi-bounded.\s
  
\subsection {Multi-bounded sets for lattice multi-norms}\label{Multi-bounded sets for lattice multi-norms}

Let $E$ be a Banach lattice. The lattice multi-norm  $(\norm_n^L: n\in \N)$ based on $E$  was defined in Definition \ref{2.3c}.\s

\begin{proposition}\label{6.11a} 
Let $E$ be a Banach lattice.  Then   each  order-bounded subset  of $E$ is multi-bounded  with respect to the lattice multi-norm.\s 
\end{proposition}

\begin{proof}   Suppose  that $B$ is order-bounded in $E$, so that  there exists  $y\in E^+$ such that  $\lv x \rv \leq y\,\;(x \in B)$.  Let
$n\in \N $,  and choose $x_1,\dots,x_n\in B$; define  $$x = \lv x_1 \rv \vee \cdots \vee \lv x_n\rv \,,
$$
 so that $x \leq y$. Then $ \lV (x_1,\ldots,x_n )\rV_n^L = \lV x \rV \leq \lV y \rV $.
Thus we see that $B \in {\mathcal {MB}}(E)$ (with $c_B \leq  \lV y \rV$).\end{proof}\smallskip

\begin{proposition}\label{6.11b} 
Let $E$  be a Banach lattice. For each pairwise-disjoint, multi-bounded sequence $(x_i)$ in $E$ and each null sequence $ (\alpha_i)$, the series 
$\sum_{i=1}^\infty \alpha_i x_i$   converges in $E$.\end{proposition}

\begin{proof} Set $c = \sup\{\lV \,\lv  x_1\rv \vee  \cdots\vee  \lv x_n\rv \,\rV : n\in\N\}$.   For each $\varepsilon >0$, take $i_0 \in\N $ such that
 $\lv \alpha_i\rv< \varepsilon\,\;(i\geq i_0)$. Now take $m, n\in\N$ with $i_0\leq m<n$. Then, using equation (\ref{(2.4ab)}), we have
$$
\lV \sum_{i=m}^n \alpha_i x_i\rV = \lV\, \lv \alpha_m\rv\lv x_m\rv\vee  \cdots\vee \lv \alpha_n\rv\lv x_n\rv\,\rV \leq \varepsilon c\,,
$$
and so  $$\left(\sum_{i=1}^n \alpha_i x_i: n\in\N\right)$$  is Cauchy, and hence convergent, in $E$.
\end{proof}\s

A `monotonically bounded  Banach lattice'   was defined in Definition \ref{2.3j}(i).\smallskip

\begin{theorem}\label{6.11} 
Let $E$ be a  monotonically bounded Banach lattice. Then a subset of $E$ is order-bounded if and only if it is multi-bounded. \end{theorem}

\begin{proof} It follows from Proposition \ref{6.11a}  that we  must show  just that a multi-bounded set in $E$ is order-bounded.

Let $B$ be a multi-bounded subset of $E$,  and let ${\mathcal F} ={\mathcal P}_f(B)$, the family of finite subsets of $B$, so that 
 ${\mathcal F} $ is a directed set when ordered by inclusion.  For each $F \in  {\mathcal F}$, set
$$
y_F = \max \{ \lv x \rv :x \in F\}\,.
$$
Then $\{ y_F :F \in  {\mathcal F}\}$ is an increasing net in $E_{\R}$. Since $B$ is multi-bounded, the net $\{ y_F :F \in  {\mathcal F}\}$
 is bounded in $(E, \norm)$, and so, since $E$ is monotonically bounded,  there exists $y \in E$  with  $y_F \leq y\,\;(F \in  {\mathcal F})$.  
Thus $y$ is an upper bound for $B$, and so $B$ is order-bounded.
\end{proof}\smallskip

In particular, take  $E = C(K)$, where $K$ is a compact space, and let $\{E^n : n\in \N\}$ have the minimum multi-norm,  which is the lattice multi-norm
 from the Banach lattice $E$. Then the multi-bounded sets and  the order-bounded sets coincide, and these are just the $\norm$-bounded subsets of $E$. On the other
 hand, let $B = \{ e_n :n \in \N\}\subset c_{\,0}$. Then $B$ is  multi-bounded, but not order-bounded, in $c_{\,0}$. 

Now let $E= L^{p}(\Omega)$, where $\Omega $ is a measure space and $p\geq 1$, and let the family $\{E^n : n\in \N\}$ 
have the standard $p\,$-multi-norm, which, as we noted in Example \ref{2.3ca}, is the  lattice multi-norm from the Banach
 lattice $E$. Then the multi-bounded sets and  the order-bounded sets coincide.

Further, let $K$ be  a compact space.   Then again the multi-bounded sets for the standard $1$-multi-norm based on $M(K)$ and 
 the order-bounded sets of  $M(K)$ coincide; this follows from  Theorem \ref{4.1d}.\s
 
\subsection{Multi-bounded operators} The above notion of a multi-bounded set leads immediately to the definition of a multi-bounded operator.\s

 \begin{definition}\label{6.2}
Let $E$ and $F$ be  multi-topological linear spaces, and let $T \in {\mathcal L}(E,F)$. Then $T$ is a  {\it multi-bounded operator} if
$$
T(B)\in {\mathcal {MB}}(F)\quad (B \in {\mathcal {MB}}(E))\,.
$$
 The collection of multi-bounded linear maps  from $E$ to $F$ is denoted  by ${\mathcal M}(E,F)$. We write ${\mathcal M}(E)$ for ${\mathcal M}(E,E)$ in the case 
where $E$ and $F$ are equal as multi-topological linear spaces.
\end{definition}\smallskip

\begin{proposition}\label{6.3}
Let $E$, $F$, and $G$  be  multi-topological linear spaces.   Then:\s

{\rm (i)} ${\mathcal M }(E,F)$ is a linear subspace of ${\mathcal L}(E,F)\,$;\s

{\rm (ii)}  $T\,\circ\,S \in {\mathcal M }(E,G)$  whenever $S\in {\mathcal M }(E,F)$ and  $T\in {\mathcal M }(F,G)$.
\end{proposition}

\begin{proof} This is immediate from a  remark above.
\end{proof}\smallskip

\begin{proposition}\label{6.3a}
 Let
$((E^n, \norm_n) : n\in \N)$  and $((F^n, \norm_n) : n\in \N)$
 be   two   multi-normed spaces, and let $T\in {\mathcal M}(E, F)$. Then
$$
\sup \{c_{T(B)} : B \in {\mathcal {MB}}(E) \;\,{\rm with}\,\; c_B \leq 1\}  < \infty\,.
$$
\end{proposition}

\begin{proof}  Assume towards a contradiction  that the specified supremum is infinite.  Then, for each $n\in\N$, there exists
  $B_n \in {\mathcal {MB}(E)}$  such that $c_{B_n}\leq 1/n^2$, but $c_{T(B_n)}>  n$, and there exist $x_{1,n}, \dots, x_{k_n,n} \in B_{n}$ such that
$
\lV (x_{1,n}, \dots, x_{k_n,n})\rV_{k_n} < {1}/{n^2}$ and 
\begin{equation}\label{(6.10a)}
  \lV (Tx_{1,n},\dots, Tx_{k_n,n})\rV_{k_n} > n\,.
\end{equation}
Consider the subset
$$
B: =
\{x_{1,1},\dots, x_{k_1,1}, x_{1,2},\dots, x_{k_2,2}, \dots, x_{1,n},\dots, x_{k_n,n},\dots\}
$$
of $E$. Set $K_n = \sum_{i=1}^nk_i$ for $n\in\N$.  For each $y_1, \dots, y_m\in B$, there exists $n\in\N$ such that
$$
\{y_1, \dots, y_m\} \subset \{x_{1,1},\dots, x_{k_1,1}, x_{1,2},\dots, x_{k_2,2}, \dots, x_{1,n},\dots, x_{k_n,n}\}\,,
$$
and so, by  Lemmas \ref{2.1b} and \ref{2.1a},
\begin{eqnarray*}
\lV (y_1, \dots, y_m)\rV_m 
&\leq &
 \lV (x_{1,1},\dots, x_{k_1,1}, x_{1,2},\dots, x_{k_2,2}, \dots, x_{1,n},\dots, x_{k_n,n})\rV_{K_n}\\
&\leq &
\sum_{j=1}^n\lV (x_{1,j},\dots, x_{k_j,j})\rV_{k_j} \leq \sum_{j=1}^n\frac{1}{j^2}\,.
\end{eqnarray*}
This shows that $B \in {\mathcal {MB}(E)}$. Thus there exists $M > 0$ such that 
$$
\lV (Ty_1, \dots, Ty_m)\rV_m\leq M\quad (y_1, \dots, y_m\in B,\, m\in \N)\,.
$$  But this contradicts equation (\ref{(6.10a)}).

Thus the result holds.
\end{proof}\smallskip

The above proposition shows that the following definition of $\lV T \rV_{mb}$ always gives a number in $\R^+$.\smallskip

 \begin{definition}\label{6.4b}
 Let $((E^n, \norm_n) : n\in \N)$ and $((F^n, \norm_n) : n\in \N)$  be two multi-normed spaces, and let $T \in {\mathcal M }(E,F)$.  Then 
$$
\lV T \rV_{mb} = \sup \{c_{T(B)} : B \in {\mathcal {MB}}(E) \;\,{\rm with}\,\; c_B \leq 1\}  \,.
$$
The map $T$ is a {\it multi-contraction}  if $\lV T \rV_{mb} \leq 1$, and $T$ is a  {\it multi-isometry} if $T$ is an isometry onto
 a closed subspace $T(E)$ of $F$ and  if $T \in {\mathcal M }(E,T(E))$  and $T^{-1} \in {\mathcal M }(T(E),E)$ are both multi-contractions.
\end{definition}\smallskip

Let $((E^n, \norm_n) : n\in \N)$ and \mbox{$((F^n, \norm_n) : n\in \N)$} be two   multi-normed spaces, and let
 $T \in {\mathcal M }(E,F)$.  Then it is immediately clear that
$T\in {\B}(E,F)$ and that $\lV T\rV  \leq \lV T \rV_{mb}$. More generally, for each  $n \in\N$, we have 
\begin{equation}\label{(6.1)}
\lV (Tx_1, \dots , Tx_n)\rV_n \leq \lV T \rV_{mb}\lV (x_1, \dots , x_n)\rV_n\quad (x_1, \dots , x_n\in E)\,.
\end{equation}
Indeed, for $n\in \N$, set
$$
p_n(T) =\sup\{\lV(Tx_1,\dots,Tx_n)\rV_n : \lV(x_1,\dots,x_n)\rV_n\leq 1\}\,.
$$
 Then $(p_n(T): n\in\N )$ is an increasing sequence with
$$
\lV T \rV_{mb} = \lim_{n\to\infty}p_n(T)\,.
$$
Explicitly,  we have
\begin{equation}\label{(6.10b)}
\lV T\rV_{mb}  = \sup_n\;\sup\left\{\frac{\lV(Tx_1,\dots,Tx_n)\rV_n}{\lV(x_1,\dots,x_n)\rV_n}:
(x_1,\dots,x_n) \neq 0\right\} < \infty\,,
\end{equation}
and so $T$  is multi-bounded if and only if $\lV T\rV_{mb}=  \sup_{n\in\N}\lV T^{(n)}\rV < \infty$, where $T^{(n)}$ is the 
$n^{\rm th.}$-amplification of $T$.\s  

We have noted in Theorem \ref{2.38a} that multi-norms correspond to $c_{\,0}$-norms on $c_{\,0}\otimes E$.
 Now take $T\in {\B}(E,F)$. Then $T$ is multi-bounded if and only if $I_{c_{\,0}}\otimes T$ is bounded as a map from $c_{\,0}\otimes E$ to $c_{\,0}\otimes F$ 
(when these spaces have the $c_{\,0}$-norms  corresponding to the respective multi-norms), and then $\lV T\rV_{mb}= \lV I_{c_{\,0}}\otimes T\rV$. Thus
our multi-bounded operators are the same as the `op{\'e}rateurs r{\'e}gulier' of \cite[D{\'e}finition 3.2]{MN}. For further details, see \cite{DDPR1}.
\medskip

\subsection{Multi-continuous operators} We shall now show that the multi-bounded operators on multi-normed spaces are exactly the `multi-continuous' ones, 
mirroring the fact that an operator on a normed space is continuous if and only if it is bounded.\s

 \begin{definition}\label{6.5}
Let $E_1$ and $E_2$ be    multi-topological linear spaces with respect to $(F_1,\tau_1)$ and $(F_2,\tau_2)$, respectively.  
Then  $T\in {\mathcal L}(E_1, E_2)$ is {\it multi-continuous}  if $(Tx_i)$ is a multi-null sequence in $E_2$ 
whenever $(x_i)$ is a multi-null sequence in $E_1$. \end{definition}\smallskip

The following result is taken from \cite{DM}, where some applications are given.\smallskip

\begin{theorem}\label{6.6}
 Let $((E^n, \norm_n) : n\in \N)$ and $((F^n, \norm_n) : n\in \N)$ be   two   multi-normed spaces.  Then a linear map from $E$ to $F$  is  multi-continuous 
 if and only if it is  multi-bounded.\end{theorem}

\begin{proof}  Suppose that $T \in{\mathcal L}(E,F)$ is  multi-bounded, and let $(x_i) $ be a multi-null sequence in $E$. Then, by Theorem \ref{5.4a},
for each $\varepsilon > 0$, there exists $n_0 \in \N$ such that
$$
\sup_{k\in \N}\lV (x_{n+1}, \dots, x_{n+k})\rV_k < \varepsilon\quad (n\geq n_0)\,.
$$
 But now
$$
\sup_{k\in \N}\lV (Tx_{n+1}, \dots, Tx_{n+k})\rV_k \leq \lV T\rV_{mb}\varepsilon\quad (n\geq n_0)\,,
$$
 and so,  by Theorem \ref{5.4a} again, $(Tx_i) $ is  a multi-null sequence in $F$. Thus $T$ is multi-continuous.

Suppose that  $T \in{\mathcal L}(E,F)$ is not multi-bounded.  Then there exists a subset $B$ of $E$ such that $B$ is multi-bounded  in $E$, but
$T(B)$ is not multi-bounded  in $F$.
For each  $n \in \N$, there exist $x_{1,n}, \dots, x_{k_n,n} \in B$ such that
$$
\lV (x_{1,n}, \dots, x_{k_n,n})\rV_{k_n} < \frac{1}{n^2}\quad {\rm and}\quad \lV (Tx_{1,n},\dots, Tx_{k_n,n})\rV_{k_n} > 1\,.
$$
We may suppose that $k_n \geq n$ for each $n\in\N$. Consider the sequence
$$
y = (x_{1,1},\dots, x_{k_1,1}, x_{1,2},\dots, x_{k_2,2}, \dots, x_{1,n},\dots, x_{k_n,n},\dots)\,.
$$
 We {\it  claim\/} that $y$ is a multi-null sequence in $E$.  Indeed, take $\varepsilon > 0$.  Then there exists
 $j\in \N$ such that $\sum_{i=j}^{\infty}1/i^2 < \varepsilon$, and then
$$
\lV (x_{1,j},\dots, x_{k_{j}, j}, \dots, x_{1,j+n},\dots, x_{k_{j+n}, j+n} )\rV_{k_j+\cdots + k_{j+n}} \leq \varepsilon\quad (n\in \N)\,,
$$
  giving the claim.  However $(Ty_i)$ is clearly not a multi-null sequence in $F$. Thus $T$ is not multi-continuous.
\end{proof}
\medskip

\section{The space ${\mathcal M}(E,F)$}

\subsection{The normed space ${\mathcal M}(E,F)$}   We shall recognize ${\mathcal M}(E,F)$  as  a normed space of operators.

Let $E$ and $F$ be normed spaces. Recall that the spaces ${\mathcal F}(E,F)$ and ${\mathcal N}(E,F)$ of  finite-rank   and nuclear operators were defined in 
Chapter 1, $\S\ref{Banach spaces and operators}$.\s

\begin{theorem}\label{6.9}
Let $((E^n, \norm_n) : n\in \N)$ and $((F^n, \norm_n) : n\in \N)$ be  multi-normed spaces, with $F$ a Banach space. Then
$$
({\mathcal M }(E,F), \lV\,\cdot\,\rV_{mb})
$$
 is a Banach space. Further: \s

{\rm (i)}   $y_0\otimes \lambda_0\in  {\mathcal M}(E,F)$ with 
$\lV y_0\otimes \lambda_0 \rV_{mb} = \lV y_0 \rV \lV \lambda_0\rV =  \lV y_0\otimes \lambda_0 \rV$ for each $\lambda_0 \in E'$ and $y_0\in F$;\s
 
{\rm (ii)} ${\mathcal N}(E,F)\subset  {\mathcal M}(E,F)$, and the natural embedding is a contraction.
\end{theorem}

\begin{proof} It is immediate that  $({\mathcal M }(E,F), \lV\,\cdot\,\rV_{mb})$ is a normed space.

Let $(T_k)$ be a Cauchy sequence in $({\mathcal M }(E,F), \norm_{mb})$. Then there exists $T \in {\mathcal B }(E,F)$ such that
$\lV T_k - T\rV \to 0$ as $k\to \infty$.  Take $\varepsilon >0$.  Then there exists $k_0 \in \N$ such that
$\lV T_j-T_k\rV_{mb} < \varepsilon\,\;(j,k\geq k_0)$. It follows from equation (\ref{(6.10b)}) that
  $T-T_k \in {\mathcal M }(E,F)$ and $\lV T-T_k\rV_{mb} \leq \varepsilon$ for each $j\geq k_0$.  Thus  $T_k\to T$ with respect to 
$\norm_{mb}$, and so  $({\mathcal M }(E,F), \lV\,\cdot\,\rV_{mb})$ is a Banach space.\s

(i) Let $\lambda_0 \in E'$ and $y_0\in F$, and set $T = y_0\otimes \lambda_0$. For each $n\in\N$ and $x_1,\dots,x_n \in E$, we have
\begin{eqnarray*}
\lV (Tx_1,\dots,Tx_n)\rV_n
&\leq &
\max\{ \lv \langle x_j, \lambda_0\rangle\rv : j\in\N_n\}\lV (y_0,\dots,y_0)\rV_n\\
&\leq &
 \lV y_0 \rV \lV \lambda_0\rV\max\{\lV x_j\rV  : j\in\N_n\}\\
&\leq &
 \lV y_0 \rV \lV \lambda_0\rV \lV (x_1,\dots,x_n)\rV_n\,,
\end{eqnarray*}
and so 
$$
\lV y_0\otimes \lambda_0 \rV \leq \lV y_0\otimes \lambda_0 \rV_{mb} \leq  \lV y_0 \rV \lV \lambda_0\rV =  \lV y_0\otimes \lambda_0 \rV\,.
$$
 It follows that $y_0\otimes \lambda_0  \in {\mathcal M}(E,F)$ with $\lV y_0\otimes \lambda_0 \rV_{mb}=\lV y_0\otimes \lambda_0 \rV$, and hence we have 
 ${\mathcal F}(E,F)\subset  {\mathcal M}(E,F)$.\s

(ii)  Let $T\in {\mathcal N}(E,F)$. Then clearly $T \in {\mathcal M}(E,F)$ with $\lV T \rV_{mb}\leq \nu(T)$, so that the natural embedding is a contraction.
\end{proof}\smallskip

We shall see in Example \ref{6.10a}, below,  that the `minimum' case for which we have  ${\mathcal N}(E,F) = {\mathcal M}(E,F)$ can occur.\s

\begin{theorem}\label{6.5c}
 Let $((E^n, \norm_n) : n\in \N)$ be a multi-normed space. Then $({\mathcal M }(E), \norm_{mb})$ is a unital Banach operator algebra.\qed
\end{theorem}\s

The following result was pointed out by Matt Daws; the result is also essentially contained in \cite[Remarque, p.\ 20]{MN}.\s

\begin{theorem}\label{6.8}
Let $((E^n, \norm_n): n\in \N)$ and $((F^n, \norm_n): n\in\N)$ be  two   multi-normed spaces.  Suppose that the multi-norm based on $F$
 is the min\-imum multi-norm, or that the multi-norm based  on $E$ is the maximum multi-norm. Then
$$
{\mathcal M }(E,F)={\mathcal B}(E,F)\quad {\rm  and}\quad \lV T \rV_{mb} =\lV T \rV\quad(T\in {\mathcal B}(E,F))\,.
$$
\end{theorem}

\begin{proof} First, suppose that the multi-norm based  on $F$ is the minimum multi-norm. We take  $T\in {\mathcal B}(E,F)$ and  $B \in {\mathcal {MB}}(E)$. Since
$$
\textstyle{\lV (Tx_1, \dots ,Tx_n)\rV_n = \max_{i\in\N_n}\lV Tx_i\rV\quad (n\in\N)}\,,
$$
 it is clear that $c_{T(B)} \leq \lV T \rV c_B$. It follows that $T\in {\mathcal M }(E,F)$ and that $\lV T \rV_{mb}\leq \lV T \rV$. 
But always  $\lV T \rV\leq \lV T \rV_{mb}$, and so we have 
$\lV T \rV =  \lV T \rV_{mb}$, as required. \s

Second, suppose that the multi-norm based  on $E$ is the maximum multi-norm.  We take $T\in {\mathcal B}(E,F)_{[1]}$, and define
$$
\LV (x_1,\dots, x_n)\RV_n = \max\{\lV (x_1,\dots, x_n)\rV_n,\, \lV (Tx_1,\dots, Tx_n)\rV_n\}
$$
for $x_1,\dots,x_n \in E$. It is easy to check that $((E^n, \LV\,\cdot\,\RV_n) : n\in \N)$  is a multi-normed space and that
$$
\LV x \RV = \max \{ \lV x \rV,\, \lV Tx\rV\} = \lV x \rV\quad (x \in E)\,.
$$
Since the multi-norm  based on $E$ is the maximum multi-norm, it follows that
$$
\lV (Tx_1,\dots, Tx_n)\rV_n \leq \lV (x_1,\dots, x_n)\rV_n\quad (x_1,\dots,x_n \in E)
$$
 for each $n\in\N$, and so $T\in {\mathcal M }(E,F)$ with $\lV T \rV_{mb} \leq 1$. This shows  that  we have  ${\mathcal M }(E,F)={\mathcal B }(E,F)$, and also that
$\lV T \rV_{mb} =\lV T \rV$ for each $T\in\B(E,F)$.
\end{proof}
\medskip

\subsection{A multi-norm based on ${\mathcal M}(E,F)$} We shall now see that there is a natural multi-normed structure based on ${\mathcal M}(E,F)$.\s

 \begin{definition}\label{6.7}
Let $((E^n, \norm_n) : n\in \N)$ and $((F^n, \norm_n) : n\in \N)$ be two   multi-normed spaces, and let $n\in \N$ and $T_1,\dots, T_n  \in {\mathcal M }(E,F)$. Then
$$
\lV (T_1,\dots, T_n)\rV_n^{mb} = \sup \{c_{T_1(B)\cup \cdots \cup T_n(B)} : B \in {\mathcal {MB}}(E) \;\,{\rm with}\,\; c_B \leq 1\}\,.
$$
\end{definition}\smallskip

Let $T \in {\mathcal M }(E,F)$. Then, by the definition,   $\lV T \rV_1^{mb}$ is exactly $\lV T \rV_{mb}$.
We have a somewhat more explicit formula for  $\lV (T_1,\dots, T_n)\rV_n^{mb}$. 

\begin{proposition}\label{6.7d}
Let $((E^n, \norm_n) : n\in \N)$ and $((F^n, \norm_n) : n\in \N)$
 be two   multi-normed spaces, and let $n\in \N$ and $T_1,\dots, T_n  \in {\mathcal M }(E,F)$. Then
\begin{equation}\label{(6.10)}
\lV (T_1,\dots, T_n)\rV_n^{mb} = \sup \lV (T_ix_j: i\in\N_n,\,j\in\N_k)\rV_{nk} \,,
 \end{equation}
where the supremum is taken over $x_1,\dots,x_k \in E$ with $\lV (x_1,\dots,x_k)\rV_k\leq 1$.
\end{proposition}

\begin{proof} Denote the left- and right-hand sides of the equation (\ref{(6.10)}) by $a$ and $b$, respectively.

  Take $x_1,\dots,x_k \in E$ with  $\lV (x_1,\dots,x_k)\rV_k\leq 1$, and set $B =\{x_1,\dots,x_k\}$. Then $c_B\leq 1$  and  
$\{T_ix_j: i\in\N_n,\,j\in\N_k\}\subset T_1(B)\cup \cdots \cup T_n(B)$.  Since $c_{T_1(B)\cup \cdots \cup T_n(B)}\leq a$, we have
$\lV (T_ix_j: i\in\N_n,\,j\in\N_k)\rV_{nk}\leq a$.  Hence $b\leq a$.

Take $\varepsilon >0$. Then there exists a set $B$ in $ E$ such that   $c_B \leq 1$ and $$c_{T_1(B)\cup \cdots \cup T_n(B)}\geq a-\varepsilon\,,$$  
and  there exist $k_1,\dots,k_n\in\N$ and $x_{1,i},\dots, x_{k_i,i}\in B$ for $i\in \N_n$ such that
$$\lV (T_ix_{r,i}: r\in \N_{k_i},\,i\in\N_n)\rV_k >  c_{T_1(B)\cup \cdots \cup T_n(B)}-\varepsilon\,,
$$
 where $k= k_1+ \cdots + k_n$.  Let $x_1,\dots,x_k$ be a listing of the elements $x_{r,i}$. By Lemma \ref{2.1b},
$$
\lV (T_ix_j: i\in \N_n,\,j\in\N_k)\rV_{nk}\geq \lV (T_ix_{r,i}: r\in \N_{k_i},\,i\in\N_n)\rV_k\,,
$$
and so  $b > a-2\varepsilon$. This holds true for each $\varepsilon >0$, and so $b\geq a$.
\end{proof}\s

\begin{theorem}\label{6.7a}
Let $((E^n, \norm_n) : n\in \N)$ and $((F^n, \norm_n) : n\in \N)$ be two   multi-normed spaces. Then $\norm_n^{mb}$  is a norm on the linear space ${\mathcal M }(E,F)^n$,
 and 
$$
(({\mathcal M }(E,F)^n, \norm_n^{mb}):  n\in \N)
$$
 is a  multi-normed space with $\lV T \rV_1^{mb}=\lV T \rV_{mb}$; it is a  multi-Banach space in the case where $F$ is a Banach space.
\end{theorem}

\begin{proof}  This now follows easily.  \end{proof}\smallskip

\begin{definition}\label{6.7b}
The   multi-norm  $(\norm_n^{mb}: n\in\N)$ is the  {\it multi-bound\-ed multi-norm\/} based on ${\mathcal M }(E,F)$.
\end{definition}\smallskip

\begin{theorem}\label{6.8a}
Let $((E^n, \norm_n): n\in \N)$ and $((F^n, \norm_n): n\in\N)$ be  two   multi-normed spaces, with $E\neq \{0\}$. Then the multi-bounded
multi-norm based on ${\mathcal M }(E,F)$ is the minimum multi-norm if and only if the multi-norm based on $F$ is the  minimum multi-norm.
\end{theorem}

\begin{proof}  Suppose that the multi-norm based on $F$ is the  minimum multi-norm.  

Let $n\in\N$  and   $T_1,\dots,T_n\in {\B}(E,F)$. 
For $k\in\N$, take $x_1,\dots,x_k\in E$ such that  $\lV (x_1,\dots,x_k)\rV_k\leq 1$. Then $\lV x_j\rV \leq 1\,\;(j\in \N_k)$, and so 
$\lV T_ix_j\rV \leq \lV T_i\rV\,\;(i\in\N_n,\,j\in\N_k)$. It follows from equation (\ref{(6.10)}) that 
$$
\lV (T_1,\dots,T_n)\rV_n^{mb} \leq \max_{i\in\N_n}\lV T_i\rV\,.
$$
By  Theorem  \ref{6.8}, $\lV T_i\rV_{mb} = \lV T_i\rV\,\;(i\in\N_n)$, 
and hence  $(\norm_n^{mb} : n\in\N)$ is the  minimum multi-norm based on ${\mathcal M }(E,F)$.

Conversely, suppose that $$\lV (T_1,\dots,T_n)\rV_n^{mb} = \max_{i\in\N_n}\lV T_i\rV$$ whenever $T_1,\dots,T_n\in {\B}(E,F)$ and $n\in\N$.  Fix 
 $n\in\N$, and take $y_1,\dots,y_n\in F$. Since $E\neq \{0\}$, there exist  $x_0\in E$ and $\lambda_0 \in E'$ with 
$\lV x_0\rV =\lV \lambda_0\rV = \langle x_0,\,\lambda_0\rangle =1$. For $i\in\N_n$ define
$T_i =   y_i\otimes \lambda_0$, so that  $T_i \in {\mathcal M }(E,F)$ and 
$\lV T_i\rV_{mb} =  \lV T_i\rV = \lV y_i\rV\;\,(i\in\N_n)$ by Theorem \ref{6.9}(i).
From   (\ref{(6.10)}), 
$$
\lV (T_1x_1,\dots,T_nx_n)\rV_n \leq \lV (T_1,\dots,T_n)\rV_n^{mb}\,.
$$
Hence 
$$
\lV (y_1,\dots,y_n)\rV_n\leq \max_{i\in\N_n}\lV T_i\rV = \max_{i\in\N_n}\lV y_i\rV = \lV (y_1,\dots, y_n)\rV_n^{\min}\,. 
$$
It follows that  $(\norm_n :n\in\N) \leq (\norm_n^{\min} :n\in\N) $, and so $(\norm_n :n\in\N)$ is the  minimum multi-norm based on $F$.
\end{proof}\s

\begin{corollary}\label{6.8b}
Let $((E^n, \norm_n): n\in \N)$ and $((F^n, \norm_n): n\in\N)$ be  two   multi-normed spaces, with $F$ a finite-dimensional space. 
 Then the multi-bounded multi-norm based on ${\mathcal M }(E,F)$ is equivalent to the minimum multi-norm.
\end{corollary}

\begin{proof}  By  Proposition  \ref{3.2b}, the  multi-bounded multi-norm is equivalent to the minimum multi-norm  based  on $F$, and so this follows by  a 
slight variation of the above proof
\end{proof}\s

In particular,   ${\mathcal M }(E,\C) = E'$, and the multi-bounded multi-norm  based  on ${\mathcal M }(E,\C)$ is  just the minimum 
multi-normed space $((E')^n, \norm_n^{\min} : n\in\N)$. We shall discuss in Chapter 7 a different way of constructing multi-norms based on dual spaces.
\medskip

 \section{Examples}

\noindent We give some specific examples of the Banach spaces ${\mathcal M}(E, F)$ and the Banach algebras ${\mathcal M}(E)$.\smallskip
 
\subsection{Algebras of operators} Let $E$ and $F$ be normed spaces.    The linear space of  compact operators on a normed space $E$  is 
denoted by ${\mathcal K}(E)\/$, as in $\S\ref{Banach spaces and operators}$. 

 In the first example, we shall  show that it may be that
 ${\mathcal K}(E)\not\subset {\mathcal M}(E)$, and hence that  ${\mathcal M}(E)\subsetneq {\mathcal B}(E)$.   \smallskip

\begin{example}\label{6.10}
{\rm Let $H$ be the Hilbert space $\ell^{\,2}(\N)$, with the stand\-ard $2$-multi-norm $(\norm_n^{[2]} :n\in\N)$ based on $H$ of Definition  \ref{4.1a}. As before,
 $(\delta_n : n\in \N)$ is the stand\-ard basis of $H$; the inner product in $H$ is denoted by $[\,\cdot\,,\,\cdot\,]$.

Consider the system of vectors $(x_r^s : r\in \N_s, s\in\N)$ in $H$ defined as follows:  $x_r^s (k)= 0$ except when $k\in \{2^{s-1},\dots , 2^s-1\}$; 
at  the $2^{s-1}$ numbers $k$ in the set $\{2^{s-1},\dots , 2^s-1\}$, $x_r^s (k)= \pm 1/\sqrt{2^{s-1}}$, the values $\pm 1$ being chosen
 so that $[x_{r_1}^s, x_{r_2}^s]= 0$ when $r_1,r_2 \in \N_s$ and $r_1 \neq r_2$.  Such a choice is clearly possible.  Then
$$
S: = \{x_r^s : r\in \N_s, s\in\N\}
$$
 is an orthonormal set in $H$.  Order the set $S$ as $(y_n)$ by using the lexicographic order on the pairs $(s,r)$
 (so  that $y_1 = x^1_1$, $y_2 = x^2_1$, $y_3 =x^2_2$, $y_4 = x^3_1$, etc.).

Let $(\alpha_i) \in \ell^{\,\infty}$. We define an operator $T$ by setting 
$$
Tx_r^s  = \alpha_s\delta_n\quad {\rm  when}\quad x_r^s =y_n\,;
$$
 clearly $T$ extends by linearity and continuity to become an operator in ${\mathcal B}(H)$. It is also clear that, 
in the case where $(\alpha_i) \in c_{\,0}$, we have $T\in {\mathcal K}(H)$.

For $k \in \N$, set  $N_k =\sum_{i=1}^k i = k(k+1)/2$. We see that $\lV (y_1, y_2, \dots, y_{N_k})\rV_{N_k}^{[2]} = \sqrt{k}$.  However
$$
\lV (Ty_1, Ty_2, \dots, Ty_{N_k})\rV_{N_k}^{[2]} =  \lV (\alpha_1\delta_1,\alpha_2\delta_2, \alpha_2\delta_3, 
\alpha_3 \delta_4,\dots, \alpha_k\delta_{N_k})\rV_{N_k}^{[2]}  = \left({\sum_{i=1}^k i\lv \alpha_i\rv^2}\right)^{1/2}\,.
$$
Now take $\gamma \in ( 0,1/2)$, and set $\alpha_i = i^{-\gamma}\,\;(i\in \N)$, so that $(\alpha_i) \in c_{\,0}$ and $T\in {\mathcal K}(H)$.  Then
$$\sum_{i=1}^k i\lv \alpha_i\rv^2=\sum_{i=1}^k i^{1-2\gamma} \geq \int_1^k t^{1-2\gamma}{\dd}t\geq \frac{1}{2-2\gamma}(k^{2-2\gamma}-1)\,.
$$
Thus
$$
\frac{\lV (Ty_1, Ty_2, \dots, Ty_{N_k})\rV_{N_k}}{\lV (y_1, y_2, \dots, y_{N_k})\rV_{N_k}} \geq ck^{(1-2\gamma)/2}
$$
for a constant $c>0$. Since $\gamma < 1/2$, we have $T \not\in {\mathcal M}(H)$.
Thus ${\mathcal K}(H)\not\subset {\mathcal M}(H)$. In particular, ${\mathcal M}(H)\subsetneq {\mathcal B}(H)$.
 However  ${\mathcal M}(H)\not\subset {\mathcal K}(H)$ because $I_H \in {\mathcal M}(H)$.

 Now  consider the Hilbert multi-norm $(\norm_n^H:n\in\N)$ based on $H$.  By Theorem  \ref{4.16},  the Hilbert multi-norm   
is equivalent to the maximum  multi-norm $(\norm_n^{\max}:n\in\N)$ based on $H$, and so it follows from Theorem \ref{6.8}  that ${\mathcal M}(H) = {\mathcal B}(H)$
 in this case.\qed}\end{example}\smallskip
 
\begin{example}\label{6.10a}
{\rm   In this example, we shall show that the inclusion 
$$
{\mathcal N}(E,F)\subset  {\mathcal M}(E,F)
$$
given  in Theorem \ref{6.9}(ii) is best possible.

One might guess that a form of Banach's isomorphism theorem would hold for multi-bounded operators.
This would assert that $T^{-1} \in {\mathcal M}(F, E)$  whenever both $$((E^n, \norm_n): n\in\N)\quad  {\rm and}\quad   ((F^n, \norm_n): n\in\N)$$ 
are multi-normed spaces, $T \in {\mathcal M}(E, F)$,  and $T$ is a bijection. However we shall show that this is not the case; 
this will also be shown, in stronger form, in Examples \ref{6.35} and  \ref{6.16g}, below.

Let $E = \ell^{\,1}$.  Then $((E^n, \norm_n^{[1]}): n\in\N)$ is a multi-normed space, where  we are writing $(\norm_n^{[1]}: n\in\N)$
for  the standard $1$-multi-norm of Definition \ref{4.1a}.  In this case,
$$
\lV (\delta_1, \dots, \delta_n) \rV_n ^{[1]}  = n\quad (n\in\N)\,,
$$
as in equation (\ref{(4.2)}). By Example \ref{2.3ca}, $(\norm_n^{[1]}: n\in\N)$ coincides with the lattice multi-norm  
$(\norm_n^{L}: n\in\N)$ on the Banach lattice $E$. However, also let $F = \ell^{\,1}$, and consider the minimum multi-norm 
$(\norm_n^{\min} : n\in\N)$ based on $F$, so that 
$$
\lV (\delta_1, \dots, \delta_n) \rV^{\min}_n = 1\quad (n\in\N)\,.
$$
Since $\norm_n^{\min} \leq \norm_n^{[1]}\,\;(n\in\N)$,   the identity map $I_E$ on $E$, regarded as map from $E$ to $F$  belongs to ${\mathcal M }(E,F)$ 
 (and $I_E$ is a multi-contraction). However the above two equations show that  $I_E^{-1} : F \to E$ is not multi-bounded.

Indeed, by Theorem \ref{6.8}, ${\mathcal M }(E,F)= {\mathcal B }(E,F)$ and, by Theorem \ref{6.8a}, the multi-bounded multi-norm based on 
${\mathcal M }(E,F)$ is the minimum multi-norm.

  We shall now identify  ${\mathcal M }(F,E)$.
Take $T \in  {\mathcal M }(F,E)$.   The unit ball $F_{[1]}$ of $F$ is multi-bounded, and so $T(F_{[1]})$ is multi-bounded in $E$.
  Since the Banach lattice $\ell^{\,1}$ is mono\-tonically bounded, it follows from Theorem \ref{6.11} that  $T(F_{[1]})$
 is order-bounded in  $\ell^{\,1}$, and  so there exists $x = (x_i) \in \ell^{\,1}$ with
$$
\lv (Ty)_i \rv\leq x_i\quad (i\in\N)
$$
for each $y \in F_{[1]}\,$; further, $\sum_{i=1}^\infty x_i\geq \lV T \rV_{mb}$.  Take $i\in\N$, let $\pi_i : z \mapsto z_i\delta_i$ 
be the rank-one operator on $\ell^{\,1}$, and set $T_i  = \pi_i\,\circ\,T = (\delta_i\otimes T')(\delta_i)$. For each $y\in F_{[1]}$, we have
$$
\lv \langle y,\,T'(\delta_i)\rangle\rv = \lv \langle Ty,\, \delta_i \rangle\rv = \lv (Ty)_i \rv\leq x_i\,,
$$
and so $\lV T'(\delta_i)\rV \leq x_i$, whence $\nu (T_i) =  \lV T'(\delta_i)\rV \lV \delta_i\rV \leq x_i$.
Clearly, we  have  $T =  \sum_{i=1}^\infty  T_i$, and hence  $\nu (T) \leq  \sum_{i=1}^\infty x_i < \infty$.   Thus $T \in {\mathcal N}(F, E)$.

In summary, in this case we have
$$\vspace{-\baselineskip}
{\mathcal M }(E,F)= {\mathcal B }(E,F)\quad{\rm and}\quad  {\mathcal M }(F,E)= {\mathcal N }(F,E)\,.
$$
\hspace*{\stretch{1}}\qed}\end{example}\s

\subsection{Partition multi-norms}
 We present an example that was suggested by Michael Elliott.

Take $p\geq 1$, and consider  $\ell^{\,p} =\ell^{\,p}(\N)$; the norm on $\ell^{\,p}$ is denoted by $\norm$.\s

\begin{definition}\label{6.31}
  For each partition $\Pi$ of $\N$ and $n\in \N$, set 
$$
\lV (f_1,\dots,f_n)\rV_n^{\Pi}= \left(\sum\left\{\max_{k\in\N_n}\lV f_k \mid P\rV^p : P \in \Pi\right\}\right)^{1/p}\quad (f_1,\dots, f_n\in \ell^{\,p})\,.
$$
\end{definition}\s

It is easy to check that $\norm_n^{\Pi}$ is a norm on $(\ell^{\,p})^n$ for each $n\in \N$ and that $$(((\ell^{\,p})^n, \norm_n^{\Pi}): n\in\N)$$ is a multi-normed space. 

By taking $\Pi$ to be the singleton $\{\N\}$, we see that we obtain the minimum multi-norm $(\norm_n^{\min}: n\in\N)$ as an example;   by taking  $\Pi$ to be 
the collection of singletons $\{n\}$ in $\N$, we obtain the lattice multi-norm $(\norm_n^L: n\in\N)$ based on $\ell^{\,p}$.\s

\begin{definition}\label{6.32}
  For each partition $\Pi$ of $\N$, the  above multi-norm $(\norm_n^{\Pi}: n\in\N)$ is the {\it partition multi-norm\/} based on $\ell^{\,p}$.
\end{definition}\s

For $\sigma \in {\mathfrak S}_{\N}$ and  $S\subset \N$, we set $\sigma(S) = \{\sigma(n) : n\in S\}$, and we define
$$
T_{\sigma} :f\mapsto f\,\circ\,\sigma\,,\quad \ell^{\,p}\to \ell^{\,p}\,,
$$
so that $T_{\sigma} :\ell^{\,p}\to \ell^{\,p}$ is an isometry.

Let $\Pi$ be a partition of $\N$, and   define the sets
$$
\Pi_{\sigma}(P) = \{Q \in \Pi: \sigma(Q)\cap P \neq \emptyset\}\quad {\rm and} \quad \Pi_{\sigma}^{-1}(P) = \{Q \in \Pi: \sigma(P)\cap Q \neq \emptyset\} 
$$
for each  $P\in \Pi$, so that $\Pi_{\sigma}^{-1}(P)= \Pi_{\sigma^{-1}}(P)$ and
$\sigma(P)$ is contained in the pairwise-disjoint union of the family  $\Pi_{\sigma}^{-1}(P)$ of subsets of $\N$. \s

\begin{lemma}\label{6.33}
Let $\Pi$ be a partition of $\N$, and take $\sigma \in {\mathfrak S}_{\N}$. Then 
$$
\lV (T_\sigma f)\mid P\rV  \leq  \left(\sum\left\{ \lV f\mid Q\rV^p : Q\in \Pi_{\sigma}^{-1}(P)\right\}\right)^{1/p}\quad (f\in \ell^{\,p},\,P \in\Pi)\,.
$$
\end{lemma}

\begin{proof}  Take $f\in \ell^{\,p}$. Then
\begin{eqnarray*}
\lV (T_\sigma f)\mid P\rV^p  &=& \sum_{n \in P}\lv f(\sigma(n))\rv^p  \leq  \sum\left\{\sum_{m \in Q}\lv f(m)\rv^p : Q\in\Pi_{\sigma}^{-1}(P)\right\}  \\
&=&
 \sum\{\lV f\mid Q\rV^p : Q\in\Pi_{\sigma}^{-1}(P)\} \,,
\end{eqnarray*}
giving the stated result.
\end{proof}\s

Let $\Pi$ be a partition of $\N$, and take $\sigma \in {\mathfrak S}_{\N}$. Then we define
$$m_\sigma = \sup \{\lv \Pi_\sigma(P)\rv : P\in \Pi\}\,,
$$
so that $m_\sigma \in \N\cup\{\infty\}$.\s

\begin{theorem}\label{6.34}
Let $\Pi$ be a partition of $\N$, and consider the multi-norm $(\norm_n^{\Pi}: n\in\N)$ based on $ \ell^{\,p}$, where $p\geq 1$. Take 
$\sigma \in {\mathfrak S}_{\N}$.  Then $T_\sigma : \ell^{\,p}\to \ell^{\,p}$ is multi-bounded
with respect to this multi-norm if and only if $m_\sigma< \infty$; in this latter case,  $\lV T \rV_{mb} = m_\sigma^{1/p}$.
\end{theorem}

\begin{proof}  Suppose that $m_\sigma< \infty$. Take $n\in\N$ and $f_1,\dots,f_n \in \ell^{\,p}$.  Then
\begin{eqnarray*}
\left(\lV (T_\sigma f_1,\dots,T_\sigma f_n)\rV_n^\Pi \right)^p&=& 
\sum_{P\in\Pi}\max_{k\in\N_n}\lV (T_\sigma f_k) \mid P\rV^p\\
&\leq & 
\sum_{P\in\Pi}\max_{k\in\N_n}\sum_{Q\in \Pi_{\sigma}^{-1}(P)} \lV f_k\mid Q\rV^p\quad \mbox{by Lemma \ref{6.33}} \\ 
&\leq &
\sum_{P\in\Pi}\sum_{Q\in \Pi_{\sigma}^{-1}(P)} \max_{k\in\N_n}\lV f_k\mid Q\rV^p\\
&= &\sum_{Q\in\Pi}\sum_{P\in \Pi_{\sigma}(Q)} \max_{k\in\N_n}\lV f_k\mid Q\rV^p\\
&\leq &
\sum_{Q\in\Pi}\lv \Pi_{\sigma}(Q)\rv\max_{k\in\N_n}\lV f_k\mid Q\rV^p \leq 
 m_\sigma  \left(\lV (f_1,\dots,f_n)\rV_n^\Pi\right)^p\,,
\end{eqnarray*}
and so $T_\sigma \in {\mathcal M}(\ell^{\,p})$ with $\lV T \rV_{mb} \leq  m_\sigma^{1/p}$.

We continue to suppose that $m_\sigma< \infty$,  say $k= m_\sigma\in\N$.   Then there exists  $P\in \Pi$  and $k$ pairwise-disjoint  sets 
$Q_1,\dots,Q_k\in \Pi_\sigma(P)$.  For each $j\in \N_k$, choose $n_j \in Q_j$ with $\sigma(n_j)\in P$, and set $f_j=\delta_{\sigma(n_j)}$. Then
$$
\lV (f_1,\dots,f_k)\rV_k^\Pi = \max_{j\in\N_k}\lV f_j\mid P\rV =1
$$
and  $T_\sigma f_j = \delta_{n_j}$, so that $\lV (T_\sigma f_j)\mid Q_j\rV = 1$ and $\lV (T_\sigma f_j)\mid Q\rV = 0$ for $Q\in \Pi$ with $Q\neq Q_j$. 
 Thus $\lV (T_\sigma f_1,\dots,T_\sigma f_n)\rV_n^\Pi = k^{1/p}$.
 
 It follows that  $\lV T \rV_{mb} = m_\sigma^{1/p}$ in the case where $m_\sigma< \infty$.
 
 In the  case where we have $m_\sigma =\infty$, the argument of the last paragraph shows that, for each $k\in  \N$, there exist $f_1,\dots,f_k\in \ell^{\,p}$ such that  
 $\lV (f_1,\dots,f_k)\rV_k^\Pi =1$ and $\lV (T_\sigma f_1,\dots,T_\sigma f_k)\rV_k^\Pi = k^{1/p}$, and so $T$ is not multi-bounded.
\end{proof}\s

The  next example shows a failure of the `Banach isomorphism theorem  for multi-normed spaces' in the special case where the two multi-norms are equal.\s
 
\begin{example}\label{6.35}
{\rm   Let $\Pi$ be a partition of $\N$ into infinitely many infinite  subsets,
 say $P_1, P_2, \dots$, where the sets $P_j$ are distinct.  Take 
$$
Q_0 = P_1 \cup P_3\cup P_5\cup\cdots\quad{\rm and}\quad Q_j =P_{2j}\,\;(j\in\N)\,,
$$
 so  that $\{Q_0, Q_1, Q_2, \dots\}$  is also a partition of $\N$ into infinite sets.  For each $k\in \N$, 
let $\sigma_k : P_k\to Q_{k-1}$
be a bijection, and define $\sigma\in {\mathfrak S}_\N$ by setting $\sigma(n) =\sigma_k(n)$ when $n\in P_k$.

Consider the partition multi-norm $(\norm_n^{\Pi}: n\in\N)$ based on $ \ell^{\,1}$, and set $T=T_\sigma$ in the above notation.  For each 
$P_i\in\Pi$, we have $\Pi_\sigma(P_i)=\{P_i\}$, and so $\lv \Pi_\sigma(P_i)\rv =1$. By Theorem \ref{6.34}, $T  \in {\mathcal M}(\ell^{\,1})$ with $\lV T \rV_{mb} =1$.
On the other hand, $ T^{-1}= T_{\sigma^{-1}}\in {\B}(\ell^{\,1})$, and 
$$
\Pi_{\sigma^{-1}}(P_1) = \{P_j : Q_0 \cap P_j\neq \emptyset\} = \{P_1,P_3,P_5, \dots\}\,,
$$
an infinite set, so that  $ T^{-1}$ is not multi-bounded.}\qed
\end{example}\medskip

\section{Multi-bounded operators on Banach lattices}

\noindent Our next aim is  to identify the space ${\mathcal M }(E,F)$ of multi-bounded operators in the case where $E$ and $F$ are
 Banach lattices. Throughout this  section, we are taking  the  lattice multi-norms 
$(\norm_n^L:n\in\N)$ of Definition \ref{2.3c} as the multi-norms on both of the families $\{E^n :n\in\N\}$ and $\{F^n :n\in\N\}$.

\subsection{Multi-bounded and order-bounded operators}   Let $E$ and $F$ be Banach lattices.
Recall that the space ${\B}_b(E,F)$  of order-bounded operators   from $E$ to $F$ and the   norm  $\lV T \rV_b$
of $T\in {\B}_b(E,F)$ were   defined in $\S1.3.4$.  In this subsection, we shall compare ${\B}_b(E,F)$  with  ${\mathcal M}(E,F)$.\s

\begin{theorem}\label{6.16a}
Let $E$ and $F$ be Banach lattices. Then each order-bounded operator $T$ from $E$ to $F$ is multi-bounded, and 
$\lV T \rV_{mb} \leq \lV T \rV_{b}$.\end{theorem}

\begin{proof}  Let $T\in {\B}_b(E,F)$, and suppose that  $B\in {\mathcal {MB}}(E)$.  Now take $x_1,\dots, x_n\in B$,
and set $v = \lv x_1\rv\vee \cdots \vee \lv x_n\rv$, so that $\lV v \rV \leq c_B$.   For each $x\in \Delta_v$ and   $\varepsilon >0$,
 there exists $w \in F$ such that $$\lv Tx\rv \leq w\quad {\rm  and}\quad \lV w \rV < \lV T \rV_b\lV v \rV + \varepsilon\,.
$$
   For  $i\in\N_n$, we have  $\lv Tx_i\rv \leq w$, and so $\lv Tx_1\rv \vee\cdots\vee \lv Tx_n\rv \leq w$.  Thus 
$$
\lV (Tx_1,\dots,Tx_n)\rV_n^L = \lV\,\lv Tx_1\rv \vee\cdots\vee\lv Tx_n\rv\,\rV \leq \lV w \rV < \lV T \rV_bc_B + \varepsilon\,.
$$
This holds true for each $\varepsilon >0$, and so $T\in {\B}_b(E,F)$ with $c_{T(B)} \leq  \lV T \rV_bc_B$.  Thus $T\in {\mathcal M}(E,F)$ with 
$\lV T \rV_{mb} \leq \lV T \rV_{b}$, as claimed.
\end{proof}\s

\begin{corollary}\label{6.16b}
Let $E$ and $F$ be Banach lattices, and let $T \in {\B}(E,F)^+$. Then $$\lV T \rV_{r} = \lV T \rV_{b} = \lV T \rV_{mb} =\lV T \rV\,.$$
\end{corollary} 

\begin{proof} Always $\lV T \rV \leq \lV T \rV_{mb}$. By the theorem, $\lV T \rV_{mb} \leq \lV T \rV_{b}$. 
 But $\lV T \rV_{b} = \lV T \rV_r =\lV T \rV$ for positive operators $T$ by equation (\ref{(1.3a)}).
\end{proof}\s

The present formulation of the following result  is due to Michael Elliott.\s

\begin{theorem}\label{6.16c}
Let $E$ and $F$ be Banach lattices. \s

{\rm (i)} Suppose that $F$ is monotonically bounded.  Then ${\B}_b(E,F)= {\mathcal M}(E,F)$ and $\norm_{mb}$ and
 $\norm_b$ are equivalent on ${\B}_b(E,F)$. \s
 
 {\rm (ii)} Suppose that $F$ has the weak Nakano property.  Then   $\norm_{mb}$ and 
 $\norm_{b}$ are equivalent on  ${\B}_b(E,F)$, with  equality of norms when $F$ has the Nakano property.\s
 
  {\rm (iii)}     Suppose that $F$ is monotonically bounded and has the Nakano property.  Then ${\B}_b(E,F)= {\mathcal M}(E,F)$ and 
 $\lV T\rV_{mb} = \lV T\rV_{b}\,\;(T\in {\B}_b(E,F))$.\s
 
 {\rm (iv)}  Suppose that $F$ is monotonically bounded and Dedekind complete.  Then  $${\B}_r(E,F) = {\B}_b(E,F)= {\mathcal M}(E,F)$$ and  $\norm_{mb}$ and 
$\norm_r$  are equivalent on ${\B}_r(E,F)$, with equality of norms when $F$ has the Nakano property.
\end{theorem}

\begin{proof}   Let $T\in {\B}(E,F)$. 
 Suppose that $T\in {\B}_b(E,F)$. Then it follows from Theorem \ref{6.16a} that  $T\in {\mathcal M}(E,F)$ with $\lV T \rV_{mb} \leq \lV T \rV_{b}$.\s
 
(i)  Suppose that $T\in {\mathcal M}(E,F)$, and take an order-bounded subset $B$ of $E$. By  Prop\-osition \ref{6.11a}, $B \in {\mathcal {MB}}(E)$, and so 
$T(B)\in {\mathcal {MB}}(F)$.  Since $F$ is monotonically bounded, it follows from Theorem  \ref{6.11} that $T(B)$ is order-bounded, and so 
$T\in {\B}_b(E,F)$.  Thus ${\mathcal M}(E,F) = {\B}_b(E,F)$. Since  $({\mathcal M}(E,F), \norm_{mb})$ and  $({\B}_b(E,F), \norm_b)$ 
are Banach spaces with their norms dominating the operator norm, the equivalence of the norms follows from the closed graph theorem.\s

(ii)  Suppose that $T\in {\B}_b(E,F)$.  Fix $\varepsilon >0$, and take $v \in E^+_{[1]}$.

  The  set $B:= \Delta_v$ is order-bounded,  and so $B \in {\mathcal {MB}}(E)$
 with $c_B \leq 1$, as in Proposition \ref{6.11a}.  Take ${\mathcal F} ={\mathcal P}_f(T(B))$, and set $y_S = \bigvee \{ \lv y \rv :y \in S\}$ 
for $S \in  {\mathcal F}$.  Then  $\{ y_S :S \in  {\mathcal F}\}$ is an increasing  net in $F^+$ such that 
 $\lV y_S \rV \leq \lV T \rV_{mb}\,\;(S \in {\mathcal F})$.
 
The set  $T(B)$ is order-bounded, and so $\{ y_S :S \in  {\mathcal F}\}$ is order-bounded. Since $F$ has the  weak  Nakano property, 
there exists $K\geq 1$ and $u\in F_{\R}$  such that   
$$
  y_S\leq u\quad(S \in  {\mathcal F})\quad{\rm and}\quad \lV u\rV \leq 
K\sup_{S\in {\mathcal F}}\lV y_S \rV+ \varepsilon\leq K\lV T \rV_{mb}+ \varepsilon\,.
$$  
It follows that $\lV T \rV_b \leq  K\lV T \rV_{mb}+ \varepsilon$. 

This holds true for each $\varepsilon >0$, and so $\lV T \rV_b \leq  K\lV T \rV_{mb}$. The result follows.\s

(iii)  This follows immediately from (i) and (ii).\s

(iv)   By Theorem \ref{6.13},   ${\B}_b(E,F) = {\B}_r(E,F)$  and $\lV T\rV_r = \lV T\rV_b$  for each $T \in {\B}_b(E,F)$, and so the result follows from (i) and (iii).
\end{proof}\s

\begin{corollary}\label{6.16i}
Let $E$   be a Banach lattices, and let $F= L^p(\Omega)$  for a measure space $\Omega$ and $p\geq 1$.  Then ${\B}_r(E,F) = {\B}_b(E,F)= {\mathcal M}(E,F)$ and 
$$
\lV T\rV_{mb}  = \lV T\rV_r = \lV T\rV_b = \lV \,\lv T \rv\,\rV\quad(T \in  {\mathcal B}_r(E,F) )\,.
$$
\end{corollary}

\begin{proof}  The hypotheses on $F$ in Theorem \ref{6.16c}(iv) are satisfied by every monotonically complete Banach lattice with order-continuous norm, and hence by
the lattices $L^p(\Omega)$.
\end{proof}\s

In the case where $E=F= L^p(\Omega)$, $p>1$, and $L^p(\Omega)$ is infinite-dimensional, it follows from Theorem \ref{1.27}(iii) that  ${\mathcal M}(E,F)$ is not dense in 
${\B}(E,F)$.  We are grateful to Anthony Wickstead for the following remarks. First, let
  $p, q \in [1,\infty]$. Then  $ {\mathcal B}_r(\ell^{\,p}, \ell^{\,q})\neq {\mathcal B}(\ell^{\,p}, \ell^{\,q})$ whenever either 
$p>1$  or $q< \infty$, and so, in the latter case, $ {\mathcal M}(\ell^{\,p}, \ell^{\,q})\neq {\mathcal B}(\ell^{\,p}, \ell^{\,q})$. 
 Second,  suppose that $1 \leq q< p< \infty$.  Then it follows from Pitt's theorem  \cite[Theorem 2.1.4]{AK}
 that ${\mathcal K}(\ell^{\,p}, \ell^{\,q}) =  {\mathcal B}(\ell^{\,p}, \ell^{\,q})$, and so 
  ${\mathcal M}(\ell^{\,p}, \ell^{\,q}) \subsetneq {\mathcal K}(\ell^{\,p}, \ell^{\,q})$ in this case.

The following easy example  shows  that `monotonically bounded' is not redundant in Theorem \ref{6.16c}, (i), (iii), and (iv).\s

\begin{example}\label{6.16d}
{\rm  Take  $E = c=c_{\,0}\oplus \C 1$, where $1 =(1,1,\dots)$,  and $F = c_{\,0}$. Then $F$ is Dedekind complete and
 has the Nakano property, but it  is not monotonically bounded.  By Theorem \ref{2.3n}(ii), the lattice multi-norm  on the $AM$-space   
$F$ is the minimum multi-norm, and so, by Theorem \ref{6.8}, we have ${\mathcal M }(E,F)={\mathcal B}(E,F)$ and $\lV T \rV_{mb} =\lV T \rV\,\;
(T\in {\mathcal B}(E,F))$. 

Consider the map 
$$T: \alpha + z 1\mapsto \alpha\,,\quad E\to F\,.
$$
 Then ${T \in \mathcal B}(E,F)$ with $\lV T \rV =2$, but $T$ is not order-bounded.  For set
$\alpha_n= \sum_{i=n}^\infty\delta_i \in E$, so that $\{\alpha_n : n\in\N\}$ is order-bounded. However, 
$\lv T(\alpha_n)\rv = \sum_{i=1}^{n-1}\delta_i \in E$, so that the set  $\{T\alpha_n : n\in\N\}$ is not order-bounded in $F$.}\qed
\end{example}\s

The following   example, also due to Michael Elliott,  shows  that `has the weak Nakano property'  is not redundant in Theorem \ref{6.16c}(ii),
 even when $F$ is Dedekind complete.  

Let $(R_n )$ denote the sequence of Rademacher functions on $\I$. Thus   $$R_1=\chi_{[0,1/2]}-\chi_{( 1/2,1]}\,,\quad  
R_2 =  \chi_{[0,1/4]}-\chi_{( 1/4,1/2]}+\chi_{(1/2,3/4]}-\chi_{( 3/4,1]}\,,
$$
 etc.; we regard these functionals as elements of the dual space $L^\infty(\I)$ of  $L^1(\I)$.  

We {\it claim\/} that,  for each $f \in L^1(\I)$, the sequence $( \langle f,\,R_n\rangle:n\in\N) $ is a null sequence. Indeed, first suppose that 
$f =\chi_{[a,b]}$ for $0\leq a \leq b\leq 1$.  Then $ \lv \langle f,\,R_n\rangle\rv \leq 1/2^{n-1}\,\;(n\in\N)$, so that the claim holds in this case. Hence 
it holds for each simple function $f$, and then for each $f \in L^1(\I)$ because the simple functions are dense in $L^1(\I)$. 
 It follows from Proposition \ref{6.11b} that $$\sum_{n=1}^\infty \lv \langle f,\,R_n\rangle \rv y_n$$ is convergent in $E$ for each pairwise-disjoint, 
multi-bounded sequence $(y_n)$ in  a Banach lattice.\s

\begin{lemma}\label{6.16e}
Let $E$ be the Banach lattice $L^1(\I)$, and let $F$ be any Banach lattice. Suppose that $(y_n)$ is  a  pairwise-disjoint, multi-bounded sequence in $F^+$,
 and define
$$
T: f\mapsto \sum_{n=1}^{\infty}\langle f,\,R_n\rangle y_n\,,\quad E\to F\,.
$$
Then $T \in {\mathcal M}(E,F)$ with $\lV T \rV_{mb} \leq \lV (y_n)\rV_{mb}$. 
\end{lemma}\s

\begin{proof} Note that, using equation (\ref{(2.4ab)}), we have
$$
\lV y_1+\cdots+y_n\rV  =   \lV y_1\vee  \cdots\vee y_n\rV=  \lV (y_1,\dots, y_n)\rV^L_n \quad (n\in\N)\,,
$$
and so $$\lV (y_n)\rV_{mb}=\sup\{\lV y_1+\cdots+y_n\rV: n\in\N\}\,.$$

 Let $B \in {\mathcal MB}(E)$ with $c_B\leq 1$, so that $B\subset E_{[1]}$, and take $\{z_1, \dots,z_k\}$ to be a finite subset of $T(B)$.  
For each $j\in \N_k$, choose $f_j\in B$ with $Tf_j= z_j$, and then, for $n\in\N$,  define the numbers
$$
s_n = \lv \langle f_1, R_n\rangle\rv \vee \cdots\vee  \lv \langle f_k, R_n\rangle\rv\,, \quad
 t_n = \lv \langle f_1, R_n\rangle\rv +  \cdots +  \lv \langle f_k, R_n\rangle\rv\,.
$$

Fix $\varepsilon >0$. For each $j\in \N_k$,  there exists $i\in\N$ such that
$$
\lV  \sum_{n=i}^{\infty}\lv \langle f_j,\,R_n\rangle \rv y_n\rV< \frac{\varepsilon}{k}\quad (j\in \N_k)\,,
$$
and so 
$$
\lV  \sum_{n=i}^{\infty} s_n y_n\rV \leq \lV  \sum_{n=i}^{\infty} t_ny_n\rV < \varepsilon\,.
$$
However, $\lv z_1\rv \vee \cdots \vee \lv z_k\rv = \lv Tf_1\rv \vee \cdots \vee \lv Tf_k\rv = \sum_{n=1}^\infty s_ny_n$, and so 
$$
\lV \,\lv z_1\rv \vee \cdots \vee \lv z_k\rv\,\rV \leq \lV \sum_{n=1}^{i-1}y_n\rV + \varepsilon = \lV (y_1,\dots,y_{i-1})\rV_{i-1}^L + \varepsilon \leq
\lV (y_n)\rV_{mb} + \varepsilon\,.
$$
Thus $$\lV (z_1,\dots,z_k)\rV_k^L \leq \sup\{\lV y_1+\cdots+y_n\rV: n\in\N\}+ \varepsilon\,.$$  
This holds true \vspace{.5mm}for each $\varepsilon >0$, and so
 $$\lV (z_1,\dots,z_k)\rV_k^L \leq \sup\{\lV y_1+\cdots+y_n\rV: n\in\N\}\,.
$$ Hence   the result follows.
\end{proof}\s

 \begin{theorem}\label{6.16f}
 Let $E$ be the Banach lattice $L^1(\I)$,  and let $F$ be any Dedekind complete lattice.  Then $F$ has the weak $\sigma$-Nakano 
property if and only if  $\norm_{mb}$ is equivalent to $\norm_r$ on ${\B}_r(E,F)$.
\end{theorem}

\begin{proof}   Recall from  Theorem \ref{6.13} that  ${\B}_r(E,F) =  {\B}_b(E,F)$ and that  $\lV T \rV_{b} =\lV T\rV_r$ for $T\in {\B}_r(E,F)$. 
Thus, when $F$ has the weak  Nakano property, the norms   $\norm_{mb}$ and  $\norm_r$ are equivalent on ${\B}_r(E,F)$ by Theorem \ref{6.16c}(ii); 
since $E$ is separable, a trivial  variation of the argument shows this when $F$ has just the weak
 $\sigma$-Nakano property. 

Conversely, suppose that   $K\geq 1$ with  $\lV T\rV_r\leq K\lV T\rV_{mb}\,\;(T \in {\B}_r(E,F))$.

Let $(x_n)$   be an increasing, order-bounded sequence in $E_\R$. Since $F$ is Dedekind complete, the set $\{x_n: n\in\N\}$ has  a supremum, say $y$.   

Define $y_1=x_1$ and $y_n=x_n-x_{n-1}$ for $n\geq 2$, so that $(y_n)$ is pairwise-disjoint and $y_1+\cdots + y_n =x_n\,\;(n\in\N)$.  
The sequence $(y_n)$ is order-bounded,  and so, by Proposition \ref{6.11a},  $(y_n)$ is multi-bounded. Thus Lemma \ref{6.16e} applies to the sequence 
$(y_n)$ and the operator $T$ defined in that lemma.  The operator $T$ is bounded above by the positive operator 
$$
S: f\mapsto \langle f,\,1\rangle y\,,\quad E\to F\,,
$$ 
and so $T \in {\B}_r(E,F)$; clearly $S =\lv T \rv$, so that $\lV T \rV_r =\lV S \rV = \lV y\rV$.

By Lemma \ref{6.16e}, $T \in {\mathcal M}(E,F)$  with $\lV T \rV_{mb}\leq \sup_{n\in\N} \lV x_n\rV$. It follows that $$\lV y\rV \leq K\sup_{n\in\N} \lV x_n\rV\,,
$$
 and so $F$ has the weak $\sigma$-Nakano property.
\end{proof}\s

An example of a  Dedekind complete Banach lattice  without  the weak $\sigma$-Nakano property was given in Example \ref{2.41a}.

We now note that, even in the case where $E$ is a  monotonically complete lattice with the Nakano property, it is not necessarily the case that every
 compact operator on $E$ is multi-bounded.\s
 
 \begin{example}\label{6.16h}
{\rm   Let $n\in\N$. Essentially as  in \cite[Example 16.6]{AB}, there is  $T_n \in {\M}_{2^n}(\C)$  with  $\lV T_n \rV = 1$ and
 $\lV\,\lv  T_n\rv\, \rV = 2^{\,n/2}$ (where ${\C}^n$ has the Euclidean norm).  Let $E$ be the ${\ell}^{\,2}$-sum of 
the spaces $(\C^n, \norm_2)$ (not the $c_{\,0}$-sum  given in \cite{AB}). Then $E$ is a KB-space,  and so 
satisfies the conditions on $F$ in Theorem \ref{6.16c}(iv).  Let
$$
T((x_n)) = (2^{-n/3}T_nx_n)\quad ((x_n) \in E)\,.
$$
Then, as in  \cite{AB},  $T \in {\mathcal K}(E)$, but $T$ is not regular. Thus $T \in {\mathcal K}(E) \setminus {\mathcal M}(E)$.

As remarked in \cite[Example 16.6]{AB}, a compact operator need not have a modulus, and a  compact operator can have a modulus that is not 
compact (see also \cite{AW1}).}\qed
\end{example}\s
  
  In Examples \ref{6.10a} and \ref{6.35}, we showed that the multi-bounded version of Banach's isomorphism theorem  might fail. We now give another example of 
this failure; it applies even in the special case when we consider one Banach lattice and the lattice multi-norm.\s

\begin{example}\label{6.16g}
{\rm Let $E$ be the Banach lattice  $L^2(\T)$, and consider the lattice multi-norm based on $E$.  By Corollary   \ref{6.16i}, ${\B}_r(E)= {\mathcal M}(E)$. 

As in Example \ref{6.15c}, there exists  $T \in  {\mathcal K}(E)\cap {\B}_r(E)$ with 
 $\sigma_o(T) \supsetneq \sigma(T)$; choose  $z \in \sigma_o(T) \setminus \sigma(T)$. Then $zI_E-T \in {\mathcal M}(E)$  and $zI_E-T$ 
is invertible in ${\B}(E)$, so that $zI_E-T :E\to E$ is a linear  isomorphism. However, $zI_E-T$ is not invertible in the  Banach algebra ${\mathcal M}(E)$.}\qed
\end{example}\s

We now enquire when we have ${\mathcal M}(E,F) = {\B}(E,F)$.\s

\begin{theorem}\label{6.22}
Let $E$  and $F$ be Banach lattices. Suppose that either $E$ is an  $AL$-space or that $F$ is an $AM$-space.
Then ${\mathcal M}(E,F) = {\B}(E,F)$ and $\lV T\rV_{mb}= \lV T\rV\,\;(T\in {\B}(E,F))$.
\end{theorem}

\begin{proof} In the two cases, by Theorem \ref{2.3n}, the lattice multi-norms based on $E$ and $F$  are the maximum and minimum multi-norms, 
respectively.  The result now follows from Theorem \ref{6.8}.
\end{proof}\s

\begin{corollary}\label{6.23}
 Let $E$  and $F$ be Banach lattices.  Suppose that $F$ is a Dedekind complete $AM$-space   with an order-unit. Then
$$
{\B}_r(E,F)={\B}_b(E,F) ={\mathcal M}(E,F) = {\B}(E,F) 
$$
and  $\lV T \rV_{r} = \lV T \rV_{b} = \lV T \rV_{mb} =\lV T \rV\,\;(T\in {\B}(E,F)$.
\end{corollary}

\begin{proof}  This follows from Theorems \ref{6.16c}(iv) and \ref{6.22}, where we note that an $AM$-space   with an order-unit is monotonically bounded and has 
the Nakano property whenever it is Dedekind complete.
\end{proof}\medskip


\subsection{The multi-bounded multi-norm}  We shall now extend Theorem \ref{6.16c} by considering the multi-bounded multi-norm $(\norm_n^{mb}: n\in\N)$. 
We shall  show that, for all Banach lattices $E$  and  suitable   Banach lattices $F$, the multi-norm based on $ {\mathcal M}(E,F) $ 
 is not greater than  the lattice multi-norm, with equality when $E$ is the space $\ell^{\,1}$. However, an  example 
will show that these multi-norms are not necessarily equivalent  when $E=\ell^{\,p}$ for $p> 1$. 

We first note the following formula. Let $E$ and $F$ be Banach lattices, and take  $T_1,\dots,T_n \in  {\mathcal M}(E,F)$.
 Then it follows from equation (\ref{(6.10)}) that 
\begin{equation}\label{(6.11)}
\lV (T_1, \dots,T_n)  \rV_n^{mb} = \sup\left\{\lV\bigvee\{\lv T_ix_j \rv:i\in\N_n,\,j\in\N_k\}\rV \right\}\,,
\end{equation}
where the supremum is taken over all  $x_1,\dots,x_k\in E$ with   $\lV\,\lv x_1\rv \vee \cdots\vee \lv x_k\rv\,\rV\leq 1$. 

 Recall from
Theorem \ref{6.16c}(iv) that ${\B}_r(E,F) = {\B}_b(E,F)= {\mathcal M}(E,F)$, with 
equality of norms, whenever $F$ is Dedekind complete, monotonically bounded, and has the Nakano property, and so ${\mathcal M}(E,F)$ is
a Banach lattice with respect to the lattice multi-norm  $(\norm_n^L: n\in\N)$ in this case. \s

\begin{theorem}\label{6.16}
Let $E$ and $F$ be Banach lattices such that $F$ is Dedekind complete, monotonically bounded, and has the Nakano property. Let
$T_1,\dots,T_n \in  {\mathcal M}(E,F)$. Then 
\begin{equation}\label{(6.12)}
\lV (T_1, \dots,T_n)  \rV_n^{mb}\leq \lV \,\lv T_1\rv\vee \cdots \vee \lv T_n\rv\,\rV =  \lV (T_1, \dots,T_n)\rV_n^L\,.
\end{equation}
\end{theorem}

\begin{proof} 
We set $T = \lv T_1\rv\vee \cdots \vee \lv T_n\rv$.
Take $x_1,\dots,x_k\in E$  and set $x= \lv x_1\rv \vee \cdots\vee \lv x_k\rv$, so that  $\lV x \rV \leq 1$.  Since 
$\lv x_j\rv \leq x\,\;(j\in\N_k)$ and $ \lv T_i\rv \leq T\,\;(i\in\N_n)$, it follows from  Theorem \ref{6.13} that
$$
\lv T_ix_j \rv \leq \lv T_i\rv (\lv x_j\rv ) \leq\lv T_i\rv (x)\leq Tx\quad (i\in\N_n,\,j\in\N_k)\,,
$$
and so
$$
\lV\bigvee\{\lv T_ix_j \rv:i\in\N_n,\,j\in\N_k\}\rV \leq \lV Tx\rV \leq \lV T\rV\,.
$$
By  equation (\ref{(6.11)}), $\lV (T_1, \dots,T_n)  \rV_n^{mb} \leq  \lV T\rV$, as required.
\end{proof}\s

 \begin{theorem}\label{6.17}
Let $E$  be the Banach lattice  $\ell^{\,1}$,    and suppose that $F$ is a Dedekind complete, monotonically bounded Banach lattice
with  the Nakano property. Take $n\in\N$ and $T_1,\dots,T_n \in  {\mathcal M}(E,F)$.  Then
\begin{equation}\label{(6.13)}
\lV (T_1, \dots,T_n)  \rV_n^{mb}=  \lV \,\lv T_1\rv\vee \cdots \vee \lv T_n\rv\,\rV\,.
\end{equation}
\end{theorem}

\begin{proof}  Set $T = \lv T_1\rv\vee \cdots \vee \lv T_n\rv \in {\B}(E,F)^+$. By Theorem \ref{6.16}, $\lV (T_1, \dots,T_n)  \rV_n^{mb}\leq \lV T\rV$;
 we must prove the opposite inequality.
 
 We know that  $\lV T \rV =\sup_{i\in\N}\lV T(\delta_i)\rV$.  Take $i\in\N$.  Then the only way that we can write $\delta_i$ as $f_1 + \cdots + f_n$, 
where $f_1,\dots,f_n \in (\ell^{\,1})^+$ is to take $f_j = \alpha_j \delta_i$, where $\alpha_1,\dots,\alpha_n \in\I$ and $\alpha_1 + \cdots + \alpha_n= 1$.
 In this case,
$$
 \lv  T_1\rv(f_1)+ \cdots +  \lv  T_n\rv(f_n) =\alpha_1  \lv  T_1\rv(\delta_i) +\cdots +\alpha_n  \lv  T_n\rv(\delta_i) 
 \leq \lv  T_1\rv(\delta_i)\vee \cdots \vee \lv  T_n\rv(\delta_i)
$$
 using Proposition \ref{2.17}(v), and so  $T(\delta_i) = \lv T_1\rv(\delta_i)\vee \cdots \vee \lv  T_n\rv(\delta_i)$ by  equation (\ref{(1.3b)}).  Thus
$\lV T \rV = \sup_{i\in\N}\lV\,\lv T_1\rv(\delta_i)\vee \cdots \vee \lv  T_n\rv(\delta_i)\,\rV$.  However, by equation (\ref{(6.11)}),
$$
\lV (T_1, \dots,T_n)  \rV_n^{mb}\geq \lV\,\lv T_1\rv(\delta_i)\vee \cdots \vee \lv  T_n\rv(\delta_i)\,\rV\quad (i\in\N)\,,
$$
and so $\lV T\rV \leq \lV (T_1, \dots,T_n)  \rV_n^{mb}$, as required.
 \end{proof}\smallskip
 
 \begin{theorem}\label{2.3r}
Let $E$ be an $AM$-space,  and let $F$ be an $AL$-space. Then the lattice multi-norm on ${\B}_b(E,F)$ is the maximum multi-norm.
\end{theorem}

\begin{proof}  By \cite[Exercise 15.3, p.\ 263]{AB}, the Banach lattice ${\B}_b(E,F)$ is an $AL$-space. Thus the result follows from Theorem \ref{2.3n}(i).
\end{proof}\medskip

\begin{example}\label{6.18}
{\rm  We take $E = \ell^{\,p}$ and $F = \ell^{\,q}$, where $p, q\geq 1$. For  $n\in\N$, set $$e_n = \sum_{j=1}^n\delta_j=(1, \dots,1,0,\dots)\,.$$

For $j \in\N_n$, we define  $T_j : (\alpha_i)\mapsto \alpha_j e_n,\;\, E\to F$, so that $T_j\geq 0$ and $$\lV T_j\rV = \lV e_n\rV_{\ell^{\,q}} = n^{1/q}\,.$$
 Set $T = T_1\vee \cdots \vee T_n$.  Then, using (\ref{(1.3b)}), we see that 
 $$
 T(e_n) \geq \sum_{j=1}^n T_j(\delta_j) =ne_n\,,
 $$
 and so $\lV T \rV   \geq  n\,\cdot\, n^{ {1}/{q}- {1}/{p}} = n^{1+ {1}/{q}- {1}/{p}}$.
 
Now take $x_1,\dots, x_k\in E$ with $\lV\, \lv x_1\rv \vee \cdots \vee \lv x_k\rv\,\rV \leq 1$. Then each component of each $x_j$ 
has modulus at most $1$, and so $\lv T_ix_j\rv \leq e_n$ for $i\in\N_n$ and $j\in\N_k$. By (\ref{(6.11)}),
 $\lV (T_1, \dots,T_n)  \rV_n^{mb}\leq \lV e_n\rV_{\ell^{\,q}} = n^{1/q}$, and so 
$$\lV T \rV = \lV (T_1, \dots,T_n)  \rV_n^{L}\geq n^{1-1/p}\lV (T_1, \dots,T_n)  \rV_n^{mb}\,.
$$

This shows that  the multi-norms  $(\norm_n^{mb} : n\in\N)$ and  $(\norm^L_n  : n\in\N)$ based on ${\B}_b(E,F)$ are not equivalent whenever $p>1$.}\qed
\end{example}\medskip

\section{Extensions of multi-norms}

\noindent In this section, we shall show how to take various extensions of multi-norms.\s

\subsection{Definitions}

\noindent Let   $((E^n, \norm_n) : n\in \N)$  be a multi-normed space, and let   ${\mathcal F}$  be a fixed family in 
${\mathcal B }(E)_{[1]}$ such that   $I_E\in {\mathcal F}$.
Then we can define a  multi-norm structure on $\{E^n :  n\in\N\}$ by using ${\mathcal F}$: indeed, for $n\in\N$ and $x_1,\dots,x_n \in E$, set
\begin{equation}\label{(7.1)}
\LV (x_1,\dots,x_n)\RV_n = \sup\{\lV (T x_1,\dots,T x_n)\rV_n  : T  \in {\mathcal F}\}\,.
\end{equation}
We see  that $((E^n, \LV\,\cdot\,\RV_n) : n\in \N)$  is a multi-normed space and that $$\LV x\RV_n \geq \lV x \rV_n\quad (x \in E^n,\,n\in\N)\,;
$$ it is the  {\it extension\/}  of the given multi-norm by ${\mathcal F}$. \smallskip

In particular, let us take ${\mathcal F}$ to be the family ${\mathcal B }(E)_{[1]}$ or the family  of all isometric isomorphisms on $E$.  The  multi-normed
 structure that we obtain is the {\it balanced extension\/} or {\it isometric extension\/}, respectively. \smallskip

\begin{definition} \label{7.2}
A multi-normed space $((E^n, \norm_n) : n\in \N)$  is:\s

{\rm (i)} {\it balanced} if  $\lV T \rV_{mb} = \lV T \rV\,\;(T \in {\mathcal B}(E))$;\s

{\rm (ii)} {\it isometric} if  $\lV T \rV_{mb} = 1$ for each isometric isomorphism $T \in {\mathcal B}(E)$.
\end{definition}\smallskip

Thus $((E^n, \norm_n) : n\in \N)$ is balanced if and only if  $( {\mathcal M }(E), \norm_{mb})$ is isometrically isomorphic to $({\B}(E), \norm)$; since 
$\lV T \rV \leq  \lV T \rV_{mb}\,\;(T \in {\mathcal M}(E))$, this holds if and only if, for each $T \in {\mathcal B}(E)$ and $n\in\N$, we have 
 $$
\lV (Tx_1,\dots, Tx_n)\rV_n \leq \lV T \rV\,\lV (x_1,\dots, x_n)\rV_n\quad (x_1,\dots, x_n \in E)\,.
$$

Clearly a balanced multi-norm is isometric, and  the balanced or isometric extension of a multi-norm  is balanced or isometric, respectively.\s

\subsection{Examples of balanced multi-normed spaces}

\begin{example}\label{7.3}
{\rm  Let $E$ be any normed space, and let $$(\norm_n^{\min}: n\in\N)\quad {\rm  and }\quad (\norm_n^{\max}: n\in\N)$$ be the minimum and maximum multi-norms on  the family
$\{E^n : n\in \N\}$, respectively.  Then it follows from Theorem \ref{6.8} that both these multi-normed  spaces are balanced.
}\qed\end{example}\smallskip

\begin{example}\label{7.1}
{\rm   Let $E$ be a normed space,  and take $1\leq p\leq q< \infty$. Consider the $(p,q)$-multi-norm $(\norm_n^{(p,q)}: n\in \N)$   based 
on  $E$. Take $n\in\N$.

For each  $T\in {\B}(E)_{[1]}$, we have  $\lV T'\rV \leq  1$,  and so, by equation  (\ref{(3.2aa)}),  
 $$
\mu_{p,n}(T'\lambda_1,\dots, T'\lambda_n) \leq   \mu_{p,n}(\lambda_1,\dots, \lambda_n)\quad 
((\lambda_1,\dots, \lambda_n) \in (E')^n)\,.
$$
Let $(x_1,\dots,x_n)\in E^n$.  Since $\lv \langle Tx_i,\,\lambda_i\rangle \rv = \lv \langle  x_i,\,T'\lambda_i\rangle \rv \,\;(i\in\N_n)$, 
it follows from equation (\ref{(4.0d)})  that $\lV (Tx_1,\dots,Tx_n)\rV_n^{(p,q)} \leq  \lV (x_1,\dots,x_n)\rV_n^{(p,q)} $. Thus $((E^n, \norm_n^{(p,q)}) : n\in\N)$ 
is a balanced multi-normed space.}\qed
\end{example}\s

 The following result is a special case of \cite[Proposition 7.3]{DDPR1}.\s

\begin{theorem}\label{7.9}
Let $(\Omega, \mu)$  be a measure space,  and suppose  that  $1\leq p\leq q < \infty$ and   $L^p(\Omega, \mu)$ is infinite-dimensional. 
 Then the balanced extension of the standard  $q\,$-multi-norm based on   $L^p(\Omega, \mu)$ is the $(p,q)$-multi-norm.\qed
\end{theorem}\s

\subsection{Examples of isometric multi-normed spaces}
We now consider when  some  examples of  multi-normed spaces are isometric.\s

\begin{theorem}\label{7.5}
Let $(\Omega, \mu )$  be a measure space,  and  and suppose  that  $1\leq p\leq q < \infty$ with $p\neq 2$.   Then the standard
  $q\,$-multi-norm   based on $L^{p}(\Omega, \mu)$   is isometric.
\end{theorem}

\begin{proof}  Let $U$ be an isometric isomorphism  on $L^{p}(\Omega, \mu)$. Since $p\neq 2$,  $U$ has  the form  of equation (\ref{(7.2)}), where $\sigma$
 is a regular set isomorphism on $\Omega$ and $$\int _{\sigma (X)} \lv h \rv ^{\,p} {\dd}\mu_2 = \mu_1 (X)$$
for each measurable subset $X$ of $\Omega$.

For $n \in \N$, let ${\bf X} = (X_1, \dots , X_n)$ be an ordered partition of $\Omega$,  
 and define  
$$
Y_j =\sigma^{-1} (X_j)\quad (j\in\N_n)\,.
$$
 Then clearly  ${\bf Y} = (Y_1, \dots , Y_n)$ is  an ordered partition of $\Omega$.
For each $j \in \N_n$ and a measurable subset $X$ of $\Omega$, we  have
\begin{eqnarray*}
\int_{X_j}\lv U\chi_X\rv^p &=&  \int_\Omega \chi_{X_j}\lv h\rv^p\chi_{\sigma(X)}= \int_\Omega \lv h\rv^p\chi_{\sigma(X\cap Y_j)}\\
&=&
\int_\Omega \lv U\chi_{X\cap Y_j}\rv^p = \int_\Omega \chi_{X\cap Y_j}= \int_{Y_j}\chi_X\,,
\end{eqnarray*}
and so  $\int_{X_j}\lv Uf\rv^p= \int_{Y_j}\lv f \rv^p$ for all $f\in L^{p}(\Omega, \mu)$.
Take  $f_1,\dots, f_n \in L^{p}(\Omega, \mu)$. Then 
$r_{\bf X}((Uf_1,\dots,Uf_n))  =   r_{\bf Y}((f_1,\dots,f_n))$.

It follows from the definition  in equation (\ref{(4.1a)}) that   $$\lV(Uf_1, \dots, Uf_n)\rV_n^{[q]} = \lV(f_1, \dots, f_n)\rV_n^{[q]}\,,
$$
and hence we obtain an isometric multi-norm.
\end{proof}\smallskip

\begin{example}\label{7.8}
{\rm  In this example, we shall show that  the constraint that $p\neq 2$ in Theorem \ref{7.5} is necessary.

Set $H = \ell^{\,2}$, and consider Example \ref{6.10}.  In that example, we obtained an orthonormal subset $S= \{x_r^s : r\in \N_s, s\in\N\}$ of $H$.
 As before, enumerate $S$ as a sequence $(y_n)$, and  now choose a sequence $T=  (z_n)$ in $H$ such that $S\cup T$ is an orthonormal basis of $H$. 
 Define a bounded linear operator  $U\in {\mathcal B}(H)$ by requiring that
$$
Uy_n = \delta_{2n}\,,\quad Uz_n = \delta_{2n-1}\quad (n\in\N)\,.
$$
Clearly, $U$ is an isometric isomorphism  on $H$.

Consider  the standard $2\,$-multi-norm  on $\{H^n : n \in \N\}$. As in Example \ref{6.10} (in the elementary case where
  $\alpha_i= 1 \,\;(i\in \N)$), $U$ is not even a multi-bounded map with respect to this multi-norm.}\qed
\end{example}
\medskip

\chapter {Orthogonality and duality}

\noindent In this final chapter, we shall  discuss a notion of orthogonality in multi-normed spaces; we are seeking 
a theory of orthogonality  involving multi-norms that extends the classical notions of orthogonality in Hilbert spaces and Banach lattices 
to more general Banach spaces. These ideas will be used to define the multi-dual of a multi-normed space; our motivation is to try to 
establish a satisfactory duality theory for general multi-normed spaces.

A `test question' for our approach is the following. Let $E= L^{p}(\Omega)$, where $\Omega$ is a measure space and $1< p< \infty$, and 
let $\{E^n : n\in\N\}$ have the standard $p\,$-multi-norm $(\norm_n^{[p]}: n\in\N)$ of Definition \ref{4.1a}. 
Let $q$ be the conjugate index to $p$, and set $F = E' = L^q(\Omega)$. 
Then we expect that the `multi-dual' of the multi-normed space $((E^n, \norm_n^{[p]}) : n\in\N)$ should be $((F^n, \norm_n^{[q]}) : n\in\N)$, and hence that
$$((E^n, \norm_n^{[p]}) : n\in\N)$$  is `multi-reflexive'. We also expect that the `multi-dual' of the lattice multi-norm on the multi-normed space
$((E^n, \norm_n^L) : n\in\N)$, where $E$ is a Banach lattice, will be the lattice multi-norm on  $\{(E')^n : n\in\N\}$.  We should formulate the notion
 of `multi-dual' to achieve these aims.  This seems to be not completely straightforward.

In this chapter, we consider  Banach spaces over only the complex field.\medskip

\section{Decompositions}

\noindent  We recall that the notion of a direct sum decomposition of a Banach space was  given $\S1.2.3$; this included the notion of a
 `closed family of decompositions'.\s

\subsection{Hermitian decompositions of a normed space}

The first  decomposition that we consider is essentially known, and does not involve multi-norms.\s

\begin{definition}\label{11.1}
Let $ E = E_1\oplus \cdots \oplus E_k$ be a direct sum decomposition of a normed space $(E, \norm)$.  Then the decomposition is {\it hermitian} if  
\begin{equation}\label{(11.1a)}
\lV \zeta_1 x_1 + \cdots + \zeta_kx_k \rV \leq \lV x_1 + \cdots + x_k \rV
\end{equation}
whenever $\zeta_1, \dots, \zeta_k \in \overline{\D}$ and $x_1\in E_1, \dots, x_k \in E_k$.
\end{definition}\smallskip

In particular, we see that
$\lV \zeta_1 x_1 + \cdots + \zeta_kx_k \rV = \lV x_1 + \cdots + x_k \rV$
when  $\zeta_1, \dots, \zeta_k \in {\T}$ and $x_1\in E_1, \dots, x_k \in E_k$.  Further, it follows from a simple remark on page~\pageref{simpleremark}
that this condition implies that the decomp\-osition is hermitian. 

The reason for the above terminology (suggested by \cite{Ka}) is the following.   
Suppose that $E= F\oplus G$ is a decomposition.  Then the decomp\-osition is hermitian if and only if 
$\lV \zeta x+  y \rV= \lV x+  y \rV\,\;(x\in F,\,y\in G, \zeta \in \T)$.
Let $P:E\to F$ be the projection. Then 
$$
\exp({\rm i}\theta  P)(x+y) = {\rm e}^{{\rm i}\theta}x + y\quad (x\in F,\,y\in G,\,\theta \in \R)\,,
$$
 and so the decomposition is  hermitian if and only if $P$ is a hermitian operator.

We see that  trivial decompositions are hermitian.  For example, let us identify 
$\C^n$ as $\C \oplus \cdots \oplus \C$, and suppose that $\norm$ is a norm on $\C^n$.  Then this  decomposition is a hermitian decomposition
of $(\C^n, \norm)$  if and only if  $\norm$ is a lattice norm on $\C^n$.

A decomposition $E=F\oplus G$ of a Banach space $E$ is said to be an {\it $M$-decomposition\/} if $\lV y+z \rV =\max\{\lV y\rV, \lV z \rV\}$ and an 
 {\it $L$-decomposition\/}  if $\lV y+z \rV = \lV y\rV+\lV z \rV $   for all $y\in F$ and $z\in G$; in these cases, $F$ and $G$ are $M$- and 
{\it $L$-summands\/}, respectively.    Clearly, $M$- and $L$- decompositions are hermitian. See \cite{HWW} for a discussion of $M$- and $L$- decompositions.
 
There have been many generalized versions of `orthogonality' in the theory of normed linear spaces; our concept of a hermitian decomposition $E=F\oplus G$ 
 implies that we have 
$\lV x- y\rV  =\lV x+y\rV$ for each $x \in F$ and $y\in G$; thus $x$ and $y$ are `isosceles orthogonal' in the sense of \cite[Definition 2.1]{James}. Indeed, 
$\lV x- ky\rV  =\lV x+ky\rV$ for each $x \in F$, $y\in G$, and $k\in \C$, and so  $x$ and $y$ are `orthogonal' in the sense of the early paper 
\cite{Roberts}.  See also the notion of {\it $h$-summand\/} in \cite{GKS}.\s

 \begin{definition}\label{11.1a}
Let $(E, \norm)$ be a normed space. Then the family of all hermitian decomp\-ositions of $E$  is ${\mathcal K}_{\rm herm}$.
\end{definition}\smallskip
 
 It is clear that  ${\mathcal K}_{\rm herm}$ is a closed family of direct sum decompositions.  Let $(E, \norm)$ be a  normed space,
 and consider a family $\mathcal K $ of hermitian decomp\-ositions of $E$. Then  the smallest closed  family $\mathcal L$ of hermitian decompositions of $E$  
such that $\mathcal L$  contains $\mathcal K $ is the  {\it hermitian closed family generated by\/}
 $\mathcal K$. \smallskip
 
 \begin{example}\label{8.10}
{\rm   Let $B$ be the subset of $\C^{\,2}$ which is the absolutely convex hull of the set 
consisting of the three points $(1,0)$, $(0,1)$, and $(2,2)$. Then $B$ is the closed unit ball of a norm, say $\norm$,  on $\C^{\,2}$. 
Then the obvious direct sum decomposition  
$$
\C^{\,2} = (\C \times\{0\})\oplus (\{0\}\times \C)
$$
  is not  a hermitian decomposition of $(\C^{\,2}, \norm)$.  Indeed, $\lV (2,2)\rV = 1$, but $\lV (2,0)\rV =2$. \qed}
\end{example}\s

 \begin{example}\label{8.10a}
{\rm  Let $E = \ell^{\,p}_2$, where $p\geq 1$. Then $E= (\C \times\{0\})\oplus (\{0\}\times \C)$ is a hermitian decomposition.

We consider which other non-trivial direct sum decompositions of $E$ are hermitian. Indeed, for $\alpha \in\C$, set 
$E_\alpha= \{(z,w) \in \C^2: w=\alpha z\}$.  
Then   $E =E_\alpha\oplus E_\beta$ whenever $\alpha\neq \beta$, and every such decomposition has this form for some 
$\alpha, \beta\in \C$ with $\alpha\neq \beta$, say $\alpha \neq 0$.  Take $x_1 = (1,\alpha)\in E_\alpha$ 
and $x_2 = (\zeta, \beta \zeta)\in E_\beta$, where $\zeta \in \C$. Then  $\lV x_1 + x_2 \rV = \lV x_1 - x_2 \rV $ only if 
\begin{equation}\label{(7.10)}
\lv 1+ \zeta\rv^p + \lv \alpha + \beta \zeta\rv^p = \lv 1- \zeta\rv^p + \lv \alpha - \beta \zeta\rv^p\quad (\zeta \in \C)\,.
\end{equation}
Thus $\beta\neq 0$.  In the case where $-\beta/\alpha \not\in \R^+$, there exists $\zeta \in \T$ with $\Re \zeta >0$ and $\Re (\beta \zeta/\alpha)>0$, and then 
$\lv 1+ \zeta\rv  > \lv 1- \zeta\rv$ and $\lv \alpha + \beta \zeta\rv > \lv \alpha - \beta \zeta\rv$, a contradiction of (\ref{(7.10)}). 
 Thus we have $ \beta =-\alpha r$ for some $r>0$.  For $t\in\R$ with $\lv t\rv < \min\{r,1\}$, we have 
\begin{equation}\label{(7.11)}
( 1+ t)^p - (1- t)^p =   \lv \alpha\rv^p((1 + rt)^p - ( 1- rt)^p)\,.
\end{equation} 

Suppose that $p\neq 1,2$. Then, by equating the first and third derivatives at $t=0$ of both sides of (\ref{(7.11)}), we see that $ \lv \alpha\rv^pr= \lv\alpha\rv^pr^3=1$,
 and so $r= \lv \alpha\rv =1$, say $\alpha ={\rm e}^{{\rm i}\theta}$,  and then $\beta =-{\rm e}^{{\rm i}\theta}$.  

Suppose that $p=1$. Then, from (\ref{(7.11)}), $ \lv \alpha\rv r= 1$, and so, from (\ref{(7.10)}),
$$
\lv \zeta + 1\rv - \lv \zeta - 1\rv = \lv \zeta + 1/r\rv - \lv \zeta - 1/r\rv \quad (\zeta \in \C)\,.
$$
By taking $\zeta  = 1+{\rm i}$, we see that this is only possible  when $r=1$, and again we have  $\alpha ={\rm e}^{{\rm i}\theta}$,  
and then $\beta =-{\rm e}^{{\rm i}\theta}$. 

Thus, for $p\neq 2$, we obtain a hermitian decomposition only if $\alpha ={\rm e}^{{\rm i}\theta}$   and   $\beta =-{\rm e}^{{\rm i}\theta}$ for
 some  $\theta \in [0,2\pi)$.  But finally take $x_1= (1,{\rm e}^{{\rm i}\theta})$ and $x_2= (1,-{\rm e}^{{\rm i}\theta})$.  Then 
$$
\lV x_1+ x_2\rV = 2^p \neq 2\,\cdot\,2^{p/2} = \lV x_1+ {\rm i}x_2\rV\,,
$$
and so there is no hermitian decomposition  of this form.

Thus  the only hermitian decompositions of $E = \ell^{\,p}_2$ for $p\geq 1$ and $p\neq 2$ are $$
E= (\C \times\{0\})\oplus (\{0\}\times \C)\quad {\rm  and}\quad  E= (\{0\}\times \C)\oplus (\C \times\{0\})\,.$$

A  similar argument shows that this is also true for $E = \ell^{\,\infty}_2$.
 
 Suppose that $p=2$. Then, from (\ref{(7.10)}), $\Re (\zeta) =- \Re(\overline{\alpha}\beta\zeta)$ for all $\zeta \in \C$, and so there exist 
$\theta \in [0,2\pi)$ and $r>0$ with  $\alpha =r{\rm e}^{{\rm i}\theta}$ and $\beta =-{\rm e}^{{\rm i}\theta}/r$. Each   decomposition  
corresponding to  such a choice of $\alpha $ and $\beta$ is hermitian.

More general results about hermitian decompositions of $\ell^{\,p}$ follow from   Theorem \ref{1.17} and \cite[Theorem 5.2.13]{FJ}. 
\qed}
\end{example}\s

Let $E = E_1\oplus \cdots \oplus E_k$ be a hermitian decomposition of a normed space $E$. Then
 the maps $P_j$ are continuous, and    $\lV P_j\rV =1$ when $E_j \neq\{0\}$ (even in the case where  $(E, \norm)$ is not necessarily complete), 
and the maps $P_j' :E_j' \to E'$ are isometric embeddings. Again, $E' = E_1'\oplus \cdots \oplus E_k'$.

\begin{proposition}\label{11.3}
Let $ E = E_1\oplus \cdots \oplus E_k$ be a hermitian decomposition of a normed space $(E, \norm)$.
 Then the decomposition  $E' = E_1'\oplus \cdots \oplus E_k'$ is  also hermitian.
\end{proposition}

\begin{proof}   
Let $\zeta_1, \dots, \zeta_k \in \overline{\D}$ and $\lambda _i\in E_i'\,\;(i\in \N_k)$. Then
\begin{eqnarray*}
\lV \zeta_1 \lambda_1 + \cdots + \zeta_k\lambda_k \rV &= &
\sup_{x \in E_{[1]}} \lv \langle x, \zeta_1\lambda_1\rangle + \cdots + \langle x, \zeta_k\lambda_k\rangle \rv \\
&= & \sup_{x \in E_{[1]}} \lv \langle \zeta_1 P_1x, \lambda_1\rangle + \cdots + \langle \zeta_kP_kx, \lambda_k\rangle \rv \\
&= & \sup_{x \in E_{[1]}} \lv \langle \zeta_1 P_1x+ \cdots + \zeta _kP_kx,\lambda _1+ \cdots + \lambda _k\rangle\rv\,.
\end{eqnarray*}
But $$\lV \zeta_1 P_1x+ \cdots + \zeta_kP_kx \rV\leq \lV  P_1x+ \cdots + P_kx \rV =\lV x \rV \leq 1\quad (x \in E_{[1]})\,,
$$
 and it follows that 
$\lV \zeta_1 \lambda_1 + \cdots + \zeta_k\lambda_k \rV \leq \lV  \lambda_1 + \cdots + \lambda_k \rV$, giving the result. 
\end{proof}\smallskip

The above result also follows from \cite[\S9, Corollary 6(ii)]{BD1}, where it is stated  that $P'\in {\B}(E')$ is hermitian if and only if $P \in {\B}(E)$ is hermitian.\s

\begin{proposition}\label{10.5a}
Let $(E, \norm )$ be a normed space, and  let $k\in \N$.   Suppose that $E$ has two hermitian decompositions
$$
E = E_1 \oplus \cdots \oplus E_k = F_1 \oplus \cdots \oplus F_k\,.
$$
 For $j\in \N_k$, let  $Q_j :E \to F_j$  be the natural projections. Then  
\begin{equation}\label{(10.1a)}
\lV Q_1x_1+ \cdots + Q_kx_k\rV \leq\lV x_1+ \cdots + x_k\rV\quad (x_1\in E_1,\dots, x_k\in E_k)\,.
\end{equation}
\end{proposition}

\begin{proof}  Set $\zeta = \exp(2\pi{\rm i}/k)$. Then we note that
$$
Q_{\ell} =\frac{1}{k}
\sum_{i=1}^k   \sum_{j=1}^k  \zeta^{j(i-\ell)}Q_i   \quad (\ell \in \N_k)\,.
$$

Take $x_i \in E_i \,\;(i\in \N_k)$.  Then
\begin{eqnarray*}
\lV Q_1x_1+ \cdots + Q_kx_k\rV &=&
\frac{1}{k}
\lV \sum_{\ell =1}^k \sum_{i=1}^k   \sum_{j=1}^k  \zeta^{j(i-\ell)}Q_ix_{\ell}\rV\\
&\leq & \frac{1}{k} \sum_{j=1}^k
\lV \sum_{i =1}^k \sum_{\ell =1}^k  \zeta^{j(i-\ell)}Q_ix_{\ell}\rV\\
&= & \frac{1}{k} \sum_{j=1}^k \lV \left(\sum_{i =1}^k\zeta^{ji}Q_i\right)\left( \sum_{\ell =1}^k \zeta^{-j\ell}x_{\ell}\right) \rV\\
&= & \frac{1}{k} \sum_{j=1}^k \lV \sum_{\ell =1}^k \zeta^{-j\ell}x_l\rV
\end{eqnarray*}
because  the decomposition   $E = F_1 \oplus \cdots \oplus F_k$ is hermitian, and so
$$
\lV Q_1x_1+ \cdots + Q_kx_k\rV \leq \frac{1}{k} \sum_{j=1}^k\lV  \sum_{\ell =1}^k \zeta^{-j\ell}x_{\ell}\rV= \lV   x_1+\cdots + x_k\rV
$$
because the decomposition $E = E_1 \oplus \cdots \oplus E_k$ is hermitian.  Equation (\ref{(10.1a)}) follows.
\end{proof}\s

We now give some examples of hermitian decompositions of particular Banach spaces.\s

  \begin{theorem}\label{10.5ba}
Let $K$ be a compact space, and let $C(K) = E_1\oplus \cdots \oplus E_k$ be a hermitian decomposition.
 Then there exist  clopen subspaces $K_1,\dots,K_k$ of $K$ such that $E_j= C(K_j)\,\;(j\in \N_k)$.  In particular, in the case where $K$ is connected, 
there are no non-trivial hermitian decompositions of $C(K)$.
\end{theorem}
 
   \begin{proof} Take  $j\in \N_k$,  and let $P_j$ be the projection of $C(K)$ onto $E_j$, so that $P_j$ is a hemitian operator. By Theorem \ref{1.16},
 there exists $h_j\in C_\R(K)$ with  $P_jf=h_jf\,\;(f\in C(K))$.  Since $P_j=P_j^2$, we have $h_j=h_j^2$ in $C(K)$, and so $h_j$ is the 
characteristic function of a subset, say  $K_j$, of $K$.  Clearly, $K_j$ is clopen and $E_j=C(K_j)$.  
 \end{proof}\s
 
   \begin{proposition}\label{10.5bb}
Take  $p\in [1,\infty]$ with $p\neq 2$, and let $\ell^{\,p}= E_1\oplus \cdots \oplus E_k$
 be a hermitian decomposition. Then there exist subsets $S_1, \dots, S_k$ of $\N$ such that $E_j= \ell^{\,p}(S_j)\,\;(j\in\N_k)$.
\end{proposition}

\begin{proof} This follows similarly, now using Theorem \ref{1.17}. \end{proof}
\medskip

\subsection{Small decompositions of multi-normed spaces}
  We now turn to decompositions of normed spaces $E$ with respect to multi-norms based on $E$.\smallskip

\begin{definition}\label{10.24}
Let   $((E^n, \norm_n) : n\in \N)$ be a  multi-normed space,  let $k\in \N$,and let  $E = E_1\oplus \dots\oplus E_k$  be a direct sum decomposition of $E$. 
 Then the decomposition is {\it small} (with respect to the multi-norm) if
$$
\lV P_1x_1+\cdots + P_kx_k\rV \leq \lV (x_1,\dots, x_k)\rV_k\quad (x_1,\dots, x_k \in E)\,.
$$
\end{definition}\smallskip

We shall see in Example \ref{10.24c} that the notion of a small decomposition of a  normed space $E$ depends on the multi-norm $((E^n, \norm_n) : n\in \N)$, 
and is not intrinsic to the normed space $E$.

Clearly $\lV P_j\rV \leq 1\,\;(j\in\N_n)$ for each such small decomposition.\smallskip

\begin{proposition}
\label{10.24e}
Let   $((E^n, \norm_n) : n\in \N)$ be a  multi-normed space, and suppose that  $E = E_1\oplus \dots\oplus E_k$ is a  small decomposition of $E$.
  Then the decomposition is   hermitian. Further,  
 \begin{equation}\label{(11.1b)}
 \lV (x_1,\dots,x_k)\rV_k = \lV x_1 + \cdots+ x_k\rV\quad (x_1\in E_1, \dots, x_k\in E_k)\,.
 \end{equation}
\end{proposition}

\begin{proof}   Take  $\zeta_1, \dots, \zeta_k \in \overline{\D}$ and $x_1\in E_1, \dots, x_k \in E_k$, and then  set $x = x_1 + \cdots +x_k$. 
Clearly  $P_jx= x_j\,\;(j\in\N_k)$, and so
\begin{eqnarray*}
\lV \zeta_1x_1 + \cdots + \zeta_1x_1\rV
&=&
\lV P_1(\zeta_1x)+ \cdots + P_k(\zeta_kx)\rV  \leq  \lV (\zeta_1x, \dots,\zeta_kx)\rV_k\\
&\leq &
\lV (x, \dots,x)\rV_k = \lV x \rV = \lV x_1 + \cdots +x_k \rV\,,
\end{eqnarray*}
and so the decomposition is hermitian.\s

Now take $x_1\in E_1, \dots, x_k\in E_k$, and set $\zeta =\exp(2\pi{\rm i}/k)$. Then
\begin{eqnarray*}
\lV x_1 + \cdots+ x_k\rV &=& \lV P_1x_1 + \cdots+ P_kx_k\rV  
\leq   \lV (x_1,\dots,x_k)\rV_k\\ 
&\leq &  \frac{1}{k}\sum_{j=1}^k \lV \sum_{m=1}^k \zeta^{jm}x_m\rV\quad \mbox{by Proposition \ref{2.5a}} \\
& \leq &  \max_{j\in \N_k}\lV \sum_{m=1}^k \zeta^{jm}x_m\rV \leq \lV x_1 + \cdots+ x_k\rV\,,
\end{eqnarray*}
which gives the equality (\ref{(11.1b)}).
\end{proof}\smallskip

\begin{example}\label{10.24b}
{\rm Let $E = \ell^{\,p}(\N)$, where $p \geq 1$, and consider the lattice multi-norm  based on $E$, namely $(\norm^L_n : n\in\N)$; by Example \ref{2.3ca}, this is the 
standard $p\,$-multi-norm on $E$. 

For $k\in\N$, take $(S_1, \dots,S_k)$ to be an ordered partition of $\N$, and set $E_i = \ell^{\,p}(S_i)$ for $i\in \N_k)$.
Then it is clear that $E = E_1\oplus \cdots \oplus E_k$ is a  small decomposition with respect to the 
lattice multi-norm because  
$$
\lV f_1\mid S_1 + \cdots+   f_k\mid S_1\rV \leq \lV \,\lv f_1\rv \vee\cdots \vee \lv f_k\rv\,\rV = \lV (f_1,\dots,f_k)\rV_k^L
$$
for all $f_1,\dots,f_k \in E$. The collection of all such decompositions is a closed family. 
}\qed\end{example}\smallskip

The following remark will be generalized later, in  Theorem \ref{10.3g}.\smallskip

\begin{proposition}\label{10.24d}
Let $(E, \norm )$ be a normed space, and suppose that $E=E_1\oplus E_2$ is a hermitian decomposition of $E$. For $x_1,x_2\in E$, set
$$
\lV (x_1,x_2)\rV_2 = \max\{
\lV x_1\rV, \lV x_2\rV, \lV P_1x_1+P_2x_2\rV, \lV P_1x_2+P_2x_1\rV\}\,.
$$
Then $(\norm, \norm_2)$ is a multi-norm of level 2 on $\{E,E^{\/2}\}$, and the direct sum decomposition $E=E_1\oplus E_2$  is small with respect to this multi-norm.
\end{proposition}

\begin{proof}  It is clear that $\norm_2$ is a norm on $E^{\/2}$ and that  $\norm_2$ satisfies (A1); $\norm_2$ satisfies (A2) because the decomposition  is hermitian. 

 Let $x \in E$. Then $\lV (x,0)\rV_2 = \lV x \rV $ because $\lV P_1\rV, \lV P_2\rV \leq 1$, so that (A3) holds, and
$\lV (x,x)\rV_2 =  \lV x \rV $ because $P_1x+P_2x =x$,  so that (A4) holds. Thus $(\norm, \norm_2)$ is a multi-norm of level 2 on $\{E,E^{\/2}\}$.

Clearly the decomposition $E=E_1\oplus E_2$  is small with respect to the multi-norm $(\norm, \norm_2)$.
\end{proof}\smallskip

\begin{definition}\label{11.1b}
Let   $((E^n, \norm_n) : n\in \N)$ be a  multi-normed space. Then the family of all small  decompositions of $E$ is ${\mathcal K}_{\rm small}$.
\end{definition}\smallskip

\begin{proposition}
\label{10.25}
Let   $((E^n, \norm_n) : n\in \N)$ be a  multi-normed space.  Then ${\mathcal K}_{\rm small}$ is a closed family of direct sum decompositions.
\end{proposition}

\begin{proof}  Clearly,
Axiom  (C1) of Definition \ref{10.5} is satisfied, and   (C3) is trivially satisfied .

Take  $k \geq 3$, and let $E =E_1\oplus \dots\oplus E_k$ be a small decomposition.  Take $x, x_3, \dots, x_k \in E$. Then the projection 
of $x$ with kernel $E_3\oplus \cdots\oplus E_k$ onto the space $E_1\oplus E_2$ is $P_1x + P_2x$, and so
$$
\lV (P_1x + P_2x)+ P_3x_3+ \cdots+ P_{k}x_{k}\rV \leq  \lV (x,x,x_3, \dots ,x_{k})\rV_k 
=\lV (x,x_3, \dots ,x_{k})\rV_{k-1}\,.
$$
Hence  $E = (E_1\oplus E_2)\oplus E_3\oplus \dots\oplus E_k$ is a small decomposition of $E$, and so Axiom (C2) is satisfied.
 
Thus ${\mathcal K}_{\rm small}$ is a closed family.
\end{proof}\s

\subsection{Orthogonal decompositions of multi-normed spaces}

\noindent  We now move to consideration of orthogonal decompositions of multi-normed spaces. It will be seen later that such decompositions 
generalize various classical notions of orthogonality.\smallskip

  Let $E$ be a linear space.  We recall that a `coagulation' of  an element  $(x_1,\dots, x_n) \in E^n$ was defined on page~\pageref{coagulation}.\s
 
\begin{definition}\label{10.1}
Let   $((E^n, \norm_n) : n\in \N)$ be a  multi-normed space, let $k\in\N$, and let $E = \{E_1, \dots, E_k\}$ be family of closed subspaces of $E$. Then
 $\{E_1,\dots,E_k\}$  is an {\it orthogonal family} in  $E$ if, for each  $x_1\in E_1, \dots, x_k\in E_k$ and each coagulation $(y_1,\dots,y_j)$ of
 $(x_1,\dots,x_k)$, we have 
$$
  \lV (y_1,\dots, y_j)\rV_j= \lV(x_1,\dots, x_k)\rV_k\,.
$$
  A subset $\{x_1,\dots, x_k\}$ of  $E$ is {\it orthogonal} if the  family  $\{\C x_1,\dots, \C x_k\}$ of subspaces is an orth\-ogonal family.
\end{definition}\smallskip

Again, the notion of an orthogonal family depends on the multi-norm structure; it is not intrinsic to the normed space $E$.  The definition depends on only  the set
$\{E_1,\dots,E_k\}$, and not on the ordering of the spaces $E_1,\dots,E_k$. 

For example, a trivial direct sum decomposition of $E$ is orthogonal for any multi-normed space $((E^n, \norm_n) : n\in \N)$; this follows from the basic Axiom (A3).

Let  $\{E_1,\dots,E_k\}$  be an  orthogonal family of subspaces of $E$. Then certainly
\begin{equation}\label{(10.1)}
\lV (x_1,\dots, x_k)\rV_k = \lV x_1 + \cdots + x_k\rV\quad (x_1\in E_1, \dots, x_k\in E_k)\,.
\end{equation}
Indeed, suppose that $x_i\in E_i\,\;(i\in\N_k)$. Then
\begin{equation}\label{(10.1d)}
\lV (x_1,\dots, x_k)\rV_k = \lV\zeta_1 x_1 + \cdots + \zeta_kx_k\rV\quad (\zeta_1, \dots,\zeta_k \in \T)\,.
\end{equation}

\begin{lemma}
\label{10.2}
 Let   $((E^n, \norm_n) : n\in \N)$ be a  multi-normed space, let $k \in \N$, and let $\{E_1,\dots,E_k\}$ be an orthogonal family in $E$. Then:\smallskip

{\rm (i)}    for $i,j \in \N_k$ with $i\neq  j$, we have $ E_i\cap E_j = \{0\}$\,; \smallskip 

{\rm (ii)} $\{E_1\oplus E_2, E_3, \dots, E_k\}$ is an  orthogonal family in $E$ {\rm (}whenever  $k\geq 3${\rm)};\smallskip

{\rm (iii)}  for $j \in \N_k$ such that $E_j\neq\{0\}$, the norm of the projection from  $(E_1\oplus \cdots \oplus E_k, \norm)$  onto $(E_j, \norm)$ is $1$.
\end{lemma}

\begin{proof}  These are immediate. \end{proof}\smallskip

\begin{definition}\label{10.3}
Let $((E^n, \norm_n) : n\in \N)$ be a  multi-normed space, let $k\in\N$, and let  $E = E_1\oplus\cdots\oplus E_k$
 be a direct sum decomposition.  Then  the decomposition is  {\it orthogonal} (with respect to the multi-norm  of $E$) if  $\{E_1,\dots,E_k\}$ is an orthogonal family.
\end{definition}\smallskip

We make the following remark, without proof. 

 Let   $((E^n, \norm_n) : n\in \N)$ be a  multi-normed space,  and let $\mathcal K$ be a closed family of 
hermitian decompositions of $E$.  Suppose that, for each decomposition $E =E_1\oplus \cdots\oplus E_k$ in $\mathcal K$, we have 
$$
\lV (x_1,\dots, x_k)\rV_k \geq \lV x_1 + \cdots + x_k\rV\quad (x_1\in E_1, \dots, x_k\in E_k)\,.
$$
Then each decomposition in $\mathcal K$ is orthogonal.\s

\begin{definition}\label{11.1c}
 Let   $((E^n, \norm_n) : n\in \N)$ be a  multi-normed space. Then the family of all orthogonal   decompositions of $E$ is ${\mathcal K}_{\rm orth}$.
\end{definition}\smallskip

 \begin{proposition}\label{10.2a}
 Let   $((E^n, \norm_n) : n\in \N)$ be a  multi-normed space. Then ${\mathcal K}_{\rm orth}$ is a closed family of direct sum decompositions.
\end{proposition}

\begin{proof}  Clearly trivial  direct sum decompositions of $E$ are orth\-ogonal, and so this follows from  Lemma \ref{10.2}.
\end{proof}\smallskip

\begin{theorem}\label{10.3f}
Let   $((E^n, \norm_n) : n\in \N)$ be a  multi-normed space. Then:\s

{\rm (i)}  each orthogonal decomposition of $E$ is hermitian;\s 

{\rm (ii)} each small decomposition of $E$   is orthogonal.
\end{theorem}

\begin{proof} (i) This is immediate from equation (\ref{(10.1d)}).\s

(ii) Let $E = E_1\oplus \dots\oplus E_k$ be a small decomposition of $E$, and then  take elements $x_1\in E_1, \dots, x_k\in E_k$.  Suppose that  $(y_1,\dots,y_j)$ 
is a coagulation  of $(x_1,\dots,x_k)$, and let $\{S_j : j\in \N_k\}$ be a partition   of $\N_n$ such that  
$$
y_j = \sum\{x_i : i \in S_j\}\quad (j\in \N_k)\,.
$$ 
 Set $F_j= \oplus\{E_i : i\in S_j\}\,\;(j\in \N_k)$.  Then $E =F_1\oplus \dots\oplus F_j$ is a direct sum decomposition of $E$,
 and, by Proposition  \ref{10.25}, it is a small decomposition of $E$.  By equation (\ref{(11.1b)}), 
$\lV (y_1,\dots, y_j)\rV_j = \lV y_1+ \cdots +y_j\rV$ and  $\lV (x_1,\dots, x_k)\rV_k = \lV x_1+ \cdots +x_k\rV$. But 
$\lV y_1+ \cdots +y_j\rV = \lV x_1+ \cdots +x_k\rV$, and so $\lV (y_1,\dots,y_j)\rV_j= \lV (x_1,\dots, x_k)\rV_k$. Thus the decomposition is orth\-ogonal.
\end{proof}\smallskip

{\bf Question}  Let $(\norm_n : n\in\N)$ be  multi-norm based on a Banach space $E$. We regret that we do not know whether  every orthogonal decomposition
 with respect to this multi-norm is  necessarily small.  If this is not true in  general, one could seek classes of multi-norms or of Banach spaces $E$
 for which it is true. \s

\begin{proposition}\label{10.24f}
Let $E$ be a Banach space.  Then every orthogonal  decomposition of $E$ with respect to the minimum multi-norm is small with respect to this multi-norm.
\end{proposition}

\begin{proof}  Take $k\in\N$, and let $ E = E_1\oplus \cdots \oplus E_k$ be an orthogonal, and hence hermitian,  decomposition of $E$. 
Then $\lV P_j\rV \leq 1\,\;(j\in \N_k)$.

Take $x_1,\dots, x_k \in E$. Since the decomposition is orthogonal, we have  
$$
\lV P_1x_1+\cdots + P_kx_k\rV = \lV (P_1x_1,\dots, P_kx_k)\rV_k^{\min}\,,
$$
 and so 
$$
\lV P_1x_1+\cdots + P_kx_k\rV \leq   \max\{\lV x_j\rV :j\in\N_n\} = \lV (x_1,\dots, x_k)\rV_k^{\min}\,.
$$
Hence the decomposition is small.
\end{proof}\smallskip

\begin{proposition}\label{10.3e}
 Let   $((E^n, \norm_n) : n\in \N)$ be a  multi-normed space. Let $k\in \N$,  and let $E = E_1\oplus \cdots\oplus E_k$ be an  
 orthogonal decomposition  of $E$, with corresponding projections $P_1,\dots,P_k$. Take $\lambda_1, \dots \lambda_k  \in E'$. Then
$$
\sup\left\{\lv \sum_{i=1}^k  \langle P_ix, \lambda_i\rangle \rv   : x\in E_{[1]}\right\} 
 = \sup \left\{\lv \sum_{i=1}^k  \langle x_i,\,\lambda_i\rangle \rv : x_i \in E_i, \,\lV (x_1,\dots,x_k)\rV_k \leq 1\right\}\,.
$$
\end{proposition}

\begin{proof} Let the left-hand and right-hand sides of the above equation be $A$ and $B$, respectively.

Take $x \in E_{[1]}$. Then $P_ix\in E_i\,\;(i\in \N_k)$, and so
$$
\lV (P_1x,\dots,P_kx)\rV_k  = \lV P_1x +\cdots + P_kx\rV = \lV x \rV \leq 1\,.
$$
Thus  $\lv \sum_{i=1}^k  \langle P_ix, \lambda_i\rangle \rv \leq B$, and so $A \leq B$.

Take elements $x_i \in E_i\,\;(i\in\N_k)$ with $\lV (x_1,\dots,x_k)\rV_k \leq 1$, and   set $x = x_1+\cdots + x_k$. Then, by equation
 (\ref{(10.1)}), $x \in E_{[1]}$, and  $P_ix =x_i\,\;(i\in\N_k)$. Thus $\lv \sum_{i=1}^k  \langle x_i,\,\lambda_i\rangle \rv \leq A$, and so $B\leq A$.

The result follows.
\end{proof}\smallskip

\begin{proposition}\label{10.23}
Let  $((E^n, \norm_n) : n\in \N)$ and $((F^n, \norm_n) : n\in \N)$ be multi-normed spaces,  let $ k\in\N$, and let 
$E = E_1\oplus \cdots \oplus E_k$ be an  orth\-ogonal decomposition  of $E$.  Then
\begin{equation}\label{(10.10)}
\lV (Tx_1, \dots, Tx_k)\rV_k \leq \lV T\rV\lV (x_1, \dots, x_k)\rV_k
\end{equation}
for $x_1 \in E_1, \dots, x_k \in E_k$ and $T \in {\mathcal B}(E,F)$.
\end{proposition}

\begin{proof} Take  $x_j \in E_j$ for each $j \in \N_k$. By  Proposition \ref{2.5a}, 
$$
\lV (Tx_1,\dots, Tx_k)\rV_k \leq \frac{1}{k}\sum_{j=1}^k \lV \sum_{m=1}^k \zeta^{jm}Tx_m\rV\,,
$$
where  $\zeta =\exp(2\pi{\rm i}/k)$. However,
$$
\lV \sum_{m=1}^k \zeta^{jm}Tx_m\rV \leq \lV T \rV \lV \sum_{m=1}^k \zeta^{jm}x_m\rV = \lV T\rV\lV (x_1, \dots, x_k)\rV_k
$$
for each $j\in \N_k$ by equation (\ref{(10.1d)}), and now (\ref{(10.10)}) follows.
\end{proof}\s

\subsection{Elementary examples}

We give four  elementary examples involving hermitian, small, and orthogonal  decompositions; further examples will be given later.\smallskip

\begin{example}\label{10.25a}
{\rm Let $E = {\ell}_2^{\,p}$, where $p\in [1,\infty]$; the norm on $E$ is $\norm$.  Set $E_1 = \C \times \{0\}$ and $E_2=\{0\}\times \C$, so that 
$E=E_1 \oplus E_2$  is a hermitian decomposition for each $p\in [1,\infty]$.

Let $(\norm, \norm_2^{\min})$ be the minimum multi-norm of level 2 on $\{E,E^{\/2}\}$.

First, suppose that $p< \infty$, and take $x_1 =(1,0)$ and $x_2=(0,1)$, so that $x_1 \in E_1$ and $x_2  \in E_2$.  Then $\lV x_1 + x_2\rV = \lV (1,1)\rV = 2^{1/p}$,
 whereas 
$$\lV (x_1,x_2)\rV_2^{\min} = \max \{\lV x_1\rV, \lV x_2\rV\} =1\,;
$$ since $2^{1/p}>1$, the decomposition is not orthogonal.  We conclude that there are hermitian decompositions that are not orthogonal with respect to a particular multi-norm.

Second, suppose that $p= \infty$, and let $(\norm, \norm_2)$ be any multi-norm of level 2 on $\{E,E^{\/2}\}$.
Take $x_1 =(z_1,w_1)$ and $x_2=(z_2,w_2)$ in $E$.  Then
 $$
\lV P_1x_1 + P_2x_2\rV =\max\{\lv z_1\rv, \lv w_2\rv\}\,,
$$
 whereas 
$$
\lV (x_1,x_2)\rV_2\geq \lV (x_1,x_2)\rV_2^{\min} = \max \{\lv z_1\rv, \lv z_2\rv,\lv w_1\rv, \lv w_2\rv\}\geq \lV P_1x_1 + P_2x_2\rV\,,
$$
 and so the decomposition is small.}\qed
\end{example} \smallskip

\begin{example}\label{10.24c}
{\rm Let $B$ be the subset of $\C^{\,2}$ which is the absolutely convex hull of the set 
consisting of the three points $x_1 = (1,0)$, $x_2 = (0,1)$, and $x_1+x_2 = (1,1)$. Then $B$ is the closed unit ball of a norm, say $\norm$,  on $\C^{\,2}$.
Again set $E_1 = \C \times \{0\}$ and $E_2=\{0\}\times \C$, so that $E=E_1 \oplus E_2$.
We have $x_1 \in E_1$ and $x_2 \in E_2$. Also $\lV x_1+x_2\rV =1$, but $\lV x_1-x_2\rV=2$, and so the decomposition is not hermitian.

Let $(\norm_n : n\in\N)$ be any  multi-norm based on $E$.   
Then, by Theorem \ref{10.3f}(i), the decomposition $E=E_1 \oplus E_2$ is not orthogonal  with respect to this  multi-norm because it is not hermitian.

Next set $F_1 = \{(z,z) : z\in\C\}$ and $F_2 = \{(z,-z) : z\in\C\}$. Then  $E=F_1 \oplus F_2$ is a direct sum decomposition; say
 the projections onto $F_1$ and $F_2$ are $Q_1$ and $Q_2$, respectively. Simple geometrical considerations show that this decomposition is hermitian.

Let $(\norm, \norm_2^{\min})$ be the minimum   multi-norm of level 2 on  $\{E,E^{\/2}\}$, and now take $x_1 = (1,1)\in F_1$ and $x_2 = (1/2,-1/2)\in F_2$, so that
we have $\lV x_1\rV =  \lV x_2\rV =1$ and  $\lV (x_1, x_2)\rV_2^{\min} =1$.  Further,  
$$
 \lV  x_1  +  x_2\rV = \lV \left(\frac{3}{2}, \frac{1}{2}\right)\rV >1\,.
$$
 Thus the decomposition $E=F_1 \oplus F_2$ is not orthogonal with respect to the minimum multi-norm.  Again we see that there are hermitian
 decompositions that are not orthogonal with respect to a particular multi-norm.

As in Proposition \ref{10.24d}, there is a multi-norm of level 2 on  $\{E,E^{\/2}\}$ with respect to which the decomposition $E=F_1 \oplus F_2$ is small.}\qed
\end{example}\s

\begin{example}\label{10.3b}
{\rm  This example shows that we cannot determine the orthogonality of a set just by looking at pairs of elements in the set.

Let $E$ be the space $\C^{\,4}$, with the norm $\norm$ given by
$$
\lV (z_1,\dots,z_4)\rV = \max\{\lv z_1\rv,\dots,\lv z_4\rv\}\quad(z_1,\dots,z_4\in \C)\,,
$$
so that $E=\ell^{\,\infty}_4$.  Then  $((E^n, \norm_n) : n\in \N)$ is a  multi-normed space for the minimum multi-norm.

Set $f_1 = (1,0,0,1/2)$, $f_2 = (0,1,0,1/2)$, and $f_3 = (0, 0,1,1/2)$. It is immediate that $\lV f_1\rV=\lV f_2\rV=\lV f_3\rV=1$.

We {\it claim\/} that $\{f_1,f_2\}$ is orthogonal. Indeed take $\zeta_1,\zeta_2 \in \C$. Then we see that $\lV (f_1,f_2)\rV_2 =\max\{\lv \zeta_1\rv,\lv \zeta_2\rv\}$ and
$$
\lV \zeta_1f_1 + \zeta_2f_2\rV = \max\{\lv\zeta_1\rv,\lv\zeta_2\rv, \lv(\zeta_1+\zeta_2)/2\rv\}= \max\{\lv \zeta_1\rv,\lv \zeta_2\rv\}\,,
$$
as required.  Similarly,  $\{f_1,f_3\}$ and  $\{f_2,f_3\}$ are orthogonal. However, we calculate that  $f_1+ f_2+ f_3 = (1,1,1,3/2)$,  so that 
$$
\lV f_1+ f_2+ f_3 \rV = 3/2>1 = \lV(f_1,f_2,f_3)\rV_3\,. 
$$
 Thus $\{f_1,f_2,f_3\}$ is not orthogonal.\qed}
\end{example}\s

\begin{example}\label{10.3c}
{\rm  Let $E = C(\I)$, with the uniform norm $\lv \,\cdot\,\rv_{\I}$, and consider the minimum multi-norm based on $E$. We ask when $\{f_1, f_2\}$ is orthogonal.  
This is  certainly the case whenever  $f_1$ and $f_2$ have disjoint supports. However this may occur in other cases. 

For example, define a function $f_1 \in E$  by requiring that $f_1(0)=1$,  that $f_1(1)=0$,   that $f_1(t_0)= 1/2$ for some $t_0\in (0,1)$, and that 
$f_1$ be  linear on $[0,t_0]$ and $[t_0,1]$,  and then set  $f_2(t) =f_1(1-t)\,\;(t\in \I)$.  Then it is easy to see that $\{f_1, f_2\}$ is orthogonal 
if and only if $t_0 \geq 1/2$.  

Let $K $ be a compact space, and let $f \in C(K)$ be such that $f(K) = \I$. Then $\{f, 1-f\}$ is an ortho\-gonal set.}\qed
\end{example}\medskip

\subsection{Decompositions of the spaces $C(K)$}  Throughout this sub\-section, $K$ is a non-empty, compact space.\s

\begin{proposition}\label{10.3ca}
Let $(\norm_n: n\in\N)$ be any multi-norm based on $C(K)$, and   suppose that $\{K_1, \dots, K_k\}$ is  a partition of $K$ into clopen  subspaces.
Then the decomposition $C(K)=  C(K_1)\oplus \cdots\oplus C(K_k) $ is  small with respect to this multi-norm.
\end{proposition}

\begin{proof}  We write $P_j : f\mapsto f\mid K_j$ for $j\in\N_n$.
Let $f_1,\dots,f_k\in C(K)$. Then  
\begin{eqnarray*}
\lv P_1f_1+\cdots+ P_kf_k\rv_K 
&=&
\max \{\lv P_jf_j\rv_K:j\in\N_k\}\leq \max \{\lv f_j\rv_K:j\in\N_k\} \\
&=&
\lV (f_1,\dots,f_k)\rV_k^{\min} \leq \lV (f_1,\dots,f_k)\rV_k\,,
\end{eqnarray*}
and so the decomposition is small.\end{proof}\smallskip

The following theorem gives more information about  decompositions of the space $C(K)$. Recall from Theorem \ref{2.3n}(ii)
that the lattice multi-norm based on $ C(K)$ is just the minimum multi-norm.\smallskip

\begin{theorem} \label{10.3cb}
 Let $C(K) =  E_1\oplus \cdots \oplus E_k$ be a direct sum decomposition  of  $C(K)$, and let $(\norm_n: n\in\N)$ be a multi-norm based on $C(K)$.
 Then the following are equivalent:\s
 
  {\rm (a)} $E_j=C(K_j) \,\;(j\in\N_k)$ for some  partition $\{K_1, \dots, K_k\}$  of $K$ into clopen  subspaces\,;  \smallskip

 {\rm (b)} the decomposition is small with respect to the  lattice multi-norm\,;  \smallskip

 {\rm  (c)} the decomposition is orthogonal  with respect to the lattice  multi-norm\,;\s
 
 {\rm  (d)} the decomposition is hermitian.
\end{theorem}

\begin{proof} (a) $\Rightarrow$ (b)  This follows  from Proposition \ref{10.3ca}.\s

(b) $\Rightarrow$ (c) $\Rightarrow$ (d)   This follows from Theorem \ref{10.3f}.\s

(d) $\Rightarrow$ (a) This follows from Theorem \ref{10.5ba}.
\end{proof} \medskip

\subsection{Decompositions of Hilbert spaces}  Let $H$ be a Hilbert space. Recall that the Hilbert multi-norm $(\norm^H_n: n\in\N)$  based on $H$
 was defined in Definition \ref{4.3b}; orthogonal decompositions of $H$ were defined in Chapter 1, $\S2.6$.  We again denote the
 inner product on $H$ by $[\,\cdot\,,\,\cdot\,]$.\s
 
 \begin{theorem}\label{10.3d}
Let $H$ be a Hilbert space, and let $H = H_1\oplus \cdots \oplus H_k$ be  a direct sum decomposition of $H$. Then the following are equivalent:\s

 {\rm (a)} the decomposition  is   orthogonal\,;  \smallskip

 {\rm (b)} the decomposition is small with respect to the Hilbert multi-norm\,;  \smallskip

 {\rm  (c)} the decomposition is orthogonal with respect to the Hilbert multi-norm\,;  \smallskip

 {\rm  (d)} the decomposition is hermitian.\end{theorem}
 
\begin{proof}  (a) $\Rightarrow$ (b) $\Rightarrow$ (c) $\Rightarrow$ (d) These are immediate from the definition of the Hilbert multi-norm and Theorem \ref{10.3f}.\s

(d) $\Rightarrow$ (a) 
First, let  $H = H_1\oplus H_2$ be a hermitian  decomposition, with  $H_1,H_2\neq \{0\}$.  We choose $x_1\in H_1$ and $x_2 \in H_2$ with $\lV x_1 \rV =  \lV x_2 \rV =1$, 
and set $\zeta = [x_1,x_2]$, so that $\lv \zeta \rv \leq 1$.  We shall show that $\zeta = 0$, and hence deduce that $H = H_1\oplus H_2$ is an orthogonal decomposition. 

 We may suppose that $\zeta \leq 0$. Write $y= x_1 - \zeta x_2 $, so that  $[x_2,y] =0$.  Then
$$
1 + \lV y \rV^2 = \lV y-x_2\rV^2 = \lV x_1-(1+ \zeta )x_2\rV^2 \leq  \lV x_1+x_2\rV^2
$$
because $\lv 1+ \zeta \rv \leq 1$ and the decomposition is hermitian, and so
$$
1 + \lV y \rV^2 \leq \lV (1+\zeta)x_2 +y\rV^2 =(1+\zeta)^2 +  \lV y \rV^2\,.
$$  Thus $\zeta = 0$, giving the claim.

The general case follows by induction.
\end{proof}\medskip

\subsection{Decompositions of lattices} Let $E$ be a Banach lattice. We recall that a direct sum decomposition  $E = E_1 \oplus \cdots \oplus E_k$  is
a   band decomposition, written 
$$
E = E_1 \oplus_{\perp} \cdots \oplus_{\perp} E_k\,,
$$
if each of $E_1,\dots, E_k$ is a band, and that $\lV P_j\rV \leq 1\,\;(j\in\N_k)$ in this  case.\s
 
 \begin{theorem}\label{10.24g}
Let $E$ be a   Banach lattice. Then every band  decomp\-osition of $E$ is small with respect to the lattice multi-norm.  
\end{theorem}

\begin{proof}  Suppose that $E = E_1 \oplus_{\perp} \cdots \oplus_{\perp} E_k$ is a band decomposition, and take elements   $x_1,\dots, x_k \in E$. Then 
$$
\lV P_1x_1+\cdots + P_kx_k\rV  =  \lV \,\lv P_1x_1\rv\vee \cdots\vee\lv P_kx_k\rv\, \rV   
 \leq   \lV \,\lv  x_1\rv\vee \cdots\vee\lv  x_k\rv\, \rV
$$
by (\ref{(2.4ac)}), and so $\lV P_1x_1+\cdots + P_kx_k\rV \leq \lV (x_1, \dots,x_k)\rV_k^L$.  Thus
  the decomp\-osition is small with respect to the lattice multi-norm.
\end{proof}\s

Thus every band decomposition of a Banach lattice is orthogonal with respect to the lattice multi-norm.  We enquire whether  the converse to this statement holds.  
For example, take $K$ to be  a non-empty, locally compact space,  and suppose that $M(K)= E\oplus F$ is an orthogonal
 decomposition  with respect to the 
lattice multi-norm.  Then it follows from remarks on page \pageref{band} that this is a band decomposition, and so the converse holds in this case; further, 
 $M(K)= M_d(K)\oplus M_c(K)$ is an example of such a decomposition. 

First we note that the above converse need not hold in  the case when $E$ is a {\it real\/}  Banach lattice, as the following example shows.\smallskip

 \begin{example}\label{11.2a}
{\rm   Consider the space $E= \R^2$, with the $\ell^{\,1}$-norm, so that $E$ is   a Banach lattice.  Set
$$
 E_1 =\{(x,x) : x\in\R\}\,,\quad  E_2 =\{(x,-x) : x\in\R\}\,.
 $$
 Then $E=E_1\oplus E_2$ is a direct sum decomposition.  We note that, for $x,y\in \R$, so that $(x,x)\in E_1$ and $(y,-y)\in E_2$, we have   
 \begin{equation}\label{(10.7b)}
 \lV \,\lv (x,x)\rv\wedge \lv (y,-y)\rv\,\rV =\lV (\lv x\rv,\lv x\rv)\wedge (\lv y\rv,\lv y\rv)  \rV = 2\max\{\lv x\rv, \lv y\rv\}
\end{equation}
 and
 \begin{equation}\label{(10.7c)}
  \lV   (x,x)+ (y,-y) \rV = \lv x+y\rv + \lv x-y\rv = 2\max\{\lv x\rv, \lv y\rv\} \,.
 \end{equation}
 Hence  $E=E_1\oplus E_2$ is an orthogonal decomposition with respect to the lattice multi-norm.
 
However,  it is not true that $\lv (x,x)\rv\wedge \lv (y,-y)\rv = 0$ for each $x,y \in \R$, and so $E=E_1\oplus E_2$ is  not  a band decomposition.  \qed}
\end{example}\smallskip

However this leaves open the converse for (complex) Banach lattices.  We are very grateful to the late Professor Nigel 
Kalton for responding to a question by proving the converse in this case; see \cite[Theorem 4.2]{Ka}.\s

 \begin{theorem}  \label{10.24h}
Let  $E=E_1\oplus\cdots\oplus E_k$ be a direct sum decomposition of a   Banach lattice $E$. Suppose that 
$$
\lV x_1+\cdots+x_k\rV = \lV \, \lv x_1\rv \vee \cdots \vee  \lv x_k\rv\,\rV\quad (x_j \in E_j,\,j\in\N_k)\,.
$$
Then the decomposition is a band decomposition.\qed
\end{theorem}\s

The following theorem is now a consequence of Theorems \ref{10.3f}(ii), \ref{10.24g}, and \ref{10.24h}.\s

 \begin{theorem}  \label{10.24j}
Let $E=E_1\oplus\cdots\oplus E_k$ be a direct sum decomposition of a   Banach lattice $E$. Then the following  are equivalent:\s

{\rm (a)} the decomposition is orthogonal with respect to the lattice multi-norm;\s

{\rm (b)}  the decomposition is small with respect to the lattice multi-norm;\s

{\rm (c)}  the decomposition is a band decomposition.\qed
\end{theorem}\s

It is not true that every hermitian decomposition of a Banach lattice is a band decomposition.  For let $X =  \ell^{\,2}_2$, and set 
$$E= \{(z,z) : z\in \mathbb C\}\quad {\rm and}\quad F= \{(w,-w) : w\in \mathbb C\}\,,
$$
 so that $X = E\oplus F$. For $x = (z,z)\in E$ and $y= (w,-w)\in F$, we have 
$$
\lV x + e^{i\theta}y \rV^2 = 2(\lv z \rv^2 + \lv w \rv^2) \quad (\theta \in [0,2\pi))\,,
$$ 
 and so the decomposition  is hermitian. However it is not a band decomposition.

In fact, in \cite[Theorems 5.4 and 5.5]{Ka}, Kalton proved the following stronger and considerably deeper result.\s

 \begin{theorem}  
\label{10.24i}
Let $E=F\oplus G$ be a direct sum decomposition  of a Banach lattice $E$. \s

{\rm (i)} Suppose that the decomposition is hermitian. Then 
$$\lV x +y \rV =\lV (\lv x \rv^2 + \lv y\rv^2)^{1/2}\rV\quad (x\in F, \, y\in G)\,.
$$

{\rm (ii)} Suppose that, for some $p\in [1,\infty)$ with $p\neq 2$, we have
$$
\lV x +y \rV =\lV (\lv x \rv^p + \lv y\rv^p)^{1/p}\rV\quad (x\in F, \, y\in G)\,.
$$
Then the decomposition is a band decomposition.\qed
\end{theorem}\medskip

\subsection{Decompositions of $L^p$-spaces} We have seen that Theorem \ref{10.24h} does not extend to all real Banach lattices $E$.
 However,  by an argument due to Hung Le Pham, it does extend to certain real Banach lattices\s

We first make a remark.   Take $p\geq 1$.  Then we  have the inequality 
\begin{equation}\label{(10.1c)}
\frac{1}{2}(\lv z+w \rv^p + \lv z-w \rv^p) \geq  \lv z\rv^p\quad (z, w \in \C)
\,.
\end{equation}
Now suppose that $\lv z \rv\geq \lv w \rv$.  In the case where $p>1$,
 equality holds in the above if and only if  $w=0$;   in the case where $p=1$,  equality holds in the above if and only if  $z=\alpha w$  for some $\alpha \in \R$.
 
\begin{proposition}\label{11.4c}
Let $ (\Omega, \mu) $ be a measure space, and take  $E $ to be  $L^p(\Omega, \mu)$ or $  L^p_\R(\Omega, \mu)$, where  $p>1$, or $  L^1(\Omega, \mu)$. 
Suppose that $E=F\oplus G$ is an orthogonal decomposition with respect to the lattice   multi-norm.  Then   $E=F\oplus G$ is a  band decomposition.
\end{proposition}

\begin{proof}  Take  $f\in F$ and $g\in G$, and set $A =\{x\in \Omega : \lv f(x)\rv \geq \lv g(x)\rv\}$
 and $B = \Omega \setminus A$, so that $A$ and $B$ are Borel measurable subsets of $\Omega$.  Since the decomposition is orthogonal, we have
$$
\lV \,\lv f\rv \vee  \lv g\rv\,\rV =\lV f + g\rV  = \lV f - g\rV\,, 
$$
and so
\begin{eqnarray*}
\lV \,\lv f\rv \vee  \lv g\rv\,\rV^p
&=&
\frac{1}{2} \lV f + g\rV^p  + \frac{1}{2} \lV f - g\rV^p\\
&=&
\frac{1}{2}\int_\Omega \left(\lv f + g\rv^p {\dd}\mu +  \lv f - g\rv^p\right) {\dd}\mu\\
&=&
\left(\int_A + \int_B\right)\frac{1}{2}
\left(\lv f + g\rv^p + \lv f - g\rv^p \right){\dd}\mu\\
&\geq &
\int_A \lv f\rv ^p {\dd}\mu + \int_B \lv g\rv ^p {\dd}\mu\quad \;{\rm  by \;\;(\ref{(10.1c)})}
 \\
&=&   \lV \,\lv f\rv \vee  \lv g\rv\,\rV^p\,.
\end{eqnarray*}
In the case where  $p>1$, it follows that $g=0$ almost everywhere on $A$ and $f=0$ almost everywhere on $B$, and so $\lv f \rv\wedge \lv g \rv =0$.

Now suppose that $E = L^1(\Omega, \mu)$.  Then $g(x) =\alpha(x) f(x)$ for almost all $x\in A$, where $\alpha (x) \in \R\,\;(x\in A)$. 
By repeating the argument with $g$ replaced  by ${\rm i}g$ (which does not change the sets $A$ and $B$), we see that  ${\rm i}g(x) =\beta(x) f(x)$ for almost 
all $x\in A$, where $\beta (x) \in \R\,\;(x\in A)$. Thus again $g=0$ almost everywhere on $A$ and $f=0$ almost everywhere on $B$.
\end{proof} \s

 Let $\Omega$ be a $\sigma$-finite measure space, and take $p\geq 1$. Then $L^p(\Omega)$ has a weak order unit, say $e$.
 Suppose that $L^p(\Omega) = E_1\oplus \cdots\oplus E_k$ is a band decomp\-osition, with corresponding projections $P_1,\dots,P_k$. Set $v_j=P_je\,\;(j\in\N_n)$. 
Then, as remarked on page~\pageref{characteristic}, each $P_{v_j}$ is just multiplication of elements of $L^p(\Omega)$ by the characteristic function 
of a measurable set, say $S_j$; since we have $P_j= P_{v_j}$, the range $E_j$ of $P_j$ is just $L^p(S_j)$. Thus each band decomposition  of $L^p(\Omega)$ has the form 
 $L^p(S_1)\oplus\cdots \oplus L^p(S_k)$ for a measurable partition $\{S_1,\dots,S_k\}$ of $\Omega$.  This may not be true when  $\Omega$ is not $\sigma$-finite.
However   the following result applies even when $S$ is not countable.\s

\begin{corollary}\label{11.4d}
Let $S$ be a non-empty set, and take $p\geq 1$. Suppose that $$\ell^{\,p}(S) = E_1\oplus \cdots\oplus E_k$$
 is an orthogonal decomposition with respect to the standard $p\,$-multi-norm. Then there is  a partition
 $\{S_1,\dots,S_k\}$ of $S$ such that  $E_j= \ell^{\,p}(S_j)\,\;(j\in\N_k)$.
\end{corollary}

\begin{proof}  By Example \ref{2.3ca},  the  standard $p\,$-multi-norm is the lattice multi-norm.

The result follows from Proposition \ref{10.5bb}  and Theorem \ref{10.3f}(i), and also by an easy direct version of the above argument.
\end{proof}\s

\begin{corollary}\label{11.4e}
Let $S$ be a non-empty set, and suppose  that  $1\leq p<q$. Then there are no non-trivial
 decompositions of $\ell^{\,p}(S)$ which are orthogonal with respect to the standard $q\,$-multi-norm.\end{corollary}

\begin{proof} Suppose that  $\ell^{\,p}(S) = F\oplus G$ is an orthogonal  decomposition with respect to the standard $q\,$-multi-norm. 
For $f\in F$ and $g\in G$, we have 
$$
 \lV \,\lv f\rv \vee  \lv g\rv\,\rV  = \lV (f,g)\rV_2^{[p]} \geq \lV (f,g)\rV_2^{[q]}   =\lV f + g\rV  = \lV f - g\rV\,,
$$
and so, by  the argument  in Proposition \ref{11.4c}, $\lV \,\lv f\rv \vee  \lv g\rv\,\rV  = \lV (f,g)\rV_2^{[q]}$. Thus the decomposition is also orthogonal
with respect to the standard $p\,$-multi-norm.  By Corollary \ref{11.4d},  there are subset $S_F$ and $S_G$ of $S$ with  $F= \ell^{\,p}(S_F)$ and $G= \ell^{\,p}(S_G)$.

Assume towards a contradiction that  both $S_F$ and $S_G$ are non-empty, and take $s\in S_F$ and $t\in S_G$. Then 
$$
2^{1/p} = \lV \delta_s+ \delta_t\rV = \lV (\delta_s, \delta_t)\rV_2^{[q]}\leq 2^{1/q}\,,
$$
a contradiction because $q>p$.  Thus the decomposition is trivial.
\end{proof}\s


\medskip

\section{Multi-norms generated by closed families}

\noindent  We now discuss multi-norms that are generated by various closed families of direct sum decompositions of Banach spaces; this will lead 
to a theory of `multi-duals' of multi-normed spaces.\s

\subsection{Generation of multi-norms} 

\begin{definition}\label{10.3gb}
Let   $(E, \norm)$ be a  normed space, and consider a closed family $\mathcal K$ of hermitian decompositions  of $E$. For $n\in \N$ and $x_1,\dots, x_n \in E$, set
$$
\lV (x_1, \dots, x_n)\rV_n^{\mathcal K} = \sup\{\lV P_1x_1 + \cdots  +P_nx_n \rV : E=E_1\oplus\cdots\oplus E_n\}\,,
$$
where the supremum is taken over all decompositions in $\mathcal K$ of length $n$.
\end{definition}\s

\begin{theorem}\label{10.3g}
 Let   $(E, \norm)$ be a  normed space, and let $\mathcal K$  be a closed family of hermitian decompositions of $E$.    Then  
 $((E^n, \norm_n^{\mathcal K}) : n\in \N)$ is a  multi-normed space, and each direct sum decomposition in $\mathcal K$ is small with respect to this multi-norm.
\end{theorem}

\begin{proof} Let $n\in\N$. Then it  is clear that $\norm_n$ is a seminorm on $E^n$.  By considering the  trivial decompositions in  $\mathcal K$, we see that
$$
\lV (x_1, \dots, x_n)\rV_n \geq \max\{\lV x_1\rV, \dots,\lV  x_n\rV\}\quad (x_1,\dots, x_n \in E)\,,
$$
and so  $\norm_n$ is a  norm on $E^n$.

It is now easy to see  that $((E^n, \norm_n^{\mathcal K}) : n\in \N)$  is a  multi-normed space; Axioms (A1), (A3), and (A4) hold because the family 
 $\mathcal K $ is closed, and (A2) holds because all the decompositions in the family $\mathcal K$ are hermitian.

Take  a decomposition $E= E_1\oplus \cdots\oplus E_n $ in  the family  $\mathcal K$, and take 
 $x_1,\dots, x_n \in E_n$.  Then $\lV P_1x_1 + \cdots  +P_nx_n \rV \leq \lV (x_1, \dots, x_n)\rV_n^{\mathcal K} $, and so the decomposition is small 
 with respect to the multi-norm.
\end{proof}\smallskip

\begin{definition}\label{10.3j}
 Let   $(E, \norm)$ be a  normed space, and let  $\mathcal K $ be a closed family of hermitian  decompositions of $E$.  Then the {\it multi-norm generated by\/} 
$\mathcal K$  is the multi-norm $(\norm_n^{\mathcal K} : n\in\N)$. 
\end{definition}\smallskip

\begin{example}\label{10.3gc}
{\rm   Let $\mathcal K$ be the family of all trivial decompositions of a Banach space $E$. Then the  multi-norm generated by $\mathcal K$ is the minimum
 multi-norm.}\qed
\end{example}

We now consider when the multi-norm generated by ${\mathcal K}_{\rm herm}$ is the maximum multi-norm.\s

 \begin{example}\label{10.3ga}
{\rm (i) Let $K$ be an infinite,  connected compact space.  By Theorem \ref{10.3cb}, the only decompositions 
of $C(K)$ in  ${\mathcal K}_{\rm herm}$  are trivial, and so the multi-norm generated by ${\mathcal K}_{\rm herm}$ is the minimum multi-norm. By  Corollary \ref{3.19c}
 (or Theorem \ref{3.15}), the minimum multi-norm is not equivalent to the maximum multi-norm.\s
 
 (ii) Let  $E=\ell^{\,p}$ with $p\neq 2$. By Proposition \ref{10.5bb}, each hermitian decomposition of $E$ has the form
 $E= \ell^{\,p}(S_1)\oplus \cdots \oplus \ell^{\,p}(S_k)$, where $k \in \N$ and $\{S_1,\dots,S_k\} $ is a partition of $\N$.  Thus the family
 ${\mathcal K}_{\rm herm}$ generates the standard $p\,$-multi-norm $(\norm_n^{[p]}: n\in\N)$.  By Corollary \ref{4.1ai}, 
this multi-norm is equivalent to the maximum  multi-norm if and only if $p=1$ {\rm(}with equality of multi-norms when $p=1${\rm)}.\s
 
(iii)  Let $H$ be a Hilbert space. The   Hilbert multi-norm
 $(\lV \,\cdot\,\rV^H_n : n\in\N)$ based on $H$   was defined in Definition \ref{4.3b}.  It was shown in Theorem \ref{10.3d} that the following closed 
 families are equal: (a)  the family of all orthogonal decompositions;   (b)  ${\mathcal K}_{\rm small}$\,; (c)  ${\mathcal K}_{\rm orth}$\,;
(d)  ${\mathcal K}_{\rm herm}$.   Let these families be called ${\mathcal K}$.  Then it is clear from the definition of the Hilbert multi-norm that 
 $  (\lV \,\cdot\,\rV^{\mathcal K}_n : n\in\N)=  (\lV \,\cdot\,\rV^H_n : n\in\N)$.
 
 As we remarked on page~\pageref{Hmultinorm},  the Hilbert multi-norm is equivalent to   the maximum  multi-norm, but is  not equal to it, whenever $\dim H$
 is sufficiently large.  \qed}
\end{example}\smallskip
  
  Let   $(E, \norm)$ be a  normed space, and let $\mathcal K $  and $\mathcal L $ be two closed families of
hermitian decompositions of $E$ with $\mathcal K \subset \mathcal L$. Then clearly
$$
\lV (x_1, \dots, x_n)\rV_n^{\mathcal K}\leq \lV (x_1, \dots, x_n)\rV_n^{\mathcal L}\quad (x_1,\dots, x_n \in E, n\in\N)\,,
$$
 and so  $(\norm_n^{\mathcal K}: n\in\N) \leq (\norm_n^{\mathcal L} : n\in\N)$ with respect to the ordering  of ${\mathcal E}_E$ given in  Definition \ref{2.5e}.
 
 The next example shows that two different families of decompositions may generate the same multi-norm.\s
  
 \begin{example}\label{10.31}
 {\rm  Let $K$ be a compact space, and consider the lattice multi-norm based on $C(K)$; this is just the minimum multi-norm based on $C(K)$.  

Let  ${\mathcal K}$  be the family of trivial decompositions of $C(K)$, and let ${\mathcal L}$ be the family of decompositions of the form
 $C(K_1) \oplus \cdots \oplus C(K_k)$, where $\{K_1, \dots, K_k\}$ is a partition of $K$
into clopen subsets.   By Theorem \ref{10.3cb},  ${\mathcal L} =  {\mathcal K}_{\rm small} = {\mathcal K}_{\rm orth}= {\mathcal K}_{\rm herm}$. 
Thus $\mathcal K \subset \mathcal L$, and   
$\mathcal K \neq \mathcal L$ as soon as $K$ is not connected. However the multi-norm generated by both $\mathcal K $
 and $\mathcal L$ is the lattice multi-norm based on $C(K)$. 
\qed}\end{example} \medskip

\subsection{Orthogonality with respect to families}  

\begin{definition}\label{10.3ab}
 Let   $((E^n, \norm_n) : n\in \N)$ be a  multi-normed space, and let    $\mathcal K$ be a closed family of small decompositions of $E$. 
 Then  the multi-norm is {\it orthogonal with respect to} $\mathcal K$ if 
\begin{equation}\label{(10.4a)}
\lV (x_1,\dots,x_n) \rV_n = \lV (x_1,\dots,x_n) \rV_n^{\mathcal K}
\end{equation}
for each $n \in  \N$ and $x_1,\dots,x_n \in E$. The multi-norm  is {\it orthogonal} if it is 
 orthogonal with respect to ${\mathcal K}_{\rm small}$.\end{definition}\smallskip

Thus, in this case, the given multi-norm  $(\norm_n : n\in\N)$ is the multi-norm generated by $\mathcal K$.

Of course, it is automatically the case that
$$
\lV (x_1,\dots,x_n) \rV_n^{\mathcal K}\leq \lV (x_1,\dots,x_n) \rV_n \quad (x_1,\dots,x_n \in E,\,n\in\N)\,.
$$

We see that a multi-norm $(\norm_n : n\in\N)$ based on a normed space $E$  is orthogonal  if and only if, for each $n \in  \N$, each $x_1,\dots,x_n \in E$,
and   each $\varepsilon > 0$,  there is a direct sum decomposition $E = E_1\oplus\cdots \oplus E_n$ of $E$ such that
$$
 \lV (x_1,\dots,x_n) \rV_n -\varepsilon\leq \lV P_1x_1+\cdots+P_nx_n \rV \leq  \lV (x_1,\dots,x_n) \rV_n\,.
$$

For example, it follows from  Example \ref{10.3ga}(iii) that the Hilbert multi-norm  based on a Hilbert space is orthogonal, and from 
Example \ref{10.31} that the lattice multi-norm based on $C(K)$ is orthogonal. However, Example \ref{11.4bc}, below, will 
give an example of a lattice multi-norm that is not orthogonal.\medskip

\subsection{Orthogonality and Banach lattices}
Let $E$ be a Banach lattice, and  let $\mathcal K$ be the family of all band  decompositions of $E$. Clearly $\mathcal K$ is a closed family of
 decompositions that are small with respect to the lattice multi-norm. \smallskip

\begin{theorem}
\label{11.4}
Let $E$ a Banach lattice which is either an  $AM$-space  or $\sigma$-Dedekind complete. Then the lattice multi-norm based on $E$ is orthogonal with respect 
to the family of band decompositions  of $E$, and hence is the multi-norm generated by the  band decompositions.\end{theorem}

 \begin{proof} We must show that, for each $n \in  \N$ and $x_1,\dots,x_n \in E$, we have
\begin{equation}\label{(10.12)}
\lV \,\lv x_1\rv \vee \cdots \vee \lv x_n\rv\,\rV = \sup \lV  P_1x_1 +\cdots + P_nx_n  \rV\,,
\end{equation}
where the supremum is taken over  the band decompositions of length $n$.
It is sufficient to suppose that $x_1,\dots,x_n \in E^+$, and we do this. 
By (\ref{(2.4ae)}), it is sufficient to prove that  
\begin{equation}\label{(10.13)}
x \leq  \sup \{P_1 x_1 +\cdots + P_nx_n\}
\end{equation}
and that the supremum on the right is attained, where  $x =  x_1 \vee \cdots \vee  x_n$.

In the case where $E$ is an  $AM$-space, the result follows by a slight variation of the argument in Example \ref{10.31}

Now we consider the case where $E$ is $\sigma$-Dedekind complete.

First suppose that $n=2$, and   set
$$
y = (x_1-x_2)^+\quad {\rm and}\quad z= -(x_1-x_2)^-\,,
$$
 and let $B_y$ be the band generated by $y$. By Proposition \ref{2.3fb}(ii) (which applies because $E$ is $\sigma$-Dedekind complete), $E = B_y \perp B_y^{\perp}$;
  the  projections onto $B_y$ and $B_y^{\perp}$ are  $P_y$ and $Q_y$, respectively, say.  We have $y =P_y(x_1-x_2)$ and $z = -Q_y(x_1-x_2)$, and so
$$
P_y(x_1 \vee x_2) = P_y(x_2) + P_y((x_1-x_2)\vee 0)  = P_y(x_2) + (P_y(x_1-x_2))\vee 0) = P_y(x_2) + y\,.
$$
It follows that $P_y(x_1 \vee x_2) = P_y x_1 \vee P_yx_2 \geq P_y x_2$,
and so $P_y (x_1 \vee x_2) = P_yx_1$. Similarly, $Q_y(x_1 \vee x_2) = Q_yx_2$.  Thus $x_1\vee x_2 = P_yx_1 + Q_yx_2$.
This establishes the result in the special case where $n=2$.

The general case follows easily by induction.\end{proof}
\smallskip

\begin{corollary}\label{11.4b}
Let $E$ be a $\sigma$-Dedekind complete Banach lattice.  Then the lattice multi-norm based on $E$ is the multi-norm generated by the family of all band 
decompositions of $E$.\qed
\end{corollary}\s

\begin{corollary}\label{11.4ba}
 Take $p\geq 1$. Then the multi-norm generated by the family of all decompositions of $\ell^{\,p}$ as 
$\ell^{\,p}(S_1) \oplus\cdots \oplus \ell^{\,p}(S_k)$, where $\{S_1,\dots,S_k\}$ is a partition of $\N$,  is the standard $p\,$-multi-norm. 
\end{corollary}

\begin{proof}  By Theorem \ref{10.24j} and Corollary \ref{11.4d}, the specified family is the family of all band decompositions of $\ell^{\,p}$. 
 By Corollary \ref{11.4b}, this family generates the lattice multi-norm; by Example \ref{2.3ca}, this is the standard \mbox{$p\,$-multi-norm}.
\end{proof}\s

\begin{corollary}\label{11.4bb}
Let $E$ be a Banach lattice with no non-trivial band decompositions.  Then the lattice multi-norm based on $E$
 is orthogonal with respect to the family of band decompositions of $E$ if and only if $E$ is an $AM$-space.
\end{corollary}

\begin{proof}  We must show that $E$ is an $AM$-space whenever the lattice multi-norm is orthogonal with respect to the family $\mathcal K$ of band decompositions of $E$.

Take $x,y\in E^+$. Since  there are only the two trivial band decompositions of length $2$, we have
$$
\lV x\vee y\rV = \lV (x,y)\rV_2^L = \lV (x,y)\rV_2^{\mathcal K} = \lV (x,0)\rV_2^L \vee \lV (0,y)\rV_2^L = \lV x \rV \vee \lV y \rV\,. 
$$
By equation (\ref{(2.19a)}), this shows that $E$ is an $AM$-space.
\end{proof}\s

 \begin{example}\label{11.4bc}
 {\rm  Consider the Banach space $C(\I)$ with the norm $\norm$ specified by 
$$
\lV f \rV = \lv f \rv_\I + \lv f(0)\rv\quad(f\in C(\I))\,.
$$ 
Then $(C(\I), \norm)$ is a Banach lattice with no non-trivial band decompositions.  However it is not an $AM$-space  (take $f = 1/2$ and $g =Z$, so that 
$\lV f \rV = \lV g \rV  =1$ but $\lV f\vee g\rV = 3/2$), and so, by Corollary \ref{11.4bb}, the lattice multi-norm is not orthogonal 
with respect to the family of band decompositions.\qed}\end{example}
\medskip

\section{Multi-norms on dual spaces}

\noindent We now consider how to form the `multi-dual' of a multi-normed space. 

Let $((E^n,\norm_n): n\in\N)$ be a multi-normed space. It is   tempting to regard ${\mathcal M }(E,\C)$ as the  `multi-dual' of this space.
 However recall that  ${\mathcal M }(E,\C) = E'$ when we regard $\C$ as having its unique multi-norm structure, and that, as a  multi-normed space, 
${\mathcal M }(E,\C)$  has just the minimum multi-norm. Thus the approach of using this multi-normed space as a `multi-dual' is not satisfactory.

A second temptation is to look at the family   $(((E')^n,\norm'_n): n\in\N)$ for a multi-normed space $((E^n,\norm_n): n\in\N)$, where $\norm'_n$ is the 
dual of the norm $\norm_n$.   But this is an even worse failure: $(\norm'_n: n\in\N)$ is a dual multi-norm, not a multi-norm, on $\{(E')^n:n\in\N\}$.

We shall give a different approach, using the notion of orthogonal de\-compositions.     We continue to use the notation of earlier sections.\medskip

\subsection{The multi-dual space}

\noindent Here we define our concept of a multi-dual space.

Let   $(E, \norm) $ be a normed space, and let  $\mathcal K $ be a closed family of hermitian  decomp\-ositions of $E$. As in Definition \ref{10.3j}, 
 $\mathcal K $ generates a multi-norm $(\norm_n^{\mathcal K} : n\in\N)$ based on $E$.  We shall now define a multi-norm 
 on $\{(E')^n : n\in\N\}$ in terms of $\mathcal K $. Recall that the dual $ {\mathcal K}'$ of a  closed family $\mathcal K$ of direct sum 
decompositions of $E$  was defined in Definition \ref{10.4c}.\smallskip

\begin{definition}\label{10.4d}
Let   $(E, \norm) $ be a normed space, and let $\mathcal K$ be a  closed family of hermitian  decompositions of $E$.  Then the multi-norm based 
on $E'$ which is generated by ${\mathcal K}'$ is the {\it  multi-dual multi-norm\/} to the multi-norm $(\norm_n^{\mathcal K}: n\in\N)$;  it is  denoted by
$$
(\norm_{n, {\mathcal K}}^{\dag} : n\in\N)\,.
$$
The multi-normed space $(((E')^n, \norm_{n, {\mathcal K}}^{\dag}) : n\in\N)$ is the {\it  multi-dual space} (with respect to  $\mathcal K$).
\end{definition}\smallskip

Let  $\mathcal K $ be a closed family of hermitian  decomp\-ositions of $E$.  By Proposition \ref{11.3}, each member of  ${\mathcal K}'$ is a 
hermitian decomposition of $E'$, and so $(\norm_{n, {\mathcal K}}^{\dag} : n\in\N)$ is indeed a multi-norm based on $E'$ by Theorem \ref{10.3g}. 
 It is an orthogonal multi-norm.

For each  $n \in\N$ and $\lambda_1,\dots, \lambda_n \in E'$, we have
$$
\lV (\lambda_1,\dots, \lambda_n )\rV_{n,\mathcal K}^{\dag}  =
\sup \lV P'_1\lambda_1+\cdots + P'_n\lambda_n\rV =
\sup  \lV \lambda_1\,\circ\,P_1+\cdots +\lambda_n\,\circ\,P_n\rV  \,,
$$
where the supremum is taken over all the decompositions $E=E_1\oplus\cdots\oplus E_n$ in  the closed family $\mathcal K$.\s

\begin{definition}\label{10.4db}
Let   $(\norm_n: n\in\N)$ be an orthogonal multi-norm based on a normed space $E$.  Then the {\it  multi-dual multi-norm
based on $E'$} is  the multi-norm generated by the family $({\mathcal K}_{\rm small})'$.
\end{definition}\smallskip

In the above case, the multi-dual multi-norm is itself orthogonal, and so we generate multi-norms based on all the successive dual spaces  of $E$.\s

\begin{proposition}\label{10.4da}
Let   $((E^n, \norm_n) : n\in \N)$ be a  multi-normed space,  and let $\mathcal K$ be a  closed family of orthogonal decompositions 
{\rm(}with respect to the multi-norm{\rm)} of $E$. 
  Take $\lambda_1, \dots \lambda_n  \in E'$. Then
  $$
\lV (\lambda_1,\dots, \lambda_n )\rV_{n,\mathcal K}^{\dag}  =
\sup_{\mathcal K}\sup \left\{\lv \sum_{i=1}^n  \langle x_i,\, \lambda_i\rangle \rv : x_i \in E_i, \,\lV (x_1,\dots,x_n)\rV_n \leq 1\right\}\,,
  $$
  where the first supremum is taken over all decompositions $E=E_1\oplus\cdots\oplus E_n$ in $\mathcal K$. 

 Each direct sum decomposition in ${\mathcal K}'$ 
is small  with respect to the multi-dual multi-norm $(\norm_{n,\mathcal K}^{\dag} : n\in\N)$.
\end{proposition}

\begin{proof}  This follows from Proposition \ref{10.3e}.\end{proof}\s

\begin{example} {\rm Let $\mathcal K$  be the family of trivial decompositions of a normed space $E$, so that the multi-norm generated by $\mathcal K$ 
is the minimum multi-norm based on $E$. Then ${\mathcal K}'$ is the the family of trivial decompositions of $E'$, and the multi-dual multi-norm is the
 minimum multi-norm based on $E'$.\qed}\end{example}\s
 
\begin{example}\label{10.31a}
 {\rm  Let $K$ be a compact space, and consider the lattice multi-norm based on $C(K)$. Let $\mathcal K $ be the family of trivial
 decompositions of $C(K)$, and let  
$$
\mathcal L ={\mathcal K}_{\rm small} = {\mathcal K}_{\rm orth}= {\mathcal K}_{\rm herm}
$$
 be as in Example \ref{10.31}.
  Then both $\mathcal K $ and $\mathcal L$ generate the lattice multi-norm based on $C(K)$.  However ${\mathcal K}'$ is the family of trivial
 decompositions of $C(K)'=M(K)$, and so  the multi-dual multi-norm  $(\norm_{n,\mathcal K}^{\dag}: n\in\N)$ is the minimum multi-norm based on $M(K)$, 
whereas the multi-dual multi-norm  $(\norm_{n,\mathcal L}^{\dag}: n\in\N)$ is a strictly  larger multi-norm based on $M(K)$ as soon as $K$ is not connected.  Indeed, 
in the case where $K$  is a Stonean space, or, equivalently, when $C(K)$ is Dedekind complete, it follows from Proposition \ref{4.1e}   
that $(\norm_{n,\mathcal L}^{\dag}: n\in\N)$  is the standard $1$-multi-norm  based on $M(K)$; by Proposition \ref{4.1d}, this is the lattice multi-norm, and,
 by Theorem  \ref{2.3n}(i), it is the maximum multi-norm.\qed}\end{example}\smallskip

\begin{example}\label{10.13}
{\rm  Take $p\geq 1$, and let $E = \ell^{\,p}$.  We again consider  the  standard $p\,$-multi-norm, $(\norm_n^{[p]}: n\in\N)$, 
 based on $E$. By Example \ref{2.3ca}, this is the lattice multi-norm based on $E$.

Let ${\mathcal K}$ be the family of decompositions of the form $$\ell^{\,p}(S_1) \oplus \cdots \oplus \ell^{\,p}(S_n)\,,
$$
 where $\{S_1, \dots, S_n\}$  is a partition of $\N$. By Theorem \ref{10.24j} and Corollary \ref{11.4d}, we have 
 ${\mathcal K}= {\mathcal K}_{\rm small} =  {\mathcal K}_{\rm orth}$.  Then it is clear that
${\mathcal K}$ generates the standard $p\,$-multi-norm on $E$, and so this multi-norm is orthogonal with respect to ${\mathcal K}$.

Suppose that $p>1$. The conjugate index to $p$ is $q\,$; set $F = \ell^{\,q}$.  Clearly,   ${\mathcal K}'$ is the family of decompositions of the form
 $\ell^{\,q}(S_1) \oplus \cdots \oplus \ell^{\,q}(S_k)$, where $\{S_1, \dots, S_k\}$ is a partition of $\N$.  Thus ${\mathcal K}'$ generates
the standard $q\,$-multi-norm on $E'$. This shows that  the multi-dual of $(((\ell^{\,p})^n, \norm^{[p]}_n): n\in \N)$ (with respect to ${\mathcal K}$) is 
 $(((\ell^{\,q})^n, \norm^{[q]}_n): n\in \N)$, a fact that  was one of  the aims of our theory.
\qed}\end{example}\s

\begin{example}
\label{10.13a}
{\rm Let $H$ be a Hilbert space. Let    ${\mathcal K}$ be the family of orthogonal decomp\-ositions of $H$; by Theorem \ref{10.3d},
 ${\mathcal K} =  {\mathcal K}_{\rm small}= {\mathcal K}_{\rm orth}=   {\mathcal K}_{\rm herm}$. It is clear from the definition of the Hilbert multi-norm 
in (\ref{(4.3)}), that the multi-norm generated by ${\mathcal K}$  is the Hilbert multi-norm, and that this multi-norm is orthogonal.
 It is immediate that the multi-dual of $((H^n, \norm^H_n): n\in \N)$ (with respect to ${\mathcal K}$) is itself.
\qed}\end{example}\smallskip
 
Let $E$ be a  Banach lattice, and let $\mathcal K$ be the family of band decompositions of $E$. By Theorem \ref{10.24j}, 
$\mathcal K =  \mathcal K_{\rm small} =\mathcal K_{\rm orth}$.
We shall consider the multi-normed space  $$((E^n, \norm_n^L): n\in\N)\,,
$$
 where $((\norm_n^L): n\in\N)$ is the lattice multi-norm, and suppose that this multi-norm is generated by the  family $\mathcal K$.    We would like to know
 when the multi-dual (with respect to $\mathcal K$) of this multi-Banach  space  is $((E^n, \norm_n^L): n\in\N)$, where $(\norm_n^L: n\in\N)$ is now the 
lattice multi-norm on $E'$.  It follows from Example \ref{10.31a} that this is not always the case. However we have the following theorem.\s

\begin{theorem}\label{10.13b}
Let $E$ be a  Dedekind complete Banach lattice.  Then the lattice multi-norm  based on $E$ is generated by the family
$\mathcal K$ of band decompositions, and the multi-dual with respect to $\mathcal K$ is the lattice multi-norm based on $E'$.
\end{theorem}

\begin{proof}  It follows  from Corollary \ref{11.4b} that lattice multi-norms based on $E$ and $E'$ are generated by $\mathcal K$ and by the family,
 say $\mathcal L$, of band decompositions of   $E'$, respectively. 
 
 We shall show that the lattice multi-norms based on $E'$ is  also generated by  ${\mathcal K}'$. Take $n\in\N$ and $\lambda_1,\dots,\lambda_n\in E'$. Then certainly
$$
\lV (\lambda_1,\dots, \lambda_n )\rV_{n,\mathcal K}^{\dag} \leq \lV (\lambda_1, \dots, \lambda_n)\rV_n^{\mathcal L}\,.
$$ 
We shall show the reverse inequality.  We prove the result in the case where $n=2$; the general case follows by induction.

Thus take $\varepsilon >0$, and let $E' = F_1 \oplus_\perp F_2$ be a band decomposition of $E'$ such that
$$
\lV Q_1\lambda_1 +   Q_2\lambda_2\rV > \lV (\lambda_1,   \lambda_2)\rV_2^{\mathcal L} - \varepsilon\,,
$$
where $Q_i$ is the projection on $F_i$.  Thus there exist $\mu_1, \mu_2 \in E'$ such that $\lv \mu_i\rv \leq \lv \lambda_i\rv $ for $i=1,2$
 and $\mu_1\perp \mu_2$  and such that 
$$
\lV \mu_1 +   \mu_2\rV > \lV (\lambda_1,  \lambda_2)\rV_2^{\mathcal L} - \varepsilon\,.
$$
For $i=1,2$, define $X_i = \{x \in E : \langle \lv x \rv,\,\lv \mu_i\rv \rangle =0\}$.  Then $X_1$ and $X_2$ are bands in $E$, and so, 
by Proposition \ref{2.3fb}(i), they are principal bands. Set $E_1= X_1^{\perp}$ and $E_2= X_2^{\perp}$, so that $E_1\perp E_2$.  
It is clear that $\mu_1 \in E_1'$ and $\mu_2 \in E_2'$.   By enlarging   $E_1$ and $E_2$, if necessary, we may suppose that $E_1 \oplus E_2= E$, 
and so $E = E_1\oplus_\perp E_2$ is  a band decomposition of  $E$.  Thus the decomposition  $E'= E_1'\oplus_\perp E_2'$  belongs to ${\mathcal K}'$.  It follows that 
$$
\lV (\lambda_1, \lambda_2 )\rV_{n,\mathcal K}^{\dag}> \lV (\lambda_1,   \lambda_2)\rV_2^{\mathcal L} - \varepsilon\,.
$$
This holds true for each $\varepsilon > 0$, and so the result follows.
\end{proof}
\medskip

\subsection{Second dual spaces}

\noindent Let   $(E, \norm)$ be a  normed space, and let  $\mathcal K $ be a closed family of hermitian  decompositions of $E$.  Then $\mathcal K $ and
 ${\mathcal K} '$  generate multi-norms on the two families $\{E^n : n\in\N\}$ and $\{(E')^n : n\in\N\}$, respectively.  Similarly, the closed family
 ${\mathcal K}'' $  of hermitian  decompositions of $E''$ generates a multi-norm  $(\norm_{n,{\mathcal K}}^{\dag\dag}: n\in\N)$ on  $\{(E'')^n : n\in\N\}$. 

The following result can be regarded as a multi-normed form of the  Hahn--Banach theorem.\smallskip

\begin{theorem}\label{10.14}
Let   $((E^n, \norm_n) : n\in \N)$ be a  multi-normed space,  let  $\mathcal K $ be a closed family 
of small decompositions of $E$, and consider the multi-norm  $$(\norm_{n,{\mathcal K}}^{{\dag}{\dag}}: n\in\N) $$ based on $E''$. Then the 
 canonical embedding of $E$ into  $E''$ gives  a multi-isometry if and only if the multi-normed space 
$((E^n, \norm_n) : n\in \N)$ is orthogonal with respect to the family $\mathcal K $.
\end{theorem}

\begin{proof}  Let $x_1,\dots,x_k \in E$.  Then
$$
\lV (x_1,\dots,x_k) \rV_{k,{\mathcal K}}^{\dag\dag} =\sup \lV P''_1x_1+ \dots + P''_kx_k\rV\,,
$$
where the supremum is taken over all projections $P_1,\dots, P_k$ that arise from decomp\-ositions in $\mathcal K$. Since  $P_i''x_i= P_ix_i$ for $x_i\in E_i$
 and $i\in\N_n$, it follows that the canonical embedding is a multi-isometry  if and only if the multi-normed space   $((E^n, \norm_n) : n\in \N)$
 is orthogonal with respect to the family $\mathcal K $.
\end{proof}\smallskip

\begin{definition}\label{10.26}
Let   $((E^n, \norm_n) : n\in \N)$ be a  multi-normed space,  let  $\mathcal K $ be a closed family of small  decompositions of $E$.  Then 
the space $((E^n, \norm_n) : n\in \N)$  is  {\it multi-reflexive with respect to  $\mathcal K $} if the canonical embedding of $E$ into $E''$ (when 
the multi-norm based on $E''$ is taken to be $(\norm_{n,{\mathcal K}}^{{\dag}{\dag}}: n\in\N)$) 
 is a multi-isometry that is a surjection.\end{definition}\smallskip

Thus $((E^n, \norm_n) : n\in \N)$ is multi-reflexive with respect to  $\mathcal K$ if and only if $E$ is a reflexive Banach space and
 $((E^n, \norm_n) : n\in \N)$ is orthogonal with respect to the family $\mathcal K $. \smallskip
 
 \begin{example}\label{10.13c}
{\rm  Let $E$ be a Banach lattice such that $E$ is reflexive as a Banach space. Then $E$ is Dedekind complete, 
and so, by Theorem \ref{11.4}, the lattice multi-norm is orthogonal with respect to the family $\mathcal K$ of band decompositions, and so the space  $E$ is 
 multi-reflexive with respect to  $\mathcal K$. }\qed\end{example}\smallskip

\begin{example}\label{10.13d}
{\rm  Take $p>1$, and let $E= L^p(\Omega, \mu)$ for a measure space $(\Omega, \mu)$, with the standard $p\,$-multi-norm.
Then $((E^n, \norm_n) : n\in \N)$ is multi-reflexive with respect to the family  of all band decompositions of $E$.}\qed
\end{example}\smallskip


 \begin{theindex}
\addcontentsline{toc}{chapter}{Index of terms}


  \item absolutely convex, 9
  \item absolutely convex hull, 9
  \item absolutely summing operators, 69
  \item absorbing, 9
   \item $AL$-space, 37, 103, 107
   \item $AL_p$-space, 37, 103 
  \item $AM$-space, 37, 103, 107, 155 
  \item amalgamation, 110
  \item amplification, 11, 12
  \item annihilator, 13
 \item Axiom (P), 59

  \indexspace

  \item balanced, 9
  \item Banach algebra, 23, 35
  \item Banach lattice, 29, 119, 130, 150, 151
  \subitem   Dedekind complete, 26, 30 
\subitem   dual, 36
 \subitem  Fatou norm, 32
\subitem KB-space, 31 
\subitem Levi norm, 32
\subitem monotonically bounded, 31, 131
\subitem monotonically complete, 31
\subitem Nakano property, 31, 131, 135
\subitem real, 28
  \subitem  $\sigma$-Dedekind complete, 26,  108, 168
\subitem weak Nakano property, 31, 131
\subitem  weak $\sigma$-Nakano property, 31, 133
  \item Banach module, 39,  53
\subitem injective, 39
  \item Banach operator algebra, 23, 124
  \item Banach sequence algebra, 23
  \item Banach space, 12
  \item Banach--Mazur distance, 13, 80
  \item band,  29 
  \subitem principal, 29
\subitem projection, 29
  \item band decomposition, 29
\item basic set, 110
\item basis
 \subitem Schauder, 101
 \subitem unconditional, 101
 \item bounded, 110

  \indexspace
  
    \item $C^*$-algebra, 19, 23, 59
 \item canonical embedding, 13
  \item cardinality, 7
  \item Cauchy--Schwarz inequality, 19
  \item coagulation, 11, 144
  \item compactification, Stone-{\v C}ech, 24
  \item complemented, 14
 \subitem $\lambda$-, 14 
  \item complete, Dedekind,   26, 30
  \item complexification, 9, 11
  \item condition (P), 59
  \item cone, 32
  \item conjugate index,  8
  \item contraction, 13
  \item converge unconditionally, 101
  \item convex, 9
  \item convex hull, 9
  \item coordinate functional, 7
  \item cross-norm, 15
 \subitem reasonable, 15, 57
 \subitem sub-, 15

  \indexspace

  \item decomposition 
 \subitem band, 27, 29, 150, 155 
 \subitem direct sum, 9, 12, 17, 139
 \subsubitem closed family, 17
   \subsubitem generated by, 17
 \subsubitem length, 16
 \subsubitem trivial, 16
 \subitem dual family, 7
 \subitem hermitian, 139, 142
 \subsubitem  closed family, generated by, 140
 \subitem orthogonal, 19, 144, 145
 \subitem  $M$-, $L$-, 140 
 \subitem small, 143
  \item dimension, 9
 \subitem Hilbert, 19
  \item disjoint, 27, 29
  \item disjoint complement, 29
  \item dominates, 48
  \item dual
 \subitem multi-Banach space, 45, 52
 \subitem  multi-norm, 42
 \subsubitem maximum, 61, 76
 \subsubitem minimum, 76
  \subsubitem  $(r,s)$-, 89
 \subitem norm, 13
 \subitem sequence, 18
 \subitem  space, 13

  \indexspace

  \item elementary representation, 76
  \item equivalent, 48
  \item extreme point, 9
  \item extremely disconnected, 24

  \indexspace

  \item finite-dimensional space, 62
  \item functional calculus, 28

  \indexspace

  \item group algebra, 23, 39

  \indexspace

  \item hermitian element, 25
  \item Hilbert space, 18, 90
  \item hyper-Stonean 
 \subitem envelope, 24
  \subitem space, 24
  \item H{\"o}lder's inequality, 8, 21

  \indexspace

  \item inequality-of-roots, 46
  \item involution, 19
  \item isometric isomorphism, 13
  \item isometrically isomorphic, 13
  \item isometry, 13
  \item isomorphic, 13

  \indexspace

  \item Krivine calculus, 28

  \indexspace

  \item lattice, 26
   \subitem Banach, 29
 \subitem  real Banach, 28
  \item lattice norm, 28
  \item lattice operations, 26
\item linearly homeomorphic, 13
  \item local base, 109
  \item locally compact group, 39
 \subitem  amenable, 39
  \item locally convex space, 115

  \indexspace

  \item matrix, 11, 53
  \subitem column-special, 55
 \subitem row-special, 53
 \subitem transpose, 11
 \item matrix norms, 58
  \item measurable, 21, 22
  \item measure, 21
 \subitem Borel, 22
 \subitem continuous, 22
 \subitem discrete, 2
  \subitem positive, 21
    \item measure algebra, 23
  \item Minkowski functional, 10
  \item modulus, 27, 28
  \item multi-Banach space, 45
  \subitem  dual-, 45
  \item multi-bound, 118
  \item multi-bounded  
 \subitem    sequence, 119
 \subitem operator, 120, 122 
 \subitem set, 118
  \item multi-Cauchy sequence, 112
  \item multi-closure, 112
  \item multi-continuous, 122 
  \item multi-contraction, 121
  \item multi-dual multi-norm, 159
  \item multi-dual space, 159
  \item multi-isometry, 121
  \item multi-norm, 41
 \subitem abstract $q$-, 102 
 \subitem based on $E$, 41
  \subitem  compatible with the lattice structure, 106
 \subitem dual, 42
 \subitem dual lattice, 105 
 \subitem  equivalent, 48 
 \subitem extension, 136
  \subsubitem  balanced, 136
 \subsubitem isometric, 136
 \subitem generated by $\mathcal K$, 154
 \subitem Hilbert, 90
 \subitem lattice, 105, 109, 119, 130, 150, 158
 \subitem maximum, 62, 73, 76, 82, 86, 95, 107, 108 
 \subitem  minimum, 61, 82, 108
 \subitem multi-bounded, 125
  \subitem multi-dual, 157
 \subitem of level $n$, 41
 \subitem orthogonal, 155
 \subitem partition, 128
  \subitem $(p,p)$-, 92
 \subitem $(p,q)$-, 85, 95, 109
 \subitem Schauder, 102
 \subitem standard $1$-, 138
 \subitem standard $p$-,  101, 106
 \subitem standard $q$-, 93, 96, 99, 109
 \subsubitem -dual, 94
 \subitem type-$p$,   56 
  \item multi-normed space, 41
 \subitem balanced, 136
 \subitem dual, 42
 \subitem isometric, 136
 \subitem multi-reflexive, 160
 \subitem quotient, 49
 \subitem subspace, 49
  \item multi-null sequence, 112
  \item multi-topological linear space, 111

  \indexspace

  \item Nakano property, 31
 \subitem   weak, 31
  \item net, increasing, decreasing, 26
  \item norm  
 \subitem  $1$-unconditional, 101
 \subitem  $c_{\,0}$-, 57--59 
 \subitem  Chevet--Saphar, 85
 \subitem  Fatou, 32
 \subitem  injective, 16, 58
 \subitem  Levi, 32
 \subitem  order-bounded,  
  \subitem  order-continuous, 30
  \subitem  $(p,q)$--\,summing, 69
 \subitem  projective, 15, 58, 78
 \subitem  quantum, 58
 \subitem  sequential, 57
 \subitem  special, 44, 102
  \subitem  $\sigma$-order-continuous, 30
 \subitem  weak $p$--summing, 65
  \item null sequence, 20
  \item numerical range, 25
 \subitem  spatial, 25

  \indexspace

  \item operator 
 \subitem  approximable, 14
 \subitem  compact, 14, 126
 \subitem  dual, 14
 \subitem  finite-rank, 14, 123
 \subitem  ideal, 69
 \subitem  linear, 11
 \subitem   multi-bounded, 120, 122
 \subitem  nuclear, 14, 123
 \subitem  order-bounded, 26, 32, 33, 130
 \subitem  order-continuous, 33
 \subitem  order-isometric, 33
  \subitem  $(p,q)\,$--\,summing, 69
 \subitem  positive, 32
 \subitem  regular, 32, 33
  \subitem   sequence space, 57
    \item operator space theory, 58
  \item operator space, abstract, 58
  \item order
\subitem -bounded, 26, 29, 120
\subitem -closed, 27, 29
\subitem -continuous, 33
\subsubitem  $\sigma$-, 117
\subitem  -convergent, 27
\subitem  -homomorphism, 27
\subitem  -ideal, 27, 29
\subitem -interval, 26
\subitem -isometric, 28
\subitem -isometry, 28
\subitem  -isomorphic, 28
\subitem -isomorphism, 28
\subitem -limit, 27
\subitem -null, 27, 117
\subitem -spectrum, 35
\subitem -spectral radius, 35
\subitem  -unit, 22
  \item ordered linear space, 26
  \item Orlicz 
\subitem  constant, 71, 72
\subitem  property, 71 
  \item orthogonal, 19, 144
\subitem  decomposition, 145
\subitem  family, 144
\subitem  projections, 19
\subitem  with respect to $\mathcal K$, 155
  \item orthogonality, 139
  \item orthonormal, 19
\subitem  basis, 19

  \indexspace

  \item partition, 128
\subitem ordered, 21, 98
  \item partition multi-norm, 128
\item permutation, 7, 23
  \item positive cone,  26
  \item Principle of Local Reflexivity, 14, 76, 88
  \item projection, 12, 19 
  \item pseudo-amenable, 39

  \indexspace

  \item quantum functional analysis, 58

  \indexspace

  \item Rademacher functions, 132
  \item rate of growth, 60
\subitem maximum, 64
  \item reflexive, 13
  \item regular set isomorphism, 23
  \item Riesz space, 26
  \subitem Dedekind complete, 27
\subitem normed, 28

  \indexspace

  \item second dual question, 90, 103 
  \item sequence, 7
\subitem constant, 10
\subitem convergent, 20
\subitem dual, 18
\subitem multi-convergent, 112
\subitem multi-null, 112,  114, 117
\subitem null, 12, 20
\subitem pairwise-disjoint, 29
\subitem similar, 7, 61
\subitem weakly $p$-summable, 65
  \item sequential norm, 57
  \item solid, 27, 29
  \item special-Banach space, 45
  \item special-norm, 44, 102
  \item special-normed space, 44, 50
  \item spectral radius, 23
  \item spectral radius formula, 23
  \item spectrum, 23
  \item standard basis, 20
  \item state space, 24
  \item Stonean space, 24, 100
  \item sub-lattice, 26
  \item summing constant, $(p,q)$-, 68
  \item symmetric difference, 7

  \indexspace
  
  \item tensor product, 15, 57
\subitem injective, 16
\subitem projective, 15
  \item theorem 
 \subitem Ando, 35
 \subitem Arendt, 35
\subitem Banach's isomorphism, 127, 129, 134 
\subitem Banach--Alaoglu, 13
\subitem Banach--Stone, 23
\subitem Dvoretzky, 82
\subitem F.\ Riesz, 30
\subitem  F.\ Riesz and Kantorovitch, 34
\subitem Feng and Tonge, 73
\subitem Gel'fand--Naimark, 19, 59
\subitem Goldstein, 13
\subitem Gordon, 72
\subitem Hahn's decomposition, 22
\subitem Hahn--Banach, 74, 173
\subsubitem separation, 13
\subitem Kakutani, 38
\subitem Kalton, 151
\subitem Kantorovic, 32
\subitem Kolmogorov, 115
\subitem Lamperti, 24, 104
\subitem Orlicz, 72
\subitem Pisier, 59, 108
\subitem Pitt, 132
\subitem representation for multi-normed spaces, 108
\subitem Ruan, 59
\subitem Russo--Dye, 19, 91 
\subitem Szarek, 69
\subitem Tam, 26

  \indexspace

  \item underlying real-linear space, 9
  \item unit sphere, 12
  \item unitary, 19
  \item unitary group, 19

  \indexspace

  \item weak $p$--summing norm, 65
  \item weak order unit, 29 
 \item weakly  $p$--summing sequences, 65
  \label{endindex}
\end{theindex}

\addcontentsline{toc}{chapter}{Index of symbols}

\newpage

 \chapter*{Index of Symbols}
 
\renewcommand{\labelitemi}{}
\begin{multicols}{2}
\begin{itemize}\setlength{\itemsep}{0ex}\setlength{\parskip}{0ex}

  \indexspace 
 
   \item ${\rm aco\/}(S)$, 10
    \item $A_\sigma$, 12, 41, 110
  \item ${\mathcal A}(E,F)$, ${\mathcal A}(E)$, 14
   \item $\alpha = O(\beta)$, $\alpha = o(\beta)$,  $\alpha \sim  \beta$, 
		7

\indexspace 

  \item $B(A)$, $B_x$, 29
  \item ${\mathcal B}(E,F)$, ${\mathcal B}(E)$, 13
  \item ${\mathcal B}(E,F)^+$,  33
  \item ${\mathcal B}^n(E_1,\dots,E_n; F)$, 14
    \item ${\B}_b(E,F)$, 33, 130
   \item ${\mathcal B}_r(E, F)$, 33
   \item ${\mathcal B}_b(E)$, ${\mathcal B}_r(E)$, 35
  \item $\beta S$, 24 

\indexspace 

  \item ${\rm co\/}(S)$, 9
    \item $c_{\,0}(E)$, 12
  \item $c_B$, 118
  \item $c_m(E)$, $c_{m,0}(E)$, 111
  \item $c_n(E)$, 71
  \item $c_{\,0}$, $c_{\/00}$, $c$, $c_{\/0, \R}$, $c_{\,\R}$, 20
    \item $C(K)$, 22, 31
  \item $C_0(K)$, $C_{0,\R}(K)$, 22
  \item $C_q$, 72
  \item  $\C$, 7
    \item $\C^S$, 10

 \indexspace
 
   \item $\dim E$, 9
 \item $d(E,F)$, 13
 \item $\D$, $\overline{\D}$, 7
  \item $\Delta_v$, 29
    \item $\delta_n$, 20
  \item $\delta_x$, 22

 \indexspace

   \item $\ex K$, 9
     \item $E_{[r]}$, 12
  \item $E_{\R}$, 29
    \item $E^n$, 10
  \item $E^{\,\N}$, 11
 \item $E'$, $E''$, 13
  \item $E^+$, 26, 29
  \item $E^+_{[1]}$, 29
 \item $E\cong F$, $E\sim F$, 13
    \item $E_1\oplus_{\perp}\cdots\oplus_{\perp} E_n$, 29
      \item $E\otimes F$, 15
     \item $(E\injectivetensor F, \norm_\varepsilon)$, 16
  \item $(E\projectivetensor F, \norm_\pi)$, 15
    \item $({\mathcal E}_E, \leq)$, 48, 63

\indexspace

  \item $F+G$, $F\oplus G$, 9
  \item $F\oplus_{\perp} G$, 19 
  \item $f^+$, $f^-$, $f \vee g$, $f \wedge g$, $\lv f\rv$, $f\leq g$, 10
    \item ${\mathcal F}(E,F)$, ${\mathcal F}(E)$, 14
  \item $\varphi_n(E)$, 60
  \item $\varphi_n^{\,\max}(E)$, 64, 82
  \item $\varphi_n^{\,\max}(C(K))$, 82, 109
  \item $\varphi_n^{\,\max}(L^{p})$, 81, 109
  \item $\varphi_n^{\,\max}(\ell^{\,p})$, 81
  \item $ \varphi_n^{(p,q)}(E)$, 84, 109
  \item $\varphi_n^H(H)$, 91, 109
  \item $\varphi_n^L(E)$, 105, 109
   \item $\varphi_n^{[q]}(L^{p})$, 93
  \item $\varphi_n^{[q]}(\ell^{\,p})$, 94,99
  \item $\varphi_n^{[q]}(M(K))$, 996
 \indexspace 
 
  \item $H_1\oplus_{\perp} \cdots \oplus_{\perp} H_n$, 20

 \indexspace
 
   \item $I_E$, 11
     \item  $\I$,   7

  \indexspace
  
     \item ${\mathcal K}(E,F)$, ${\mathcal K}(E)$, 14
 \item $\mathcal K$, 17
  \item ${\mathcal K}'$, 17, 157
  \item ${\mathcal K}''$, 160
  \item ${\mathcal K}_{\rm herm}$, 140
  \item ${\mathcal K}_{\rm orth}$, 145
  \item ${\mathcal K}_{\rm small}$, 144
  
    \indexspace 
    
  \item $\Lim_{i\to \infty}x_i$, 112
    \item $\lin S$, 9
 \item $L^{p}(\Omega)$, 21, 92
  \item $L^{p}(\Omega,\mu)$, $L^\infty(\Omega, \mu)$, 21
\item  $L^{p}_\R(\Omega,\mu)$, $L^\infty_\R(\Omega,\mu)$, 21
  \item $L^{p}(G)$, 39
     \item $L^1(G)$, 23, 39
  \item $L^{p}(\I)$, 22
  \item $\ell^{\,p}$, 20, 142, 158
  \item ${\ell}^{\,\infty}$, ${\ell}^{\,p}_\R$, ${\ell}^{\,\infty}_\R$, ${\ell}^{\,p}_n$, 20
\item ${\ell}^{\,\infty}_n$, 21
  \item $\ell^{\,\infty}(E_\alpha)$, $\ell^{\,p}(E_\alpha)$, 18
    \item  $\ell^{\,p}(E)^{w}$, $\ell^{\,p}_n(E)^{w}$, 65
  \item ${\mathcal L}(E,F)$, ${\mathcal L}(E)$, 11
  \item ${\mathcal L}(E,F)^+$, 32
  \item ${\mathcal L}^n(E_1,\dots,E_n; F)$, 11
  \item ${\mathcal L}_b(E,F)$, ${\mathcal L}_r(E,F)$, 33

 \indexspace

  \item $M(K)$, $M_\R(K)$, 22, 99
    \item $M_c(K)$, $M_d(K)$, 22
  \item $M_\alpha$, 12, 41, 110
  \item $M(G)$, 23
  \item $m_\sigma$, 129
  \item ${\mathcal {MB}}(E)$, 118
  \item ${\mathcal M}(E,F)$, 120--123
  \item ${\mathcal M}(E)$, 120
  \item ${\mathbb M}_{\,m,n}$, ${\mathbb M}_{\,n}$, 11, 21, 73
  \item ${\mathbb M}_{\,m,n}(E)$, ${\mathbb M}_{\,n}(E)$, 11
  \item $\mu_{p,n}$, 65--66
      \item $\mu_{1,n}$, 66

 \indexspace

  \item ${\mathcal N}(E,F)$, ${\mathcal N}(E)$, 14, 123, 127
 \item $\N$, $\N_n$, 7
   \item $\nu(T)$, 14
  \item $\nu(a)$, 23
  \item $\nu_o(a)$, 35

  \indexspace

  \item $\olim_{\alpha} x_\alpha$, 27  

\indexspace

  \item $P_S$, $Q_S$, 11, 110
  \item $P_i$, $Q_i$, 11
  \item $P_v$, 32
   \item $p_K$, 10
   \item  $\Pi(E)$, 25
      \item $(\Pi_{q,p}(E,F), \pi_{q,p})$, $(\Pi_p(E,F), \pi_p)$, 69
  \item $\pi_2(E)$, $\pi_1 (E)$, $\pi_p(E)$, 69
  \item  $\pi_{q,p}^{(n)}(T)$, $\pi_{q,p}^{(n)}(E)$, 68
  \item $\pi_{p}^{(n)}(T)$, $\pi_{p}^{(n)}(E)$, 68
    \item $\overline{\pi}_1^{(n)}(E)$, 70
  
    \indexspace 
   \item $\R$,  $\R^+$,  7
   \item  $\R^S$, 10

 \indexspace

   \item $\lv S \rv$, 7
 \item $S^{\perp}$, 29
  \item $S_E$, 12
 \item $S(A)$, 24
\item $S \leq T$, 32
 \item $S \perp T$, 29
\item $S\Delta T$, 7
  \item $S\otimes T$, 15
    \item $\Sup x$, 113
      \item $s\vee t$, $s\wedge t$, 26
      \item $\bigvee S$, $\bigwedge S$, 26
 \item ${\mathfrak S}_n$, ${\mathfrak S}_{\N}$, 7
     \item $\sigma(E,E')$, $\sigma(E',E)$, $\sigma(E'',E')$, 13
  \item $\sigma(a)$, 23
  \item $\sigma_o(T)$, 35, 134

\indexspace 

  \item $T^{(n)}$, 12
  \item $T_{\C}$, 11
 
 \indexspace
 
   \item $V(A,a)$, 25
  \item $V(T)$, 25
    \item $v\oplus w$, 11

\indexspace

  \item $X^{\circ}$, 13
    \item $[x,y]$, 9, 18
  \item $[x]$, 7
    \item $x\amalg y$, $x\amalg_k x$, 110
  \item $x\perp y$, 19, 27
  \item $x^+$, $x^-$, $\lv x \rv$, 27
  \item $x^t$, 11
  \item $x^{[n]}$, 11
  \item $x_\alpha \downarrow x$, $x_\alpha \uparrow x$, 26

\indexspace

  \item $y_0\otimes \lambda_0$, 13
  
  \indexspace

  \item $z\perp w$, 29
 \item $Z_i$, 7
   \item   $\Z$, $\Z^+$, $\Z^+_n$, 7

\indexspace 

  \item $\lV a : \ell_n^{\,p}\to \ell_m^{\,q}\rV$, 21
    \item $\lV T: E\to F \rV$, 13
     \item $\lV T \rV_{mb}$, 121
 \item $\lV (T_1,\dots, T_n)\rV_n^{mb}$, 124
     \item $\norm_\pi$, $\norm_\varepsilon$,  16
      \item $\norm_b$, 33, 35, 130
        \item $\norm_r$, 33, 35
        \item $\norm_p^{weak}$, $\norm_p^w$, 65
        \item $(\norm_k : k\in \N_n)$, $(\norm_k : k\in \N)$, 41
        \item $\norm_n^{\max}$, 63, 74, 108
  \item $\norm_n^{\min}$, 61, 108
        \item $\norm^{(p,q)}_n$, 85, 86, 109
    \item $\norm_n^{[q]}$, 93, 99, 102, 109
  \item $\lV \,\cdot\,\rV^H_n$, 91, 109
  \item $\norm^L_n$, $\norm^{DL}_n$, 105, 109, 119
   \item $\norm_n^{mb}$, 125, 135 
\item $\norm_n^{\mathcal K}$, 153, 155
\item $\norm_{n, {\mathcal K}}^{\dag}$, 157
  \item $\norm_{n,{\mathcal K}}^{\dag\dag}$, 160
\item $\Norm^{(r,s)}_n$, 89
 \item $( \norm^1_k : k\in\N)\leq  ( \norm^2_k : k\in\N)$, 48
\item $( \norm^1_k : k\in\N)\preccurlyeq ( \norm^2_k : k\in\N)$, 48
 \item $( \norm^1_k : k\in\N)\cong   ( \norm^2_k : k\in\N)$, 48
\end{itemize}
\end{multicols}

\end{document}